\documentclass[a4paper,12pt,english]{amsart}
\usepackage{babel}
\usepackage[utf8]{inputenc}
\usepackage[T1]{fontenc}
\usepackage{amssymb,amsfonts,amsmath,amsthm}
\usepackage{mathdots}
\usepackage{color}
\usepackage{graphicx}
\usepackage[stretch=16,shrink=16,step=2,kerning=true,protrusion=true,final]{microtype} 
\usepackage{dynkin-diagrams}
\usepackage[colorlinks=true, linkcolor=blue, citecolor=red, urlcolor=blue]{hyperref}
\usepackage{tikz}

\usepackage{mathtools, stmaryrd}
\usepackage{xparse, etoolbox}
\usepackage{enumitem}

\calclayout

\newtheorem{theorem}{Theorem}
\newtheorem{proposition}[theorem]{Proposition}
\newtheorem{corollary}[theorem]{Corollary}
\newtheorem{lemma}[theorem]{Lemma}

\newtheorem{theorem_intro}[theorem]{Theorem}

\theoremstyle{definition}
\newtheorem{definition}[theorem]{Definition}

\newtheorem{notation}[theorem]{Notation}

\theoremstyle{remark}
\newtheorem{remark}[theorem]{Remark}
\newtheorem{remarks}[theorem]{Remarks}
\newtheorem{example}[theorem]{Example}
\newtheorem{examples}[theorem]{Examples}

\allowdisplaybreaks 

\newcommand{\ie}{i.e.\ }
\newcommand{\eg}{e.g.\ }

\newcommand{\PP}{\mathbf{P}}
\newcommand{\CC}{\mathbf{C}}

\newcommand{\RR}{\mathbf{R}}

\newcommand{\ZZ}{\mathbf{Z}}
\newcommand{\NN}{\mathbf{N}}

\newcommand{\C}{\mathbf{C}}
\newcommand{\R}{\mathbf{R}}

\newcommand{\Z}{\mathbf{Z}}
\newcommand{\N}{\mathbf{N}}
\newcommand{\K}{\mathbf{K}}

\newcommand{\GL}{\operatorname{GL}}
\newcommand{\SL}{\operatorname{SL}}
\newcommand{\PSL}{\operatorname{PSL}}
\newcommand{\SO}{\operatorname{SO}}

\newcommand{\OO}{\operatorname{O}}
\newcommand{\SU}{\operatorname{SU}}
\newcommand{\U}{\operatorname{U}}
\newcommand{\Sp}{\operatorname{Sp}}

\newcommand{\Sym}{\operatorname{Sym}}
\newcommand{\Herm}{\operatorname{Herm}}
\newcommand{\Aut}{\operatorname{Aut}}
\newcommand{\Inn}{\operatorname{Inn}}
\newcommand{\lietype}[1]{\mathrm{#1}}

\DeclareMathOperator{\Lie}{Lie}
\DeclareMathOperator{\Span}{Span}

\newcommand{\coloneqq}{\mathrel{\mathop:}=}

\newcommand{\g}{\mathfrak{g}}

\renewcommand{\u}{\mathfrak{u}}

\newcommand{\so}{\operatorname{\mathfrak{so}}}

\newcommand{\HH}{\mathbf{H}}
\newcommand{\Oc}{\mathbf{O}}
\newcommand{\Ad}{\operatorname{Ad}}
\newcommand{\End}{\operatorname{End}}
\newcommand{\ad}{\operatorname{ad}}
\newcommand{\id}{\operatorname{Id}}
\newcommand{\alphab}{{\boldsymbol{\alpha{}}}}
\newcommand{\gammab}{{\boldsymbol{\gamma{}}}}
\newcommand{\varepsilonb}{{\boldsymbol{\varepsilon{}}}}
\newcommand{\vb}{{\boldsymbol{v{}}}}

\def\setminus{\smallsetminus}

\DeclarePairedDelimiterX\scalK[1](){\sprargs{#1}}%
\NewDocumentCommand{\sprargs}{>{\SplitArgument{1}{,}}m }
 {\sprargsaux#1}
\NewDocumentCommand{\sprargsaux}{ m m }%
{\ifblank{#1}{\cdot}{#1}\nonscript\delimsize\vert\nonscript\mathopen{}\ifblank{#2}{\cdot}{#2}}

\title{Generalizing Lusztig's total positivity}

\author[O. Guichard]{Olivier Guichard}
\address{Universit\'e de Strasbourg, IRMA, 7 rue  Descartes, 67000 Strasbourg, France}
\email{olivier.guichard@math.unistra.fr}

\author[A. Wienhard]{Anna Wienhard}
\address{Max Planck Institute for Mathematics in the Sciences, Inselstr. 22, 04103 Leipzig
Germany}
\email{wienhard@mis.mpg.de}

\thanks{OG thanks the Institut Universitaire de France, he was partially supported by the Agence Nationale de la Recherche under the grant DynGeo (ANR-16-CE40-0025).
AW was partially supported by the NSF under agreement DMS-1065919, by the Sloan Foundation, by the Deutsche Forschungsgemeinschaft, by the European Research Council under ERC-Consolidator grant 614733 and under ERC-Advanced Grant 101018839,  by the Klaus Tschira Foundation, and by the National Science Foundation under Grant No. 1440140 and the Clay Foundation, while she was in residence at the Mathematical Sciences Research Institute in Berkeley, California, during the semester of fall 2019.
AW thanks the Hector Fellow Academy for support. 
Both authors acknowledge support from U.S. National Science Foundation grants DMS 1107452, 1107263, 1107367 ``RNMS: GEometric structures And Representation varieties'' (the GEAR Network).
This work was supported by the Deutsche Forschungsgemeinschaft under Germany's Excellence Strategy EXC-2181/1 - 390900948 (the Heidelberg STRUCTURES Cluster of Excellence).}

\AtBeginDocument{%
   \def\MR#1{}
}

\begin{document}

\numberwithin{theorem}{section} \numberwithin{equation}{section}

\begin{abstract}
  We introduce the notion of $\Theta$-positivity in real semisimple Lie
  groups. This notion at the same time generalizes Lusztig's total positivity
  in split real Lie groups and invariant orders in Lie groups of Hermitian
  type.  We show that there are four families of simple Lie groups which admit
  $\Theta$-positive structures and investigate fundamental properties of
  $\Theta$-positivity.  We define and describe the positive and nonnegative
  unipotent semigroups and show that they give rise to a notion of positive
  $n$-tuples in flag varieties.
\end{abstract}

\maketitle

\tableofcontents{}

\section{Introduction}
The theory of totally positive matrices arose in the beginning of the 20th
century through work of Schoenberg \cite{Schoenberg} and Gantmacher and Krein
\cite{Gantmacher_Krein}. Total positivity has since become an important
concept in several mathematical fields.  In the 1990's the theory has been
generalized widely by Lusztig \cite{Lusztig} who introduced the total positive
semigroup of a general split real semisimple Lie group. Lusztig's total
positivity plays an important role in representation theory, cluster algebras,
and has many relations to other areas in mathematics as well as in theoretical
physics.  In this article we introduce a generalization of Lusztig's total
positivity in the context of real semisimple Lie groups~$G$ that are not
necessarily split. We call this generalization \emph{$\Theta$-positivity} because it
depends on the choice of a flag variety associated with~$G$ which is
determined by a subset~$\Theta$ of the set of simple roots. The notion of
$\Theta$-positivity generalizes at the same time Lusztig's total positivity
(when the group is split and the flag variety is the complete flag variety) as
well as Lie semigroups of Lie groups of Hermitian type, which are related to
bi-invariant orders and causal structures (and the flag variety is the Shilov
boundary).  We give a classification of semisimple Lie groups admitting a
$\Theta$-positive structure. Besides split real Lie groups and Hermitian Lie
groups of tube type, two further families of simple Lie groups admit a
$\Theta$-positive structure: the groups locally isomorphic to indefinite
orthogonal groups $\SO(p,q)$, with $2\leq p <q$, and an exceptional family,
whose system of restricted roots is of type $\lietype{F}_4$ (cf.\
Theorem~\ref{thm_intro:classification} below). To our knowledge, for these two
families no positive structure was known before.

Our interest in $\Theta$-positivity arose from higher rank Teichm\"uller
theory, in particular from trying to find a unifying framework that explains
the similarities between Hitchin representations and maximal
representations. The notion of $\Theta$-positivity provides such a unifying
systematic framework. It also leads to several conjectures regarding higher
rank Teichm\"uller spaces, some of which have been formulated in \cite{GW_ECM}
and \cite{wienhardicm}, and partly proven in \cite{Collier, ABCGGO,
  Bradlow_Collier_etal}, \cite{GLW}, \cite{BP}, and \cite{BGLPW}.

However, the reach of $\Theta$-positivity goes far beyond higher rank
Teichm\"uller theory. For example, as we will discuss in a bit more
detail below, $\Theta$-positivity suggests that Hermitian Lie
groups of tube type should be considered as groups of type $\lietype{A}_1$ over
noncommutative algebras. Something, which has been made precise in
\cite{ABRRW} for most classical groups.

Before we describe the notion of $\Theta$-positivity and our results in more
detail, let us note that we described the notion of $\Theta$-positivity
several years ago in a survey paper \cite{GW_ECM} with only few proofs and a
hands on description of the notion for indefinite orthogonal groups
$\SO(p,q)$, $2 \leq p<q$. Several of the properties we described there have
been already used by other people.  This article now finally gives the
foundation of $\Theta$-positivity.  We focus here on the structure of
unipotent subgroups in~$G$ and on $\Theta$-positivity in generalized flag
varieties. In particular, we provide all the background needed to introduce
positive configurations in flag varieties, positive maps or positive
representations of surface groups. We give all the background and prove the
results used in \cite{Collier,ABCGGO, Bradlow_Collier_etal}, in
\cite{PSW2,BP}, as well as in \cite{GLW} and \cite{BGLPW}. Ideas from this
paper and \cite{GLW} were also used in \cite{BCL}.  In a forthcoming second
foundational article we will focus on the finer structure of
$\Theta$-positivity, including braid relations, Gau\ss{} decomposition
theorems, and geometric properties of positive elements in~$G$.

We now describe the results of the paper in more detail. 

\subsection{\texorpdfstring{$\Theta$}{Θ}-positive structures} 
Totally positive matrices are defined by requiring all minors  of the matrix to be
positive. Anne Whitney showed in \cite{Whitney} that the semigroup of totally
positive, or more general totally nonnegative, matrices can be generated by
an explicit set of simple matrices. This reduction theorem allowed the
generalization of total positivity to all semisimple split real Lie groups by Lusztig 
\cite{Lusztig}. The key building block in Lusztig's approach is the
construction of a closed nonnegative semigroup~$U^{\geq 0}$ in the unipotent
radical~$U$ of the Borel subgroup~$P$ of a split real Lie group~$G$. This
semigroup is constructed as follows.  The Lie algebra~$\mathfrak{u}$ can be
described as the sum of the root spaces~$\mathfrak{g}_\alpha$ for all positive
roots~$\alpha$. Since $G$~is split, any root space~$\mathfrak{g}_\alpha$ is of
dimension~$1$. Choosing an element~$X_\alpha$ in~$ \mathfrak{g}_\alpha$ it can hence be identified with~$\R$. Given the set of simple positive
roots~$\Delta$, consider the map
$x_\alpha\colon \mathfrak{g}_\alpha\cong \R \to U$,
$s \mapsto \exp(s X_\alpha)$. Lusztig's nonnegative semigroup $U^{\geq 0}$ is
then the semigroup generated by $x_\alpha(\R_{\geq0})$, $\alpha \in \Delta$.

Turning to our situation, when $G$ is a semisimple real Lie group, which is not
necessarily split, we consider a subset $\Theta\subset \Delta$ of the set of
simple roots. This defines a standard parabolic subgroup $P_\Theta$ in $G$ and
a flag variety $\mathsf{F}_\Theta = G/P_\Theta$. (The convention is 
that $P_\Delta$ is the minimal parabolic subgroup of~$G$.) The standard
opposite parabolic subgroup $P_{\Theta}^{\mathrm{opp}}$ is such that
$L_\Theta= P_\Theta \cap P_{\Theta}^{\mathrm{opp}}$ is a Levi factor and
the group~$P_\Theta$ is the semidirect product of its
unipotent radical~$U_\Theta$ and the Lie group~$L_\Theta$.  The Lie algebra $\mathfrak{u}_\Theta$ of
$U_\Theta$ carries the adjoint action by $L_\Theta$ and can hence be
decomposed into its $L_\Theta$-irreducible pieces. For every simple root
$\alpha \in \Theta$ there is an irreducible piece $\mathfrak{u}_\alpha$, which
in general is not of dimension~$1$. When $G$~is split and $\Delta = \Theta$
these are precisely the root spaces~$\mathfrak{g}_\alpha$.

We say that $G$ admits a $\Theta$-positive structure if it satisfies the
following property: For every $\alpha \in \Theta$ there exists an acute convex
cone $c_\alpha \subset \mathfrak{u}_\alpha$ that is invariant under the action of the connected component~$L_\Theta^\circ$ of~$L_
\Theta$.

In Section~\ref{sec:theta-posit-st}, we first show that the notion of
$\Theta$-positivity defined for semisimple groups reduces to the case of
simple groups (i.e.\ every simple factor must have a positive structure).
We then prove  
\begin{theorem_intro}[see Theorem~\ref{thm:classification}]
  \label{thm_intro:classification}
  A simple Lie group $G$ admits a $\Theta$-positive structure if and only if
  the pair $(G,\Theta)$ belongs to the following list:
   \begin{enumerate}
   \item $G$ is a split real form, $\Theta = \Delta$, and $\mathsf{F}_\Theta$
     is the complete flag variety of~$G$;
   \item $G$ is Hermitian of tube type and of real rank~$r$,
     $\Theta = \{ \alpha_r\}$, where $\alpha_r$ is the long simple restricted
     root, and $\mathsf{F}_\Theta$ is the Shilov boundary of the symmetric
     space associated with~$G$;
   \item $G$ is locally isomorphic to $\SO(p+1,p+k)$, $p>1$, $k>1$,
     $\Theta = \{ \alpha_1, \dots, \alpha_{p}\}$, where
     $\alpha_1, \dots, \alpha_{p}$ are the long simple restricted roots, and
     $\mathsf{F}_\Theta$ is the variety of flags $(F_1, \dots, F_p)$ with
     $F_i\subset \R^{2p+k+1}$ isotropic subspace of dimension~$i$ and
     $F_{i}\subset F_{i+1}$;
   \item $G$ is the real form of $\lietype{F}_4$, $\lietype{E}_6$, $\lietype{E}_7$, or of $\lietype{E}_8$ whose system
     of restricted roots is of type~$\lietype{F}_4$, and
     $\Theta = \{ \alpha_1, \alpha_2\}$, where $\alpha_1, \alpha_{2}$ are the
     long simple restricted roots.
  \end{enumerate}
\end{theorem_intro}

The cones $c_\alpha \subset \mathfrak{u}_\alpha$, $\alpha \in \Theta$, allow
to define the nonnegative semigroup $U_\Theta^{\geq 0}$ to be the
subsemigroup of $U_\Theta$ generated by the elements $\exp(v)$, with $v \in c_\alpha$ for some $\alpha \in \Theta$.

When $G$ is a split real Lie group which also belongs to one of the other
three families in Theorem~\ref{thm_intro:classification}, then $G$~carries two
different positive structures, one with respect to~$\Delta$ and one with
respect to a proper subset $\Theta \subset \Delta$.

Theorem~\ref{thm_intro:classification} takes a more algebraic approach to
$\Theta$-positivity, but we show further that the existence of a $\Theta$-positive structure can be characterized geometrically in terms of positive triples in the associated flag varieties $\mathsf{F}_\Theta$ (see Section~\ref{sec:posit-flag-vari}). For this it is important to prove that essentially any $L_{\Theta}^{\circ}$-invariant semigroup in $U_\Theta$ comes from a $\Theta$-positive structure. 

To give a precise statement, let us make a few observations. 
To every root system there is a corresponding \emph{opposition involution}
from~$\Delta$ to~$\Delta$, but which might not preserve~$\Theta$. However in
Theorem~\ref{thm_intro:classification}, whenever $\Theta \neq \Delta$, the
opposition involution is just the identity on~$\Delta$ and thus preserves
$\Theta$. This implies that in all cases $P_\Theta$ is conjugate to its
opposite parabolic subgroup $P_\Theta^{\mathrm{opp}}$, and the flag varieties
defined by $P_\Theta$ and by $P_\Theta^{\mathrm{opp}}$ coincide. 
Any parabolic subgroup $P$ of $G$ defines a Bruhat decomposition (see
Section~\ref{sec:bruhat-decomposition}), which is the orbit decomposition
under the action 
of $P\times P$, and that admits a unique open orbit. We denote the open Bruhat orbit with respect to $P_\Theta^{\mathrm{opp}} \times P_\Theta^{\mathrm{opp}}$ by  $\Omega_{\Theta}^{\mathrm{opp}}$, it can be described as  $\Omega_{\Theta}^{\mathrm{opp}} = \{
  g\in G\mid g P_{\Theta}^{\mathrm{opp}} g^{-1}\cap U_{\Theta}^{\mathrm{opp}}
  \text{ is trivial}\}$ ($U_{\Theta}^{\mathrm{opp}}$ is the unipotent radical of~$P_{\Theta}^{\mathrm{opp}}$).

\begin{theorem_intro}[see Theorem~\ref{theo:invar-unip-semigr-implies-pos} and
  Corollary~\ref{coro:semigroup-in-Bruhat-Theta}]
  \label{thmintro:unipotent_semigroup_implies_positivity}
  Let~$G$ be a connected simple Lie group.  Suppose that there is
  $U^{+}\subset U_\Theta$ such that $U^{+}$ is a closed
  $L_{\Theta}^{\circ}$-invariant semigroup of nonempty interior, which
  contains no nontrivial invertible element.  Then $G$ admits a
  $\Theta$-positive structure and the semigroup~$U^+$ contains the
  semi\-group~$U_{\Theta}^{\geq 0}$.
  
  If we further assume that the interior of $U^+$ is contained in the open
  Bruhat cell $\Omega_{\Theta}^{\mathrm{opp}}$ with respect to $P_\Theta^{\mathrm{opp}}$, then
  $U^+ = U_{\Theta}^{\geq 0}$.
\end{theorem_intro}

\subsection{The unipotent positive semigroup  and the
  \texorpdfstring{$\Theta$}{Θ}-Weyl group}
Whereas the definition of the nonnegative unipotent semigroup  in
Lusztig's and in our work is straight forward, the understanding of the positive
unipotent semigroup is more delicate. In Lusztig's work it involves the Weyl
group~$W$. Recall that the Weyl group is generated by the elements $s_\alpha$,
$\alpha \in \Delta$. There is a unique longest length element
$w_\Delta \in W$. To construct the positive unipotent semigroup $U^{>0}$ one 
fixes a reduced expression $s_{\alpha_{1}} \cdots s_{\alpha_{k}}$
of~$w_\Delta$. Making this choice Lusztig defines a map
$F\colon \R^k \cong \mathfrak{g}_{\alpha_{1}} \times \cdots \times
\mathfrak{g}_{\alpha_{k}} \to U$,
$ (x_1,\dots, x_k) \mapsto \exp(x_1) \cdots \exp(x_k)$.  The image of
$\R_{>0}^k $ under this map~$F$ is the positive unipotent semigroup. A key
point is of course to show that this image is independent of the choice of the
reduced expression. In \cite{Lusztig} this is done by first a reduction to the
simply-laced cases and then by giving explicit formulas for the change of
coordinates given by a braid relation, i.e. exchanging  $s_i s_{i+1}s_i =
s_{i+1} s_i s_{i+1}$ in the reduced expression. The change of coordinates is
given by  positive rational functions. (Berenstein and Zelevinsky
\cite{Berenstein_Zelevinsky} give explicit formulas for the braid relations in
the case of non simply-laced Dynkin diagrams.) Since any two reduced
expressions of $w_\Delta$ are related by a sequence of braid relations, this
thus proves that the image $F(\R_{>0}^{k})$ is independent of the chosen
reduced expression. This image is the positive unipotent semigroup~$U^{>0}$
and the map~$F$ thus gives a parametrization of $U^{>0}$. One can then conclude that the semigroup~$U^{>0}$ is in
fact the interior of~$U^{\geq 0}$.

Here we follow a slightly different approach. We define the nonnegative semigroup $U_{\Theta}^{\geq 0}$ to be the
semigroup, in the unipotent radical~$U_\Theta$ of~$P_\Theta$, generated by
$\bigcup_{\alpha\in \Theta} \exp( c_\alpha)$; we then define
the positive unipotent semigroup $U_{\Theta}^{>0}$ to be the interior
(relative to~$U_\Theta$) of the nonnegative semigroup $U_{\Theta}^{\geq 0}$.

We then give explicit parametrizations of the semigroup~$U_{\Theta}^{>0}$. For
this, the relevant object to control the combinatorial data underlying these
parametrizations is the $\Theta$-Weyl group. The $\Theta$-Weyl group
$W(\Theta)$ is a subgroup of the Weyl group~$W$ of~$G$ defined by an explicit
generating set $R(\Theta) =\{ \sigma_\alpha\}_{\alpha\in \Theta}$ (see
Section~\ref{sec:theta-weyl-group}) obtained by choosing shortest length
representatives for some cosets of the subgroup 
$W_{\Delta\smallsetminus \Theta}$ corresponding to the Levi factor
of~$P_\Theta$ in~$W$. The $\Theta$-Weyl group is isomorphic to the Weyl group of a different root system. This is 
 a priori be a bit surprising, but  indeed there is even a natural embedding of a split real Lie group of
type $W(\Theta)$ into~$G$.  The $\Theta$-Weyl group $W(\Theta)$ then plays the
same role as the Weyl group in order to parametrize the positive unipotent
semigroup and give an alternative description of it.

\begin{theorem_intro}[see Section~\ref{sec:positive_semigroup}]
  \label{thmintro:positive_semigroup-param}
  Given any reduced expression $\sigma_{\gamma_1} \cdots \sigma_{\gamma_n}$ of
  the longest length element $w_{\max}^{\Theta} \in W(\Theta)$ and setting
  $c\coloneqq c_{\gamma_1} \times \cdots \times c_{\gamma_n}$, the
  map $$F\colon c \to U_\Theta,$$ defined as the product of the exponential
  map on each factor, has the following properties:
  \begin{enumerate}
  \item $F$~is proper;
  \item $F|_{\mathring{c}}$ is open and injective;
  \item $F(\mathring{c})$ is 
    a connected component of
    $\Omega_\Theta^{\mathrm{opp}} \cap U_\Theta$ (as above
    $\Omega_\Theta^{\mathrm{opp}}$ is the open Bruhat cell
    with respect to $P_\Theta^{\mathrm{opp}}$);
  \item The image $F(\mathring{c})$ is independent of the reduced expression
    of $w^{\Theta}_{\max}$ and is equal to the positive unipotent
    semigroup~$U_{\Theta}^{>0}$.
  \end{enumerate}
\end{theorem_intro}

Further properties of the unipotent semigroups follow from this: 

\begin{theorem_intro}[see Corollary~\ref{cor:semigroupUtheta}]
  \label{thmintro:positive_semigroup-properties}
  The positive and nonnegative semigroups have the following properties:
  \begin{enumerate}
  \item The semigroup~$U_{\Theta}^{>0}$ is invariant by conjugation by
    $L_\Theta^{\circ}$.
  \item The closure of $U_{\Theta}^{>0}$ is the nonnegative semigroup
    $U_\Theta^{\geq 0}$, in particular $U_{\Theta}^{\geq 0}$~is closed.
  \item For every reduced expression of~$w_{\max}^{\Theta}$, one has
    $U_{\Theta}^{\geq 0}=F(c)$ (with $c$ and $F$ as in
    Theorem~\ref{thmintro:positive_semigroup-param}).
  \item $U_{\Theta}^{\geq 0} U_{\Theta}^{>0} \subset U_{\Theta}^{>0}$ and
    $U_{\Theta}^{>0} U_{\Theta}^{\geq 0} \subset U_{\Theta}^{>0}$.
  \end{enumerate}
\end{theorem_intro}

The proofs of Theorem~\ref{thmintro:positive_semigroup-param} and
Theorem~\ref{thmintro:positive_semigroup-properties} rely on a fine
understanding of the image of the map $F$ relative to the Bruhat
decompositions of $G$ with respect to $P_{\Theta}^{\mathrm{opp}}$ and with respect to
$P_{\Delta}^{\mathrm{opp}}$.

Analogously we construct the nonnegative semigroup
$U_{\Theta}^{\mathrm{opp}, \geq 0}$ and the positive semigroup
$U_{\Theta}^{\mathrm{opp}, > 0}$ in~$U_{\Theta}^{\mathrm{opp}}$ and similar
statements hold for them.

Using this, the $\Theta$-nonnegative semigroup $G_{\Theta}^{\geq 0}$ is defined to be the
semigroup of~$G$ generated by $U_{\Theta}^{\mathrm{opp}, \geq 0}$,
$U_{\Theta}^{\geq 0}$, and $L_\Theta^\circ$, and the $\Theta$-positive
semigroup $G_{\Theta}^{>0}$ is defined to be its interior. These semigroups,
and the geometric properties of their elements will be analyzed in a
forthcoming article.  

\begin{remark}
  Note that to prove these results, we do not need to establish the explicit
  change of coordinates with respect to modifying the reduced expression of
  $w_{\max}^{\Theta}$ by a braid relation. However, to get a finer
  understanding, it is in fact interesting to write down explicit braid
  relations. In this paper we will only shortly discuss the braid relations in
  the case that $G = \mathrm{SO}(p,q)$ derived from explicit matrix
  equations. The more general treatment of braid relations involves
  $\Theta$-versions of universal enveloping algebras, adapting the strategy
  in~\cite{Berenstein_Zelevinsky}. This will be appear in \cite{GW_pos_2}.
\end{remark}

\subsection{Positivity in flag varieties}
We now turn our attention to the flag variety
$ \mathsf{F}_\Theta \cong G/P_\Theta$. There is a unique flag
$f_\Theta \in \mathsf{F}_\Theta$ fixed by~$P_\Theta$ and a unique flag
$f^{\mathrm{opp}}_{\Theta} \in \mathsf{F}_\Theta$ fixed by the opposite
parabolic subgroup~$P_{\Theta}^{\mathrm{opp}}$ (recall that we are in a
situation where $G/P_\Theta \cong G/P_{\Theta}^\mathrm{opp}$). It is well known that
$U_\Theta$~acts simply transitively on the set of flags that are transverse
to~$f_\Theta$. The positive unipotent semigroup~$U_{\Theta}^{>0}$ allows us to
introduce the notion of positive $n$-tuples of flags in~$\mathsf{F}_\Theta$.
In order to have a more intrinsic notion, we will consider all the possible
such subsemigroups, i.e.\ the subsemigroups $U_{\Theta}^{\varepsilonb>0}$ where
$\varepsilonb$~varies in $\{\pm 1\}^\Theta$ and $U_{\Theta}^{\varepsilonb>0}$
is the image of the map~$F$ of Theorem~\ref{thmintro:positive_semigroup-param} where the factors~$\mathring{c}_\alpha$ are replaced by
$-\mathring{c}_\alpha$ when the $\alpha$-component of~$\varepsilonb$ is equal
to~$-1$. (Alternatively $U_{\Theta}^{\varepsilonb\geq 0}$ is the semigroup
generated by
$\exp\bigl( \bigcup_{\alpha\in \Theta} \varepsilon_\alpha c_\alpha\bigr)$ and
$U_{\Theta}^{\varepsilonb>0}$ is its interior.)

An $n$-tuple of flags $(f_0, f_1, \dots, f_{n-2}, f_\infty)$ is said to be
\emph{$\Theta$-positive} if there exist an element $g \in G$, $\varepsilonb$
in $\{\pm 1\}^\Theta$, and $u_1, \dots, u_{n-2} \in U_\Theta^{\varepsilonb>0}$
such that
$(f_0, f_1, \dots, f_{n-2}, f_\infty) = g \cdot (f^{\mathrm{opp}}_\Theta,
u_1\cdot f^{\mathrm{opp}}_{\Theta}, \dots, u_i \cdots u_1 \cdot
f_{\Theta}^{\mathrm{opp}}, \dots, u_{n-2} \cdots u_1 \cdot
f_{\Theta}^{\mathrm{opp}} , f_{\Theta})$.

A consequence of Theorem~\ref{thmintro:positive_semigroup-param} is 
\begin{theorem_intro}
  The set of flags $U_{\Theta}^{>0}\cdot f_{\Theta}^{\mathrm{opp}}$ is a
  connected component of the set of flags that are transverse to~$f_\Theta$
  and to~$f_{\Theta}^{\mathrm{opp}}$.  The stabilizer in~$G$ of a
  $\Theta$-positive triple is compact.
\end{theorem_intro}

Note that the set of $\Theta$-positive triples will in general not be
connected, but can consist of several connected components of the set of
triples of pairwise transverse flags.
    
We establish crucial properties of positive $n$-tuples in the flag
variety~$\mathsf{F}_\Theta$ in Section~\ref{sec:posit-flag-vari}. For this we
use the notion of diamonds, which turns out to be very helpful to keep track of
positivity of triples and more general $n$-tuples of flags. Given two points $f_0, f_1 \in \mathsf{F}_\Theta$, 
a diamond with extremities  $f_0, f_1$ is a (special) connected component of the set of flags that are transverse to $f_0$ and $f_1$. Diamonds should be thought of 
as generalized intervals. In the case of total positivity they were introduced (without using this 
terminology) in \cite{BCL}. The term ``diamond'' 
was coined in \cite{Labourie_Toulisse} in the case of $\SO(2,n)$, where 
their basic metric properties were established. For the case of
$\Theta$-positivity, diamonds were introduced and studied in
\cite{GLW}. Diamonds give a geometric characterization 
of $\Theta$-positivity. In 
Section~\ref{sec:diamonds-semigroups} we show that the existence of a family of diamonds in the flag variety ~$\mathsf{F}_\Theta$ is equivalent to the existence of a $\Theta$-positive structure. 
  This geometric characterization is used in \cite{GLW} to
investigate positive representations.

We already noticed that the Lie groups $\SO(n,n+1)$, $\Sp(2n, \R)$,
or the  split real form of~$\lietype{F}_4$ admit two $\Theta$-positive structures, one for $\Theta = \Delta$ and
one for~$\Theta$ being a strict subset of~$\Delta$. We prove that the natural
projection $\mathsf{F}_\Delta\to \mathsf{F}_\Theta$ maps positive $n$-tuples
to positive $n$-tuples (Section~\ref{sec:comp-posit-struct}).

\subsection{The \texorpdfstring{$\Theta$}{Θ}-principal $\mathfrak{sl}_{2}$ and positive circles}
It is well known that the Lie algebra of a split real Lie group contains a
distinguished conjugacy class of $3$-dimensional simple subalgebras, the
principal $\mathfrak{sl}_{2}$. This principal $\mathfrak{sl}_{2}$ plays an
important role for the moduli space of Higgs bundles, and in particular for
the Hitchin section and consequently the Hitchin component \cite{Hitchin}.
When $G$ admits a $\Theta$-positive structure, we introduce a $\Theta$-system,
a family in the Lie algebra consisting of $\mathfrak{sl}_2$-triples associated
with nilpotent elements in the cone~$\mathring{c}_\alpha$, for each
$ \alpha \in \Theta$. This $\Theta$-system generates a split real subalgebra
$\mathfrak{h}_\Theta$ in the Lie algebra $\mathfrak{g}$ of $G$, as well as a
principal $3$-dimensional simple subalgebra in $\mathfrak{h}_\Theta$. We call
this special $\mathfrak{sl}_{2}$ the $\Theta$-principal $\mathfrak{sl}_{2}$.
This $\Theta$-principal $\mathfrak{sl}_{2}$ gives rise to a positive circle in
the flag variety $\mathsf{F}_\Theta$. 
The $\Theta$-principal $\mathfrak{sl}_{2}$ is related to magical $\mathfrak{sl}_{2}$-triples introduced in  \cite{Bradlow_Collier_etal}. 
 The 
explicit relation is stated in Section~\ref{sec:sl2}.

\subsection[Positive representations and higher rank Teichm\"uller spaces]{Positive representations and higher rank Teichm\"uller\\ spaces} 
Lusztig's total positivity on the one hand and positive semigroups in
Hermitian Lie groups of tube type on the other play an important role in
higher rank Teichm\"uller theory. Hitchin components and spaces of maximal
representations have been characterized as positive representations
\cite{Labourie_anosov, Guichard_hyper, Fock_Goncharov,
  Burger_Iozzi_Labourie_Wienhard, Burger_Iozzi_Wienhard_ann}. At the same
time, total positivity plays an important role in the construction of cluster
coordinates by Fock and Goncharov \cite{Fock_Goncharov}.

As $\Theta$-positivity realizes total positivity in split real Lie groups and
Lie semigroups in Hermitian Lie groups of tube type as two incarnations of the
same concept, it provides an interesting systematic framework for higher rank
Teichm\"uller spaces. This is based on the notion of positive $n$-tuples in
the flag variety $ \mathsf{F}_\Theta$, which allows us to introduce the notion
of positive representations.

Let $\pi_1(S)$ be the fundamental group of an oriented surface of negative
Euler characteristic. Then the boundary $\partial_\infty \pi_1(S)$ carries a
natural cyclic ordering. Given a simple Lie group $G$ with a $\Theta$-positive
structure, we call a representation $\rho\colon \pi_1(S) \to G$
\emph{$\Theta$-positive} if there exists a $\rho$-equivariant map from
$\partial_\infty \pi_1(S)$ to the flag variety $ \mathsf{F}_\Theta$, which
sends every cyclically ordered $n$-tuple of points in $\partial_\infty \pi_1(S)$ to a
positive $n$-tuple of flags in~$\mathsf{F}_\Theta$.

In \cite{GW_ECM} we conjectured that the set of positive representations forms
higher rank Teichm\"uller spaces, i.e.\ it is open and closed and consists
entirely of discrete and faithful representations.  Using several of the
results obtained in this paper, this conjecture has been proven to a large
extent in \cite{GLW} and \cite{BP}, relying partly on \cite{Collier,
  Bradlow_Collier_etal}, and is fully proven in \cite{BGLPW}. 
  
However, $\Theta$-positivity suggests an even deeper connection between the
different families of higher rank Teichm\"uller spaces, including extensions
of Fock--Goncharov's construction of cluster coordinates to spaces of positive
representations.  The expectation is that there are cluster coordinates for positive
representations into the Lie group $G$ with respect to its $\Theta$-positive
structure provide examples of noncommutative cluster algebras of type
$W(\Theta)$. This has been proven for the symplectic group (as a Hermitian Lie
group of tube type) in \cite{AGRW}, where appropriate noncommutative
coordinates give a geometric realization of the noncommutative cluster
algebra of type~$\lietype{A}_1$ introduced in \cite{br}, see also
\cite{KR-FramedLocalSystems}. Moreover in \cite{ABRRW} classical Hermitian Lie
groups of tube type are realized as groups of type $\lietype{A}_1$ over a
noncommutative ring.

For further conjectures and questions regarding $\Theta$-positivity, we refer
the reader to \cite{wienhardicm}.

\subsection{Structure of the paper}
The necessary background on semisimple Lie algebras, their roots systems and
their parabolic subgroups is recalled in
Section~\ref{sec:struct-parab-subgr}. There, we explain in particular how a
semisimple group decomposes into its simple factors and the subsequent
decompositions of the parabolic subgroups, the flag varieties, etc.

We introduce and classify $\Theta$-positive structures in
Section~\ref{sec:theta-posit-st}. Their first immediate properties are then
proved (the special root, the Hermitian tube type subgroup).

In Section~\ref{sec:nonn-posit-semigr} we introduce the unipotent nonnegative
and positive semigroups and state a few simple properties of semigroups in
topological groups.

In Section~\ref{sec:explicit-description} we describe the invariant cones that
enter in the construction of the semigroups, explicitly for the family of
classical groups and with a table for the $F_4$ cases. We construct
$\Theta$-systems and prove some commutation relations used in the sequel.

We show that every $\Theta$-system generates a split real subalgebra in
Section~\ref{sec:split-group-type} and show that it gives rise to a
$\Theta$-principal three-dimensional subalgebra in Section~\ref{sec:sl2}. 
Here we also show that the longest length element of the Weyl group of this split
subgroup belongs to the image of the $\Theta$-principal $\SL_2(\R)$.

In Section~\ref{sec:Weyl} we introduce the $\Theta$-Weyl group and show it is
isomorphic to the Weyl group of the split real subgroup. The $\Theta$-length
is introduced and we show that the $\Theta$-Weyl group normalizes the
Weyl group of the Levi factor of~$L_\Theta$

In Section~\ref{sec:bruh-decomp-cones}, relying on the elements in the
cones~$c_\alpha$ constructed in Section~\ref{sec:explicit-description}, we
establish important properties of the relations of the invariant cones and
Bruhat decompositions, which are used in Section~\ref{sec:positive_semigroup}
to give explicit parametrizations of the unipotent positive
semigroup. Theorem~\ref{thmintro:positive_semigroup-param} is proved there. We also determine
the tangent cone at the identity of the semigroup~$U_{\Theta}^{\geq 0}$. 

Section~\ref{sec:orthogonal} takes a particular look at the positive structure
for the family of indefinite orthogonal groups, including a parametrization of
the nonnegative unipotent semigroup for $\SO(3,q)$.

In Section~\ref{sec:invar-unip-semigr} we prove the relation of general
invariant unipotent semigroups with those arising from a $\Theta$-positive
structure, i.e.\ Theorem~\ref{thmintro:unipotent_semigroup_implies_positivity}
above.  Finally in Section~\ref{sec:posit-flag-vari} we consider the flag
varieties, introduce the notion of diamonds and the positivity of $n$-tuples
of flags.

\smallskip

The authors thank the anonymous referee for various suggestions to improve the paper.

\section{Background and preliminaries}
\label{sec:struct-parab-subgr}
In this section we recall classical material on the structure of semisimple Lie groups
and their parabolic subgroups.

As, starting from the next section, we will mainly focus on simple Lie groups,
we explain here how a semisimple Lie group decomposes into its simple factors
and the induced decompositions of the parabolic subgroups, the flag varieties,
the Weyl groups, etc.

Unless explicitely stated otherwise all vector spaces, Lie algebras,
Lie groups, and their representations are defined over~$\R$. 

\subsection{Cartan involution}\label{sec:cartan-involution}
Let~$G$ be a connected real semisimple Lie group with finite center
\index{$G$}
\index{$G$!a semisimple Lie group}
and denote by
$\mathfrak{g}$ its Lie algebra.\index{$\mathfrak{g}$!the Lie algebra of~$G$} The simple factors of the Lie
algebra~$\mathfrak{g}$ are denoted $\{ \mathfrak{g}_s\}_{s\in
  \mathcal{S}}$\index{$\mathfrak{g}_s$ the simple factors of~$\mathfrak{g}$} so
that $\mathfrak{g} = \bigoplus_{s\in\mathcal{S}} \mathfrak{g}_s$. For
every~$s$ in~$\mathcal{S}$, the connected Lie subgroup of~$G$ with Lie
algebra~$\mathfrak{g}_s$ will be denoted by~$G_s$;\index{$G_s$ the simple
  factors of~$G$} it is a closed subgroup and
a simple Lie group. The subgroups $\{G_s\}_{s\in\mathcal{S}}$ pairwise commute
and the product map $\prod_{s\in\mathcal{S}} G_s \to G$ is onto with finite
kernel (it is sometimes said that $G$~is the \emph{almost product} of the~$G_s$).

Let $\mathfrak{k}$\index{$\mathfrak{k}$ the maximal compact algebra of~$\mathfrak{g}$} be the Lie algebra of a
maximal compact subgroup $K\subset G$ ($\mathfrak{k}$~is called a maximal
compact subalgebra).\index{$K$ the maximal compact subgroup of~$G$} Then
$\mathfrak{g} = \mathfrak{k} \oplus \mathfrak{k}^\perp$ where
$\mathfrak{k}^\perp$ is the orthogonal of~$\mathfrak{k}$ with respect to the
Killing form~$B$\index{$B$ the Killing form on~$\mathfrak{g}$}
on~$\mathfrak{g}$. We denote by~$\tau$\index{$\tau$ the Cartan involution} the {\em Cartan involution} with respect to~$\mathfrak{k}$, it is the Lie algebra automorphism $\tau\colon
\mathfrak{g}\to \mathfrak{g}$ whose fixed point set is~$\mathfrak{k}$ and which
is an involution; namely $\tau|_\mathfrak{k}= \id$ and
$\tau|_{\mathfrak{k}^\perp}= -\id$. The bilinear form $B(X,
\tau(-Y))$ on~$\mathfrak{g}$ is symmetric and positive definite.

For all~$s$, $\mathfrak{k}_s=\mathfrak{k}\cap
\mathfrak{g}_s$\index{$\mathfrak{k}_s$ the intersection of $\mathfrak{k}$ with
$\mathfrak{g}_s$} is the Lie
algebra of a maximal compact subgroup~$K_s$ of~$G_s$.\index{$K_s$ the maximal
  compact subgroup of~$G_s$} 
For all~$s$ we denote by $\mathfrak{k}_{s}^{\perp}$ the orthogonal complement
in~$\mathfrak{g}_s$ of the Lie subalgebra~$\mathfrak{k}_s$ with respect to the
Killing form~$B_s$ of~$\mathfrak{g}_s$, one has the equality
$B=\sum_{s\in\mathcal{S}} B_s$ (under the decomposition $\mathfrak{g} =
\bigoplus_{s\in\mathcal{S}} \mathfrak{g}_s$) and the decompositions
$\mathfrak{k} = \bigoplus_{s\in\mathcal{S}} \mathfrak{k}_s$,
$\mathfrak{k}^\perp = \bigoplus_{s\in\mathcal{S}}
\mathfrak{k}_{s}^{\perp}$. Furthermore the Cartan involution~$\tau$ is the sum
of the Cartan involutions~$\tau_s$ of~$\mathfrak{g}_s$ with respect
to~$\mathfrak{k}_s$ (again using the decomposition  $\mathfrak{g} =
\bigoplus_{s\in\mathcal{S}} \mathfrak{g}_s$). The Lie group~$K$ is the almost
product of the~$K_s$.

\subsection{Restricted roots}
We now choose a maximal Abelian subspace~$\mathfrak{a}$\index{$\mathfrak{a}$ a
Cartan subspace of~$\mathfrak{g}$} contained
in~$\mathfrak{k}^\perp$; $\mathfrak{a}$~is called a \emph{Cartan subspace}
of~$G$. Denoting $\mathfrak{a}_s= \mathfrak{a}\cap \mathfrak{g}_s$, one has
that $\mathfrak{a}_s$ is a Cartan subspace of~$G_s$. The equality
$\mathfrak{a}= \bigoplus_{s\in\mathcal{S}} \mathfrak{a}_s$ holds and induces
an identification of~$\mathfrak{a}^*$ with the direct sum
$\bigoplus_{s\in\mathcal{S}} \mathfrak{a}^{*}_{s}$.

We denote by\index{$\Sigma$ the system of restricted roots}
$\Sigma = \Sigma(\mathfrak{g}, \mathfrak{a})$ the system of restricted
roots, \ie $\Sigma\subset \mathfrak{a}^*$ is the set of nonzero weights for
the adjoint action of~$\mathfrak{a}$ on~$\mathfrak{g}$; the corresponding
weight spaces are denoted~$\mathfrak{g}_\alpha$\index{$\mathfrak{g}_\alpha$
  the root space for the roo~$\alpha$} ($\alpha\in\Sigma$). 
Namely, for every~$\alpha$ in~$\mathfrak{a}^*$, we set
\[\mathfrak{g}_{\alpha}\coloneqq \bigl\{ X \in \mathfrak{g} \, \big\vert \ad(H)(X) =
\alpha(H) X \text{ for all } H\in \mathfrak{a}\bigr\}\]
and $\Sigma$ is the family of elements~$\alpha$ in
$\mathfrak{a}^*\smallsetminus\{0\}$ such that $\mathfrak{g}_\alpha\neq 0$; the
subspace $\mathfrak{g}_\alpha$~is called a \emph{root space} for $\alpha\in
\Sigma$.

For every root
$\alpha\in \Sigma$, the intersection of~$\mathfrak{g}_\alpha$
and~$\mathfrak{k}$ is reduced to~$\{0\}$ and one has $\tau(
\mathfrak{g}_{\alpha}) = \mathfrak{g}_{-\alpha}$.

The root system~$\Sigma$ is the disjoint union of the root systems
$\Sigma_s=\Sigma(\mathfrak{g}_s, \mathfrak{a}_s)$ for~$s$ in~$\mathcal{S}$
and, when $\alpha$~belongs to~$\Sigma_s$, one has
$(\mathfrak{g}_s)_\alpha=\mathfrak{g}_\alpha$. Furthermore $\Sigma_s = \Sigma \cap \mathfrak{a}_{s}^{*}$.

The
weight
space~$\mathfrak{g}_0$\index{$\mathfrak{g}_0=\mathfrak{z}_{\mathfrak{g}}(\mathfrak{a})$
the centralizer of~$\mathfrak{a}$} is the centralizer
$\mathfrak{z}_{\mathfrak{g}}(\mathfrak{a})$ of~$\mathfrak{a}$
in~$\mathfrak{g}$ and is the direct sum of the centralizers $\mathfrak{z}_{\mathfrak{g}_s}(\mathfrak{a}_s)$. Also $\mathfrak{z}_{\mathfrak{g}}(\mathfrak{a}) = \mathfrak{m} \oplus \mathfrak{a}$, where \index{$\mathfrak{m}=\mathfrak{z}_{\mathfrak{k}}(\mathfrak{a})$
the centralizer of~$\mathfrak{a}$ in~$\mathfrak{k}$}
$\mathfrak{m} =
\mathfrak{z}_{\mathfrak{k}}(\mathfrak{a}) =
\mathfrak{z}_{\mathfrak{g}}(\mathfrak{a}) \cap \mathfrak{k}$ is the
centralizer of $\mathfrak{a}$ in $\mathfrak{k}$ (the decomposition
$\mathfrak{z}_{\mathfrak{k}}(\mathfrak{a}) =\bigoplus_{s\in \mathcal{S}}
\mathfrak{z}_{\mathfrak{k}_s}(\mathfrak{a}_s)$ also holds) and the sum is orthogonal with respect to the Killing form~$B$. 
Therefore we have a $B$-orthogonal decomposition  $\mathfrak{g}= \mathfrak{m} \oplus \mathfrak{a} \oplus
\bigl(\bigoplus_{\alpha\in\Sigma} \mathfrak{g}_\alpha\bigr)$. 

Let $<_{\mathfrak{a}^*}$ be a total linear ordering on~$\mathfrak{a}^*$: for
every~$\alpha$ and~$\beta$ in~$\mathfrak{a}^*$, one has either $\alpha=\beta$,
$\alpha<_{\mathfrak{a}^*} \beta$, or $\beta<_{\mathfrak{a}^*} \alpha$ (which we
also denote $\alpha >_{\mathfrak{a}^*} \beta$) and if
$\alpha<_{\mathfrak{a}^*} \beta$ then, for every~$\gamma$ in~$\mathfrak{a}^*$,
$\alpha +\gamma<_{\mathfrak{a}^*} \beta+\gamma$, and for
every~$\lambda$ in~$\R_{>0}$,
 $\lambda\alpha<_{\mathfrak{a}^*} \lambda\beta$.

We then denote by  $\Sigma^+=\{ \alpha\in \Sigma\mid \alpha>_{\mathfrak{a}^*}
0\}$\index{$\Sigma^+$ the positive roots}
the set of positive roots and  by $\Sigma^-=\{ \alpha\in \Sigma\mid
\alpha<_{\mathfrak{a}^*} 0\}$ the set of negative roots; one has
$\Sigma^{-}=-\Sigma^+$. The positive roots
that cannot be written nontrivially as a sum of positive roots are called
\emph{simple} roots; their set is denoted by\index{$\Delta$ the simple roots} $\Delta \subset \Sigma^+$. It is
well known that $\Delta$~is a basis of~$\mathfrak{a}^*$ and that, for
every~$\beta$ in~$\Sigma$, the coefficients of~$\beta$ as a linear combination
with respect to the basis~$\Delta$ are integers all nonnegative or all
nonpositive according to whether $\beta$~belongs to~$\Sigma^+$ or
to~$\Sigma^-$.

The set of positive roots~$\Sigma^+$ is the disjoint union of the sets
$\Sigma_{s}^{+} = \Sigma_s \cap \Sigma^+$ of positive roots for the
simple factors. This holds similarly for the set of negative roots. The
set~$\Delta$ of simple roots is also the disjoint union of the $\Delta_s =
\Delta\cap\Sigma_s$.

\smallskip

The Killing form induces, by restriction and duality, a Euclidean scalar
product\index{$\scalK{,}$ the Euclidean scalar product on~$\mathfrak{a}^*$
  induced by the Killing form} $\scalK{,}$ on~$\mathfrak{a}^*$. It is well known 
that, for every~$s$ in~$\mathcal{S}$, either all
the roots in~$\Sigma_s$ have the same norm or the set~$\Delta_s$ of simple
roots is the disjoint union of two 
nonempty subsets $\Delta_{s,p}$ and $\Delta_{s,g}$ with the following properties:
all the elements of~$\Delta_{s,p}$ have the same norm, all the elements
of~$\Delta_{s,g}$ have the same norm, and the norm of elements in~$\Delta_{s,g}$ is
greater than the norm of elements in~$\Delta_{s,p}$. Elements of~$\Delta_{s,g}$ are
called \emph{long roots}, elements of~$\Delta_{s,p}$ are called \emph{short
  roots}. For $\alpha \in \Delta_{s,g}$ and $\beta \in \Delta_{s,p}$, one has in fact
$\scalK{\alpha,\alpha} = 2\scalK{\beta,\beta}$ except in the case of
type~$\lietype{G}_2$ where $\scalK{\alpha,\alpha} = 3\scalK{\beta,\beta}$.

 In general the root spaces~$\mathfrak{g}_\alpha$ are not
necessarily $1$-di\-men\-sional. In fact they are all $1$-di\-men\-sional if and only if $\mathfrak{g}$~is the Lie algebra of a
split real form. However the root spaces~$\mathfrak{g}_\alpha$ are
$1$-di\-men\-sional when $\alpha$~is a long root.

The \emph{open Weyl chamber}  corresponding to~$\Sigma^+$ will be
denoted\index{$\mathfrak{a}^+$ an open Weyl chamber}
\[\mathfrak{a}^+ =\{ X\in \mathfrak{a}\mid \alpha(X)>0, \ \forall \alpha\in
\Sigma^+\};\]
its closure~$\bar{\mathfrak{a}}^+$ is called the \emph{closed Weyl
  chamber}.\index{$\bar{\mathfrak{a}}^+$ a closed Weyl chamber} One has also
\[\mathfrak{a}^+ =\{ X\in \mathfrak{a}\mid \alpha(X)>0, \ \forall \alpha\in
\Delta\}.\]
The closed Weyl chamber is the direct sum of the
\[\bar{\mathfrak{a}}^{+}_{s} =\{ X\in \mathfrak{a}_s\mid \alpha(X)\geq 0, \ \forall \alpha\in
\Sigma^{+}_{s}\} = \bar{\mathfrak{a}}^+ \cap \mathfrak{a}_s.\]
The (relative)
interior of~$\bar{\mathfrak{a}}_{s}^{+}$ is
\[\mathfrak{a}_{s}^{+}  =\{ X\in \mathfrak{a}_s\mid \alpha(X)>0, \ \forall \alpha\in
\Sigma^{+}_{s}\},\]
the closure of~$\mathfrak{a}_{s}^{+}$
is~$\bar{\mathfrak{a}}_{s}^{+}$ and the open Weyl chamber~$\mathfrak{a}^+$ is
the direct sum of the open Weyl chambers~$\mathfrak{a}_{s}^{+}$.

An element of~$\mathfrak{g}$ is said
 \emph{regular} if it is conjugate to an element of $\mathfrak{a}
\smallsetminus \bigcup_{\alpha\in\Sigma} \ker(\alpha)$. A regular element is
contained in a unique Cartan subspace and also in a unique Weyl chamber.

\subsection{Cartan matrix and Dynkin diagram}\label{sec:Dynkin}
The structure of the root system can be efficiently encoded using the Cartan matrix or the Dynkin diagram. 

The \emph{Cartan matrix} of~$\mathfrak{g}$ is the square matrix of size $\dim
\mathfrak{a} = \sharp \Delta$ whose coefficients
$(A_{\alpha,\beta})_{\alpha,\beta\in \Delta}$ are given by the
formula\index{$(A_{\alpha,\beta})$ the Cartan matrix}
\(A_{\alpha,\beta} = 2  \scalK{\alpha,\beta}/\scalK{\alpha,\alpha}.\)

The Cartan matrix of~$\mathfrak{g}$ is block diagonal with diagonal blocks
being the Cartan matrices of~$\mathfrak{g}_s$ (for~$s$ in~$\mathcal{S}$) and
the Cartan matrices of~$\mathfrak{g}_s$ are irreducible (not block diagonal
even up to permutation).

The \emph{Dynkin diagram} of~$\mathfrak{g}$ is the graph whose vertex set
is~$\Delta$ and the edges are given by the following recipe: for every
$\alpha\neq \beta$ in~$\Delta$, $A_{\alpha,\beta} A_{\beta,\alpha}$ edges are
added between~$\alpha$ and~$\beta$, an arrow is added from~$\alpha$ to~$\beta$
in the case when $A_{\alpha,\beta} A_{\beta,\alpha}\neq 0$ and
$\scalK{\alpha,\alpha}>\scalK{\beta,\beta}$. The connected components of the
Dynkin diagram of~$\mathfrak{g}$ are the Dynkin diagrams of~$\mathfrak{g}_s$
($s\in\mathcal{S}$). The Dynkin diagram of~$\mathfrak{g}_s$ is among the
following diagrams: 

\noindent $A_n$ \dynkin A{} ($n\geq 1$), $B_n$  \dynkin B{} ($n\geq 2$), 
$C_n$  \dynkin C{} ($n\geq 3$), $D_n$  \dynkin D{} ($n\geq 4$), $F_4$  \dynkin
F4, $E_6$ \dynkin E6, $E_7$ \dynkin E7, 

\noindent $E_8$ \dynkin E8, or
$G_2$ \dynkin G2.

  We already observed that the root spaces~$\mathfrak{g}_\alpha$ are not
necessarily $1$-di\-men\-sional.  
Therefore, contrary to the situation for complex simple Lie algebras, the
Dynkin diagram alone does not characterize the real Lie algebra, additional
information is needed. 

\subsection{Non-reduced root systems}\label{sec:nonreduced}
A root system is said to be \emph{reduced} if given any root $\alpha\in \Sigma$, the multiplies of~$\alpha$ that are in~$\Sigma$ are $\pm \alpha$. 
Most root systems are reduced, and when a simple real Lie algebra has
non-reduced type then it is necessarily Hermitian of non-tube type and
the corresponding root system is 
of type $\lietype{BC}_n$: if $\{\alpha_1, \dots, \alpha_n\}$ are the simple roots then  $2\alpha_n$~is also a 
root. 

In this case, $\mathfrak{g}_{2\alpha_n}$ is again $1$-dimensional but
$\mathfrak{g}_{\alpha_n}$ is not, in fact the Lie bracket from
$\mathfrak{g}_{\alpha_n}\times \mathfrak{g}_{\alpha_n}$ to
$\mathfrak{g}_{2\alpha_n}$ is onto.

\subsection{The $\mathfrak{sl}_2$-triples}\label{sec:sl_2-triples} 
Even though the root spaces~$\mathfrak{g}_\alpha$ are not necessarily
$1$-dimensional, their elements lead to embeddings of $3$-dimensional simple subalgebras.  

For every root $\alpha \in \Sigma$ and any nonzero
element~$X$ in~$\mathfrak{g}_\alpha$,  the Lie bracket
$[\tau(X), X]$ belongs to~$\mathfrak{a}$; there is furthermore a (unique) positive
multiple~$Y$ of~$X$ such that
\[
(Y, \tau(-Y),  [\tau(Y), Y])\] is an
$\mathfrak{sl}_2$-triple, that is a triple $(E,F,D)$ such that $D=[E,F]$,
$[D,E]=2E$, and $[D,F]=-2F$.

The set of elements~$X$ in~$\mathfrak{g}_\alpha$ such that $(X,
\tau(-X), [\tau(X), X])$  is an
$\mathfrak{sl}_2$-triple is precisely equal to the sphere\index{$S(\mathfrak{g}_\alpha)$
the sphere in $\mathfrak{g}_\alpha$ of elements giving rise to $\mathfrak{sl}_2$-triples} $S(\mathfrak{g}_\alpha)$ in~$\mathfrak{g}_\alpha$
defined by the equation $B(X, \tau(X))=-2/ \scalK{\alpha,\alpha}$ (cf.\ 
\cite[Prop.~6.52, p.~379]{Knapp_LieGrp})
The element $H_\alpha = [\tau(X),X]$\index{$H_\alpha$ the coroot associated to~$\alpha$}
does not depend on the initial choice of~$X$ in~$S(\mathfrak{g}_\alpha)$.
When 
$\alpha$, $\beta$, and $\alpha+\beta$ belong to~$\Sigma$, one has
$H_{\alpha+\beta} = H_\alpha+H_\beta$, furthermore 
$H_{-\alpha}= -H_\alpha$.

When $\alpha$ and $\beta$ belong to~$\Sigma$, the suprema
\begin{align*}
  p&=\sup \{n\in\NN\mid
     \beta-n\alpha\in\Sigma \}\\
  \intertext{and} q&=\sup \{n\in\NN\mid
  \beta+n\alpha\in\Sigma \}
\end{align*}
are finite and the subset $\{
\beta+n\alpha\}_{n=-p,\dots, q}$\index{$\{
\beta+n\alpha\}_{n=-p,\dots, q}$ the $\alpha$-chain containing~$\beta$} is contained in~$\Sigma$ and is called the
\emph{$\alpha$-chain containing~$\beta$}. The $\alpha$-chain is called
\emph{trivial} when $p=q=0$. The significance of the $\alpha$-chain $\{
\beta+n\alpha\}_{n=-p,\dots, q}$ is that the direct sum $\bigoplus_{n=-p}^{q}
\mathfrak{g}_{\beta+n\alpha}$ is invariant by bracketing with elements
in~$\mathfrak{g}_\alpha \cup \mathfrak{g}_{-\alpha}$, in particular it is
invariant by the $\mathfrak{sl}_2$-triple $(X, -\tau(X), [\tau(X),X])$ for
every~$X$ in~$S(\mathfrak{g}_\alpha)$.

At different places, we will use basic knowledge about
$\mathfrak{sl}_2(\R)$-modules and the fact that there is, up to
isomorphism,  a unique  $\mathfrak{sl}_2(\R)$-module of dimension~$n$. An
$\mathfrak{sl}_2(\R)$-module structure on a vector space~$V$ is equivalently
given by an $\mathfrak{sl}_2$-triple $(X,Y,H)$ in $\mathfrak{gl}(V)=
\End(V)$; if an element $v\in V$ satisfies $X(v)=0$ and is an eigenvector for~$H$, it generates an
irreducible module~$W$ which is of dimension~$d+1$ where $d$~is the
nonnegative integer such that $H(v)=dv$; a basis of~$W$ is then $\{ Y^k(v)\}_{
k\in \llbracket 0, d\rrbracket}$; the Lie algebra morphism $\mathfrak{sl}_2\to \End(W)$ lifts to an homomorphism $\SL_2(\R)\to
\GL(W)$, this homomorphism factors through $\PSL_2(\R)$ if and only if $d$~is
even (and $W$~is odd dimensional). The element $\exp( \frac{\pi}{2}(X-Y))$ is
an automorphism of~$W$, of order~$2$ if $d$~is even, and that sends the
$H$-eigenspace with respect to the eigenvalue~$k$ onto the $H$-eigenspace with
respect to~$-k$.

\subsection{The Weyl group}
\label{sec:weyl-group}

The subgroup $W\subset \GL( \mathfrak{a}^*)$ of\index{$W$ the Weyl group of~$G$} automorphisms of~$\Sigma$ is
called the \emph{(restricted) Weyl group} of~$G$. We therefore have an action
$(w, \alpha)\mapsto w(\alpha)$ of~$W$ on~$\Sigma$. The group~$W$ is a subgroup of the orthogonal group for the scalar
product $\scalK{,}$ induced by the Killing form~$B$. It is a finite Coxeter
group generated by the
orthogonal reflections\index{$s_\alpha$ the reflection in~$W$ associated to
  the root~$\alpha$} $\{ s_\alpha\}_{\alpha\in \Delta}$ such that
$s_\alpha(\alpha)=-\alpha$. The Weyl group~$W$ is (isomorphic to) the direct
product of the Weyl groups~$W_s$ of~$G_s$. The Dynkin diagram determines the
Coxeter presentation of the Weyl group. Namely a presentation of~$W$ is given
by the following relations $s_{\alpha}^2=e$ ($\alpha\in \Delta$) and $(s_\alpha
s_\beta)^{n_{\alpha,\beta}}=e$ ($\alpha\neq \beta$) where
$n_{\alpha,\beta}=2,3,4,6$ if the number of edges between~$\alpha$ and~$\beta$
in the Dynkin diagram is respectively $0,1,2,3$. The following relations
in~$W$ $(\alpha\neq \beta)$
\begin{align*}
  s_\alpha s_\beta
  &= s_\beta  s_\alpha \text{ if } n_{\alpha,\beta}=2\\
  s_\alpha s_\beta  s_\alpha
  &= s_\beta  s_\alpha s_\beta \text{ if } n_{\alpha,\beta}=3\\
    s_\alpha s_\beta s_\alpha s_\beta &=
  s_\beta s_\alpha s_\beta s_\alpha \text{ if } n_{\alpha,\beta}=4\\
    s_\alpha s_\beta s_\alpha s_\beta s_\alpha s_\beta &=
  s_\beta s_\alpha s_\beta s_\alpha s_\beta s_\alpha \text{ if } n_{\alpha,\beta}=6;
\end{align*}
are called the \emph{braid relations} and play an important role.

Note that the Weyl groups of type $\lietype{B}_n$ and $\lietype{C}_n$ are the same since the orientation
  of the arrows do not contribute to the presentation of the Weyl group.

\begin{remark}
  When the simple restricted roots are numbered, \ie when
  $\Delta= \{ \alpha_1, \alpha_2, \dots, \alpha_r\}$, we will rather write
  $s_i$ instead of $s_{\alpha_i}$.

    \end{remark}

 The closed Weyl chamber~$\bar{\mathfrak{a}}^+$ is a fundamental domain
for the action of~$W$ on~$\mathfrak{a}$. 
It is also the convex
cone generated by $\{H_\alpha\}_{\alpha\in \Delta}$.

The Weyl group~$W$ is isomorphic to the
quotient of the normalizer $N_K( \mathfrak{a})$ of~$\mathfrak{a}$ in~$K$ by
the centralizer $C_K(\mathfrak{a})$.
On different occasions for an element~$w$ in~$W$ we will denote by~$\dot{w}$ an
element of $N_K( \mathfrak{a})$ representing it. 
For every~$\alpha$ in~$\Sigma$, one has $\Ad( \dot{w}) \mathfrak{g}_\alpha =
\mathfrak{g}_{w (\alpha)}$. 

By no means the assignment
$w\mapsto \dot{w}$ is an homomorphism.\index{$\dot{w}$ an element of $N_K(
  \mathfrak{a})$ that is a lift of~$w$}
It will sometimes be convenient to have
explicit formula for these lifts. 
For example,
when $\alpha$~is a root, we will take $\dot{s}_\alpha  = \exp
\bigl(\frac{\pi}{2} (X + \tau(X))\bigr)$ for~$X$ in
$S(\mathfrak{g}_\alpha)$ \cite[Prop.~6.52]{Knapp_LieGrp}. With this choice, the automorphism $\Ad(
\dot{s}_\alpha)$ of~$\mathfrak{g}$ is of order~$2$ or~$4$ according to whether
the corresponding morphism 
from $\mathfrak{sl}_2(\R)$ to~$\mathfrak{g}$
is tangent to a representation of $\PSL_2(\R)$ or
not.

The above generating set $\{s_\alpha\}_{\alpha\in \Delta}$ of~$W$ allow to define
the word length function $\ell\colon W\to \NN$.\index{$\ell$
  the word length function on~$W$} It is well
known that $W$ admits a unique element $w_\Delta$ of longest length.\index{$w_\Delta$ the longest
  length element in~$W$} This element $w_\Delta$ is
of order~$2$, it is the element of~$W$ (that we now see as acting on~$\mathfrak{a}$) sending the Weyl
chamber~$\mathfrak{a}^+$ to its opposite~$-\mathfrak{a}^+$. If $w_{\Delta_s}$
denotes the longest length element of~$W_s$, then the $w_{\Delta_s}$ pairwise
commute and one has $w_\Delta = \prod_{\sigma\in\mathcal{S}} w_{\Delta_s}$.

The map $\iota: \alpha
\mapsto - w_{\Delta}( \alpha)$ is an involution of~$\Sigma$, it
sends~$\Delta$ to~$\Delta$ and is called the \emph{opposition
  involution}. If $\iota_s$ is the opposition involution of~$\Sigma_s$, one
has $\iota = \coprod_{s\in\mathcal{S}} \iota_s$ under the identification
$\Sigma= \coprod_{s\in\mathcal{S}} \Sigma_s$.

Given an element~$w$ in~$W$ of length $n=\ell(w)$, an $n$-tuple $(\alpha_1,
\dots, \alpha_n)$ in $\Delta^n$ such that $w= s_{\alpha_1}\cdots s_{\alpha_n}$
is called a \emph{reduced expression} of~$w$. We will sometimes say that
$s_{\alpha_1}\cdots s_{\alpha_n}$ is a reduced expression of~$w$. An important
property is that all the reduced expressions of~$w$ can be obtained from one
another by successively applying braid relations.

When $w$, $w_1$, and~$w_2$ are elements of~$W$ such that $w=w_1 w_2$ and
$\ell(w) =\ell(w_1) + \ell(w_2)$, we will say that $w_1 w_2$ is a reduced
product of~$w$.

\subsection{Parabolic subgroups}
\label{sec:parabolic-subgroups}
Let $\Theta \subset \Delta$ be a subset of the set of simple
roots.\index{$\Theta$ a subset of $ \Delta$} Setting $\Sigma_{\Theta}^+ = \Sigma^+ \smallsetminus \Span(\Delta
\smallsetminus \Theta)$, the
subspaces 
\[\mathfrak{u}_\Theta = \sum_{\mathclap{\alpha \in \Sigma_{\Theta}^+}}
\mathfrak{g}_{\alpha},\quad\mathfrak{u}^{\mathrm{opp}}_\Theta = \sum_{\mathclap{\alpha \in
  \Sigma_{\Theta}^+}} \mathfrak{g}_{-\alpha},\] 
 are nilpotent Lie subalgebras\index{$\mathfrak{u}_\Theta$ the standard
   nilpotent subalgebras} \index{$\mathfrak{u}_{\Theta}^{\mathrm{opp}}$ the standard
  opposite nilpotent subalgebras}
and the subspace\index{$\mathfrak{l}_\Theta$ the Levi factor of $\mathfrak{p}_\Theta$}
\[\mathfrak{l}_\Theta =\mathfrak{z}_{\mathfrak{g}}(\mathfrak{a}) \oplus \!\!\! \!\!\!
\sum_{\alpha \in \Span(\Delta \smallsetminus \Theta)} \!\!\! \!\!\!
\mathfrak{g}_{\alpha} =\mathfrak{z}_{\mathfrak{g}}(\mathfrak{a}) \oplus \!\!\! \!\!\! \!\!\!
\bigoplus_{\alpha \in \Span(\Delta \smallsetminus \Theta) \cap \Sigma^+} \!\!\! \!\!\! \!\!\! (\mathfrak{g}_{\alpha} \oplus \mathfrak{g}_{-\alpha})\] 
is a reductive subalgebra.

The \emph{standard parabolic subgroup} $P_\Theta$\index{$P_\Theta$ the
  standard parabolic groups} associated with
$\Theta \subset \Delta$ is the normalizer in $G$ of $\mathfrak{u}_\Theta$; it
is also the normalizer of the Lie algebra\index{$\mathfrak{p}_\Theta$ the Lie
  algebra of~$P_\Theta$} $\mathfrak{p}_\Theta =
\mathfrak{l}_\Theta \oplus \mathfrak{u}_\Theta$; one has $\mathfrak{p}_\Theta
= \Lie( P_\Theta)$ and $P_\Theta$ is its own normalizer.  We
also denote by $P^{\mathrm{opp}}_\Theta$\index{$P^{\mathrm{opp}}_\Theta$ the
  standard opposite parabolic subgroups} the normalizer in~$G$ of
$\mathfrak{u}^{\mathrm{opp}}_{\Theta}$; $P^{\mathrm{opp}}_\Theta$~is called
the \emph{standard opposite parabolic subgroup} associated with
$\Theta$. The longest length element~$w_\Delta$ conjugates opposite parabolic
and parabolic subgroups: $\Ad( \dot{w}_\Delta)
\mathfrak{p}_{\Theta}^{\mathrm{opp}} =
\mathfrak{p}_{\iota(\Theta)}$, $\Ad( \dot{w}_\Delta)
\mathfrak{p}_{\Theta}^{\mathrm{opp}} = \mathfrak{p}_{\iota(\Theta)}$, $\Ad( \dot{w}_\Delta)
\mathfrak{u}_{\Theta}^{\mathrm{opp}} =
\mathfrak{u}_{\iota(\Theta)}$, $\Ad( \dot{w}_\Delta)
\mathfrak{u}_{\Theta}^{\mathrm{opp}} = \mathfrak{u}_{\iota(\Theta)}$,
$\dot{w}_\Delta P_{\Theta}^{\mathrm{opp}} \dot{w}_{\Delta}^{-1} =
P_{\iota(\Theta)}$, and $\dot{w}_\Delta P_{\Theta} \dot{w}_{\Delta}^{-1} =
P_{\iota(\Theta)}^{\mathrm{opp}}$, where~$\iota$ is the opposition involution
$\alpha\mapsto -w_\Delta( \alpha)$.

The group~$P_\Theta$ is the semidirect product of its unipotent radical
$U_\Theta = \exp(\mathfrak{u}_\Theta)$\index{$U_\Theta$ the unipotent radical
  of $P_\Theta$} and the Levi subgroup\index{$L_\Theta$ the Levi factor of $P_\Theta$}
$L_\Theta = P_\Theta \cap P^{\mathrm{opp}}_\Theta$.  We denote by~$L_{\Theta}^{\circ}$ the connected component of the identity in~$L_\Theta$.  The Lie algebra
of~$L_\Theta$ is~$\mathfrak{l}_\Theta$. 
With this convention $P_{\emptyset} = G$ and $P_{\Delta}$~is the minimal
parabolic subgroup, more generally, if $\Theta\subset \Psi$, one has
$P_\Theta\supset P_\Psi$ and $L_{\Theta}\supset L_\Psi$.

Setting $\Theta_s = \Delta_s \cap \Theta$ and adopting similar notation with
respect to the group~$G_s$ and the Lie algebra~$\mathfrak{g}_s$, one has that
$\mathfrak{u}_\Theta$ is the direct sum $\bigoplus_{s\in \mathcal{S}}
\mathfrak{u}_{\Theta_s}$ and also that $\mathfrak{l}_\Theta =
\bigoplus_{s\in\mathcal{S}} \mathfrak{l}_{\Theta_s}$, $\mathfrak{p}_\Theta =
\bigoplus_{s\in \mathcal{S}} \mathfrak{p}_{\Theta_s}$. The (unipotent)
Lie group~$U_\Theta$ is the product of the Lie
groups~$U_{\Theta_s}$, the Lie group~$L_\Theta$ is the almost
product of the Lie groups~$L_{\Theta_s}$ and the Lie group~$P_\Theta$ is the
almost product of the Lie groups~$P_{\Theta_s}$. Similar statements hold for
$U_{\Theta}^{\mathrm{opp}}$ and~$P_{\Theta}^{\mathrm{opp}}$.

When $\Theta_s=\emptyset$, the group~$G_s$ is included in~$P_\Theta$ and this
 means that $G_s$~acts trivially on 
$G/P_\Theta$. As we would like to work with (virtually) faithful actions
of~$G$, we will rule out this case  in  Section~\ref{sec:theta-posit-st}. 
The Lie algebra~$\mathfrak{l}_\Theta$ is stable under the Cartan
involution~$\tau$, the intersection $\mathfrak{k}_\Theta = \mathfrak{l}_\Theta
\cap \mathfrak{k}$\index{$\mathfrak{k}_\Theta$ a maximal compact subalgebra in
$\mathfrak{l}_\Theta$} is a maximal compact subalgebra of~$\mathfrak{l}_\Theta$, and in fact $\mathfrak{k}$~is the unique compact maximal subalgebra
of~$\mathfrak{g}$ containing~$\mathfrak{k}_\Theta$.
One has
\begin{align}
\label{eq:max-compact-S_J}  \mathfrak{k}_\Theta &= \mathfrak{m} \oplus  \!\!\! \!\!\!
\bigoplus_{\alpha \in \Span(\Delta \smallsetminus \Theta) \cap \Sigma^+}
                        \!\!\! \!\!\! (\mathfrak{g}_{\alpha} \oplus
                        \mathfrak{g}_{-\alpha}) \cap \mathfrak{k},\\
\label{eq:max-compact-S_J-factor}  (\mathfrak{g}_{\alpha} \oplus&
  \mathfrak{g}_{-\alpha}) \cap \mathfrak{k}
                      = \{ X+\tau(X)\}_{X\in \mathfrak{g}_\alpha},\ \forall \alpha \in  \Sigma^+.
\end{align}

The derived subgroup $S_\Theta= [L_\Theta, L_\Theta]$ of~$L_\Theta$ is a
semisimple\index{$S_\Theta$ the semisimple factor of $L_\Theta$}
real Lie group whose system of  restricted roots is given by the Dynkin
diagram
$\Delta\smallsetminus \Theta$ (meaning that the vertices in~$\Theta$ are
removed as well as the edges incident to those vertices). A Cartan subspace
for~$S_\Theta$ is\index{$\mathfrak{a}_\Theta$ a Cartan subspace for $S_\Theta$}
\begin{equation}
  \label{eq:CartanS_Theta}
  \mathfrak{a}_\Theta = \bigoplus_{\mathclap{\alpha\in \Delta\smallsetminus \Theta}} \R
  H_\alpha
\end{equation}
and a Weyl chamber in this Cartan subspace is $\sum_{\alpha\in \Delta\smallsetminus \Theta} \R_{>0}
H_\alpha$.

\subsection{The flag variety}
\label{sec:flag-variety}

A subgroup of~$G$ conjugated to~$P_\Theta$ (for some
$\Theta\subset \Delta$) is a \emph{parabolic subgroup}. A \emph{parabolic subalgebra} of~$\mathfrak{g}$ is a
subalgebra conjugated to some $\mathfrak{p}_\Theta$, such a subalgebra will
sometimes be called of \emph{type}~$\Theta$. The space of all parabolic
subalgebras of type~$\Theta$ is called the \emph{flag variety} (of
type~$\Theta$) and denoted by~$\mathsf{F}_\Theta$.\index{$\mathsf{F}_\Theta$
  the flag variety associated with $P_\Theta$} The
space~$\mathsf{F}_\Theta$ is a subset of a Grassmanian variety of the real
vector space~$\mathfrak{g}$ and is endowed with the induced topology. The flag
variety~$\mathsf{F}_\Theta$ is compact and the conjugation action of~$G$ on
the subalgebras of~$\mathfrak{g}$ induces a continuous and transitive action
on~$\mathsf{F}_\Theta$;  $\mathsf{F}_\Theta$~identifies with $G/P_\Theta$ since $P_{\Theta}$ is the stabilizer of $\mathfrak{p}_\Theta$; it also
 identifies with the space of parabolic subgroups of type~$\Theta$.

The group~$G$ acts diagonally on $\mathsf{F}_\Theta \times
\mathsf{F}_{\iota(\Theta)}$. Since the parabolic subalgebra $\mathfrak{p}_{\Theta}^{\mathrm{opp}}$ is
conjugate to $\mathfrak{p}_{\iota(\Theta)} = \Ad(\dot{w}_\Delta) \mathfrak{p}_{\Theta}^{\mathrm{opp}}$,
the parabolic subalgebra~$\mathfrak{p}_{\Theta}^{\mathrm{opp}}$ is an
element of~$\mathsf{F}_{\iota(\Theta)}$. A pair $(x,y)$ in $\mathsf{F}_\Theta \times
\mathsf{F}_{\iota(\Theta)}$ will be called \emph{transverse} if it belongs to
the $G$-orbit of $( \mathfrak{p}_\Theta,
\mathfrak{p}_{\Theta}^{\mathrm{opp}})$. By definition, there is one orbit of
transverse pairs; this orbit is isomorphic to $G/L_\Theta$ and is in fact the unique
open $G$-orbit in the product $\mathsf{F}_\Theta \times
\mathsf{F}_{\iota(\Theta)}$.

 In many cases below, the involution~$\iota$ is the
identity on~$\Delta$. In such a case, or more generally under the assumption
$\iota(\Theta)=\Theta$, the notion of transversality makes sense for pairs of
elements in~$\mathsf{F}_\Theta$.

The flag variety~$\mathsf{F}_\Theta$ is isomorphic to the product $\prod_{s\in
  \mathcal{S}} \mathsf{F}_{\Theta_s}$ and this isomorphism is equivariant with
respect to the action of $\prod_{s\in\mathcal{S}} G_s$. The flag
variety~$\mathsf{F}_{\iota(\Theta)}$ identifies with $\prod_{s\in
  \mathcal{S}} \mathsf{F}_{\iota(\Theta_s)}$ and a pair $(x,y)$ in
$\mathsf{F}_\Theta\times \mathsf{F}_{\iota(\Theta)}$ is transverse if and only
if, for all~$s$, the $s$-component of~$x$ is transverse to the $s$-component of~$y$.

\subsection{The action of \texorpdfstring{$L_\Theta$}{L_Θ} on
  \texorpdfstring{$\mathfrak{u}_\Theta$}{u_Θ}}
\label{sec:acti-l-theta-on-u-theta}

This paragraph introduces the decomposition of the Lie
algebras~$\mathfrak{u}_\Theta$ and $\mathfrak{u}_{\Theta}^{\mathrm{opp}}$
under the action of the Levi factor~$L_\Theta$. This decomposition generalizes
the root space decomposition; it will be of crucial importance for the
definition of $\Theta$-positivity (Definition~\ref{defi:theta-pos-structure}).

The Levi subgroup $L_\Theta$ acts via the adjoint action on
$\mathfrak{u}_\Theta$.  Let $\mathfrak{z}_\Theta$ denote the center of
$\mathfrak{l}_\Theta$ and $\mathfrak{t}_\Theta = \mathfrak{z}_\Theta \cap
\mathfrak{a}$ its intersection with the Cartan subspace. One
has\index{$\mathfrak{z}_\Theta$ the center of $\mathfrak{l}_\Theta$}
\index{$\mathfrak{t}_\Theta$ the split factor of the center of $\mathfrak{l}_\Theta$}
\[\mathfrak{t}_\Theta = \bigcap_{\mathclap{\alpha\in \Delta\smallsetminus \Theta}} \ker
  \alpha,\]
and $\mathfrak{a}$ is the $B$-orthogonal sum  of~$\mathfrak{t}_\Theta$ and
of~$\mathfrak{a}_\Theta$ (Equation~\eqref{eq:CartanS_Theta}).
Then
$\mathfrak{u}_\Theta$ decomposes into the weight 
spaces under the adjoint action of~$\mathfrak{t}_\Theta$; for every
 $\beta \in \mathfrak{t}^*_\Theta$, set\index{$\mathfrak{u}_\alpha$! the weight
 space of $\mathfrak{t}_\Theta$ in $\mathfrak{u}_\Theta$}
\[\mathfrak{u}_\beta\coloneqq \{ N \in \mathfrak{g} \mid \mathrm{ad}(Z)N =
\beta(Z) N , \, \forall Z \in \mathfrak{t}_\Theta \},
\]
and denote\index{$P(\mathfrak{u}_\Theta, \mathfrak{t}_\Theta)$ the weights of
  $\mathfrak{t}_\Theta$ in $\mathfrak{u}_\Theta$}
\begin{equation}
  \label{eq:weights-in-uTheta}
  P(\mathfrak{u}_\Theta, \mathfrak{t}_\Theta) = \{ \beta \in
  \mathfrak{t}^*_\Theta \mid \mathfrak{u}_\beta\neq 0 \}
\end{equation}
so that 
\[\mathfrak{u}_\Theta = \bigoplus_{\mathclap{\beta\in P(\mathfrak{u}_\Theta,
      \mathfrak{t}_\Theta)}} \mathfrak{u}_\beta.\]
These weight
spaces are of course related to those of~$\mathfrak{a}$:
\begin{equation}
\mathfrak{u}_\beta =\!\!\! \bigoplus_{\alpha \in \Sigma, \,
  \alpha|_{\mathfrak{t}_\Theta} = \beta } \!\!\! \mathfrak{g}_{\alpha}, \label{eq:u_beta_sum_g_alpha}
\end{equation}
and are invariant under~$L_{\Theta}^\circ$.

By a little abuse of notation, for an element~$\beta$
in~$\Sigma$, we will denote as well by~$\beta$ its restriction
to~$\mathfrak{t}_\Theta$ and hence by~$\mathfrak{u}_\beta$ the corresponding
weight space. For such a~$\beta$, one has
\begin{equation}
  \label{eq:u_beta_sum_g_alpha_v2}
  \mathfrak{u}_\beta =\!\!\!\!\!\!\bigoplus_{\substack{\alpha \in \Sigma\\ \alpha- \beta \in
      \Span(\Delta \smallsetminus \Theta)}}\!\!\!\!\!\! \mathfrak{g}_{\alpha}.
\end{equation}
We will make use of the spaces~$\mathfrak{u}_\beta$ particularly when the
element~$\beta$ belongs to~$\Theta$. Thus, when convenient, we will
consider~$\Theta$ as a subset of~$P( \mathfrak{u}_\Theta,
\mathfrak{t}_\Theta)$.

The following results are established in~\cite{Kostant_RootLevi}:

\begin{theorem}
  \label{theo:kostant-results}
  \begin{enumerate}[leftmargin=*,nosep]
  \item For every~$\beta$ in $P(\mathfrak{u}_\Theta, \mathfrak{t}_{\Theta})$,
    $\mathfrak{u}_\beta$~is an irreducible\index{$\mathfrak{u}_\alpha$!the
      irreducible factors for the action of $\mathfrak{l}_\Theta$ on $\mathfrak{u}_\Theta$}
    representation of~$L_{\Theta}^\circ$.
  \item For every~$\beta$ and $\beta'$ in $P(\mathfrak{u}_\Theta, \mathfrak{t}_{\Theta})$, the
    relation $[\u_\beta, \u_{\beta'}] = \u_{\beta+\beta'}$ is satisfied.
  \item The Lie algebra $\u_\Theta$ is generated 
    by $\u_\beta$, $\beta \in \Theta$. Analogously
    $\u_{\Theta}^{\mathrm{opp}}$ is generated by $\u_{-\beta}$,
    $\beta \in \Theta$.
  \end{enumerate}
\end{theorem}

\subsection{Examples} 
  In these examples, we use classical notation and numbering for the root
  systems as set for example in \cite[Ch.~X, \S{}3--4, p.~461--475]{Helgason}
  or in \cite[Ch.~VI, Planches and \S{}4,
  n\textsuperscript{o}~14]{BourbakiLie456}

\subsubsection*{Split real Lie groups} 
\label{sec:1}
Let $G$ be a split real form, and $\Theta = \Delta$. Then
    $\mathfrak{t}_\Theta = \mathfrak{a}$, $P(\mathfrak{u}_\Theta, \mathfrak{t}_\Theta) = P(\mathfrak{u}_\Delta,
\mathfrak{a}) = \Sigma^+$, and 
    $\mathfrak{u}_\beta = \mathfrak{g}_\beta$ for all $\beta \in \Sigma^+$.

\subsubsection*{Lie groups of Hermitian type}
\label{sec:2}
Let $G$ be a Lie group of Hermitian type. Then the root system is of type
$\lietype{C}_r$
 if $G$ is of tube type, and of type $\lietype{BC}_r$
(non-reduced) if $G$ is not of tube type.  Let
$\Delta =\{ \alpha_1, \dots, \alpha_r\}$, and let $\Theta = \{ \alpha_r\}$ so
that $P_\Theta$~is the stabilizer of a point in the Shilov
boundary~$\mathsf{F}_\Theta$ of the Hermitian symmetric space associated
with~$G$.  Then $\mathfrak{t}_\Theta \simeq \R$, furthermore $P(\mathfrak{u}_\Theta,
\mathfrak{t}_\Theta) =\{ \alpha_r\}$ and $\mathfrak{u}_\Theta = \mathfrak{u}_{\alpha_r}$ if $G$~is of
tube type, whereas $P(\mathfrak{u}_\Theta,
\mathfrak{t}_\Theta) =\{ \alpha_r, 2\alpha_r\}$ and 
$\mathfrak{u}_\Theta= \mathfrak{u}_{\alpha_r} \oplus \mathfrak{u}_{2
  \alpha_r}$ if $G$~is not of tube type.
Note that if $G$ is not of tube type, there is a unique (up to conjugation) maximal subgroup of tube type, whose Lie algebra is generated by 
$\mathfrak{g}_{\pm \alpha_i}$ for $i<r$ and by $\mathfrak{g}_{\pm 2 \alpha_r}$. 

Let $G$ be a Hermitian Lie group of tube type. In this case
$\Theta = \{ \alpha_r\}$, and $\u_\Theta = \u_{\alpha_r}$ is Abelian.  We
now describe~$\u_{\alpha_r}$ in more detail when the real rank of~$G$ is~$2$
or~$3$. These cases are of particular importance because they also arise when
considering the positive structures for indefinite orthogonal groups
$\SO(p,q)$, $2<p<q$, and for the exceptional groups with system or restricted
roots $F_4$ (cf.\ Sections~\ref{sec:hermtian-tube-type}
and~\ref{sec:class-symm-cones}).    
    
For a Hermitian Lie group of tube type of real rank~$2$ the roots that are
equal to~$\alpha_2$ modulo $\R\alpha_1$ are
$\alpha_2, \alpha_1 + \alpha_2, 2\alpha_1 + \alpha_2=s_1(\alpha_2)$, and
$\u_{\alpha_2} = \g_{\alpha_2} \oplus \g_{\alpha_1+ \alpha_2} \oplus
\g_{2\alpha_1 + \alpha_2}$, where
$\dim \g_{\alpha_2} = 1 = \dim \g_{2\alpha_1 + \alpha_2}$; the integer
$a=\dim \g_{\alpha_1+ \alpha_2} $ is not necessarily equal to~$1$.  For the
symplectic group $\Sp(4,\R)$, $a=1$, for $\SU(2,2)$, $a=2$, for $\SO^*(8)$,
$ a= 4$, and for $\SO(2,m)$, $a = m-2$ (in fact this last family contains, up
to isogeny, the first three examples, cf.~\cite[p.~519]{Helgason}). In this
case, $\mathfrak{u}_{\alpha_2}$ is equipped with a quadratic form of signature
$(1,a+1)$ that is invariant under the adjoint action of the Levi factor (see Section~\ref{sec:class-symm-cones} for more details).

For a Hermitian Lie group of tube type of real rank~$3$ the roots that are
equal to~$\alpha_3$ modulo the span of $\{\alpha_1, \alpha_2\}$ are
$\alpha_3, \alpha_2 + \alpha_3, 2\alpha_2 + \alpha_3=s_2(\alpha_3), \alpha_1 +
\alpha_2 + \alpha_3=s_1(\alpha_2+\alpha_3), \alpha_1+ 2\alpha_2 +
\alpha_3=s_2s_1(\alpha_2+\alpha_3), 2\alpha_1+ 2\alpha_2 + \alpha_3=s_1
s_2(\alpha_3)$, so
$\u_{\alpha_3} =\g_{\alpha_3}\oplus \g_{\alpha_2 + \alpha_3}\oplus
\g_{\alpha_1 + \alpha_2 + \alpha_3} \oplus \g_{\alpha_1+ 2\alpha_2 + \alpha_3}
\oplus \g_{2\alpha_2 + \alpha_3} \oplus \g_{2\alpha_1+ 2\alpha_2 + \alpha_3}$,
where
the spaces $\g_{\alpha_3}$, $\g_{2\alpha_2 + \alpha_3}$, and $\g_{2\alpha_1+
  2\alpha_2 + \alpha_3}$ are $1$-dimensional and the spaces $\g_{\alpha_1+
  \alpha_2 + \alpha_3}$, $\g_{\alpha_2 + \alpha_3}$, and $\g_{\alpha_1+
  2\alpha_2 + \alpha_3}$ have dimension~$a$. For the symplectic
group $\Sp(6,\R)$, $a=1$, for $\SU(3,3)$, $a=2$, for $\SO^*(12)$, $ a= 4$, and
for $E_{7(-25)}$, $a= 8$.  In these cases, the Lie algebra $\u_{\alpha_3}$ can
be identified with the space of Hermitian $3\times 3$ matrices, over the reals
($a=1$), the complex numbers ($a=2$), the quaternions ($a=4$), and the
octonions ($a=8$) respectively.

\section{\texorpdfstring{$\Theta$}{Θ}-positive structures}
\label{sec:theta-posit-st}

In this section we give the definition of $\Theta$-positive structures and their
classification. We also prove that the notion of $\Theta$-positivity
reduces to the case of simple Lie groups so that we will only focus on simple
Lie groups later on. The first algebraic properties of
$\Theta$-positivity are then derived: the existence of a ``special'' root and
of a subgroup of Hermitian tube type (when $\Theta\neq \Delta$), and the
commutativity of the $\mathfrak{u}_\alpha$ ($\alpha\in\Theta$).

\subsection{Definition and characterization} 

Recall that a convex cone is \emph{acute} if the only vector space
contained in its closure is~$\{0\}$ and is \emph{nontrivial} if it contains a nonzero vector. 

\begin{definition}
  \label{defi:theta-pos-structure}
  Let $G$ be a semisimple Lie group with finite center.
  Let 
  $\Theta \subset \Delta$ be a subset of simple roots such that $G$~acts
  virtually faithfully on the flag variety~$\mathsf{F}_\Theta$.  We say that $G$~admits
  a \emph{$\Theta$-positive structure} if for all~$\beta$ in~$\Theta$ there
  exists an $L_{\Theta}^{\circ}$-invariant acute nontrivial convex cone in
  $\mathfrak{u}_\beta$.
\end{definition}

Note that such invariant cones must be of nonempty
interior since the action of~$L_{\Theta}^{\circ}$ on~$\mathfrak{u}_{\beta}$
is irreducible (Theorem~\ref{theo:kostant-results}). 

\begin{remarks}
  \begin{itemize}[leftmargin=*,nosep]
  \item In general it is not possible to request that the cones are invariant
    under the entire Levi factor~$L_\Theta$. If the space $\mathfrak{u}_\beta$
    contains an $L^{\circ}_{\Theta}$-invariant cone, then it contains exactly
    two $L^{\circ}_{\Theta}$-invariant cones, $c_\beta$ and $-c_\beta$ (cf.\
    point~(\ref{item2:prop:cone-semisimple}) in Proposition~\ref{prop:cone-semisimple} below). In
    some cases there is an element in $L_\Theta$ that interchanges $c_\beta$
    and $-c_\beta$, see Remark~\ref{rem:homog-under-group-L} for  more details.
      \item The
    faithfulness condition says that $\Theta_s=\Theta\cap \Delta_s$ is nonempty for
    every factor~$G_s$ of~$G$ 
    (cf.\ Section~\ref{sec:parabolic-subgroups}).
  \end{itemize}
\end{remarks}

\begin{proposition}\label{prop:semisimple}
  Let $\{G_s\}_{s\in \mathcal{S}}$ be the simple factors of~$G$ and use the
  notation introduced in Section~\ref{sec:struct-parab-subgr}:
  $\Theta_s\subset \Theta$, etc.

  Then $G$~admits a
  $\Theta$-positive structure if and only if, for all~$s$ in~$\mathcal{S}$, $G_s$ admits a
  $\Theta_s$-positive structure.   
\end{proposition}
\begin{proof}
  Let~$\beta$ be in~$\Theta$. There is thus~$s$ in~$\mathcal{S}$ such that
  $\beta$~belongs to~$\Theta_s$. For every~$t$ in~$\mathcal{S}$ distinct
  from~$s$, the factor~$L_{\Theta_t}^{\circ}$ acts trivially
  on~$\mathfrak{u}_\beta$; therefore $\mathfrak{u}_\beta$ is
  an irreducible summand for the action of~$L_{\Theta_s}$.
  
  This said, the irreducible representations~$\mathfrak{u}_\beta$
  of~$L_{\Theta}^{\circ}$ ($\beta\in \Theta$) are precisely the irreducible
  representations~$\mathfrak{u}_\beta$ of~$L_{\Theta_s}^{\circ}$ ($s\in
  \mathcal{S}$, $\beta\in \Theta_s$). The equivalence now follows directly from the definition.
\end{proof}

Due to Proposition~\ref{prop:semisimple}, a positive structure on a
semisimple Lie group is completely determined by the induced positive
structure on its simple factors. Therefore 
   structural results given in the sequel of the
  paper (for the unipotent positive semigroup, for the configurations of
  positive flags, etc.) will be stated only when the group~$G$ is simple but, thanks to the
  product structure of the flag variety and of the unipotent subgroups, 
 they immediately generalize to the case of semisimple Lie groups.

{\bf We assume from now on that $G$ is a simple Lie group.}
  
\begin{theorem}\label{thm:classification}\index{$G$!a simple Lie group
    (starting from Theorem~\ref{thm:classification})}
  Let~$G$ be a simple real Lie group, and let~$\Delta$ be its set of simple
  roots. Let~$\Theta$ be a nonempty subset of~$\Delta$. Then $G$~admits a $\Theta$-positive structure if and only if
  the pair $(G, \Theta)$ belongs to one of the following four cases:
  \begin{enumerate}[leftmargin=*]
  \item $G$ is a split real form, and $\Theta = \Delta$;
  \item $G$ is Hermitian of tube type and of real rank~$r$ and $\Theta = \{ \alpha_r\}$, where $\alpha_r$~is the long simple restricted root;
  \item $G$ is locally isomorphic to $\SO(p+1,p+k)$, $p>1$, $k>1$ and
    $\Theta = \{ \alpha_1, \dots, \alpha_{p}\}$, where $\alpha_1, \dots,
    \alpha_{p}$ are the long simple restricted roots; 
  \item $G$ is locally isomorphic to the real form of $\lietype{F}_4$, $\lietype{E}_6$, $\lietype{E}_7$, or
    of $\lietype{E}_8$ whose system of restricted roots
    is of type~$\lietype{F}_4$, and $\Theta = \{ \alpha_1, \alpha_2\}$, where $\alpha_1,
    \alpha_{2}$ are the long simple restricted roots.
  \end{enumerate}
\end{theorem}
\begin{remarks}\label{rem:split-and-2-str}
  \begin{itemize}[leftmargin=*,nosep]
  \item In all cases, one has $\iota(\Theta)= \Theta$ so that
    $\mathsf{F}_{\iota(\Theta)} = \mathsf{F}_\Theta$ and the notion of
    transversality makes sense in~$\mathsf{F}_\Theta$. This will be relevant
    later when introducing positive tuples in~$\mathsf{F}_\Theta$ (Section~\ref{sec:posit-flag-vari}).
  \item When $G$~is a split real form and $\Theta = \Delta$, the positive
    structure agrees with Lusztig's total positivity in split real Lie groups.
  \item 
    A split real group of type $\lietype{B}_n,$ $\lietype{C}_n$, or $\lietype{F}_4$ appears
    twice in this list of groups and carries two different positive
    structures, one corresponding to $\Theta = \Delta$ and one with respect to
    a proper subset $\Theta \subset \Delta$. The relation between these two
    different positive structures will be discussed later
    (Examples~\ref{exam:nonnegativeSplitHermitian} and Section~\ref{sec:comp-posit-struct}).
  \end{itemize}
\end{remarks}

\subsection{Even and proximal representations}
\label{sec:repr-weights}
To prove Theorem~\ref{thm:classification}, we recall the notions, for representations of reductive Lie groups, of even
representations and of proximal representations and their relation with having
an invariant cone. 
We will use these
notions mainly for representations of the semisimple factor of the Levi
components of parabolic subgroups, see
Section~\ref{sec:proof-theo-classification} below.

Any (continuous) representation $\rho\colon G \to \GL(V)$ induces a Lie
algebra morphism
$\mathfrak{\rho}_*\colon \mathfrak{g}\to \End(V)$. The restriction of $\mathfrak{\rho}_*$ 
to~$\mathfrak{a}$ admits a weight space decomposition\index{$V_\Lambda$, the weight spaces in a $G$-module
  $V$} \index{$R(V)$ the weights of the $G$-module $V$}
\begin{align*}
  V&=\bigoplus_{\mathclap{\lambda\in R(V)}}
  V_\lambda
  \intertext{where}
  V_\lambda & \coloneqq \{ v \in V \mid  \mathfrak{\rho}_* (H) (v) = \lambda
              (H)v \text{ for all } H \in \mathfrak{a}\}, \text{ and}\\
R(V) & \coloneqq \{ \lambda \in \mathfrak{a}^* \mid V_\lambda \neq \{0\} \}.
\end{align*}

The representation~$\rho$ will be called \emph{even} if
$\lambda(H_\alpha)$ is an even integer for all $\lambda\in R(V)$ and all
$\alpha\in\Sigma$ (these numbers are already known to be integers). In
classical notation this means the inclusion $R(V)\subset 2P$ where $P =\{ \lambda
\in \mathfrak{a}^* \mid \lambda(H_\alpha)\in \ZZ, \, \forall
\alpha\in\Sigma\}$ is
the weight lattice. Since $\mathfrak{g}$~is generated by~$\mathfrak{g}_0$ and
the root spaces~$\mathfrak{g}_{\pm \alpha}$ ($\alpha\in\Delta$), $\rho$~is
even if and only if $\lambda(H_\alpha)\in 2\ZZ$ for every~$\lambda$ in~$R(V)$
and every~$\alpha$ in~$\Delta$.

The
\emph{highest weight} of the representation~$\rho$ (when it is defined) is the
greatest element of~$R(V)$ for the (partial) order on~$\mathfrak{a}^*$ defined by
$\lambda\leq \mu$ $\Leftrightarrow$ $\lambda( H_\alpha)\leq \mu(H_\alpha)$ for
all $\alpha\in \Sigma^+$. An irreducible representation~$\tau$ always admit a highest
weight~$\lambda_{\tau}$ and such an irreducible representation is even if and only if
$\lambda_{\tau}(H_\alpha)\in 2\ZZ$ for every $\alpha\in \Delta$. 

A representation will be called \emph{proximal} if it admits a 
highest weight and if the corresponding weight space is $1$-dimensional.

To prove Theorem~\ref{thm:classification} we make use of the following fact, extracted from
\cite[Proposition~4.7]{Benoist_autcon} that refines Cartan--Helgason's theorem.
\begin{proposition}\label{prop:cone-semisimple}
  \begin{enumerate}[leftmargin=*]
  \item\label{item1:prop:cone-semisimple} Let~$S$ be a connected semisimple
    Lie group, $V$ a real vector space, and $\pi\colon S \to \GL(V)$ an
    irreducible representation.  Then $S$~preserves an acute closed convex
    cone~$c$ in~$V$ if and only if the representation~$\pi$ is proximal and
    even.
  \item \label{item2:prop:cone-semisimple} In this case, the union of
    $\mathring{c}$ and of $-\mathring{c}$ is exactly the set of vectors
    $v\in V\smallsetminus\{0\}$ whose stabilizers contains a maximal compact
    subgroup of~$S$; likewise a vector $v\in V \smallsetminus\{0\}$ belongs to
    $\mathring{c}$ or to $-\mathring{c}$ as soon as the Lie algebra of its
    stabilizer contains a maximal compact subalgebra.
  \end{enumerate}
\end{proposition}

In order to apply Proposition~\ref{prop:cone-semisimple} to the semisimple
part of a reductive Lie group, we establish the following result.

\begin{proposition}\label{prop:cone-reductive}
  Let $L$ 
  be a connected reductive Lie group and let $S=[L,L]$ be its semisimple part.
  Let 
  $\pi\colon L\to \GL(V)$ be an irreducible real representation.
  \begin{enumerate}[leftmargin=*]
  \item\label{item:1:prop:cone-reductive} Suppose that there is an acute $\pi(L)$-invariant cone in~$V$, then
    the restriction $\pi|_S$ is as well irreducible (and hence proximal and
    even by the previous proposition).
  \item\label{item:2:prop:cone-reductive} Suppose that $\pi|_S$ is irreducible, even, and proximal (and hence
    there is $c\subset V$ a $\pi(S)$-invariant acute closed convex cone by the previous
    proposition), then the cone~$c$ is equally $\pi(L)$-invariant.
  \end{enumerate}
\end{proposition}

\begin{proof}
  The group $L$ is the almost direct product of~$Z^\circ$, the identity component of
  its center~$Z$, and of its semisimple
  part~$S$. 

  Let us prove~(\ref{item:1:prop:cone-reductive}).  By Schur's lemma, the algebra $\End_L(V)$ of $L$-equivariant endomorphisms
  of~$V$ is a skew field; 
  in other words, $V$~is a vector
  space over a potentially bigger field and the endomorphisms~$\pi(g)$, for
  $g\in L$, are linear with respect to this field. 
  Since there is a $\pi(L)$-invariant acute convex cone, the representation~$\pi$ is
  proximal \cite[Proposition~1.2]{Benoist_autcon} in the sense that there is
  an element~$g$ in~$L$ such that $\pi(g)$ has a unique eigenvalue of highest
  modulus: this eigenvalue is real and the corresponding eigenspace is of real
  dimension~$1$. As this eigenspace is a vector space under the skew field $\End_L(V)$, we deduce
  that $\End_L(V)=\R$. 
  Furthermore, the image $\pi(Z^\circ)$ is
  contained in $\End_L(V)$ and this means that elements~$\pi(g)$, for $g\in
  Z^\circ$, are homotheties of~$V$. We deduce from this and the fact that $L$
  is the almost product of~$Z^\circ$ and~$S$, that a subspace of~$V$
  is $\pi(L)$-invariant if and only if it is $\pi(S)$-invariant. Since $\pi$
  is by hypothesis an irreducible representation of~$L$, the restriction
  $\pi|_S$ is thus an irreducible representation of~$S$. 

  We now address~\eqref{item:2:prop:cone-reductive}. We need to show that $c$ is
  invariant under the action of~$\pi(Z^\circ)$. Similarly to above, 
  the elements~$\pi(z)$, for $z\in Z^\circ$, are homotheties of
  positive dilation so that  $\pi(z)c=c$ for any $z\in Z^\circ$ since $c$~is a
  cone. This implies
  that $c$~is $\pi(L)$-invariant.
\end{proof}

\subsection{The proof of the Classification Theorem~\ref{thm:classification}}
\label{sec:proof-theo-classification}

  We apply Proposition~\ref{prop:cone-reductive} to the connected Lie group
  $L_{\Theta}^\circ$ and Proposition~\ref{prop:cone-semisimple}.(\ref{item1:prop:cone-semisimple})
  to~$S_{\Theta}^{\circ} =[L_{\Theta}^{\circ}, L_{\Theta}^{\circ}]$
  for the irreducible representations
  $\pi_\beta \colon L_{\Theta}^\circ \to \GL(\mathfrak{u}_\beta)$, $\beta \in
  \Theta$.
  
  Recall that a Cartan subspace for~$S^{\circ}_{\Theta}$ is $\mathfrak{a}_\Theta =
  \bigoplus_{\alpha\in \Delta\smallsetminus \Theta} \R H_\alpha$. For this
  proof it will be more convenient to work with the following Weyl chamber of~$S_{\Theta}^{\circ}$
  \[\sum_{\mathclap{\alpha\in \Delta\smallsetminus \Theta}} - \R_{>0} H_\alpha\] that
  is opposite to the natural one, or, what amounts to the same, we
  shall
  consider \emph{lowest weights} rather than highest weights of
  representations (for the natural ordering).

  The decomposition of Equation~\eqref{eq:u_beta_sum_g_alpha_v2}, p.~\pageref{eq:u_beta_sum_g_alpha_v2}, is
  \[ \mathfrak{u}_\beta = \mathfrak{g}_\beta \oplus \!\!\!\!\!\!\!\!
    \bigoplus_{\substack{\alpha \in \Sigma,\ \alpha\neq \beta \\
        \alpha-\beta\in \Span( \Delta\smallsetminus \Theta) }} \!\!\!\!\!\!\!\! \mathfrak{g}_{\alpha}.
  \]
Any~$\alpha$ in the above formula
  differs from~$\beta$ by a linear combination of the elements in $\Delta
  \smallsetminus \Theta$ with nonnegative coefficients. This implies, up
  to the abuse of considering the above decomposition indexed by the elements in
  $\mathfrak{a}_{\Theta}^{*}$ rather than by elements in~$\mathfrak{a}^*$, that the
  above is the weight space decomposition of the $S_{\Theta}^{\circ}$-module
  $\mathfrak{u}_\beta$. As a consequence, the lowest weight of this
  representation is equal to~$\beta$ and the corresponding weight space is~$\mathfrak{g}_\beta$.

  Thus the condition that
  $\pi_\beta$ has to be proximal is equivalent to the condition that the root space
  $\mathfrak{g}_\beta$ is $1$-dimensional. This needs to hold for all $\beta \in \Theta$.

  The evenness condition of Proposition~\ref{prop:cone-reductive} translates into the fact that $\beta(H_\alpha)$ is
  even for all~$\beta\in \Theta$ and all $\alpha\in \Delta\smallsetminus
  \Theta$ (cf.\ Section~\ref{sec:repr-weights}).

  These necessary and sufficient criteria can be easily translated into
  properties of the Dynkin diagram of the system of restricted roots $\Sigma$:
  \begin{enumerate}
  \item For all~$\beta$ in~$\Theta$, the root space~$\mathfrak{g}_\beta$ is
    $1$-dimensional. (\emph{Representation is proximal}.)
  \item For all~$\beta$ in~$\Theta$, the node of the Dynkin diagram with label~$\beta$ is either connected to the nodes in $\Delta \smallsetminus \Theta$ by a double
    arrow pointing towards $\Delta \smallsetminus \Theta$, or it is not
    connected to
    $\Delta \smallsetminus \Theta$ at all. (\emph{Representation is even}.)
  \end{enumerate}
  
  Let us consider the two cases,  $\Theta=\Delta$  and  $\Theta \neq \Delta$. 

  In the first case, when $\Theta=\Delta$, the evenness condition is empty. The multiplicity one
  condition implies that for every simple root $\alpha \in \Delta$ the root
  space $\mathfrak{g}_\alpha$ is of dimension one. This implies that $G$ is a
  split real form.

  In the case when 
  $\Theta \neq \Delta$, the evenness condition implies that every arrow from~$\Theta$ to
  $\Delta\smallsetminus \Theta$ has to be a double arrow, pointing from
  $\Theta$ to $\Delta\smallsetminus \Theta$. By property of Dynkin diagrams,
  there can be at most one double arrow in the Dynkin diagram~$\Delta$ \cite[Ch.~X, Lemma~3.18]{Helgason}, and $\Theta$ must
  be the set of long simple roots.
  
  Going through the Dynkin diagrams of the system of restricted roots
  associated with
  simple real Lie groups (see \eg \cite[Ch.~X, \S~6 and Table~VI p.~532-4]{Helgason}) we
  can now read off the 
  list of pairs $(G,\Theta)$. This completes the proof of Theorem~\ref{thm:classification}.

\subsection{The special root}
\label{sec:special-root}
  
By the classification of Dynkin diagrams (more accurately, from the fact that
a connected Dynkin diagram admits at most one double arrow), we observe that for all pairs $(G,
\Theta)$ admitting a $\Theta$-positive structures, both sets $\Theta$ and
$\Delta \smallsetminus \Theta$ are connected.

When $\Theta \neq \Delta$, there
is thus a unique root~$\alpha_\Theta$ in~$\Theta$ which is connected to
$\Delta \smallsetminus \Theta$. The root $\alpha_\Theta$ will be called the
\emph{special root}.\index{$\alpha_\Theta$ the special root (when $\Theta\neq \Delta$)} For all $\alpha \in \Theta\smallsetminus \{ \alpha_\Theta\}$ we have that $\u_\alpha = \g_\alpha$ and $\dim \mathfrak{u}_\alpha = 1$. 
For~$\alpha_\Theta$, the vector space
\[\u_{\alpha_\Theta}= \!\!\! \bigoplus_{\substack{\alpha \in \Sigma^+_\Theta, \\
  \alpha|_{\mathfrak{t}_\Theta} = \alpha_\Theta|_{\mathfrak{t}_\Theta} }} \!\!\!  \mathfrak{g}_{\alpha}\] is of
dimension greater than or equal to~$2$ since there is a least one other root
in~$\Sigma$ congruent to~$\alpha_\Theta$ modulo $\Delta\smallsetminus \Theta$. 

\subsection{The Hermitian tube type subgroup}
\label{sec:hermtian-tube-type}

Elaborating on the special root and again under the assumption $\Theta\neq
\Delta$, we notice that the subgraph of the Dynkin diagram with vertices $\{\alpha_\Theta\}\cup
\Delta\smallsetminus\Theta$ is of type~$\lietype{C}_{d+1}$ (where $d$~is the cardinality of
$\Delta\smallsetminus\Theta$). This is the type of the semisimple part
$S_{\Theta\smallsetminus \{\alpha_\Theta\}}$ of the Levi factor
$L_{\Theta\smallsetminus \{\alpha_\Theta\}}$. One has then the following

\begin{proposition}
\label{prop:hermtian-tube-type}
  Let~$G$ be a simple Lie group admitting a $\Theta$-positive structure with
  $\Theta\neq \Delta$. Then the subgroup $S_{\Theta\smallsetminus
    \{\alpha_\Theta\}}$ is Hermitian of tube type. Moreover
  $\mathfrak{u}_{\alpha_\Theta}$ is contained in the Lie algebra of
  $S_{\Theta\smallsetminus \{\alpha_\Theta\}}$ and is the Lie algebra of the
  unipotent radical of the maximal parabolic subgroup $P_\Theta \cap S_{\Theta\smallsetminus
    \{\alpha_\Theta\}}$ of $S_{\Theta\smallsetminus
    \{\alpha_\Theta\}}$ associated to the root~$\alpha_\Theta$; a Levi factor
  of this parabolic subgroup of $S_{\Theta\smallsetminus
    \{\alpha_\Theta\}}$ is $L_\Theta \cap S_{\Theta\smallsetminus
    \{\alpha_\Theta\}}$. 
\end{proposition}

Note that the maximal parabolic  $P_\Theta \cap S_{\Theta\smallsetminus
    \{\alpha_\Theta\}}$ is the stabilizer of a point in the Shilov boundary of the Hermitian symmetric space associated to $S_{\Theta\smallsetminus
    \{\alpha_\Theta\}}$. 
    
\begin{proof}
  Only the statement about~$\mathfrak{u}_{\alpha_\Theta}$ needs to be
  established.

  Let~$\mathfrak{s}$ be the Lie algebra of
  $S_{\Theta\smallsetminus \{\alpha_\Theta\}}$.
  Clearly, $\mathfrak{g}_{\alpha_\Theta}$ is contained in~$\mathfrak{s}$. Furthermore, $L_\Theta\subset
  L_{\Theta\smallsetminus\{ \alpha_\Theta\}}$ and hence $S_\Theta$ is
  contained in~$S_{\Theta\smallsetminus \{\alpha_\Theta\}}$. Thus the
  $S_\Theta$-module generated by~$\mathfrak{g}_{\alpha_\Theta}$ is contained
  in~$\mathfrak{s}$. By
  Kostant's result (Theorem~\ref{theo:kostant-results}) and by
  Proposition~\ref{prop:cone-reductive}, the $S_\Theta$-module~$\mathfrak{u}_{\alpha_\Theta}$ is irreducible. As
  $\mathfrak{u}_{\alpha_\Theta}$ contains~$\mathfrak{g}_{\alpha_\Theta}$, we
  obtain that the $S_\Theta$-module generated
  by~$\mathfrak{g}_{\alpha_\Theta}$ is equal to~$\mathfrak{u}_{\alpha_\Theta}$
  and therefore that $\mathfrak{u}_{\alpha_\Theta}$~is contained in~$\mathfrak{s}$.

  The Hermitian tube type Lie group $S_{\Theta\smallsetminus
    \{\alpha_\Theta\}}$ has a positive structure associated with the subset
  $\Theta_S=\{ \alpha_\Theta \}$ of its simple roots; in this case the equality
  $\mathfrak{u}_{\Theta_S} = \mathfrak{u}_{\alpha_\Theta}$
  holds.
\end{proof}

\subsection{The \texorpdfstring{$\mathfrak{u}_\alpha$}{uα} are Abelian}
\label{sec:u_alpha-are-abel}

\begin{proposition}
\label{prop:u_alpha-are-abel}
  Let~$G$ be a simple Lie group admitting a $\Theta$-positive structure. Then,
  for every~$\alpha$ in~$\Theta$, $\mathfrak{u}_\alpha$~is Abelian.
\end{proposition}

\begin{proof}
  When $\Theta=\Delta$, all the $\mathfrak{u}_\alpha$ are $1$-dimensional and
  therefore Abelian.

  When $\Theta\neq \Delta$, for all~$\alpha$
  in~$\Theta\smallsetminus \{ \alpha_\Theta\}$, $\mathfrak{u}_\alpha$ is
  $1$-dimensional and Abelian. For the special root $\alpha_\Theta$ the space
  $\mathfrak{u}_{\alpha_\Theta}$ is equal to the unipotent factor of the
  parabolic subalgebra associated to the Shilov boundary of the Hermitian
  subgroup (Proposition~\ref{prop:hermtian-tube-type}) and it is known that
  this unipotent algebra is Abelian.
\end{proof}

\subsection{The case of Hermitian Lie groups that are not of tube type}
\label{sec:case-hermitian-not-tube}
Note that simple Hermitian Lie groups which are not of
tube type do not appear on the list of groups admitting a positive
structure. This is because in that case the root system is non-reduced of type
$\lietype{BC}_n$, and the proximality condition is not satisfied, for every simple
root~$\alpha$, the root space $\mathfrak{g}_\alpha$ is of dimension
$>1$. 
In this case however, with $\alpha$ is the simple root such that $2\alpha$
is also a root, we have at our hand the space $\mathfrak{u}_{2\alpha}$, and the representation of $L_{\Theta}^\circ$ on $\mathfrak{u}_{2\alpha}$ is proximal (the root space $\mathfrak{g}_{2\alpha}$ is $1$-dimensional) and even.

\section{The nonnegative and positive semigroups} \label{sec:nonn-posit-semigr}
 In this section, we fix a pair $(G,\Theta)$ where $G$~is a simple real Lie group which admits a $\Theta$-positive
structure; we introduce the nonnegative semigroup $U^{\geq 0}_\Theta $ and the
positive semigroup $U^{> 0}_\Theta $ in the unipotent radical~$U_\Theta$ of
the parabolic subgroup $P_\Theta$. Their main properties will be
established in Section~\ref{sec:positive_semigroup}.

For every $\alpha \in \Theta$ we fix an $L_{\Theta}^\circ$-invariant closed
cone~$c_\alpha$ in~$\mathfrak{u}_\alpha$, and denote by $\mathring{c}_\alpha$ its
interior.\index{$c_\alpha$ an $L_{\Theta}^\circ$-invariant closed
cone in~$\mathfrak{u}_\alpha$}\index{$\mathring{c}_\alpha$ the interior of the $L_{\Theta}^\circ$-invariant closed
cone in~$\mathfrak{u}_\alpha$} 
For~$\alpha\in \Theta$, the cone~$\tau( -c_\alpha)$ is contained in
$\mathfrak{u}_{-\alpha}$ ($\tau$~is  the Cartan involution) and will be denoted 
by $c^{\mathrm{opp}}_\alpha$. \index{$c_{\alpha}^{\mathrm{opp}}$ the $L_{\Theta}^\circ$-invariant closed
cone in~$\mathfrak{u}_{\alpha}^{\mathrm{opp}}$ equal to $-\tau(c_\alpha)$}

\subsection{The nonnegative semigroup}
\label{sec:definition}

\begin{definition}\label{def:nonnegative}  The \emph{unipotent nonnegative
    semigroup} $U_{\Theta}^{\geq 0} \subset U_\Theta$
  is\index{$U_{\Theta}^{\geq 0}$ the unipotent nonnegative semigroup
    in  $ U_\Theta$}
  the subsemigroup of $U_\Theta$ generated by $\bigcup_{\alpha \in
    \Theta}\exp(c_\alpha)$.  The
  \emph{opposite unipotent nonnegative semigroup} \index{$U_{\Theta}^{\mathrm{opp}\geq 0}$ the {unipotent nonnegative semigroup}
    in  $ U_{\Theta}^{\mathrm{opp}}$}
  $U^{\mathrm{opp},\geq 0}_{\Theta} 
  $ is the
  subsemigroup generated by $\bigcup_{\alpha \in
    \Theta} \exp(c^{\mathrm{opp}}_\alpha)$.
\end{definition}

    We have 
    $U_{\Theta}^{\mathrm{opp}\geq 0}= \tau( U_{\Theta}^{\geq 0})^{-1}$ (where
    $\tau$~denotes also the group homomorphism $G\to G$ tangent to the
    involution $\tau\colon \mathfrak{g}\to \mathfrak{g}$). 
Note that it is not clear from the definition that the unipotent nonnegative semigroup is closed, this will be proven in Section~\ref{sec:positive_semigroup}.

The unipotent nonnegative semigroups give rise to a 
semigroup in~$G$. 

\begin{definition}\label{def:nonnegative_in_G}
  The \emph{nonnegative semigroup} $G^{\geq 0}_\Theta \subset G$ is defined
  to be the subsemigroup generated by $U_{\Theta}^{\geq 0}$,\index{$G^{\geq
      0}$ the nonnegative semigroup in~$G$}
  $U^{\mathrm{opp},\geq 0}_{\Theta}$, and $L_{\Theta}^{\circ}$.
\end{definition}

\begin{examples}\label{exam:nonnegativeSplitHermitian}
  \begin{enumerate}[leftmargin=*]
  \item When $G$ is a split real form, and $\Theta = \Delta$, then $U_{\Theta}^{\geq 0}$ is the nonnegative unipotent semigroup defined by Lusztig in \cite{lusztigposred}, and 
    $G^{\geq 0}_\Theta = G^{\geq 0 }$ is the semigroup of totally nonnegative
    elements defined by Lusztig \cite[\S~2.2]{lusztigposred}. 
    In the case when $G$ is $\SL_n(\R)$, the nonnegative semigroup $G^{\geq 0}_\Theta$ 
    consists of the matrices, all of whose minors are nonnegative. Similar characterizations for other split real Lie groups in terms of generalized minors exist. 
    
  \item Let~$G$ be a Hermitian Lie group of tube type and of real rank~$r$,
    and let
    $\Theta = \{ \alpha_r\}$. Then  $U_{\Theta}^{\geq 0} = \exp(c_{\alpha_r}) \subset U_\Theta$, and  $G^{\geq 0}_\Theta$ is the contraction
    semigroup $G^{\succ 0} \subset G$ \cite[Th.~4.9]{Koufany}. 
  
    In particular note that when $G$~is the real symplectic group of rank~$r$, the
    unique (up to isogeny) rank~$r$ group which is split and of Hermitian type, the
   nonnegative semigroups $G^{\geq 0}_\Delta$ and
    $G^{\geq 0}_{\{\alpha_r\}}$ are different when $r>1$.
In this case the nonnegative semigroup $G^{\geq 0}_{\{\alpha_r\}}$ 
    admits a nice characterization. 
  The symplectic group $\Sp(2n,\R)$ consists of block matrices 
$ \bigl(\begin{smallmatrix} A &B\\C&D\end{smallmatrix}\bigr)$ in $\GL_{2n}(\R)$ such that
${}^t D A - {}^t B C= \id$, and the matrices ${}^t A C$ and ${}^t DB$ are symmetric.

Such a matrix 
is then in the nonnegative semigroup 
$\Sp(2n,\R)^{\geq 0}$ 
if and only if $\det(A) >0$ and that the (necessarily symmetric) matrices $C A^{-1}$ and
$A^{-1}B$ are
semipositive.

    \item For the case when $G$~is a split real group locally isomorphic to
      $\mathrm{SO}(n,n+1)$ with $n>1$, we have two different semigroups,  the nonnegative semigroup $G_{\Theta}^{\geq 0}$, where $\Theta =\Delta\smallsetminus \{\alpha_n\}$ and Lusztig's nonnegative semigroup $G^{\geq 0}_\Delta$. 
No direct characterizations of the nonnegative semigroup $G_{\Theta}^{\geq 0}$ are known yet. 
      
    \item Also in the case when $G$~is the split real form of $\lietype{F}_4$ we have two different semigroups, the nonnegative semigroup $G_{\Theta}^{\geq 0}$, where $\Theta = \{\alpha_1, \alpha_2\}$ and Lusztig's nonnegative semigroup $G^{\geq 0}_\Delta$. 
    No direct characterizations of the nonnegative semigroup $G_{\Theta}^{\geq 0}$ are known yet. 
  \end{enumerate}
\end{examples}

\subsection{The positive semigroup} \label{sec:positive-semigroup}

\begin{definition}\label{def:positive}  
The interior (relative to~$U_\Theta$) of the nonnegative
semigroup\index{$U_{\Theta}^{>0}$ the unipotent positive semigroup in $U_\Theta$} $U_\Theta^{\geq}$
is called the \emph{unipotent positive semigroup} and denoted $U_{\Theta}^{> 0}$. 
The interior of the opposite nonnegative semigroup $
U_\Theta^{\mathrm{opp},\geq 0}$ \index{$U_{\Theta}^{\mathrm{opp}>0}$ the opposite unipotent positive semigroup in $U_{\Theta}^{\mathrm{opp}}$}
 is called the  \emph{opposite unipotent positive semigroup} and denoted $U_{\Theta}^{\mathrm{opp},> 0}$. 
\end{definition}

It is clear that $U_{\Theta}^{> 0}$ is a semigroup (cf.\
Lemma~\ref{lemm:semigroups} below), but this definition 
gives 
little information about the structure of the positive
semigroup. 
Already in Lusztig's work, the positive unipotent
semigroup is not introduced in this way, but is defined given explicit
parametrizations of a subset~$U^{>0}$ of~$U$, which is then shown to be a
semigroup and to agree with the interior of $U^{\geq 0}$. We will give explicit
parametrizations 
of the unipotent positive semigroup $U_{\Theta}^{> 0}$ in Section~\ref{sec:positive_semigroup}. 

In the case of Hermitian Lie groups of tupe type the situation simplifies
drastically, because then $\Theta = \{\alpha_\Theta\}$, $U_\Theta^{\geq 0} =
\exp(c_{\alpha_\Theta})$, and thus $U_\Theta^{>0} =
\exp(\mathring{c}_{\alpha_\Theta})$. 

Similarly to the nonnegative semigroup in~$G$, there is also a 
positive semigroup in~$G$. 

\begin{definition}\label{def:positive_in_G}
  The interior of~$G^{\geq 0}_\Theta$  is called the \emph{positive semigroup}
  and denoted~$G^{> 0}_\Theta$.\index{$G^{> 0}_\Theta$ the positive semigroup
    in $G$} 
  \end{definition}

In the case when $G = \mathrm{SL}(n,\R)$ the semigroup $G^{\geq 0}_\Delta$
(respectively $G_{\Delta}^{>0}$) is precisely the semigroup of nonnegative
(respectively totally positive) matrices.  A general geometric description of
the semigroups~$G^{\geq 0}_\Theta $ and~$G_\Theta^{>0}$ is not known and
should be the subject of further investigation. In this
article we focus on the nonnegative and positive unipotent semigroups. The nonnegative and
positive semigroups $G^{\geq 0}_\Theta$, respectively $G_\Theta^{>0}$, will be
further investigated in \cite{GW_pos_2} where it is shown that the positive
semigroup $G_\Theta^{>0}$ 
is the direct product
$U^{\mathrm{opp},>
  0}_{\Theta}\cdot L_{\Theta}^{\circ}\cdot U_{\Theta}^{> 0}$, and further
geometric properties of its elements are established, in particular
\cite[Conjecture~4.10]{GW_ECM}, stating that an element of $G_{\Theta}^{>0}$
acts proximally on~$\mathsf{F}_\Theta$, will be proved there. 

\subsection{Some facts about topological semigroups}
\label{sec:some-facts-about}

Let us note that semigroups contained in a topological group satisfy a number of immediate
properties:

\begin{lemma}
  \label{lemm:semigroups}
  Let~$G$ be a topological group. Let~$U$ be a semigroup contained
  in~$G$.\index{$G$!A topological group (in Section~\ref{sec:some-facts-about})}
  \begin{enumerate}[leftmargin=*]
  \item \label{item:1:lemm:semigroups:clos-inter} The closure $\overline{U}$
    and the interior $\mathring{U}$ are also semigroups.
  \item \label{item:4:lemm:semigroups-equiv} Let~$\ell$ be in~$G$ such that
    $\ell U \ell^{-1} =U$, then $\ell \overline{U} \ell^{-1} =\overline{U}$
    and  $\ell \mathring{U} \ell^{-1} =\mathring{U}$.
  \item \label{item:2:lemm:semigroups}  If $U$~is open, the following inclusions hold
    $\overline{U} {U} \subset {U}$ and
    $ {U} \overline{U} \subset {U}$. In the case when furthermore $\overline{U}$
    contains the neutral element, equalities hold: $\overline{U} {U} = {U}= {U} \overline{U}$.
  \item \label{item:3:lemm:semigroups-inter-of-closure}
    If $U$~is open and if $e$~belongs to~$\overline{U}$, then
    $U$~is equal to the interior of~$\overline{U}$.
  \end{enumerate}
\end{lemma}
\begin{proof}
  Points~(\ref{item:1:lemm:semigroups:clos-inter})
  and~(\ref{item:4:lemm:semigroups-equiv}) are consequences of the continuity
  of the law in~$G$.

  We now assume that $U$~is open.
    
  Let~$u$ be in~$\overline{U}$ and $v$~be in~${U}$. There exists an
  neighborhood~$\mathcal{N}$ of~$e$ in~$G$ such that
  $\mathcal{N}v\subset {U}$. Since $\mathcal{N}^{-1}$ is also a
  neighborhood of~$e$, $u\mathcal{N}^{-1}$ is a neighborhood of~$u$ and
  contains an element~$u'$ of $U$. The element~$v'$ of~$G$ such that $uv=u'v'$
  belongs to $\mathcal{N}v$ and hence to~$U$. As $U$~is a semigroup, $u'v'$~is
  in~$U$ and thus $uv$~is in~$U$. 
  This proves $\overline{U}U\subset U$ and similarly  $U\overline{U}\subset U$ follows.
  When $e$~belongs to~$\overline{U}$, the
  inclusions $U\subset \overline{U}U$ and $U\subset U \overline{U}$ are
  obvious. This proves point~(\ref{item:2:lemm:semigroups}).

  Suppose that $e\in \overline{U}$.
  Let~$V$ be the interior of~$\overline{U}$. Since $U$~is open, one has
  $U\subset V$. Let~$v$ be in~$V$.  Since $V$~is open, there is a neighborhood~$\mathcal{N}$
  of~$e$ such that $v\mathcal{N}\subset V$. By the
  assumption $e\in \overline{U}$, $U\cap \mathcal{N}^{-1}$ is not empty and
  let~$g$ be an element in this intersection. Then $vg^{-1}$ belongs to~$V$
  and thus to $\overline{U}$; therefore $v=(vg^{-1})g$ belongs to
  $\overline{U}U=U$. The equality $V=U$ is proved.
\end{proof}

\subsection{Hermitian Lie groups not of tube type}
When $G$~is a Hermitian Lie group that is not of tube type we saw that $G$~does not admit a positive structure. 
However,  
we can consider the simple root~$\alpha$ such that $2\alpha$ is a root. There is an  $L_{\Theta}^{\circ}$-invariant cone $c_{2\alpha}$ in $\mathfrak{u}_{2\alpha}$. We could thus consider the semigroups given by $\exp(c_{2\alpha})$, respectively $\exp(\mathring{c}_{2\alpha})$. These semigroups are  contained in the maximal subgroup of tube type in~$G$, and thus the discussion reduces to the case of Hermitian Lie groups of tube type. 
However, several of the properties of the positive semigroups $U_\Theta^{>0}$ established in Section~\ref{sec:positive_semigroup} in fact do not hold for the semigroup $\exp(\mathring{c}_{2\alpha})$ in Hermitian Lie groups that are not of tube type.

\section{The invariant cones} 
\label{sec:explicit-description}
In this section we fix a pair $(G,\Theta)$ such that
$G$ is a simple Lie group carrying a $\Theta$-positive structure and analyze
the invariant cones in $\mathfrak{u}_\alpha$ ($ \alpha \in \Theta$) in more
detail. 
In particular we introduce the notion of a $\Theta$-system, that will be used
in Section~\ref{sec:split-group-type} to generate a split subalgebra
of~$\mathfrak{g}$ related with the positive structure. 

\subsection{The description of the invariant cones}
\label{sec:class-symm-cones}
For every~$\alpha$ in $\Theta$, 
$\mathfrak{u}_\alpha$~contains
an $L_{\Theta}^\circ$-invariant closed convex 
cone, which we denote by $c_\alpha$. 
This cone is unique up to the action of~$-\id$ (cf.\ Proposition~\ref{prop:cone-semisimple}.(\ref{item2:prop:cone-semisimple})).
The interior of the cone $c_{\alpha} \subset
\mathfrak{u}_{\alpha}$ is denoted by~$\mathring{c}_\alpha$. We now describe these cones
explicitly.

When $\u_\alpha \cong \R$ (this happens for all~$\alpha$ in the case when
$\Theta=\Delta$ and for all~$\alpha\neq \alpha_\Theta$ in the case $\Theta\neq
\Delta$), the cone $c_\alpha$ can be identified with the cone $\R_{\geq 0} \subset \R$. 
Thus we now focus on the cone in $\mathfrak{u}_{\alpha_\Theta}$ (and
$\Theta\neq \Delta$).

Let us start with the case when $G$~is Hermitian of tube type. This means that $G$~is (up to
isogeny) $\Sp(2n, \R)$, $\SU(n,n)$, $\SO^*(4n)$, $\SO(2,q)$, or the exceptional group
$\lietype{E}_{7(-25)}$. The $3$~first families of classical groups can be uniformly described: 
Let~$\K$ be either~$\R$, $\C$, or~$\HH$, the space $V=\K^{2n}$ is a right
vector space and is endowed with a anti-Hermitian form defined
by\index{$\omega_\K$ the standard anti-Hermitian form on $\K^{2n}$ ($\K=\R$,
  $\C$, or $\HH$)}
\[
  \omega_\K(v,v')= \sum_{i=1}^{n} \bar{x}_i y_{i}^{\prime} - \bar{y}_i x_{i}^{\prime}
\]
with $v=(x_1, \dots, x_n, y_1, \dots, y_n)$ and $v'=(x'_1, \dots, x'_n, y'_1, \dots, 
y'_n)$. The subgroup~$G$ in $\GL_\K(V)$ of automorphisms of~$V$
stabilizing~$\omega_\K$ is Hermitian of tube type; it is $\Sp(2n,\R)$,
$\SU(n,n)$, or $\SO^*(4n)$ following that $\K$~is~$\R$, $\C$, or~$\HH$
respectively. 
The Levi factor~$L_\Theta$ in this case consists of the elements
\[ \ell_A =
  \begin{pmatrix}
    A & 0 \\ 0 & {}^t \overline{A}^{-1}
  \end{pmatrix}, \ A\in \GL_n(\K ),
\]
and the unipotent algebra $\mathfrak{u}_\Theta$ is equal to
$\mathfrak{u}_{\alpha_\Theta}$. It consists of the matrices
\[
u_M  =
  \begin{pmatrix}
    0 & M \\ 0 & 0
  \end{pmatrix}, \ M\in M_n(\K ) \text{ such that } {}^t\overline{M}=M;
\]
thus $\mathfrak{u}_{\alpha_\Theta}$~identifies with the space $\Herm(n,\K)$ of Hermitian matrices over~$\K$.
Since $ \ell_A u_M \ell_{A}^{-1} = u_{AM{}^t \overline{A}}$, we get that $\Herm^{\geq
  0}(n,\K)$ is the sought for $L_\Theta$-invariant cone.

For
$\SO(2, k+1)$, let us consider the quadratic form\index{$q_{2,k+1}$ a
  quadratic form of signature $(2,k+1)$}
$q_{2,k+1}= -x_1 x_{k+3}+ x_2 x_{k+2} -\sum_{j=3}^{k+1} x_{j}^{2}$ on
$V=\R^{k+3}$. The group $G=\SO(q_{2,k+1})$ is also an Hermitian Lie group of tube type. The Levi
factor~$L_\Theta$ consists of the matrices
\[
  \ell_{\lambda, g} =
  \begin{pmatrix}
    \lambda \det g& 0 & 0 \\ 0 & g & 0 \\ 0 & 0 & \lambda^{-1}
  \end{pmatrix},\ \lambda\in \R^*, \ g\in \OO(q_{1,k}),
\]
where $q_{1,k}$\index{$q_{1,k}$ a quadratic form of signature $(1,k)$} is given by restricting $q_{2,k+1}$ on the codimension~$2$
space defined by $x_1=x_{k+3}=0$. The unipotent algebra is composed by the
matrices
\[
u_v =
\begin{pmatrix}
  0 & {}^t v J & 0 \\ 0 & 0 & v \\ 0 & 0 & 0
\end{pmatrix}, \ v\in \R^{k+1},
\]
where $J$~is the matrix of the form~$q_{1,k}$. 
One has in this case $\ell_{\lambda,g} u_v \ell_{\lambda,g}^{-1} = u_{\lambda g v}$ and the cone
defined by $x_2>0$, $q_{1,k}>0$ is invariant by $L_{\Theta}^{\circ}$.

The exceptional Hermitian simple Lie group of tube type is $\lietype{E}_{7(-25)}$ (or
EVII in Cartan's notation). In this case, the semisimple part of the
group~$L_\Theta$ is $\lietype{E}IV =  \lietype{E}_{6(-26)}$ and the action of~$L_\Theta$ on
$\mathfrak{u}_{\alpha_\Theta}$ is the action of the group of isomorphisms of $
\Herm(3,\Oc)$, in it $\Herm^{\geq 0}(3,\Oc) $ is a natural invariant
cone (cf.\ \cite[p.~213]{Faraut_Koranyi}).

\smallskip

We now turn to the general case, i.e.\ when $\Theta\neq \Delta$ and the
cardinality of~$\Theta$ is $>1$.
The description of the cone $c_{\alpha_\Theta}$ reduces to the Hermitian
case by Proposition~\ref{prop:hermtian-tube-type}. 
Let us describe concretely what happens in those cases, namely
for orthogonal groups of forms of higher signatures and for the groups whose
system of restricted roots is~$\lietype{F}_4$.

When $G = \SO(p+1, p+k)$, then 
(using standard numbering for the simple roots)
$\alpha_\Theta = \alpha_{p} $, and $\u_{\alpha_\Theta}\cong
(\R^{1,k},q_{1,k})$ with $c_{\alpha_\Theta} =  \{ v \in \R^{1,k} \mid
v_1\geq 0, q_{1,k}(v)\geq 0\}$ ($q_{1,k}$ denotes here the standard quadratic form
of signature~$(1,k)$ on~$\R^{1+k}$). 

When $G$ is one group in the exceptional family whose system of restricted roots is of type~$\lietype{F}_4$, then 
$\mathfrak{u}_{\alpha_2}$ and the cone
$c_{\alpha_2}$ are given in Table~\ref{tab:cones-F4}.

\begin{table}[ht]
  \centering
 \begin{tabular}{llll} 
   $G$ & $\u_{\alpha_2}$ & $S_\Theta$ & $c_{\alpha_2}$  \\
   \hline
 $\operatorname{FI}=\operatorname{F}_{4(4)}$ & $\Sym(3,\R)$ & $\Sp(6,\R)$ & $\Sym^{\geq
                                                              0}(3,\R)$ {\rule{0pt}{2.6ex}}  \\ 
 $\operatorname{EII}=\lietype{E}_{6(2)}$ & $\Herm(3,\CC)$ & $\SU(3,3)$ & $\Herm^{\geq 0}(3,\CC)$ \\
$\operatorname{EVI}=\lietype{E}_{7(-5)}$ & $\Herm(3,\HH)$ & $ \SO^*(12)$ &$\Herm^{\geq 0}(3,\HH)$\\
$\operatorname{EIX}=\lietype{E}_{8(-24)}$ &$ \Herm(3,\Oc)$ & $\lietype{E}_{6(-26)}$ & $\Herm^{\geq 0}(3,\Oc) $
\end{tabular}
  \caption{The nontrivial cone for groups in the exceptional family}\label{tab:cones-F4}
\end{table}

\subsection{Homogeneity}
\label{sec:homogeneity}
The explicit description of the invariant cones $\mathring{c}_\alpha$ for all
$\alpha \in \Theta$ also gives us further information on the action of
$L_{\Theta}^{\circ}$ on the product of cones $\prod_{\alpha\in \Theta}
\mathring{c}_\alpha$. Whereas the representation of $L_{\Theta}^{\circ}$ in
$\mathrm{GL}(\mathfrak{u}_\alpha)$ can have a nontrivial kernel, the kernel
of the representation of $L_{\Theta}^{\circ}$ in $\mathrm{GL}(\prod_{\alpha\in
  \Theta}\mathfrak{u}_\alpha)$ is reduced to the intersection of the center of~$G$ with~$L_{\Theta}^{\circ}$. Thus we get the following proposition.

\begin{proposition}
  \label{prop:homogeneity-Ltheta}
  Let, for each~$\alpha$ in~$\Theta$, $c_\alpha$ be an
  $L^{\circ}_{\Theta}$-invariant acute convex cone in~$\mathfrak{u}_\alpha$
  and let $\mathring{c}_\alpha$ be its interior. Then the diagonal action
  of~$L^{\circ}_{\Theta}$ on $\prod_{\alpha\in \Theta} \mathring{c}_\alpha$ is
  transitive and proper. Precisely the stabilizers are
  maximal compact subgroups of~$L_{\Theta}^{\circ}$, therfore $\prod_{\alpha\in
    \Theta} \mathring{c}_\alpha$ is a model for the Riemannian symmetric space associated
 with~$L_{\Theta}^{\circ}$. Furthermore an element in $\prod_{\alpha\in \Theta}
 {c}_\alpha \smallsetminus \prod_{\alpha\in \Theta} \mathring{c}_\alpha$ has a
 noncompact stabilizer in~$L^{\circ}_{\Theta}$.
\end{proposition}
\begin{proof}
  When $\Theta=\Delta$ (and we are thus in the split case), we have that
  $L_{\Theta}^{\circ}\simeq \R_{>0}^{\Theta}$, $\prod_{\alpha\in \Theta}
  \mathring{c}_\alpha \simeq \R_{>0}^{\Theta}$, and the action is given by
  coordinate-wise multiplication. The action of $L_{\Theta}^{\circ}$ on
  $\prod_{\alpha\in \Theta} \mathring{c}_\alpha$ is therefore simply
  transitive. If $\mathbf{v}=(v_\alpha)_{\alpha\in\Theta}$ belongs to $\prod_{\alpha\in \Theta}
 {c}_\alpha \smallsetminus \prod_{\alpha\in \Theta} \mathring{c}_\alpha$, it
 means that one coordinate $v_\alpha$ of~$\mathbf{v}$ is zero and the
 $\alpha$-th factor of the group $\R_{>0}^{\Theta} \simeq L_{\Theta}^{\circ}$
 is contained in the stabilizer of~$\mathbf{v}$ which is therefore not compact.
 \smallskip
 
 Let us investigate now the case $\Theta\neq \Delta$. The Lie
 algebra~$\mathfrak{l}_\Theta$ of~$L_\Theta$ is the direct sum $\mathfrak{z}_\Theta\oplus
 \mathfrak{s}_\Theta$ of its center and its semisimple factor. Furthermore
 $\mathfrak{z}_\Theta$ is the direct sum of a compact factor
 $\mathfrak{m}_\Theta$ and the split factor $\mathfrak{t}_\Theta$ (cf.\ Section~\ref{sec:acti-l-theta-on-u-theta}). The map
 $X\mapsto (\alpha(X))_{\alpha\in \Theta}$ is an isomorphism
 from~$\mathfrak{t}_\Theta$ to~$\R^\Theta$. Furthermore a
 Cartan subspace for $\mathfrak{s}_\Theta$ is $\mathfrak{a}_\Theta =
 \bigoplus_{\alpha\in \Delta\setminus \Theta} \R H_\alpha$.

 The same decomposition holds for the Levi factor
 $\mathfrak{l}_{\Theta\smallsetminus \{\alpha_\Theta\}}$ and its semisimple
 factor $\mathfrak{s}_{\Theta\smallsetminus \{\alpha_\Theta\}}$. Therefore
 $\mathfrak{t}_{\Theta\smallsetminus\{ \alpha_\Theta\}}$ is contained in
 $\mathfrak{t}_\Theta$. Furthermore $\mathfrak{j}= \mathfrak{l}_\Theta \cap
 \mathfrak{s}_{\Theta\smallsetminus \{\alpha_\Theta\}}$ is the Levi factor
 in $\mathfrak{s}_{\Theta\smallsetminus \{\alpha_\Theta\}}$ corresponding to
 $\{\alpha_\Theta\}$ (Proposition~\ref{prop:hermtian-tube-type}).

 We therefore get that, up to a compact factor,
 $\mathfrak{l}_{\Theta\smallsetminus \{\alpha_\Theta\}}$ is the direct sum of
 $\mathfrak{t}_{\Theta\smallsetminus \{\alpha_\Theta\}}$ and
 of~$\mathfrak{j}$. Let us denote
 by~$T_{\Theta\smallsetminus\{\alpha_\Theta\}}$ and~$J$ the corresponding
 connected Lie groups. From the construction (and by a similar argument as in
 the split case) we get that $T_{\Theta\smallsetminus\{\alpha_\Theta\}}$ acts
 simply transitively on $\prod_{\alpha\in \Theta\setminus \{\alpha_\Theta\}}
 \mathring{c}_\alpha$ and acts trivially on
 $\mathring{c}_{\alpha_\Theta}$. Similarly $J$~acts trivially on
 $\prod_{\alpha\in \Theta\setminus \{\alpha_\Theta\}} \mathring{c}_\alpha$ and
 it acts transitively on $\mathring{c}_{\alpha_\Theta}$ with stabilizers being
 compact maximal subgroups. This implies the result for the action
 of~$L_\Theta$ on $\prod_{\alpha\in \Theta} \mathring{c}_\alpha$.

 Now let  $\mathbf{v}=(v_\alpha)_{\alpha\in\Theta}$ belongs to $\prod_{\alpha\in \Theta}
 {c}_\alpha \smallsetminus \prod_{\alpha\in \Theta} \mathring{c}_\alpha$. Either
 there is $\alpha\neq \alpha_\Theta$ such that $v_\alpha$ is zero and the
 stabilizer of~$\mathbf{v}$ in $T_{\Theta\smallsetminus\{\alpha_\Theta\}} $ is
 not compact, or $v_{\alpha_\Theta}$ belongs to $
 c_{\alpha_\Theta}\smallsetminus  \mathring{c}_{\alpha_\Theta}$
 and the stabilizer of~$\mathbf{v}$ in~$J$ is not compact
 (Proposition~\ref{prop:cone-semisimple}). In all cases, the stabilizer
 of~$\mathbf{v}$ in $L^{\circ}_{\Theta}$ is not compact.
\end{proof}

\subsection{The group of automorphisms of~$\mathfrak{g}$}
\label{sec:group-autom-mathfr}

In order to have stronger (and more natural) transitivity results for the
action on the cones (Proposition~\ref{prop:homog-under-aut} below), we now consider $\Aut(\mathfrak{g}) \subset \GL( \mathfrak{g})$ the group of
automorphisms of the Lie algebra~$\mathfrak{g}$.\index{$\Aut(\mathfrak{g})$ the group of
automorphisms of~$\mathfrak{g}$} It is a Lie group, that is
not necessarily connected, and whose Lie algebra is equal to~$\mathfrak{g}$
since $\mathfrak{g}$ is semisimple. The adjoint action induces a continuous
homomorphism $G\to \Aut(\mathfrak{g})$; this homomorphism has finite kernel
and its image is $\Inn(\mathfrak{g})$ the connected component of the identity in~$\Aut(\mathfrak{g})$.

By construction, the group $\Aut( \mathfrak{g})$ acts on the Lie
algebra~$\mathfrak{g}$ and also on the subalgebras of~$\mathfrak{g}$; this
action is orthogonal with respect to the Killing form. Automorphisms
of~$\mathfrak{g}$ act on the space of maximal compact subalgebras as
well as on the space of Cartan subspaces.

Since the action of $\Inn(\mathfrak{g})$ is transitive on the Weyl
chambers
(cf.\ \cite[Theorem~6.51]{Knapp_LieGrp}), one can
associate to any element~$g$ of $\Aut(\mathfrak{g})$ an automorphism~$\phi_g$ of the Dynkin
diagram~$D$ of~$\mathfrak{g}$. These combine to an homomorphism $\Aut(
\mathfrak{g})\to \Aut(D)$. Let~$g$ be in $\Aut( \mathfrak{g})$, 
for every~$\Theta$
contained in~$\Delta$ and for every~$\mathfrak{p}$ in~$\mathsf{F}_\Theta$, \ie
$\mathfrak{p}$~is a parabolic algebra of type~$\Theta$, the
image $g(\mathfrak{p})$ is a parabolic subalgebra of type~$\phi_g( \Theta)$.

The kernel $\Aut_1( \mathfrak{g})$ of this homomorphism\index{$\Aut_1(
  \mathfrak{g})$ the subgroup of $\Aut( \mathfrak{g})$ of elements acting
  trivially on the Dynkin diagram} 

contains $\Inn(\mathfrak{g})$; it acts on
the flag variety~$\mathsf{F}_\Theta$ since, for every~$g$ in $\Aut_1(
\mathfrak{g})$ and every~$\mathfrak{p}$ in~$\mathsf{F}_\Theta$,
$g(\mathfrak{p})$ belongs to $\mathsf{F}_{\phi_g(\Theta)} = \mathsf{F}_\Theta$.

The actions of~$G$ and of $\Aut_1( \mathfrak{g})$ on~$\mathsf{F}_\Theta$ are of
course related by the homomorphism $G\to \Aut_1(\mathfrak{g})$. 

\subsection{Homogeneity under the group of automorphisms of $\g$}
\label{sec:homog-under-group}
We prove here that the choices of the cones~$c_\alpha$ in
$\mathfrak{u}_\alpha$ ($\alpha\in\Theta$) are unimportant up to the action of
the group $\Aut_1( \mathfrak{g})$. Precisely, we have
\begin{proposition}
  \label{prop:homog-under-aut}
  For every $\{ \varepsilon_\alpha\}_{\alpha\in
    \Theta}$ in $\{ \pm 1\}^\Theta$, there is an element~$g$ in~$\Aut_1( \g)$
  such that
  \begin{itemize}[leftmargin=*]
      \item $g$ preserves every factor in the decomposition $\mathfrak{g}= \mathfrak{g}_0\oplus
    \bigoplus_{\alpha\in \Sigma} \mathfrak{g}_\alpha$, in particular it
    stabilizes~$\mathfrak{u}_\alpha$ for every~$\alpha$ in~$\Theta$;
  \item for every~$\alpha$ in~$\Theta$, the restriction $g|_{
      \mathfrak{u}_\alpha}$ is $\varepsilon_\alpha \id$.
  \end{itemize}
  One can even choose~$g$ such that
  \begin{itemize}[leftmargin=*]
  \item the restriction $g|_{\mathfrak{g}_0}$ is $+\id$, and
  \item for every~$\alpha$ in~$\Theta$, there restriction $g|_{
      \mathfrak{u}_{-\alpha}}$ is $\varepsilon_\alpha \id$.
  \end{itemize}
\end{proposition}

\begin{proof}
  Let~$\phi$ be the homomorphism of the additive group $\Span_\Z
  (\Delta)$ to the multiplicative group $\{\pm 1\}$, defined on the
  basis~$\Delta$ by
  \begin{itemize}
  \item $\phi( \alpha)= \varepsilon_\alpha$ for all~$\alpha$ in~$\Theta$ and
    $\phi(\alpha)=1$ for all~$\alpha$ in $\Delta\setminus \Theta$.
  \end{itemize}
  Define now~$g$  by $g|_{\mathfrak{g}_\gamma} =
  \phi(\gamma) \id$ for every $\gamma$ in $\Sigma \cup \{0\}$. Thus
  $g$~belongs to $\GL(\mathfrak{g})$.
  \begin{enumerate}[leftmargin=*]
  \item\label{item:1:proof-prop-homo-Aut} The relations $\phi(\gamma + \gamma') = \phi(\gamma)
    \phi(\gamma')$, for all $\gamma$ and $\gamma'$ in
    $\Sigma\cup \{ 0\}$,  imply that $g$~is an automorphism of
    $\mathfrak{g}$.
  \item \label{item:2:proof-prop-homo-Aut} One has $\phi(0) = 1$ and,
    for every $\alpha\in \Theta$ and every $\gamma$ in $\Sigma$ that is
    congruent to $\pm \alpha$ modulo the span of $\Delta \setminus \Theta$,
    $\phi(\gamma) = \varepsilon_\alpha$ and therefore
    $g|_{\mathfrak{g}_0}=\id$ and $g|_{\pm
      \mathfrak{u}_\alpha} = \varepsilon_\alpha \id$.
  \end{enumerate}
                By construction $g$~preserves
  the root space decomposition; such an automorphism of~$\mathfrak{g}$
  automatically preserves the Cartan subspace and all the roots, \ie it acts as
  the identity on
  the Dynkin diagram and 
  belongs to $\Aut_1( \mathfrak{g})$. The automorphism~$g$ has hence all the wanted properties.
  \end{proof}

As a corollary of Propositions~\ref{prop:homogeneity-Ltheta} and~\ref{prop:homog-under-aut} we get 
\begin{corollary}
  \label{coro:homog-under-aut}
  The Levi factor of the parabolic subgroup of $\Aut_1( \mathfrak{g})$
  associated with~$\Theta$ acts transitively on $\prod_{\alpha\in\Theta}
  \bigl( \mathring{c}_\alpha  \cup -\mathring{c}_\alpha \bigr)$.
\end{corollary}

\begin{remark}
  \label{rem:homog-under-group-L}
  It is an interesting question to determine the orbits of $L_\Theta$ on the
  product $\prod_{\alpha\in\Theta}
  \bigl( \mathring{c}_\alpha  \cup -\mathring{c}_\alpha \bigr)$. For the cases not corresponding to total positivity (i.e.\ when $\Theta \neq
  \Delta$), one has the following result:
  \begin{itemize}
  \item When $\sharp \Theta$ is even, the action of~$L_\Theta$ on
    $\prod_{\alpha\in\Theta} \bigl( \mathring{c}_\alpha \cup
    -\mathring{c}_\alpha \bigr)$ is transitive,
  \item When $\sharp \Theta$ is odd, this action has two orbits.
  \end{itemize}
\end{remark}

\subsection{Construction of \texorpdfstring{$\Theta$}{Θ}-systems}
\label{sec:putting-an-hand}
We explain in this paragraph an explicit construction of elements $E_\alpha$ in $\mathring{c}_\alpha \cup - \mathring{c}_\alpha$ from ``standard'' elements presenting the Lie
algebra~$\mathfrak{g}$. This will allow us to construct what we call a
$\Theta$-system (Definition~\ref{def:theta_base}).   We show in Section~\ref{sec:split-group-type}  that a $\Theta$-system generates 
a split real Lie subalgebra of $\mathfrak{g}$  and in Section~\ref{sec:sl2} that is determines a special three dimensional simple subalgebra  in $\mathfrak{g}$, that we call the $\Theta$-principal  $\mathfrak{sl}_2$. 

To start we consider the following: \index{$X_\alpha$ an element of
  $\mathfrak{g}_\alpha$ generating, together with $\tau(-X_\alpha)$, an
  $\mathfrak{sl}_2$-triple ($\alpha\in\Theta$), cf.\ also Section~\ref{sec:sl_2-triples}}
\begin{itemize}
\item for each~$\alpha$ in~$\Theta$, we choose once and for all an
  element~$X_\alpha$ in~$\mathfrak{g}_\alpha$ such that,  
  $(X_\alpha, \tau(- X_{\alpha}), [X_\alpha, \tau(-X_{\alpha})])$ is an $\mathfrak{sl}_2$-triple (cf.\ Section~\ref{sec:sl_2-triples}).
\end{itemize}
Recall that $X_\alpha$~must satisfy the relation $B( X_\alpha,
\tau(X_\alpha))=-2/\scalK{\alpha,\alpha}$ (i.e.\ $X_\alpha\in
S(\mathfrak{g}_\alpha)$, Section~\ref{sec:sl_2-triples}).
Since $\mathfrak{g}_\alpha$~is of dimension~$1$, the
element~$X_\alpha$ is determined uniquely up to sign. 

In all cases when $\Theta=\Delta$ or for $\alpha \neq \alpha_\Theta$ when
$\Theta\neq \Delta$, we have $\mathfrak{u}_\alpha =
\mathfrak{g}_\alpha$, and thus $X_\alpha$ determines the cone $\mathring{c}_\alpha = \RR_{>0} X_\alpha$. 

When $\Theta\neq \Delta$ and for $\alpha = \alpha_\Theta$, we will use the element
$ X_{\alpha}\in \mathfrak{g}_{\alpha}$ as a starting point to
define, using the action of~$W_{\Delta\smallsetminus\Theta}$ the Weyl group
of~$S_\Theta$,  a new element~$E_{\alpha_\Theta}$
(see Theorem~\ref{thm:one-element-in-the-cone}). 
This element will be contained in $\mathring{c}_{\alpha_\Theta} \cup
-\mathring{c}_{\alpha_\Theta}$ and we will pick the cone so that $E_\alpha \in \mathring{c}_{\alpha_\Theta}$.

From the classification (or more precisely from the fact that $\Theta$ and
  $\Delta\smallsetminus\Theta$ are connected)    we know that
the semi\-simple Lie group $S_{\Theta}$ is of type~$\lietype{A}_d$ where $d$~is the
cardinality of~$\Delta\smallsetminus \Theta$. We can thus enumerate the roots
in \index{$\chi_i$ ($i$ in $\llbracket 1,d\rrbracket$) the roots in
  $\Delta\smallsetminus \Theta$}
$\Delta\smallsetminus \Theta=\{ \chi_1, \dots, \chi_d\}$ so that
$\chi_1$ is connected (by a double arrow) to the special root~$\alpha_\Theta$ and,
for all~$i$ in $\llbracket 1, d-1\rrbracket$, $\chi_i$ is connected to~$\chi_{i+1}$. The
corresponding reflections in the Weyl group~$W$ will be denoted by~$s_1,
\dots, s_d$.\index{$s_i$ ($i$ in $\llbracket 1,d\rrbracket$) the reflection
  corresponding to the root $\chi_i$}

One has then the following equalities:
\begin{align*}
  s_1 (\alpha_\Theta) & = \alpha_\Theta + 2 \chi_1, \\
\intertext{since $\alpha_\Theta$ and $\chi_1$ are connected by a double arrow
  from~$\alpha_\Theta$ to~$\chi_1$;}
  s_i(\alpha_\Theta) & = \alpha_\Theta \text{ for all } i>1, \\
  \intertext{since $\alpha_\Theta$ and $\chi_i$ are not connected in the Dynkin
  diagram; and the following equalities in type~$\lietype{A}_d$:}
  s_i( \chi_{i-1}) & = \chi_{i-1} + \chi_i \text{ for all } i> 1, \\ s_{i}(
  \chi_{i+1}) &
                       = \chi_{i} + \chi_{i+1} \text{ for all } i<d,\\
  s_i(\chi_i) &=-\chi_i \text{ for } i\geq 1,\\
  s_i(\chi_j) & = \chi_j \quad \text{if } |i-j|>1.
\end{align*}
From this the equalities $s_i( \chi_{i-1} + \chi_i +\chi_{i+1}) = \chi_{i-1} +
\chi_i +\chi_{i+1}$ follow for all $i$ in $\llbracket 2, d-1\rrbracket$.

\begin{notation}
  \label{nota:gammak}
  Denote, for all~$k$ in $ \llbracket 0, d \rrbracket $\index{$\gamma_k$ ($k$ in
    $\llbracket 0,d\rrbracket$) the roots in the orbit of~$\alpha_\Theta$
    under the group $W_{\Delta\smallsetminus \Theta}$}
  \begin{equation*}
    \gamma_k = \alpha_\Theta + 2 \sum_{1\leq j \leq k} \chi_j.
  \end{equation*}
  \end{notation}
  In particular $\gamma_0=\alpha_\Theta$, and the roots~$\gamma_k$ are equal to~$\alpha_\Theta$ modulo the span
  of~$\chi_1, \dots, \chi_d$, i.e.\ $\mathfrak{g}_{\gamma_k}\subset
  \mathfrak{u}_{\alpha_\Theta}$ for all~$k$.

  \begin{lemma}\label{lem:orbit-of-alpha-J}
  The orbit of the special root~$\alpha_\Theta$ under the group
  $W_{\Delta\smallsetminus\Theta}   $ is
  $\{\gamma_0,\dots, \gamma_d\}$.
                             \end{lemma}

\begin{proof}
    For all $i$ in $\llbracket 2, d\rrbracket$ and all~$k$ in $\llbracket 0, i-2\rrbracket$, one has
  \begin{align*}
  s_i( \gamma_k)&= s_i\Bigl( \alpha_\Theta + 2
                  \sum_{1\leq j \leq k} \chi_j\Bigr)\\
    &= s_i(\alpha_\Theta) + 2 \sum_{1\leq j \leq k}
      s_i(\chi_j)\\
    &= \alpha_\Theta + 2 \sum_{1\leq j \leq k} \chi_j = \gamma_k.  
  \end{align*}
   For
  all~$i$ in $\llbracket 1, d-1\rrbracket$ and for all~$k$ in $\llbracket i+1, d\rrbracket$, one has
  \begin{align*}
  s_i(\gamma_k)
  &= s_i\Bigl( \gamma_{i-2}+ 2(\chi_{i-1}+\chi_i+\chi_{i+1})+2\sum_{j=i+2}^{k}
    \chi_j\Bigr)\\
    &= \gamma_{i-2} + 2(\chi_{i-1}+\chi_i+\chi_{i+1})+2\sum_{j=i+2}^{k}
  \chi_j= \gamma_k.  
  \end{align*}
   Similarly,  $s_i( \gamma_{i-1}) = s_i( \gamma_{i-2} +2\chi_{i-1}) =
   \gamma_{i-2} +2\chi_{i-1} +2\chi_i =
  \gamma_i$, and $s_{i}(\gamma_{i}) = \gamma_{i-1}$ for all~$i$ in $\llbracket 1, d\rrbracket$.
            
  From this we directly observe that the reflections $s_1, \dots, s_d$
  stabilize the set $\{ \gamma_0, \dots, \gamma_d\}$ and act transitively on
  it. This implies the result since $W_{\Delta\smallsetminus\Theta}$ is generated
  by~$s_1, \dots, s_d$.
\end{proof}

We now establish preliminary results concerning the $3$-dimensional Lie
subalgebras associated with elements in~$\mathfrak{g}_{\chi_i}$, and their
action on the root spaces $\mathfrak{g}_{\gamma_k}$. Recall that for~$X$ an
element in the sphere $S(\mathfrak{g}_{\chi_i})$ of~${\mathfrak{g}_{\chi_i}}$,
i.e.\ $B(X, \tau(X))=-2/\scalK{\chi_i, \chi_i}$, the triple
$(X, \tau(-X), [\tau(X), X])$ is an $\mathfrak{sl}_2$-triple (cf.\
Section~\ref{sec:sl_2-triples}), thus the Lie algebra~$\mathfrak{s}$ generated
by~$X$ and $\tau(X)$ is isomorphic to $\mathfrak{sl}_2(\R)$.

\begin{lemma}
  \label{lem:g-gamma-k-trivial-frak-si}
  Let~$k$ be in $\llbracket 0, d \rrbracket$ and let~$i\notin
  \{k,k+1\}$. Let~$X$ an element in the sphere $S(\mathfrak{g}_{\chi_i})$ and
  let~$\mathfrak{s}$ be the Lie algebra
  generated by~$X$ and $\tau(X)$.
  The root space~$\mathfrak{g}_{\gamma_k}$ is stable by~$\mathfrak{s}$ and
  is the trivial $\mathfrak{s}$-module.  
\end{lemma}
\begin{proof}
  Since the $\chi_i$-chain containing~$\gamma_k$ is trivial, the
  subspace~$\mathfrak{g}_{\gamma_k}$ is indeed an
  $\mathfrak{s}$-module (cf.\ Section~\ref{sec:sl_2-triples}). Since it is $1$-dimensional, it is the trivial module.
\end{proof}

\begin{lemma}
  \label{lem:g-gamma-i-not-trivial-frak-si}
  Let~$i$ be in $\llbracket 1, d\rrbracket$, and let~$X$ be
  in~$S(\mathfrak{g}_{\chi_i})$, so that the Lie algebra~$\mathfrak{s}$
  generated by~$X$ and $\tau(X)$ is isomorphic to $\mathfrak{sl}_2(\R)$.
  The $\mathfrak{s}$-modules
  generated by~$\mathfrak{g}_{\gamma_{i-1}}$ and by~$\mathfrak{g}_{\gamma_i}$
  coincide and are isomorphic to the $3$-dimensional irreducible module of
  $\mathfrak{s} \simeq \mathfrak{sl}_2(\R)$.
\end{lemma}
\begin{proof}
    Since $X$~belongs to~$S(\mathfrak{g}_{\chi_i})$, 
  $(X,
  Y=\tau(-X), H=[X,Y])$ is an $\mathfrak{sl}_2$-triple.

  The $\chi_i$-chain containing~$\gamma_{i-1}$ is $(\gamma_{i-1},
  \gamma_{i-1}+\chi_i, \gamma_{i-1}+2\chi_i=\gamma_i)$; it is also the
  $\chi_i$-chain containing~$\gamma_i$.   This means in particular that
  \[ 
    \mathfrak{g}_{\gamma_{i-1}} \oplus
    \mathfrak{g}_{\gamma_{i-1}+\chi_i} \oplus \mathfrak{g}_{\gamma_i}\]
  is invariant under the adjoint action of~$\mathfrak{s}$ and that the
  lowest weight space for this action is the $1$-dimensional
  space~$\mathfrak{g}_{\gamma_{i-1}}$. Together with the fact that
  $\gamma_{i-1}(H)=-2$, one gets that the 
   $\mathfrak{s}$-module generated
  by~$\mathfrak{g}_{\gamma_{i-1}}$ is the irreducible $3$-dimensional
  $\mathfrak{s}$-module (cf.\ Section~\ref{sec:sl_2-triples}) and contains the
  line $\mathfrak{g}_{\chi_i}$. By irreducibility, the $\mathfrak{s}$-module
  generated by~$\mathfrak{g}_{\chi_i}$ is equal to the 
   $\mathfrak{s}$-module generated
  by~$\mathfrak{g}_{\gamma_{i-1}}$.
\end{proof}

The above lemmas enable us to understand the orbit
of~$X_{\alpha_\Theta}$ under the subgroup~$\mathcal{W}$\index{$\mathcal{W}$ a subgroup of
  $N_K( \mathfrak{a})$ lifting $W_{\Delta\smallsetminus\Theta}$}
of $N_K( \mathfrak{a})$ generated by the elements $\exp( \frac{\pi}{2}(X+\tau(X)))$
where $X$~varies in $\bigcup_{i=1}^{d} S(\mathfrak{g}_{\chi_i})$; by
\cite[Prop.~6.52]{Knapp_LieGrp}, for~$X$ in $S(\mathfrak{g}_{\chi_i})$, the
element $\exp( \frac{\pi}{2}(X+\tau(X)))$ belongs to $N_K( \mathfrak{a})$ and represents~$s_i$.

\begin{notation}
  \label{nota:ZkYkWk}
  For all~$i$ in $\llbracket 1,d \rrbracket$, let~$X_i$ be the element 
  in~$S(\mathfrak{g}_{\chi_i})$ so that $(X_i, \tau(-X_i), [\tau(X_i), X_i])$ is an $\mathfrak{sl}_2$-triple and set $\dot{s}_i=\exp\bigl(\frac{\pi}{2}(X_i+\tau(X_i))\bigr)$.

  We denote by  $Z_0, \dots Z_d$  the elements defined by the
  following\index{$Z_k$ ($k$ in $\llbracket 0,d\rrbracket$) the orbit of
    $X_{\alpha_\Theta}$ under a subgroup of $N_K(\mathfrak{a})$ lifting
    $W_{\Delta\smallsetminus \Theta}$} equalities  $Z_0 = X_\alpha \in \mathfrak{g}_{\alpha_\Theta}$ and
  \[ Z_1=\Ad( \dot{s}_1)Z_0,\  Z_2=\Ad( \dot{s}_2)Z_1, \dots, Z_d=\Ad(
    \dot{s}_d)Z_{d-1}. \]
  We denote, for all~$k$ in $\llbracket 0,d\rrbracket$,
  $Y_k=-\tau(Z_k)$\index{$Y_k$ ($k$ in $\llbracket 0,d\rrbracket$) the element
  $\tau(-Z_k)$} and
  $W_k=[Z_k, Y_k]$.\index{$W_k$ ($k$ in $\llbracket 0,d\rrbracket$) the element $[Z_k, Y_k]$}
\end{notation}
Since $\Ad( \dot{s}_i)$ commutes with~$\tau$, one has $Y_i= \Ad( \dot{s}_i)
Y_{i-1}$ and $W_i= \Ad( \dot{s}_i) W_{i-1}$.

\begin{proposition}\label{prop:orbit-of-X_alpha-J}
  The elements $Z_0, \dots, Z_d$ do not depend on the choices of the~$X_i$.
  
  They all belong to~$\mathfrak{u}_{\alpha_\Theta}$, and are the elements of
  the orbit of~$Z_0$ under the group~$\mathcal{W}$.
\end{proposition}

\begin{proof}   For every root~$\alpha$, one has $\Ad( \dot{s}_i) \mathfrak{g}_\alpha =
  \mathfrak{g}_{ s_i(\alpha)}$. Therefore, for all~$k$ in $\llbracket 0, d\rrbracket$, $Z_k$
  belongs to $\mathfrak{g}_{\gamma_k}$ which is contained
  in~$\mathfrak{u}_{\alpha_\Theta}$.

  Let us denote, for~$k=0$ in $\llbracket 0, d\rrbracket$, by $\mathcal{W}_k$ the subgroup of $N_K(
  \mathfrak{a})$  generated by the elements $\exp( \frac{\pi}{2}(X+\tau(X)))$
where $X$~varies in $\bigcup_{i=1}^{k} S(\mathfrak{g}_{\chi_i})$. The 
subgroup~$\mathcal{W}_0$ is then the trivial subgroup and $\mathcal{W}_d=
\mathcal{W}$.

We will establish by induction on~$k$ that the orbit of~$Z_0$ under the action
of~$\mathcal{W}_k$ is equal to $\{ Z_0, \dots, Z_k\}$ and that $Z_k$ does not
depend on the choice of~$X_k$.

This is trivially satisfied for $k=0$. Let us thus assume the result for a
given~$k<d$.

By Lemma~\ref{lem:g-gamma-k-trivial-frak-si}, the action of $\exp(
\frac{\pi}{2}( X+\tau(X)))$ is trivial on~$\mathfrak{g}_{k+1}$ for all~$X$ in $\bigcup_{i=1}^{k} S(\mathfrak{g}_{\chi_i})$. This
implies that $Z_{k+1}$ is fixed by $\mathcal{W}_k$. By the same argument, $\exp(
\frac{\pi}{2}( X+\tau(X)))$ fixes the elements $Z_0, Z_1, \dots, Z_{k-1}$ for
all~$X$ in $S(\mathfrak{g}_{\chi_{k+1}})$.

Let us now prove that,  for
all~$X$ in $S(\mathfrak{g}_{\chi_{k+1}})$, one has $\exp(
\frac{\pi}{2}( X+\tau(X)))\cdot Z_k = Z_{k+1} $ and  $\exp(
\frac{\pi}{2}( X+\tau(X)))\cdot Z_{k+1} = Z_{k} $. Let~$\mathfrak{s}$ be the
$3$-dimensional algebra generated by~$X$ and~$\tau(X)$.

By Lemma~\ref{lem:g-gamma-i-not-trivial-frak-si} and the knowledge of the
$3$-dimensional irreducible representation of $\mathfrak{sl}_2( \R)$, the
element $\exp(
\frac{\pi}{2}( X+\tau(X)))$ acting on the $\mathfrak{s}$-module generated by $Z_k$ is of
order~$2$ (cf.\ Section~\ref{sec:sl_2-triples}); this implies that it is enough to establish the identity $\exp(
\frac{\pi}{2}( X+\tau(X)))\cdot Z_k = Z_{k+1} $. This implies also that the
map $X\mapsto \exp(
\frac{\pi}{2}( X+\tau(X)))\cdot Z_k$ is invariant by $X\mapsto -X$. Hence this map
factors to a continuous map from the projective space
$S(\mathfrak{g}_{\chi_i})/\{\pm 1\}$ to a $0$-dimensional
sphere $S(\mathfrak{g}_{\gamma_{k+1}})$ (since $\mathfrak{g}_{\gamma_{k+1}}$
is $1$-dimensional and since $\exp(\frac{\pi}{2}( X+\tau(X)))$~acts orthogonally with respect to the quadratic form~$B(\cdot,\tau(\cdot))$), it is hence the constant map.

This proves that $Z_{k+1}$ does not depend on~$X_{k+1}$ and also that the orbit
of~$Z_0$ under $\mathcal{W}_{k+1}$ is equal to $\{ Z_0, \dots, Z_k,
Z_{k+1}\}$, hence the induction step.
                                                    \end{proof}

The nontrivial Lie brackets between elements among $\{Z_k, Y_k, W_k\}  $ are zero.
  \begin{lemma}
  \label{lemma:zero_bracket_ZkYkWk}
  \begin{enumerate}[leftmargin=*]
  \item For all~$k$ and~$\ell$ in $\llbracket 0,d\rrbracket$, $[Z_k,
    Z_\ell]=0$, $[Y_k, Y_\ell]=0$, and  $[W_k, W_\ell]=0$.
  \item For all~$k$ and~$\ell$ in $\llbracket 0,d\rrbracket$, if $k\neq \ell$,
    $[Z_k, Y_\ell]=0$, $[Z_k, W_\ell]=0$, and $[Y_k, W_\ell]=0$.
  \end{enumerate}
\end{lemma}
\begin{proof}
  \begin{enumerate}[leftmargin=*]
  \item This follows since $\mathfrak{u}_{\pm\alpha_\Theta}$ is Abelian and since
    $\mathfrak{a}$~is Abelian.
  \item Let us prove that $[Z_k, Y_\ell]=0$. Using the action of~$\tau$ we can
    assume $k>\ell$. The bracket $[Z_k, Y_\ell]$ belongs to the weight space
    associated with $\gamma_k-\gamma_\ell = 2( \chi_{\ell+1}+\cdots+\chi_k)$
    which is not a root (since the root system is reduced and
    $\chi_{\ell+1}+\cdots+\chi_k$ is a root); this bracket is therefore zero.

    The identities $[Z_k, W_\ell]=0$ and $[Y_k, W_\ell]=0$ follow now using
    the definition of~$W_\ell$ and the Jacobi identity.
  \end{enumerate}
\end{proof}

From this we deduce
\begin{theorem}\label{thm:one-element-in-the-cone}
    The element $E_{\alpha_\Theta}$
  of~$\mathfrak{u}_{\alpha_\Theta}$ defined by\index{$E_{\alpha_\Theta}$ the
    sum $Z_0+\cdots+Z_d$} 
  \[ E_{\alpha_\Theta} \coloneqq Z_0 + Z_1 + \cdots + Z_d,\]
        belongs to $\mathring{c}_{\alpha_\Theta} \cup -\mathring{c}_{\alpha_\Theta}$.
    \end{theorem}

\begin{proof}
  By Proposition~\ref{prop:cone-semisimple}.(\ref{item2:prop:cone-semisimple})
  it is enough to show that the Lie algebra~$\mathfrak{s}$ of the stabilizer of
  $Z_0+\cdots+Z_d$ contains~$\mathfrak{k}_\Theta$. By Proposition~\ref{prop:orbit-of-X_alpha-J} this
  sum is invariant under the action of~$\mathcal{W}$, therefore $\mathfrak{s}$~is invariant by~$W_{\Delta\smallsetminus \Theta}$.

  Equation~(\ref{eq:max-compact-S_J}), p.~\pageref{eq:max-compact-S_J}, reads
  \begin{equation*}
    \mathfrak{k}_\Theta = \mathfrak{m} \oplus  \!\!\! \!\!\!
\bigoplus_{\alpha \in \Span(\Delta \smallsetminus \Theta) \cap \Sigma^+}
                        \!\!\! \!\!\! (\mathfrak{g}_{\alpha} \oplus
                        \mathfrak{g}_{-\alpha}) \cap \mathfrak{k}.
  \end{equation*}
  Here $\Span(\Delta \smallsetminus \Theta) \cap \Sigma$ is the root system of
  the Lie group~$S_\Theta$ and is of type~$\lietype{A}_d$. In this
  situation the corresponding Weyl group~$W_{\Delta\smallsetminus \Theta}$
  acts transitively on the roots. Therefore it will be enough to establish the
  invariance under $\mathfrak{m}=
  \mathfrak{z}_{\mathfrak{k}}( \mathfrak{a})$ and under
  $(\mathfrak{g}_{\chi_1} \oplus \mathfrak{g}_{-\chi_1}) \cap \mathfrak{k}$
  (since $\chi_1$~is one of the root of~$S_\Theta$).

  For all~$k$ in $\llbracket 0, d\rrbracket$,
  the root space~$\mathfrak{g}_{\gamma_k}$ is a $1$-dimensional
  representation of the (compact) Lie algebra~$\mathfrak{m}$ and is thus the
  trivial representation. Hence the stabilizer of~$Z_k$
  contains~$\mathfrak{m}$.  Therefore $\mathfrak{s}$ contains~$\mathfrak{m}$.

  Let us now establish that $\mathfrak{s}$~contains   \[
  (\mathfrak{g}_{\chi_1} \oplus
  \mathfrak{g}_{-\chi_1}) \cap \mathfrak{k}
                      = \{ X+\tau(X)\}_{X\in \mathfrak{g}_{\chi_1}},\]
  (cf.\ Equation~(\ref{eq:max-compact-S_J-factor}), p.~\pageref{eq:max-compact-S_J-factor}).
  Let~$X$ be an element of~$\mathfrak{g}_{\chi_1}$, and $Y=-\tau(X)$.
  We want to show that the stabilizer of $Z_0+\cdots+Z_d$ contains $X-Y$. As the conclusion
  holds trivially if $X=0$, we can assume that $X$~is nonzero. Up to
  multiplying~$X$ by a positive real number, we can assume that $(X,Y,
  [X,Y])$
  is an $\mathfrak{sl}_2$-triple. Denote by~$\mathfrak{f}$ the Lie algebra it
  generates. By Lemma~\ref{lem:g-gamma-i-not-trivial-frak-si}, we know that the $\mathfrak{f}$-module generated by~$Z_0$
  contains~$Z_1= \Ad(\dot{s}_1) Z_0$ (again with $\dot{s}_1 =\exp(\frac{\pi}{2}(X-Y))$) and is the irreducible $3$-dimensional
  $\mathfrak{f}$-module. Explicit knowledge of the $3$-dimensional
  irreducible $\mathfrak{sl}_2$-module shows directly that the stabilizer of
  $Z_0+Z_1$ contains $X-Y$. By Lemma~\ref{lem:g-gamma-k-trivial-frak-si}, for every~$k$ in $\llbracket 2, d\rrbracket$, the stabilizer of~$Z_k$ contains $X-Y$. From this we have the sought
  for result: $\{X+\tau(X) \}_{X \in \mathfrak{g}_{\chi_1}}$ is
  included in~$\mathfrak{s}$. \end{proof}

\begin{example}
  We illustrate the construction of the elements above in the example when
  $G = \mathrm{Sp}(2n,\R)$ and $\Theta = \{\alpha_n\}$. In this case
  $\mathfrak{u}_{\alpha_\Theta}$ naturally identifies with $\mathrm{Sym}_n(\R)$, the
  space of real $n\times n$ symmetric matrices (cf.\ Section~\ref{sec:class-symm-cones}) and
  $W_{\Delta\smallsetminus \Theta}$ is isomorphic to the symmetric
  group~$S_{n}$ acting on $\mathrm{Sym}_n(\R)$ by conjugation by the corresponding
  permutation matrices and $d=n-1$. We use the standard basis
  $(E_{i,j})_{1\leq i,j\leq n}$ of $M_n(\R)$ (the only nonzero coefficient
  of $E_{i,j}$ is in place $(i,j)$ and is equal to~$1$) so that
  $\{ E_{i,i}\}_i \cup \{ E_{i,j}+E_{j,i}\}_{i<j}$ is a basis of
  $\mathrm{Sym}_n(\R)$. The root space~$\mathfrak{g}_{\alpha_n}$ is equal (as a subspace
  of $\mathfrak{u}_{\alpha_\Theta}\simeq \mathrm{Sym}_n(\R)$) to the line generated
  by~$E_{1,1}$. Thus, one can take $Z_0=E_{1,1}$ and one has $Z_1 = E_{2,2},
  \dots, Z_i = E_{i+1,i+1}, \dots, Z_{n-1} =E_{n,n}$. Hence the sum
  $E_{\alpha_n} =Z_0+\cdots + Z_{n-1}$ is the identity matrix of~$\mathrm{Sym}_n(\R)$.
\end{example}

As mentioned already, we now fix the invariant cone $c_{\alpha_\Theta}$ by
choosing the cone containing~$E_{\alpha_\Theta}$.

\begin{lemma}
  \label{lem:inter-spanZ_i-with-cone}
  The intersection of the cone~$\mathring{c}_{\alpha_\Theta}$ with the vector space
  generated by $Z_0, \dots, Z_d$ is equal to $\sum_{i=0}^{d} \R_{> 0} Z_i$ and the intersection of the cone~$c_{\alpha_\Theta}$ with the vector space
  generated by $Z_0, \dots, Z_d$ is equal to $\sum_{i=0}^{d} \R_{\geq 0} Z_i$. In
  particular $X_{\alpha_\Theta}$ belongs to~$c_{\alpha_\Theta}$.
\end{lemma}
\begin{proof}
  By the action of $\exp(\mathfrak{a})$, one finds that $\sum_{i=0}^{d} \R_{> 0}
  Z_i$ is contained in $\mathring{c}_{\alpha_\Theta}$. To prove the reverse
  inclusion, it is enough to prove that no element of $(\sum_{i=0}^{d} \R_{\geq 0}
  Z_i) \setminus (\sum_{i=0}^{d} \R_{> 0}
  Z_i)$ belongs to~$\mathring{c}_{\alpha_\Theta}$, this follows from
  Proposition~\ref{prop:cone-semisimple} since these elements have noncompact
  stabilizers.
\end{proof}

\begin{remark}
  \label{remark:normalize_cones}
  In view of the results in this section, one can go back and forth between
  the normalization (up to $\pm 1$) of the elements~$X_\alpha$ in~$\mathfrak{g}_\alpha$
  ($\alpha\in \Theta$) and the normalization (up to $\pm \id$) of the
  cones~$c_\alpha$ ($\alpha\in \Theta$). If the element~$X_\alpha$ is chosen,
  there is a unique cone~$c_\alpha$ in~$\mathfrak{u}_\alpha$ such that
  $c_\alpha$ contains~$X_\alpha$. Conversely, if the cone~$c_\alpha$ is
  chosen, there is a unique element~$X_\alpha$ in $S(\mathfrak{g}_\alpha)$
  such that $X_\alpha$~belongs to~$c_\alpha$.
\end{remark}

Starting with the elements $Y_\alpha=\tau(-X_\alpha)$, we analogously
determine elements~$F_\alpha$ in~$\mathring{c}^{\mathrm{opp}}_\alpha$. One has $F_\alpha=\tau(-E_\alpha)$.

\begin{definition}\label{def:theta_base}
Let $E_\alpha \in \mathring{c}_\alpha$, $F_\alpha \in
\mathring{c}^{\mathrm{opp}}_\alpha$  be the elements constructed as above, and
set $D_\alpha \coloneqq [E_\alpha, F_\alpha]$ for all $\alpha \in
\Theta$. Then the\index{$E_\alpha$, $D_\alpha$, $F_\alpha$ ($\alpha$ in
  $\Theta$) the elements of a $\Theta$-system}
family $(E_\alpha, F_\alpha, D_\alpha)_{\alpha \in \Theta}$ is called \emph{a $\Theta$-system} of $\mathfrak{g}$. 
\end{definition}

\begin{remark}
  In the case when $\Theta=\Delta$, i.e.\ when $G$~is a split real group, the
  elements $\{E_\alpha,  F_\alpha, D_\alpha \}_{\alpha\in \Delta}$ are  the Chevalley or Cartan--Weyl basis
  of~$\mathfrak{g}$.
\end{remark}

\section{Subalgebras associated to a $\Theta$-system}
\label{sec:split-group-type}
In this section we show that a $\Theta$-system generates a real split Lie subalgebra $\mathfrak{h}_\Theta$ in
$\mathfrak{g}$.                 Besides being remarkable, this subalgebra will play a 
crucial role in Section~\ref{sec:independence-gammab}. 

We assume that $G$~is a simple Lie group admitting a $\Theta$-positive
structure. We let~$\mathfrak{g}$ be its Lie algebra and
$(E_\alpha, F_\alpha, D_\alpha)_{\alpha \in \Theta}$ a $\Theta$-system, with
$E_\alpha$~belonging to~$ \mathring{c}_\alpha$ and $F_\alpha$~belonging
to~$ \mathring{c}^{\mathrm{opp}}_\alpha$ for~$\alpha$ in~$\Theta$.

\begin{theorem}
  \label{theorem:real-split-Theta-subalgebra_1}
The Lie algebra $\mathfrak{h}_\Theta \subset
\mathfrak{g}$ generated by $(E_\alpha, F_\alpha)_{\alpha\in\Theta}$ is a split real simple Lie
algebra,\index{$\mathfrak{h}_\Theta$ the split real simple Lie algebra
  generated by the $\Theta$-system} more precisely $(E_\alpha,
F_\alpha, D_\alpha)_{\alpha \in \Theta}$ are Chevalley generators
of~$\mathfrak{h}_\Theta$ and satisfy the Serre relations.  When $G$~is split and $\Theta = \Delta$, we have $ \mathfrak{h}_\Delta= \mathfrak{g}$. 
  If $G$~is of type~$\lietype{C}_n$ and $\Theta = \{ \alpha_n\}$, then $\mathfrak{h}_\Theta$~is of type~$\lietype{A}_1$. 
  If $G$~is of type~$\lietype{B}_n$ and $\Theta = \Delta \smallsetminus \{ \alpha_n\}$, then $\mathfrak{h}_\Theta$~is of type~$\lietype{B}_{n-1}$. 
  If $G$~is of type~$\lietype{F}_4$ and $\Theta =  \{ \alpha_1,\alpha_2\}$, then $\mathfrak{h}_\Theta$~is of type~$\lietype{G}_{2}$.
\end{theorem}

\begin{proof}
Note that the conclusion of the theorem is plainly satisfied in the split case, where $E_\alpha = X_\alpha$ for every~$\alpha$
in~$\Delta$. The Lie algebra generated by
$\{ X_\alpha, \tau( X_\alpha)\}_{\alpha\in \Delta}$ is then equal
to~$\mathfrak{g}$.
  
The case when $\sharp \Theta =1$ (Hermitian groups) is also easy to deal
with. We thus focus on the case when $\Theta\neq \Delta$ and $\sharp \Theta>1$. 

Let  $\{X_i\}_{i\in \llbracket 1,d\rrbracket}$ $\{ Z_k, Y_k, W_k\}_{k \in \llbracket 0,d\rrbracket}$ (where $d$~is
the cardinality of $\Delta\smallsetminus\Theta$) be
the elements given in Section~\ref{sec:putting-an-hand} so that
$E_{\alpha_\Theta} = Z_0+ \cdots + Z_d$ belongs to the cone
$\mathring{c}_{\alpha_\Theta}$, and for each~$k$ in $\llbracket 1, d\rrbracket$, the relations $Y_k = \Ad( \dot{s}_k) Y_{k-1}$ and
$W_k = \Ad( \dot{s}_k) W_{k-1}$ hold (with $\dot{s}_k =
\exp\bigl(\frac{\pi}{2} (X_k+\tau(X_k))\bigr)$).

For the rest of this proof we will write $\Theta = \{\alpha_1, \dots,
\alpha_p\}$ (with $p$ the cardinality of~$\Theta$ so that $p\geq 2$) with $\alpha_p=
\alpha_\Theta$ and, for all~$i$ in $\llbracket 1, p-1\rrbracket$, $\alpha_i$ and $\alpha_{i+1}$
are connected in the Dynkin diagram; we will write
$E_i$, $F_i$, $D_i$ instead of $E_{\alpha_i}$, $F_{\alpha_i}$, $D_{\alpha_i}$.

We now prove that\index{$E_i$, $F_i$, $D_i$ the $\Theta$-system (in the proof
  of Theorem~\ref{theorem:real-split-Theta-subalgebra_1})}
\( E_i, F_i, D_i, \ (i \in \llbracket 1,p\rrbracket)\)
satisfy the Serre relations\index{Serre relations} (cf.\ \cite[Ch.~VIII, \S~4, n\textsuperscript{o}~3,
Théorème~1]{BourbakiLie78}, note that we are using here a different normalization for $\mathfrak{sl}_2$-triples). This means here the following identities 
\begin{align}
  \label{eq:1}
  [D_{i}, E_{i}] &= 2 E_{i}, \ [D_{i}, F_{i}] = -2 F_{i}, \
                    [E_{i},  F_{i}] = D_{i}, \ \forall i\in \llbracket 1,p\rrbracket\\ 
  \label{eq:2}
  [E_{i}, F_{j}] &=0, \ \forall i\neq j \in \llbracket 1,p\rrbracket \\
  \label{eq:3}
  [E_{i}, E_{j}] &=0, \ [F_{i}, F_{j}] =0,\
                   [D_{i}, E_{j}] =0,\ \text{and} \\
  \notag [D_{i},  F_{j}]
           &=0,\  \forall i,j\in \llbracket 1,p\rrbracket \text{ with }
             |{i-j}|>1 \\
  \label{eq:4}
  (\ad E_{i})^2 E_{j} &=0,\ (\ad F_{i})^2 F_{j} =0,\
                        [D_{i}, E_{j}] =-E_{j},\ 
                                 \text{and}\\
  \notag  [D_{i}, F_{j}] &=F_{j},\ 
                           \forall i\in \llbracket 1,p-1\rrbracket  \text{ and }
                           j\in\{i\pm 1\} \\
  \label{eq:6} [ D_{p}, E_{{p-1}}] &= -(d+1)E_{{p-1}}, \
                                     [ D_{p}, F_{{p-1}}] =(d+1)F_{{p-1}} \\
  \label{eq:5}
  (\ad E_{p})^{d+2} & E_{{p-1}} =0, \ (\ad F_{p})^{d+2}
                                 F_{{p-1}} =0.
\end{align}
(This list seems longer than usual due to the distinguished role of the root $\alpha_\Theta=\alpha_p$).

Equations~\eqref{eq:1} are the facts that the triples $(E_{i},
F_{i}, D_{i})$ are 
$\mathfrak{sl}_2$-triples, only the case when $i=p$ needs more explanation; 
this case  follows from the equalities $[W_k, Z_j]=0$ and $[Y_k, Z_j]=0$
for all $k\neq j$ (Lemma~\ref{lemma:zero_bracket_ZkYkWk}) and the fact that
$(Z_k, Y_k, W_k)$ are $\mathfrak{sl}_2$-triples.

Equations~\eqref{eq:2} are inherited from the
corresponding identities in~$\mathfrak{g}$, again only the case when $i$ or $j$
is equal to~$p$ needs a comment. Let's consider, for example, $i=p$; one has
$E_{p}= E_{\alpha_\Theta} = Z_0+ Z_1 +\cdots + Z_d$, and, for all~$k$ in $\llbracket 1,d\rrbracket$, $Z_k = \Ad(\dot{s}_k) Z_{k-1}$. Since, from the known
identities in~$\mathfrak{g}$, $[Z_0, F_{j}]=0$ and since
$\Ad(\dot{s}_k) F_{j}=F_{j}$ for all~$k$ in $\llbracket 1,d\rrbracket$, one gets
that, for all $k$, $[Z_k, F_{j}]= \Ad(\dot{s}_k)[Z_{k-1},
F_{j}]=0$ and also $[E_{p}, F_{j}] = \sum_{k=0}^{d} [Z_k,
F_{j}]=0$.

Equations~\eqref{eq:3} are also inherited from the corresponding identities
in~$\mathfrak{g}$ (with again a special treatment when one of the indices is
equal to~$p$). Equations~\eqref{eq:4} equally follow from the known equalities
in~$\mathfrak{g}$.

We now prove Equations~\eqref{eq:6}. Note first that $D_{p} = W_0+W_1+\cdots+W_d$. One has, again from identities valid
in~$\mathfrak{g}$, $[W_0, E_{{p-1}}] = -E_{{p-1}}$.
From the
recursive relation $\Ad( \dot{s}_k)
W_{k-1} =W_k$ and since  $\Ad( \dot{s}_k)
E_{{p-1}} =E_{{p-1}}$
we deduce that $[W_k, E_{{p-1}}] =
-E_{{p-1}}$ and, summing over~$k$, $[D_{p}, E_{{p-1}}] =
-(d+1) E_{{p-1}}$. The identity with $F_{{p-1}}$ follows by a
similar argument (or by applying~$\tau$).

Let us now address Equations~\eqref{eq:5}. We will prove the seemingly
stronger identity $(\ad E_{p})^{d+2} =0$. Since $\ad E_{p} =
\sum_{k=0}^{d} \ad Z_k$, one has
\[ (\ad E_{p})^{d+2} = \sum_{\mathclap{(k_0, \dots, k_{d+1})\in \llbracket 0,d\rrbracket^{d+2}} } \ad Z_{k_0}  \ad
  Z_{k_1} \cdots \ad Z_{k_{d+1}}. \]
Every term in this last sum is zero: indeed the order 
in a
product $\ad Z_{k_0}  \ad
  Z_{k_1} \cdots \ad Z_{k_{d+1}}$ does not matter (see
  Proposition~\ref{prop:u_alpha-are-abel}) and, since $d+2>d+1$, at least two indices
  coincide $k_m=k_\ell$ so that $\ad Z_{k_0}  \ad
  Z_{k_1} \cdots \ad Z_{k_{d+1}}=0$ since $(\ad Z_{k_\ell})^2=0$.  Similar
  arguments imply the equality
  $(\ad F_{p})^{d+2}=0$.
\end{proof}

Using the homogeneity under the action of~$L^{\circ}_{\Theta}$,
one gets

\begin{corollary}
  \label{coro:real-split-Theta-subalgebra}
  Let $V=(V_\alpha)_{\alpha\in \Theta}$ belong to $\prod_{\alpha\in \Theta}
  (\mathring{c}_\alpha \cup -\mathring{c}_\alpha)$. Then there exists a unique maximal compact subgroup~$K'$
  of~$G$ such that $K'\cap L^{\circ}_{\Theta}$ is the stabilizer of~$V$ in~$L^{\circ}_{\Theta}$.

  Let~$\sigma\colon \mathfrak{g} \to \mathfrak{g}$ be Cartan involution 
  associated with~$K'$. Then the Lie algebra $\mathfrak{h}_\Theta$ generated by $\{
  V_\alpha,  \sigma( V_\alpha)\}_{\alpha\in \Theta}$ is isomorphic to a real
  split Lie algebra of type given by the above theorem.          \end{corollary}

\begin{proof}
Changing one component~$V_\alpha$ of~$V$ to $-V_\alpha$ does not affect the
hypothesis nor the conclusion; hence we can as well assume that $V$~belongs to
$\prod_{\alpha\in\Theta} \mathring{c}_\alpha$.
By the transitivity of the action of~$L^{\circ}_{\Theta}$ on $\prod_{\alpha\in \Theta}
\mathring{c}_\alpha$ (Proposition~\ref{prop:homogeneity-Ltheta}) it is enough
to prove the result for one specific element in this product of cones.
We can thus assume that, for every~$\alpha$ in~$\Theta$,
$V_\alpha= E_\alpha$ (cf.\ Definition~\ref{def:theta_base}).
We already observed (Theorem~\ref{thm:one-element-in-the-cone}) that the Lie
algebra of the stabilizer of $E=( E_\alpha)_{\alpha\in \Theta}$ contains the
Lie algebra of $K\cap L^{\circ}_\Theta$; therefore
(cf.\ Proposition~\ref{prop:homogeneity-Ltheta}) the stabilizer of~$E$ in~$L^{\circ}_{\Theta}$ is
equal to~$K\cap L^{\circ}_{\Theta}$.
The Cartan involution is then~$\tau$. As $F_{\alpha}=-\tau(E_{\alpha})$, Theorem~\ref{theorem:real-split-Theta-subalgebra_1} implies the wanted
conclusion. \end{proof}

\section{Positive $\SL_2$}\label{sec:sl2}

This section continues the investigation of the special Lie algebraic
properties when a simple Lie group~$G$ admits a $\Theta$-positive structure;
here we stress the existence of a special $3$-dimensional subalgebra. We will
draw several consequences from this, in particular the existence of a positive
circle in the flag variety $\mathsf{F}_\Theta$ (see
Section~\ref{sec:positive-maps}), that is used in
\cite[Proposition~2.10]{GLW}.

\subsection{The \texorpdfstring{$\Theta$}{Θ}-principal subalgebra}
\label{sec:theta-princ-subalg}

The split Lie subalgebra~$\mathfrak{h}_\Theta$ constructed in the previous section admits a special subalgebra
(rather a conjugacy class of subalgebras), called the \emph{principal}
$\mathfrak{sl}_2$. In the correspondence with nilpotent elements given by the
Jacobson--Morozov theorem, it is the $\mathfrak{sl}_2$-subalgebra
corresponding to the (conjugacy class of) regular nilpotent element. A regular nilpotent element
in~$\mathfrak{h}_\Theta$ is for example
$\sum_{\alpha\in \Theta} E_\alpha$ (where $E_\alpha$ are given in
Section~\ref{sec:putting-an-hand}).
\begin{definition}
  \label{defi:theta-princ-subalg}
  The \emph{$\Theta$-principal subalgebra} is\index{$\Theta$-principal subalgebra} (the conjugacy class of) the
  subalgebra of~$\mathfrak{g}$, isomorphic to $\mathfrak{sl}_2(\R)$ and
  represented by the principal subalgebra of~$\mathfrak{h}_\Theta$.
\end{definition}
The induced morphism $\pi_\Theta\colon \mathfrak{sl}_2(\R)\to \mathfrak{g}$
will be called\index{$\pi_\Theta$ the $\Theta$-principal embedding}
\emph{$\Theta$-principal embedding}.  
Kostant \cite[Lemma~5.2]{Kostantsl2} gives formulas for a principal
$\mathfrak{sl}_2$-triple in term of the Chevalley generators, namely there are
positive integers $q_\alpha$ ($\alpha\in\Theta$) such that
\begin{equation*}
    E \coloneqq \sum_{\alpha\in\Theta} q^{1/2}_{\alpha} E_\alpha, \
  F \coloneqq \sum_{\alpha\in\Theta} q^{1/2}_{\alpha} F_\alpha,\ 
  D \coloneqq \sum_{\alpha\in\Theta} q_\alpha D_\alpha
\end{equation*}
is a principal $\mathfrak{sl}_2$-triple.
The following lemma is certainly well known.
\begin{lemma}
  \label{lem:principal-weyl-group}
  Let $H_\Theta$ be the connected subgroup of~$G$ whose Lie algebra is~$\mathfrak{h}_\Theta$.
  
  Let\index{$W({H}_\Theta)$ the Weyl group of~$\mathfrak{h}_\Theta$} $W({H}_\Theta)$ be the Weyl group of~${H}_\Theta$ (for
  the choice of Cartan subalgebra $\mathfrak{c}=\bigoplus_{\alpha\in\Theta}
  \R H_\alpha$). Then $\dot{s}=\exp\bigl(\frac{\pi}{2}(E-F)\bigr)$ belongs to the
  normalizer of~$\mathfrak{c}$ in the maximal compact subgroup
  $K\cap H_\Theta$, and it represents the longest
  length element of~$W({H}_\Theta)$.
\end{lemma}
\begin{proof}
  One has $\Ad(s)\cdot D = -D$
  (from relation valid in $\mathfrak{sl}_2(\R)$). Also $s$~belongs to $K\cap
  H_\Theta$ since $\tau(E-F)=E-F$ belongs to~$\mathfrak{k}$.
  
  The regular semisimple element~$D$ belongs to a unique Cartan
  subalgebra of~$\mathfrak{h}_\Theta$; thus $\Ad(s)\cdot
  \mathfrak{c}$ is the unique Cartan subalgebra containing~$-D$ and is
  therefore equal to~$\mathfrak{c}$. This means that $s$~belongs to the
  normalizer of~$\mathfrak{c}$ and represents an element in the Weyl
  group~$W({H}_\Theta)$. 

  Moreover $D$~belongs to a unique Weyl chamber $\mathfrak{c}^+$
  of~$\mathfrak{c}$ and reasoning as above we get that $\Ad(s)\cdot
  \mathfrak{c}^+= -\mathfrak{c}^+$. Since this equality characterizes
  the representatives of the longest length element of~$W({H}_\Theta)$, we
  get the wanted result.  
\end{proof}

\begin{remarks}
  \begin{enumerate}[leftmargin=*]
  \item The Cartan subspace~$\mathfrak{c}$ is equal to
    $\mathfrak{a}_{\Delta\smallsetminus \Theta}$, \ie{} to the Cartan subspace
    of $S_{\Delta\smallsetminus \Theta}$.
  \item The conclusion of the lemma holds as soon as the element $D=[E,F]$ is regular.
  \end{enumerate}
\end{remarks}

\begin{examples}
  When $G=\Sp(2n,\R)$, let us explain the subgroup corresponding to the
  $\Theta$-principal subalgebra. For the case $\Theta=\Delta$, the subgroup is
  the image of the irreducible representation of $\SL_2(\R)$ of
  dimension~$2n$. For the case when $\Theta =\{ \alpha_n\}$, the subgroup is the image of the diagonal embedding of $\SL_2(\R)$, i.e.\ the set of block matrices, with blocks of size~$n$ all scalar multiples of the
  identity matrix, alternatively there is a natural embedding of $\SL_2(\R)^n$
  into $\Sp(2n, \R)$ and the image of the diagonal subgroup of $\SL_2(\R)^n$
  by this embedding is the $\Theta$-principal $\SL_2(\R)$.

  When $G=\SO(p+1, p+k)$ and $\Theta=\{\alpha_1, \dots, \alpha_p\}$, the
  group~$H_\Theta$ is $\SO^\circ(p+1,p)$ (naturally embedded in~$G$) and the
  $\Theta$-principal subgroup is the irreducible $\SL_2(\R)$ sitting in~$H_\Theta$.
\end{examples}

\begin{remark}
Motivated by the introduction of $\Theta$-positivity, Bradlow, Collier, García-Prada, Gothen and Oliveira introduce in~\cite{Bradlow_Collier_etal} the notion of magical nilpotent elements and of magical $\mathfrak{sl}_{2}$-triples. They further observe that given a complex Lie group and a magical nilpotent element~$E$ there is a canonical real form $\mathfrak{g}$ associated with~$E$, and this real form admits a $\Theta$-positive structure. 
It is clear from the above construction, that the split real form
$\mathfrak{h}_\Theta$ is the split real subalgebra denoted $\mathfrak{g}(E)$ in~\cite{Bradlow_Collier_etal}, and the image of the  embedding  $\pi_\Theta\colon \mathfrak{sl}_{2} \rightarrow \mathfrak{g}$ is the (real) magical $\mathfrak{sl}_{2}$-triple in $\mathfrak{g}$. 
\end{remark}

\section{The \texorpdfstring{$\Theta$}{Θ}-Weyl group}\label{sec:Weyl}
In this section we introduce the $\Theta$-Weyl group, a specific subgroup of
the Weyl group of~$G$. The $\Theta$-Weyl
group will be crucial in the parametrization of the unipotent positive
semigroup (see Section~\ref{sec:positive_semigroup}). We further introduce the notion of the
$\Theta$-length on~$W$ which is useful to establish
some properties of the $\Theta$-Weyl group and is also used in Section~\ref{sec:bruh-decomp-cones}.  We show that the $\Theta$-Weyl
group normalizesf~$W_{\Delta\smallsetminus\Theta}$.

Let~$W$ be the Weyl group of $G$. Recall that the generators corresponding to the simple
roots $\alpha \in \Delta$ are denoted by $s_\alpha$, and that $\ell(w)$ is the length of an
element~$w$ with respect to this generating set (Section~\ref{sec:weyl-group}).

\subsection{Longest length elements and associated involutions}
\label{sec:long-elem-some}

We first recall some general properties of Weyl groups and their subgroups.
In this paragraph, unless stated otherwise,
 $\Theta$ is any subset of $\Delta$, not necessarily one that is associated to a $\Theta$-positive structure. 

For a subset~$F$ of~$\Delta$ the subgroup generated by
$\{s_\alpha\}_{\alpha\in F}$\index{$W_F$ the subgroup generated by
  $\{s_\alpha\}_{\alpha\in F}$ (for $F\subset \Delta$)}
will be denoted by~$W_F$, in particular  $W_{\Delta \smallsetminus \Theta}$~is 
the subgroup generated by $s_\alpha$ with $\alpha \in \Delta \smallsetminus
\Theta$, \ie it is the Weyl group of~$S_\Theta$. Note that the restriction
of~$\ell$ to~$W_F$ is equal to the word length on~$W_F$ equipped with its
generating set $\{s_\alpha\}_{\alpha\in F}$.
We
denote by~$w_\Delta$ the longest length element in~$W$, and by~$w_F$ the
longest length element in~$W_F$,\index{$w_F$ the
longest length element in~$W_F$} therefore $w_{\Delta \smallsetminus \Theta}$ is
the longest length element in~$W_{\Delta \smallsetminus \Theta}$.

\begin{lemma}\cite[Proposition~3.9]{BorelTits}\label{lem:decomp}
  The element $w^{\Theta}_{\max} \in W$ defined by the equality
  \index{$w^{\Theta}_{\max}$ the shortest length element in the coset
    $w_\Delta W_{\Delta\smallsetminus \Theta}$}
  \[w_\Delta = w^{\Theta}_{\max} w_{\Delta \smallsetminus \Theta}\] satisfies
  $\ell(w_\Delta) = \ell(w^{\Theta}_{\max}) + \ell(w_{\Delta \smallsetminus \Theta})$. Furthermore
  $w^{\Theta}_{\max}$ is the unique element of  minimal length in the coset
  $w_\Delta W_{\Delta \smallsetminus \Theta}$.
\end{lemma}

In the case when $G$ carries a $\Theta$-positive structure we want to apply this lemma 
when we consider the Weyl
group associated 
to the diagram
$ \{ \alpha_\Theta\} \cup\Delta \smallsetminus \Theta \subset \Delta$. Applying Lemma~\ref{lem:decomp} we then get the following
corollary.
\begin{corollary}\label{cor:decomp_Herm}
If $G$ admits a $\Theta$-positive structure with $\Theta\neq \Delta$. We consider 
  the\index{$\sigma_{\alpha_\Theta}$ the shortest length element in $w_{\{
      \alpha_\Theta\} \cup \Delta \smallsetminus \Theta}
    W_\Delta\smallsetminus \Theta$}
  element $\sigma_{\alpha_\Theta} \in W_{\{ \alpha_\Theta\} \cup\Delta \smallsetminus \Theta}
  $ defined by the equality
  \[w_{\{ \alpha_\Theta\} \cup \Delta \smallsetminus \Theta} = \sigma_{\alpha_\Theta}
  w_{\Delta \smallsetminus \Theta}.\] This element satisfies $\ell( w_{\{ \alpha_\Theta\}
    \cup \Delta \smallsetminus \Theta}) = \ell(\sigma_{\alpha_\Theta}) +\ell(
  w_{\Delta \smallsetminus \Theta}) $, it is the element of minimal length in
  the coset $w_{\{ \alpha_\Theta\} \cup \Delta \smallsetminus \Theta}
  W_{\Delta\smallsetminus \Theta}$.
\end{corollary}

We now establish the following lemma. Applying it to the case when $G$~admits
a $\Theta$-positive structure, we get that the above elements are of order~$2$
(cf.\ Corollary~\ref{cor:order2}).

\begin{lemma} \label{lemma:order2-theta-weyl-group} 
  Let~$W=\langle
    s_\alpha\rangle_{\alpha\in \Delta}$ be a finite Coxeter group, let
    $w_\Delta\in W$ be its longest length element, and let $\alpha\mapsto
    \bar{\alpha} $ be the involution of~$\Delta$ defined by the equalities
    $w_\Delta s_\alpha w_{\Delta}^{-1} =s_{\bar{\alpha}}$. Let also
    $\Theta\subset \Delta$ be 
    invariant by the involution
    $\alpha\mapsto \bar{\alpha}$ and denote by $w_{\Delta\smallsetminus
      \Theta}$ the longest length element of the finite Coxeter group $\langle
    s_\alpha\rangle_{ \alpha\in \Delta\smallsetminus\Theta}$. Finally set
    $w_r = w_{\Delta} w^{-1}_{\Delta\smallsetminus\Theta}$. Then
    \begin{enumerate}
    \item \label{item:i:lemma-order2} The elements $w_\Delta$ and
      $w_{\Delta\smallsetminus \Theta}$ are of order~$2$.
    \item \label{item:ii:lemma-order2} One has $w_\Delta
      w_{\Delta\smallsetminus \Theta} w_{\Delta}^{-1} =
      w_{\Delta\smallsetminus \Theta}$. 
    \item \label{item:iii:lemma-order2} The element~$w_r$ commutes
      with~$w_\Delta$ and with~$w_{\Delta\smallsetminus \Theta}$ and is of order~$2$.
    \end{enumerate}
\end{lemma}
\begin{proof}
Point~(\ref{item:i:lemma-order2})
    follows from the fact that the inverse of a reduced
  expression of $w_{\Delta}$ is a reduced expression of~$w^{-1}_{\Delta}$ and
  hence uniqueness of the longest length element implies the equality
  $w^{-1}_{\Delta}=w_\Delta$. Point~(\ref{item:ii:lemma-order2}) follows from
  the fact that if $s_{\alpha_1}\cdots s_{\alpha_N}$ is a reduced expression
  of $w_{\Delta\smallsetminus \Theta}$ then $s_{\bar{\alpha}_1} \cdots
  s_{\bar{\alpha}_N}$ is a reduced expression of $w_\Delta
  w_{\Delta\smallsetminus \Theta} w^{-1}_{\Delta}$ and hence the equality $w_\Delta
  w_{\Delta\smallsetminus \Theta} w^{-1}_{\Delta} =   w_{\Delta\smallsetminus
    \Theta}$ holds again from the uniqueness of the longest
  length element (applied in the subgroup $\langle
    s_\alpha\rangle_{ \alpha\in \Delta\smallsetminus\Theta}$ this
    time). 
    Point~(\ref{item:iii:lemma-order2}) is now an immediate consequence
  of~(\ref{item:i:lemma-order2}), (\ref{item:ii:lemma-order2}), and the
  definition of~$w_r$.
\end{proof}

\begin{corollary}\label{cor:order2}
  If $G$ admits a $\Theta$-positive structure with $\Theta\neq
  \Delta$, then the elements~$w_\Delta$ and
  $w_{\{\alpha_\Theta\} \cup \Delta\smallsetminus\Theta}$ commute, and 
  the elements
    $w^{\Theta}_{\max}$ and $\sigma_{\alpha_\Theta}$ are of
    order~$2$.
\end{corollary}

\begin{proof}
Observe that whenever the Dynkin diagram contains a double arrow,  the
involution $\alpha \to \bar{\alpha}$ is the identity and the hypothesis of
Lemma~\ref{lemma:order2-theta-weyl-group} are automatically
satisfied.

The Dynkin diagram of the root system~$\Delta$ contains a double arrow, hence 
Lemma~\ref{lemma:order2-theta-weyl-group} implies  that
$w^{\Theta}_{\max}$ is of order~$2$ and, when applied to $\Theta\smallsetminus
\{ \alpha_\Theta\}$, that $w_\Delta$ and
$w_{\{\alpha_\Theta\} \cup \Delta\smallsetminus\Theta}$ commute.

  The Dynkin diagram of the root system $\{\alpha_\Theta\} \cup \Delta\smallsetminus \Theta$ contains a double arrow.    Lemma~\ref{lemma:order2-theta-weyl-group} applied to this root system
    implies that $\sigma_{\alpha_\Theta}$ is of order~$2$. 
\end{proof}

\subsection{The \texorpdfstring{$\Theta$}{Θ}-length}
\label{sec:theta-length}

We introduce now a function on the Weyl group~$W$ that depends on
$\Theta\subset \Delta$. This function can be introduced in general, but is particularly useful and has good properties when  $G$ admits a
$\Theta$-positive structure, because then the sets $\Theta$ and
$\Delta\smallsetminus \Theta$ are connected by a double arrow (we will assume
this starting from Lemma~\ref{lem:theta-length-on-reduced-expression} below). 

For every~$w$ in~$W$, we denote by $\ell_\Theta(w)$ the minimal number
of occurrences of elements in $\{s_\alpha\}_{\alpha\in \Theta}$ when $w$~is
written as an expression
in the generating set
$\{s_\alpha\}_{\alpha\in \Delta}$, 
\begin{multline}
  \label{eq:lThetaw}
  \ell_\Theta(w) = \min \bigl\{ k\in \N \mid \exists N\in \N,\, (\alpha_1, \dots,
  \alpha_N)\in \Delta^N \text{ with }\bigr.\\
\bigl. w=s_{\alpha_1}\cdots s_{\alpha_N} \text{
  and } k=\sharp\{ j\in \llbracket 1, N\rrbracket \mid \alpha_j \in \Theta\} \bigr\}.
\end{multline}

\begin{definition}
  The function \(\ell_\Theta\colon W\to \N\) is called the
  \emph{$\Theta$-length}.\index{$\ell_\Theta$ the
  {$\Theta$-length}}
\end{definition}

One has $\ell_\Theta(s_\alpha)=1$ for every~$\alpha$ in~$\Theta$,
$\ell_\Theta(s_\alpha)=0$ for every~$\alpha$
in~\mbox{$\Delta \smallsetminus \Theta$}, and $\ell_\Theta(ab)\leq
\ell_\Theta(a)+\ell_\Theta(b)$ for every~$a$ and~$b$ in~$W$.

When $\Theta=\Delta$, the $\Theta$-length coincides with the length function
on~$W$ already introduced.

The $\Theta$-length gives a simple characterization
of~$W_{\Delta\smallsetminus\Theta}$.

\begin{lemma}
  \label{lemma:theta-length-0=W_Levi}
  The subgroup~$W_{\Delta\smallsetminus\Theta}$ is the set of elements of zero
  $\Theta$-length:
  \[ W_{\Delta\smallsetminus\Theta} =\{ x\in W \mid \ell_\Theta(x)=0\}.\]
\end{lemma}
\begin{proof}
  Indeed, for any~$x$ in~$W$, the following are equivalent:
  \begin{itemize}
  \item $\ell_\Theta(x)=0$
  \item  there exist $N\in \N$ and $\alpha_1, \dots, \alpha_N$
    in $\Delta$ such that $x= s_{\alpha_1}\cdots s_{\alpha_N}$ and $\{ i \in
    \llbracket 1, N \rrbracket \mid \alpha_i\in \Theta\} =\emptyset$,
  \item  there exist $N\in \N$ and $\alpha_1, \dots, \alpha_N$
    in $\Delta\smallsetminus\Theta$ such that $x= s_{\alpha_1}\cdots s_{\alpha_N}$,
  \item  $x$~belongs to~$W_{\Delta\smallsetminus\Theta}$.\qedhere
  \end{itemize}
\end{proof}

More is true: the $\Theta$-length is $W_{\Delta\smallsetminus\Theta}$-invariant.
\begin{lemma}\label{lem:theta-length-invariant}
  The function~$\ell_\Theta$ is invariant under the subgroup
  $W_{\Delta\smallsetminus \Theta}$:     for every $w\in W$ and
  every $x\in W_{\Delta\smallsetminus\Theta}$ one has $\ell_{\Theta}(w)= \ell_{\Theta}(xw)= \ell_{\Theta}(wx) $.
\end{lemma}
\begin{proof}
  By symmetry we will prove the equality only for right multiplication by~$W_{\Delta\smallsetminus\Theta}$. 
  By the subadditivity
  of~$\ell_\Theta$ and Lemma~\ref{lemma:theta-length-0=W_Levi}, 
for every $w\in W$ and every $x\in W_{\Delta\smallsetminus\Theta}$, one has
  \(\ell_\Theta(wx) \leq \ell_\Theta(w).\)

Applying the above bound to the pair $(wx, x^{-1})$ replacing $(w,x)$, we
obtain $\ell_\Theta( (wx)x^{-1})\leq \ell_\Theta(wx)$ and hence the sought for
equality since $(wx)x^{-1}=w$.
\end{proof}

A direct reformulation is
\begin{corollary}
  \label{coro:theta-length-factors}
  The $\Theta$-length factors through 
  \( W_{\Delta\smallsetminus\Theta}\backslash W /
  W_{\Delta\smallsetminus\Theta}\), the quotient
  of~$W$ by the left-right action of~$W_{\Delta\smallsetminus\Theta}$.
\end{corollary}

The next result is specific to $\Theta$-positivity, more accurately it relies
on the fact that $\Theta$ and $\Delta\smallsetminus \Theta$ are connected only
by double arrows.

\begin{lemma}\label{lem:theta-length-on-reduced-expression}
  Suppose that $G$ admits a $\Theta$-positive structure.
 The minimum in Equation~\eqref{eq:lThetaw} is then achieved on all the reduced
    expressions: for every $w\in W$  and for
    every $(\alpha_1, \dots, \alpha_{\ell(w)})\in \Delta^{\ell(w)}$ such that $w=s_{\alpha_1}
    \cdots s_{\alpha_{\ell(w)}}$, one has
    \[ \ell_\Theta(w) = \sharp \{ j\in \llbracket 1, \ell(w)\rrbracket \mid \alpha_j\in \Theta\}.\]
\end{lemma}
\begin{proof}
  For this proof we will denote, for every~$n$ in~$\N$ and every $(\alpha_1,\dots,
  \alpha_n)\in \Delta^n$, by $\hat{\ell}_\Theta(\alpha_1,\dots, \alpha_n)$ the
  number of occurrences of elements in~$\Theta$ in this finite sequence:
  \[ \hat{\ell}_\Theta(\alpha_1,\dots, \alpha_n) = \sharp \{ j\in \llbracket 1, n\rrbracket \mid \alpha_j
    \in \Theta\}.\]
  This function is additive:
  \[\hat{\ell}_\Theta(\alpha_1, \dots, \alpha_n, \alpha_{n+1},\dots, \alpha_m) =
  \hat{\ell}_\Theta(\alpha_1, \dots, \alpha_n) + \hat{\ell}_\Theta(\alpha_{n+1}, \dots,
  \alpha_m).\]

  By definition, for every element~$w$ of~$W$,  $\ell_\Theta(w)$ is
  the infimum of the $\hat{\ell}_{\Theta}( \alpha_1,\dots, \alpha_n)$ for
  $(\alpha_1, \dots, \alpha_n)$ varying among the expressions of~$w$:
  $w=s_{\alpha_1}\cdots s_{\alpha_n}$. 

  If $n'\leq n$ and if $(\alpha^{\prime}_{1}, \dots \alpha^{\prime}_{n'})$ is
  obtained from $(\alpha_1, \dots, \alpha_n)$ by removing some of the entries,
  then $\hat{\ell}_\Theta(\alpha^{\prime}_{1}, \dots \alpha^{\prime}_{n'}) \leq
  \hat{\ell}_\Theta(\alpha_1, \dots, \alpha_n)$.

  If $(\alpha_1, \dots, \alpha_n)$ and $(\beta_1, \dots, \beta_n)$ differ by a
  braid relation then
  \[\hat{\ell}_\Theta (\alpha_1, \dots, \alpha_n) =
  \hat{\ell}_\Theta(\beta_1,\dots, \beta_n).\] Indeed, by the additivity of the
  function $\hat{\ell}_\Theta$ on expressions, one needs to check this equality only
  in the case when $n$~is the order of $s_{\alpha}
  s_{\alpha'}$ ($\alpha,\alpha'\in \Delta$) and one has $(\alpha_1, \alpha_2,
  \dots) = (\alpha, \alpha', \dots)$ and $(\beta_1, \beta_2,
  \dots)=(\alpha',\alpha,\dots)$ (precisely  $(\alpha_j, \beta_j)
  =(\alpha,\alpha')$ for every odd $j\leq n$ and  $(\alpha_j, \beta_j)
  =(\alpha',\alpha)$ for every even $j\leq n$). When $\alpha$ and $\alpha'$
  both belong to~$\Theta$, we have $\hat{\ell}_\Theta(\alpha_1, \dots, \alpha_n)=n$ and 
  $\hat{\ell}_\Theta(\beta_1, \dots, \beta_n)=n$; When $\alpha$ and $\alpha'$
  both belong to~$\Delta\smallsetminus\Theta$, we have $\hat{\ell}_\Theta(\alpha_1, \dots, \alpha_n)=0$ and 
  $\hat{\ell}_\Theta(\beta_1, \dots, \beta_n)=0$.
  The last case to consider (up to exchanging the $2$ sequences) is when $\alpha$ belongs to~$\Theta$ and
  $\alpha'$ belongs to~$\Delta\smallsetminus\Theta$; in this case, we
  know that $n$~is equal to~$4$ (since $\alpha$ and $\alpha'$ are connected by
  a double arrow) and  we have $\hat{\ell}_\Theta(\alpha_1, \dots, \alpha_n)=2$ and 
  $\hat{\ell}_\Theta(\beta_1, \dots, \beta_n)=2$. In every case, the equality
  $\hat{\ell}_\Theta(\alpha_1, \dots, \alpha_n)=\hat{\ell}_\Theta(\beta_1, \dots,
  \beta_n)$ holds.

  The conclusions of the lemma follow from these remarks since:
    \begin{itemize}[leftmargin=*]
  \item from every expression $(\alpha_1, \dots, \alpha_n)$ of~$w$ one can
    deduce a reduced expression $(\beta_1, \dots, \beta_m)$ by
    applying successively a finite number of braid relations or removing subexpressions of the
    form $(\alpha,\alpha)$ ($\alpha\in \Delta$) so that
    $\hat{\ell}_\Theta(\beta_1, \dots, \beta_m) \leq
    \hat{\ell}_\Theta(\alpha_1, \dots, \alpha_n)$, therefore
    \[
    \ell_\Theta(w) = \inf \{ \hat{\ell}_\Theta(\beta_1, \dots, \beta_m) \mid
    (\beta_1, \dots, \beta_m) \text{ is a reduced expr.\ of } w\},\]
    and
  \item two reduced expressions $(\beta_1, \dots, \beta_m)$ and $(\gamma_1,
    \dots, \gamma_m)$ of~$w$ can be obtained from one another by
    applying successively a finite number of braid relations so that
    $\hat{\ell}_\Theta(\beta_1, \dots, \beta_m) = \hat{\ell}_\Theta(\gamma_1, \dots, \gamma_m)$.
  \end{itemize}
  Thus $\ell_\Theta(w) = \hat{\ell}_\Theta(\beta_1, \dots, \beta_m)$ for
  every reduced expression $(\beta_1, \dots, \beta_m)$ of~$w$.
\end{proof}

\subsection{The \texorpdfstring{$\Theta$}{Θ}-Weyl group and its generators}
\label{sec:theta-weyl-group}

We assume for the rest of this section that $G$~is a simple Lie group admitting a $\Theta$-positive structure. 

\begin{notation}\label{nota:sigma-alpha}
  We set
  \begin{itemize}[leftmargin=*]
  \item When $\Theta=\Delta$, $\sigma_\alpha=s_\alpha$ for all~$\alpha$
    in~$\Delta$;\index{$\sigma_\alpha$ ($\alpha$ in $\Theta$) the element
      of~$W$ equal to $s_\alpha$ (if $\alpha\neq\alpha_\Theta$) or to
      $\sigma_{\alpha_\Theta}$  (if $\alpha=\alpha_\Theta$)}
  \item When $\Theta\neq \Delta$, for every $\alpha\in \Theta\smallsetminus\{ \alpha_\Theta\}$,
    $\sigma_\alpha=s_\alpha$;
  \item And  $R(\Theta)$ to be equal to  $\{
    \sigma_\alpha\}_{\alpha\in\Theta}$ (recall that $\sigma_{\alpha_\Theta}$
    is defined in Corollary~\ref{cor:decomp_Herm}).\index{$R(\Theta)$ the subset
      $\{\sigma_\alpha\}_{\alpha\in \Theta}$ of~$W$}
  \end{itemize}
\end{notation}

\begin{definition}\label{defi:theta-weyl-group}
  We define the \emph{$\Theta$-Weyl group} to be the subgroup~$W(\Theta)$
  of~$W$ generated by~$R(\Theta)$.\index{$W(\Theta)$ the subgroup of¬$W$
    generated by $R(\Theta)$} 
\end{definition}

We first investigate the Weyl group of the split
real form~$\mathfrak{h}_\Theta$ as a subgroup of~$W$.

\begin{lemma}
  \label{lemma:Weyl-group_H_J}
  Let $G$~be a simple Lie group admitting a $\Theta$-positive structure with
  $\Theta\neq \Delta$. Then the standard generating Coxeter system of the Weyl
  group $W({H}_\Theta)$ of the split real form~$\mathfrak{h}_\Theta$
  is given by 
  \[ s_\alpha\ \text{for all }\alpha\in \Theta\smallsetminus\{\alpha_\Theta\},\
    \text{and }
    w_{\{\alpha_\Theta \}\cup \Delta\smallsetminus \Theta}.\]
\end{lemma}
\begin{proof}
  Let $(E_\alpha, F_\alpha, D_\alpha)_{\alpha\in \Theta}$ be the
  $\Theta$-system generating~$\mathfrak{h}_\Theta$. Then the generating set of
  $W({H}_\Theta)$ is represented by the elements $\exp(
  \frac{\pi}{2}(E_\alpha-F_\alpha))$ for~$\alpha$ varying in~$\Theta$.

  For $\alpha\neq \alpha_\Theta$, one has $(E_\alpha, F_\alpha)=(X_\alpha,
  -\tau(X_{\alpha}))$ and $\exp(
  \frac{\pi}{2}(E_\alpha-F_\alpha)) = \dot{s}_\alpha$ represents~$s_\alpha$.

  For $\alpha=\alpha_\Theta$, one has (with Notation~\ref{nota:ZkYkWk}), $E_\alpha-F_\alpha = (Z_0-Y_0)+
  (Z_1-Y_1)+\cdots + (Z_d-Y_d)$. Since the elements $(Z_k-Y_k)$ pairwise
  commute
  (Lemma~\ref{lemma:zero_bracket_ZkYkWk}), 
  \begin{align*}
    \exp\Bigl( \frac{\pi}{2}(E_\alpha-F_\alpha)\Bigr)
    &= \prod_{k=0}^{d} \exp\Bigl( \frac{\pi}{2}(Z_k-Y_k)\Bigr)\\
    \intertext{and since, for all~$k$ in $\llbracket 1, d\rrbracket$ (again
    with $\dot{s}_k =\exp\bigl(\frac{\pi}{2}(X_k+\tau(X_k))\bigr)$, cf.\ Notation~\ref{nota:ZkYkWk}), $Z_k-Y_k=
  \Ad( \dot{s}_k)(Z_{k-1}-Y_{k-1})$,}
    \exp\Bigl( \frac{\pi}{2}(Z_k-Y_k)\Bigr)
    &= \dot{s}_k \exp\Bigl( \frac{\pi}{2}(Z_{k-1}-Y_{k-1})\Bigr) \dot{s}_{k}^{-1}.
  \end{align*}
  Denote  $\dot{s}_0= \exp(
  \frac{\pi}{2}(Z_0-Y_0))$ so that $\dot{s}_0,\dot{s}_1,\dots, \dot{s}_d$
  are lifts of the standard generators $s_0, s_1,\dots, s_d$ of the Weyl group
  $W_{\{\alpha_\Theta\}\cup \Delta\smallsetminus\Theta}$ which is of type
  $\lietype{C}_{d+1}$. Since the longest length element in $W_{\{\alpha_\Theta\}\cup
    \Delta\smallsetminus\Theta}$ is the product of $s_0$ and its conjugate by
  $s_1$, $s_2 s_1$, \dots, $s_d\cdots s_1$ (cf.\
  Appendix~\ref{app:bn}) and since the above formula shows
  that  $\exp( \frac{\pi}{2}(E_\alpha-F_\alpha))$ is the product of $\dot{s}_0$ and its conjugate by
  $\dot{s}_1$, $\dot{s}_2\dot{s}_1$, \dots, $\dot{s}_d\cdots \dot{s}_1$, the
  result follows.
\end{proof}

\begin{corollary}
  \label{cor:w-h-theta-centralizes}
  The group $W({H}_\Theta)$ centralizes $W_{\Delta \smallsetminus \Theta}$.
\end{corollary}
\begin{proof}
    Let~$\alpha$ be in~$\Theta$. If $\alpha\neq \alpha_\Theta$, it means that
  there are no arrows in the Dynkin diagram between~$\alpha$ and
  $\Delta\smallsetminus \Theta$; thus $s_\alpha$ commutes with~$s_\beta$ for every~$\beta$ in $\Delta\smallsetminus\Theta$, this means that
  $s_\alpha$~centralizes $W_{\Delta\smallsetminus\Theta}$.

  If $\alpha=\alpha_\Theta$,  we know that in this
  situation the element $w_{\{\alpha_\Theta\} \cup
    \Delta\smallsetminus\Theta}$ is in the center of $W_{\{\alpha_\Theta\} \cup
    \Delta\smallsetminus\Theta}$ since it is $-\id$ as en element of
  $\GL(\mathfrak{a}^*)$ (cf.\ Appendix~\ref{app:bn}). Thus the generating set
  of $W({H}_\Theta)$ is contained in the centralizer
  of~$W_{\Delta\smallsetminus\Theta}$, this implies the statement.
\end{proof}

We next show that the $\Theta$-Weyl group is isomorphic to the Weyl group of
the split real form~$\mathfrak{h}_\Theta$.

\begin{lemma}
  \label{lem:WHTheta-iso-WTheta}
  Let $G$~be a simple Lie group admitting a $\Theta$-positive structure with
  $\Theta\neq \Delta$.

  There exists then a unique homomorphism $\phi\colon
  W({H}_\Theta)\to W$ such that, for all~$\alpha$ in
  $\Theta\smallsetminus \{ \alpha_\Theta\}$, $\phi( s_\alpha)=e$ and
  $\phi(w_{\{\alpha_\Theta\}\cup \Delta \smallsetminus \Theta})=w_{\Delta\smallsetminus\Theta}$.

  The map $\psi\colon W({H}_\Theta)\to W, \,  w\mapsto w \phi(w)$ is a group homomorphism, it is one-to-one and
  its image is equal to the $\Theta$-Weyl group $W(\Theta)$.
\end{lemma}
\begin{proof}
  Uniqueness of~$\phi$ follows from the fact it is defined on a generating set
  of~$W({H}_\Theta)$. To prove existence, we have to check that
  $\phi$~preserves the relations in $W({H}_\Theta)$, i.e.\ the fact
  that the generators are of order~$2$ and the braid relations. Since
  $w_{\Delta\smallsetminus \Theta}$ is of order~$2$, we just have to check the
  compatibility with the braid relations (cf.\ Section~\ref{sec:weyl-group}); for~$\alpha$ and~$\beta$ in $\Theta\smallsetminus
  \{\alpha_\Theta\}$, the $2$~members of the braid relation involving
  $s_\alpha$ and $s_\beta$ are sent to~$e$ in~$W$; when one of $\alpha$
  or~$\beta$ is equal to~$\alpha_\Theta$, the members of the braid relations
  have length~$2$ or~$4$ and their images under~$\phi$ are both equal to
  $w_{\Delta\smallsetminus \Theta}$ (if the length is~$2$) or both equal to $w_{\Delta\smallsetminus
    \Theta}^{2}=e$ (if the length is~$4$). This concludes the existence of~$\phi$.

  Since $w_{\Delta\smallsetminus\Theta}$ centralizes $W({H}_\Theta)$ (Corollary~\ref{cor:w-h-theta-centralizes}),
  the map~$\psi$ is indeed an homomorphism.
  The image by~$\psi$ of the generating set of $W({H}_\Theta)$ (cf.\ Lemma~\ref{lemma:Weyl-group_H_J}) is equal to
  $R(\Theta)$ and thus the image of~$\psi$ is equal to $W(\Theta)$.
  For all~$w$
  in~$W({H}_\Theta)$, $\ell_\Theta( \psi(w))=\ell_\Theta(w)$ since
  $\phi(w)$ belongs to $W_{\Delta\smallsetminus \Theta}$ (Lemma~\ref{lem:theta-length-invariant}). Thus
  the kernel of~$\psi$ is contained in $W({H}_\Theta) \cap
  W_{\Delta\smallsetminus\Theta}$ and is thus reduced to~$e$ since every
  nontrivial element of
  $W({H}_\Theta)$ has positive $\Theta$-length (cf.\ Lemma~\ref{lem:theta-length-on-reduced-expression}).
\end{proof}

As a corollary we have
that
$W(\Theta)$ is a Coxeter group of the same type as the type of the Lie algebra~$\mathfrak{h}_\Theta$.

\begin{corollary}\label{coro:WTheta}
  Let $G$~be a simple Lie group admitting a $\Theta$-positive structure.  Then
  $(W(\Theta), R(\Theta))$ is a Coxeter system of the following type:
  \begin{enumerate}[leftmargin=*]
  \item If $G$ is a split real form and $\Theta = \Delta$, $(W(\Theta),
    R(\Theta) )$ has the same type as~$G$.
  \item If $G$ is Hermitian of 
    tube type and of real rank~$r$ and if $\Theta = \{ \alpha_r\}$, then
    $(W(\Theta), R(\Theta) )$ is of type~$\lietype{A}_1$.
  \item If $G$ is locally isomorphic to $\SO(p+1,p+k)$, $p>0$, $k>1$, and
    $\Theta = \{ \alpha_1, \dots, \alpha_{p}\}$, then
    $(W(\Theta), R(\Theta) )$ is of type~$\lietype{B}_{p}$.
  \item If $G$ is the real form of $\lietype{F}_4, \lietype{E}_6, \lietype{E}_7$, or of~$\lietype{E}_8$, whose reduced
    root 
    system is of type~$\lietype{F}_4$, and $\Theta = \{ \alpha_1, \alpha_2\}$, then
    $(W(\Theta) , R(\Theta))$ is of type~$\lietype{G}_2$.
  \end{enumerate}
\end{corollary}

\begin{proof}
  When $G$ is a split real form, then $\Theta = \Delta$, then
  $R(\Theta)=\{ s_\alpha\}_{\alpha \in \Delta}$, hence $W(\Theta) = W$.

  Let $G$ be of real rank~$r$ and of Hermitian tube type, and let
  $\Theta = \{ \alpha_r\} = \{ \alpha_\Theta\}$. Consequently $W(\Theta)$ is
  generated by the single order~$2$ element $\sigma_{\alpha_\Theta}$, thus the
  result.

  The other cases follow directly from Lemma~\ref{lem:WHTheta-iso-WTheta}.
\end{proof}

\subsection{The longest length element of \texorpdfstring{$W(\Theta)$}{W(Θ)}}
\label{sec:long-elem-wtheta}

\begin{proposition}\label{prop:longestelementWTheta}
  Let $G$ be a simple Lie group admitting a $\Theta$-positive structure. Let
  $w_\Delta$ be the longest length element in $W$,
  $w_{\Delta \smallsetminus \Theta} 
  $ the
  longest length element in $W_{\Delta \smallsetminus \Theta}$, and
  $ w^{\Theta}_{\max} = w_\Delta w^{-1}_{\Delta \smallsetminus \Theta}$ so that
  $w^{\Theta}_{\max}$ is the unique element of minimal length in the coset
  $w_\Delta W_{\Delta \smallsetminus \Theta}$ (Lemma~\ref{lem:decomp}). Then
  the element $w^{\Theta}_{\max}$ belongs to $W(\Theta)$ and is the
    longest length element of the Coxeter system $( W(\Theta),
    R(\Theta))$.
            \end{proposition}
\begin{proof}
    
  We prove the proposition   by a case by
  case analysis, using the computations done in Appendices~\ref{app:bn}
  and~\ref{app:f4}. We adopt below the notation $x^y=y^{-1}xy$ in a group, one has thus
    $(x^y)^z=x^{yz}$.   \begin{enumerate}[leftmargin=*]
  \item In the case when $G$ is a split real simple Lie group there is nothing
    to prove.
  \item When $G$ is a simple Lie group of Hermitian 
    tube
    type, then $w^{\Theta}_{\max}  =
    \sigma_{\alpha_\Theta}$, which is the longest length element in~$W(\Theta)$ that
    is of type~$\lietype{A}_1$.
  \item When $G = \SO(p+1,p+k)$ with $p>0$ and $k>1$, one has   $ \Delta = \{ \alpha_1 , \dots, \alpha_{p+1}\} $,
    $ \Theta = \{ \alpha_1 , \dots, \alpha_{p}\} $, and
    $\Delta \smallsetminus \Theta = \{ \alpha_{p+1}\}$.\footnote{Observe that
      the standard numbering used here for the simple roots does not
      coincide with the ``constructive'' numbering of the general results
      proved earlier (cf.\ the proof of Lemma~\ref{lemma:Weyl-group_H_J} for
      example), the reindexing being $i \in \llbracket 1,p+1\rrbracket\mapsto
      p+1-i \in \llbracket 0,p\rrbracket$.} Writing as usual $s_i$
    for $s_{\alpha_i}$ and for short $\sigma_p$ for $\sigma_{\alpha_p}$ (and
    $\sigma_i=s_i$ for $i<p$), one
    has $w_{\Delta \smallsetminus \Theta}= s_{p+1}$ and
    a direct calculation 
    gives $w_{\{\alpha_\Theta\}\cup \Delta\smallsetminus \Theta}= s_p
    s_{p+1}s_p s_{p+1}$ (so that $\sigma_p=s_p s_{p+1} s_p$).
        Appendix~\ref{app:bn} gives the following reduced expression for the longest length element
    $w_\Delta$:
    \begin{align}\label{eq:red-exp-w-Delta}
  w_\Delta & = s_{p+1}^{s_{p} \cdots s_1} \, s_{p+1}^{s_{p} \cdots s_2} \cdots s_{p+1}^{s_{p}s_{p-1}}\,
             s_{p+1}^{s_{p}} \, s_{p+1}\\
\label{eq:red-exp-w-max-Theta}  w_{\max}^{\Theta} & = s_{p+1}^{s_{p} \cdots s_1} \, s_{p+1}^{s_{p} \cdots s_2}
                      \cdots  s_{p+1}^{s_{p}s_{p-1}}\, s_{p+1}^{s_{p}}\\
      \notag     & = \sigma_{p}^{s_{p-1} \cdots s_1} \, \sigma_{p}^{s_{p-1}
             \cdots s_2} \cdots \sigma_{p}^{s_{p-1}} \,
             \sigma_{p}\\ 
\notag           & = \sigma_{p}^{\sigma_{p-1} \cdots \sigma_1} \, \sigma_{p}^{\sigma_{p-1}
             \cdots \sigma_2} \cdots \sigma_{p}^{\sigma_{p-1}} \,
             \sigma_{p}. 
\end{align}
 In particular we get 
 that $w^{\Theta}_{\max}$ belongs to $W(\Theta)$ and is the longest length element in
 the Coxeter system $(W(\Theta), R(\Theta))$ of type~$\lietype{B}_p$.
  \item When $G$ is the real form of $\lietype{F}_4, \lietype{E}_6, \lietype{E}_7, \lietype{E}_8$, whose 
    system of restricted roots is of type $\lietype{F}_4$, we have that
    $\Delta = \{ \alpha_1, \alpha_2, \alpha_3, \alpha_4\}$,
    $\Theta = \{ \alpha _1, \alpha_2\}$ and
    $\Delta \smallsetminus \Theta = \{ \alpha_3, \alpha_4\}$. We write again
    $s_i$ for $s_{\alpha_i}$, $\sigma_2$ for $\sigma_{\alpha_2}$, and
    $\sigma_1=s_1$. By Appendix~\ref{app:f4},
    $\sigma_{2} = s_2 s_3
s_4 s_2 s_3 s_2 = 
(s_4 s_2)^{s_3 s_2}$;
    one has $w_{\Delta \smallsetminus \Theta} = 
    s_{4}^{ s_3}$ and a direct calculation given in Appendix~\ref{app:f4} gives the following reduced
    expression for the longest length element 
\begin{align}
\label{eq:red-exp-w-Delta-F4}w_\Delta & = (s_4 s_2)^{ s_3 s_2}\, s_1\, (s_4 s_2)^{ s_3 s_2}\, s_1\, (s_4 s_2)^{ s_3 s_2}\, s_1\,  s_{4}^{s_3},\\
\intertext{thus}
\label{eq:red-exp-w-max-Theta-F4}  w_{\max}^{\Theta} & = (s_4 s_2)^{ s_3 s_2}\, s_1\, (s_4 s_2)^{ s_3 s_2}\, s_1\, (s_4 s_2)^{ s_3 s_2}\, s_1\\
\notag  & = \sigma_2 \sigma_1 \sigma_2 \sigma_1 \sigma_2 \sigma_1
\end{align}
this is precisely the longest length element of $(W(\Theta), R(\Theta))$ which
is of type~$\lietype{G}_2$.\qedhere
\end{enumerate}
\end{proof}

The previous proof also gives the following information about reduced
expressions of elements in $W(\Theta)$.
\begin{lemma}
  \label{lem:reduced-exp-inWTheta-and-W}
  Under the hypothesis of Proposition~\ref{prop:longestelementWTheta},
  let~$x$ be in $W(\Theta)$. Let $\sigma_{\alpha_1}\cdots \sigma_{\alpha_k}$
  be a reduced expression of~$x$ in $(W(\Theta), R(\Theta))$.
  Then, for the lengths in~$W$, one has the equality
  \[\ell(x) = \ell(\sigma_{\alpha_1})+\cdots +\ell(\sigma_{\alpha_k}),\]
  so that $\sigma_{\alpha_1}\cdots \sigma_{\alpha_k}$
              is a reduced product of~$x$ in $(W,\Delta)$.
\end{lemma}

\begin{proof}
  The result is obvious in the split case or in the Hermitian tube type case.
  
  Let $\mathbf{X}\subset R(\Theta)^k$ be the set of reduced expressions of~$x$
  in $W(\Theta)$. We will first prove that the map
  \[ L\colon \mathbf{X}\to \N \mid (\sigma_{\alpha_1},\dots, \sigma_{\alpha_k}
    )\mapsto \ell(\sigma_{\alpha_1})+\cdots +\ell(\sigma_{\alpha_k})\]
  is constant. Since two reduced expressions are obtained from one another by
  successive applications of braid relations, it is enough to prove that the
  map~$L$ coincides on reduced expressions differing by a braid relation; let
  $\alpha$ and $\beta$ be the $2$~elements of~$\Theta$ involved in the braid
  relation; the result is obvious when $\alpha\neq \alpha_\Theta$ and
  $\beta\neq \alpha_\Theta$ since in this cass $\sigma_\alpha=s_\alpha$ and
  $\sigma_\beta=s_\beta$; the result also holds when $\alpha$ or $\beta$ is
  equal to~$\alpha_\Theta$ since the braid relations (in $(W(\Theta),
  R(\Theta))$) involving $\alpha_\Theta$ have the same number of occurences
  of~$\alpha_\Theta$ in both members of the braid relation. Hence $L$~is the
  constant function.

 Let us now prove the result for $w_{\max}^{\Theta}$, the longest length element of the Coxeter group
  $W(\Theta)$. It is therefore enough to find \emph{one} reduced expression $\sigma_{\alpha_1} \cdots
  \sigma_{\alpha_r}$ of $w_{\max}^{\Theta}$ that is a reduced product
  of $w_{\max}^{\Theta}$ in $(W,\Delta)$ or, what amounts to the same, such
  that $\sigma_{\alpha_1} \cdots
  \sigma_{\alpha_r} w_{\Delta\smallsetminus \Theta}$ is a reduced product
  of~$w_\Delta$ in~$W$. For the case of orthogonal groups, the reduced expression of~$w_\Delta$ with the desired property is given in Equation~\eqref{eq:red-exp-w-Delta} and the prefix
  producing $w_{\max}^{\Theta}$ is given in 
  Equation~\eqref{eq:red-exp-w-max-Theta}. For the $\lietype{F}_4$ case the corresponding expressions are given in Equation~\eqref{eq:red-exp-w-Delta-F4} 
  and Equation~\eqref{eq:red-exp-w-max-Theta-F4}.

  We now turn back to the element~$x$. There exists a reduced expression $\sigma_{\alpha_1} \cdots
  \sigma_{\alpha_r}$ of $w_{\max}^{\Theta}$ such that $\sigma_{\alpha_1} \cdots
  \sigma_{\alpha_k}$ is a reduced expression of~$x$. If the result holds for
  $w_{\max}^{\Theta}$, one has
  \begin{multline*}
   \ell(w_{\max}^{\Theta})\leq  \ell(x) + \ell(x^{-1}w_{\max}^{\Theta}) \leq
          \ell(\sigma_{\alpha_1})+\cdots +\ell(\sigma_{\alpha_k}) + \ell(\sigma_{\alpha_{k+1}} \cdots
    \sigma_{\alpha_r})\\ \leq 
    \ell(\sigma_{\alpha_1})+\cdots +\ell(\sigma_{\alpha_k}) +
    \ell(\sigma_{\alpha_{k+1}}) + \cdots
    +\ell(\sigma_{\alpha_r})=\ell(w_{\max}^{\Theta}), 
  \end{multline*}
  and thus all the inequalities here are equalities and in particular
  $\ell(x)= \ell(\sigma_{\alpha_1})+\cdots +\ell(\sigma_{\alpha_k})$ which,
  together with the first part of the proof, implies the result.
\end{proof}

\subsection{Normalizer}
\label{sec:normalizer}

\begin{proposition}
  \label{prop:normalizer-WDeltaminusTheta}
  The group $W(\Theta)$ normalizes $W_{\Delta\smallsetminus \Theta}$: 
\[ W( \Theta) \subset N_W( W_{\Delta\smallsetminus\Theta}) = \{ x\in W \mid x
  W_{\Delta \smallsetminus\Theta} x^{-1}=
  W_{\Delta \smallsetminus\Theta}\}.\]
\end{proposition}
\begin{proof}
  As $W(\Theta)$ is contained in the group generated by
  $W({H}_\Theta)$ and $W_{\Delta \smallsetminus \Theta}$, this
  follows from Corollary~\ref{cor:w-h-theta-centralizes}.
\end{proof}

 \section{Bruhat decomposition and the invariant cones}
\label{sec:bruh-decomp-cones}

In this section we recall some facts on  Bruhat decompositions. We then establish, in the case of $\Theta$-positivity, several facts giving a precise control of the Bruhat cells
containing the elements $\exp(v)$ when $\alpha\in \Theta$ and $v\in \mathring{c}_\alpha$. These properties will play a central role in Section~\ref{sec:positive_semigroup} when we give a parametrization of the positive unipotent semigroup. 

\subsection{Bruhat decomposition}
\label{sec:bruhat-decomposition}
Let $G$ be a simple Lie group. 
We recall here that $G$~can be decomposed under the left-right multiplication
by $P_{\Delta}^{\mathrm{opp}}\times P_{\Delta}^{\mathrm{opp}}$, where
$P_{\Delta}^{\mathrm{opp}}$ is the minimal standard opposite parabolic subgroup.  These orbits are
 called the Bruhat cells and are indexed by elements of~$W$.
We also consider decompositions with respect to the action of $P_{\Delta}^{\mathrm{opp}}\times P_\Theta^{\mathrm{opp}}$ and $P_\Theta^{\mathrm{opp}}\times P_\Theta^{\mathrm{opp}}$, where $\Theta$~is a subset of~$\Delta$. 

Explicitely, for every~$w$ in $W\simeq N_K( \mathfrak{a}) /
C_K(\mathfrak{a})$ (the quotient of the normalizer of~$\mathfrak{a}$ in~$K$ by
its centralizer), let~$\dot{w}$ in $N_K( \mathfrak{a})$ be a representative
of~$w$, the subset
\[ P_{\Delta}^{\mathrm{opp}} \dot{w} P_{\Delta}^{\mathrm{opp}}\]
depends only of~$w$ (and not on the choice of~$\dot{w}$) since
$P_{\Delta}^{\mathrm{opp}}$ contains the centralizer $C_K( \mathfrak{a})$; it is called
the \emph{Bruhat cell indexed by~$w$}; the following equalities hold
\[ P_{\Delta}^{\mathrm{opp}} \dot{w} P_{\Delta}^{\mathrm{opp}} = U_{\Delta}^{\mathrm{opp}} \dot{w}
  P_{\Delta}^{\mathrm{opp}} = P_{\Delta}^{\mathrm{opp}} \dot{w} U_{\Delta}^{\mathrm{opp}}\]
and we will remove the dots in the further notation.

It is well known that (see \eg \cite{Warner}): 
\begin{enumerate}[leftmargin=*]
\item \label{item:1:Bruhat_dec} $G$ is the disjoint union of the
  $P_{\Delta}^{\mathrm{opp}} {w} P_{\Delta}^{\mathrm{opp}}$ for~$w$ varying
  in~$W$.\index{$P_{\Delta}^{\mathrm{opp}} {w} P_{\Delta}^{\mathrm{opp}}$ the
    Bruhat cell associated with~$w$}
\item \label{item:2:Bruhat_dec} $P^{\mathrm{opp}}_{\Theta}$ is the disjoint union of the
  $P_{\Delta}^{\mathrm{opp}} {w} P_{\Delta}^{\mathrm{opp}}$ for~$w$ varying in~$W_{\Delta \setminus
    \Theta}$.
\item \label{item:3:Bruhat_dec} For all~$w_1$ and~$w_2$ in~$W$ such that
  $\ell(w_1 w_2) =\ell(w_1)+\ell(w_2)$, one has
  \[ (P_{\Delta}^{\mathrm{opp}} {w_1} P_{\Delta}^{\mathrm{opp}}) (P_{\Delta}^{\mathrm{opp}} {w_2}
    P_{\Delta}^{\mathrm{opp}}) = P_{\Delta}^{\mathrm{opp}} {w_1} P_{\Delta}^{\mathrm{opp}} w_2 P_{\Delta}^{\mathrm{opp}}
     = P_{\Delta}^{\mathrm{opp}} {w_1} w_2 P_{\Delta}^{\mathrm{opp}}.\]
\item \label{item:4:Bruhat_dec} For all~$\alpha$ in~$\Delta$ and for all~$X$
  in~$\mathfrak{g}_\alpha$, if $X\neq 0$ then
  \[ \exp(X) \in P_{\Delta}^{\mathrm{opp}} s_\alpha P_{\Delta}^{\mathrm{opp}}.\]  
\end{enumerate}

Observe also that $\dot{s}^{-1}_{\alpha}$ also belongs to $P_{\Delta}^{\mathrm{opp}}
s_\alpha P_{\Delta}^{\mathrm{opp}}$ (simply because $\dot{s}^{-1}_{\alpha}$ is a
representative of the class of~$s_\alpha$).
\begin{enumerate}[resume*]
\item \label{item:5:Bruhat_dec} For every~$\alpha$ in~$\Delta$,
  \[P_{\Delta}^{\mathrm{opp}} s_\alpha P_{\Delta}^{\mathrm{opp}} s_\alpha P_{\Delta}^{\mathrm{opp}} =
    P_{\Delta}^{\mathrm{opp}}  \sqcup P_{\Delta}^{\mathrm{opp}} s_\alpha P_{\Delta}^{\mathrm{opp}}.\]
\item For every~$w$ in~$W$, the $P_{\Delta}^\mathrm{opp}\times
  P_{\Theta}^{\mathrm{opp}}$-orbit
  \[ P_{\Delta}^\mathrm{opp} w
    P_{\Theta}^{\mathrm{opp}}\]
  depends only on the class~$x$ of $w$ in $W/ W_{\Delta\setminus\Theta}$
  and
  will sometimes be denoted $P_{\Delta}^\mathrm{opp} x
  P_{\Theta}^{\mathrm{opp}}$.\index{$P_{\Delta}^\mathrm{opp} x
  P_{\Theta}^{\mathrm{opp}}$ the $P_{\Delta}^\mathrm{opp} \times
  P_{\Theta}^{\mathrm{opp}}$-orbit associated with the class $x\in
  W/W_{\Delta\smallsetminus \Theta}$} 
\item The group~$G$ is the disjoint union
  \[ G = \bigsqcup_{\mathclap{x\in W/W_{\Delta\setminus\Theta}}} P_{\Delta}^\mathrm{opp}x
  P_{\Theta}^{\mathrm{opp}}.\]
\item Similar notation will be adopted for the action of $P_{\Theta}^{\mathrm{opp}}\times
  P_{\Delta}^{\mathrm{opp}}$ and of  $P_{\Theta}^{\mathrm{opp}}\times
  P_{\Theta}^{\mathrm{opp}}$: a double orbit $P_{\Theta}^{\mathrm{opp}} w
  P_{\Theta}^{\mathrm{opp}}$ depends only on the class~$a$ of~$w$ in
  $W_{\Delta\setminus\Theta}\backslash W /W_{\Delta\setminus\Theta}$, will be
  denoted $P_{\Theta}^{\mathrm{opp}} a
  P_{\Theta}^{\mathrm{opp}}$\index{$P_{\Theta}^\mathrm{opp} a
  P_{\Theta}^{\mathrm{opp}}$ the $P_{\Theta}^\mathrm{opp} \times
  P_{\Theta}^{\mathrm{opp}}$-orbit associated with the class $a\in
  W_{\Delta\setminus\Theta}\backslash W/W_{\Delta\smallsetminus \Theta}$}  and 
  \[ G =\!\!\!\!\!\!\!\!\!\bigsqcup_{a\in W_{\Delta\setminus\Theta}\backslash W
      /W_{\Delta\setminus\Theta}}\!\!\!\!\!\!\!\!\! P_{\Theta}^{\mathrm{opp}} a
  P_{\Theta}^{\mathrm{opp}}.\]
\end{enumerate}

Application of~(\ref{item:3:Bruhat_dec}) and~(\ref{item:5:Bruhat_dec}) above gives
that, for every $w_1$ and~$w_2$ in~$W$ and for every~$g$ in
$P_{\Delta}^\mathrm{opp} w_1 P_{\Delta}^\mathrm{opp} w_2 P_{\Delta}^\mathrm{opp}$, there exist a
prefix~$x_1$ of~$w_1$ and a suffix~$x_2$ of~$w_2$ such that $g$~belongs to
$P_{\Delta}^\mathrm{opp} x_1 x_2 P_{\Delta}^\mathrm{opp}$. In particular: 
\begin{lemma}
  \label{lem:J-length-prod-bruhat-cells}
  Let $w_1$, $w_2$, and~$x$ be in~$W$ such that $P^{\mathrm{opp}}_{\Theta} x
  P^{\mathrm{opp}}_{\Theta} \subset P^{\mathrm{opp}}_{\Theta} w_1 P^{\mathrm{opp}}_{\Theta} w_2
  P^{\mathrm{opp}}_{\Theta}$. Then
  \[ \ell_\Theta(x)\leq \ell_\Theta(w_1)+\ell_\Theta(w_2).\]
\end{lemma}

The following lemma states the equality between some of the orbits for the two
 groups~$P_{\Delta}^\mathrm{opp}$ and~$P_{\Theta}^\mathrm{opp}$.

\begin{lemma}
  \label{lemma:PThetaequalsPforWTheta}
  For every~$w$ in~$W(\Theta)$, one has
  \[  P_{\Theta}^{\mathrm{opp}} w
  P_{\Theta}^{\mathrm{opp}} = P_{\Delta}^{\mathrm{opp}} w
  P_{\Theta}^{\mathrm{opp}}= U_{\Theta}^{\mathrm{opp}} w
  P_{\Theta}^{\mathrm{opp}} .\]
\end{lemma}

\begin{proof}
  Since $P_{\Theta}^{\mathrm{opp}} = \bigsqcup_{x\in
    W_{\Delta\setminus\Theta}} P_{\Delta}^\mathrm{opp} x P_{\Delta}^{\mathrm{opp}}$, one has
  \[  P_{\Theta}^{\mathrm{opp}} w
    P_{\Theta}^{\mathrm{opp}} = \bigcup_{x\in
      W_{\Delta\setminus\Theta}} P_{\Delta}^\mathrm{opp} x P_{\Delta}^{\mathrm{opp}} w
    P_{\Theta}^{\mathrm{opp}}.\]
  It is thus enough to prove the equality for every~$x$ in
  $W_{\Delta\setminus\Theta}$
  \[  P_{\Delta}^\mathrm{opp} x P_{\Delta}^{\mathrm{opp}} w
    P_{\Theta}^{\mathrm{opp}} =  P_{\Delta}^{\mathrm{opp}} w
    P_{\Theta}^{\mathrm{opp}}.\]
  Since $x$~belongs to~$W_{\Delta\setminus\Theta}$, it is a suffix of the
  longest length element $w_{\Delta\setminus\Theta}$. Similarly $w$ is a prefix of
  $w_{\max}^{\Theta}$. We deduce from this that $xw$ is a subword of
  $w_{\Delta\setminus\Theta} w_{\max}^{\Theta} = w_\Delta$ and in particular
  $\ell(xw) = \ell(x)+\ell(w)$. By property~(\ref{item:3:Bruhat_dec}) above,
  one has
  \begin{align*}
      P_{\Delta}^\mathrm{opp} x P_{\Delta}^{\mathrm{opp}} w
    P_{\Theta}^{\mathrm{opp}} &=  P_{\Delta}^{\mathrm{opp}} x w
    P_{\Theta}^{\mathrm{opp}}\\
    &=  P_{\Delta}^{\mathrm{opp}} w (w^{-1}xw)
    P_{\Theta}^{\mathrm{opp}}\\
\intertext{and, as $w^{-1}xw$ belongs to~$W_{\Delta\setminus\Theta}$
      (Proposition~\ref{prop:normalizer-WDeltaminusTheta}),}
    &=  P_{\Delta}^{\mathrm{opp}} w 
    P_{\Theta}^{\mathrm{opp}}
  \end{align*}
  which is the sought for equality.

  Let us prove now the second equality. 
  Since $U_{\Delta}^{\mathrm{opp}}\subset
  U_{\Theta}^{\mathrm{opp}} L_{\Theta}^{\circ}$, 
  it is enough to prove that
  $L_{\Theta}^{\circ} w P_{\Theta}^{\mathrm{opp}} = w
  P_{\Theta}^{\mathrm{opp}}$ which in turn follows from the equalities
  $\dot{w}^{-1} L_{\Theta}^{\circ} \dot{w} =L_{\Theta}^{\circ}$ (where
  $\dot{w}$ is a representative of~$w$). This equality can be proved at the
  level of the Lie algebra: $\Ad(\dot{w}) \mathfrak{l}_\Theta
  =\mathfrak{l}_\Theta$. Since $\Ad(\dot{w})$ stabilizes~$\mathfrak{a}$, it
  stabilizes also $\mathfrak{z}_{\mathfrak{g}}(\mathfrak{a})$ and, thanks to
  the equality $\mathfrak{l}_\Theta  =\mathfrak{z}_{\mathfrak{g}}(\mathfrak{a}) \oplus 
\bigoplus_{\alpha \in \Span(\Delta \smallsetminus \Theta)\cap \Sigma }
\mathfrak{g}_{\alpha}$ and using $\Ad(\dot{w}) \mathfrak{g}_\alpha =
\mathfrak{g}_{w(\alpha)}$, we are reduced to prove that $w( \Span(\Delta
\smallsetminus \Theta)\cap \Sigma ) = \Span(\Delta \smallsetminus \Theta)\cap
\Sigma$. Since $w(\Sigma)=\Sigma$ and since $W(\Theta)$ is generated by
$R(\Theta)$, we will just check the equality  $w( \Span(\Delta
\smallsetminus \Theta) ) = \Span(\Delta \smallsetminus \Theta)$ for~$w$ in~$R(\Theta)$.
   The case $w=s_\alpha$, for $\alpha\in\Theta\smallsetminus
  \{\alpha_\Theta\}$ is direct since $s_\alpha$ fixes pointwise $\Delta\smallsetminus \Theta$. The case $w=w_{\{\alpha_\Theta\} \cup
    \Delta\smallsetminus \Theta} w_{\Delta\smallsetminus \Theta}$ follows
  equally since $w_{\{\alpha_\Theta\} \cup
    \Delta\smallsetminus \Theta}$ induces $-\id$ in $\Span(\Delta
  \smallsetminus \Theta)$ and since $w_{\Delta\smallsetminus \Theta}$ induces
  a permutation of $\Span(\Delta \smallsetminus \Theta) \cap \Sigma$.
\end{proof}

\subsection{Dimensions}
\label{sec:dimensions}

We give here recursive information on the dimensions of the Bruhat cells
when the set $\Theta$ is fixed under the opposition involution or, what
amounts to the same, the dimensions of their images in the flag
variety~$\mathsf{F}_{\iota( \Theta)} = \mathsf{F}_\Theta$.  This information will be of importance when proving the
openness of the parametrizations introduced in
Section~\ref{sec:positive_semigroup}.

Recall that when $G$~admits a $\Theta$-positive structure for a proper subset~$\Theta$ in~$\Delta$,  the opposition involution~$\iota$ is always
the identity. 

For every~$w$ in~$W(\Theta)$ we will denote by $C(w)$\index{$C(w)$  the
$P_{\Theta}^{\mathrm{opp}}$-orbit 
of $w\cdot \mathfrak{p}_{\Theta}^{\mathrm{opp}}$ in $\mathsf{F}_{\Theta}$} the
$P_{\Theta}^{\mathrm{opp}}$-orbit 
of $w\cdot \mathfrak{p}_{\Theta}^{\mathrm{opp}}$ in $\mathsf{F}_{\Theta}$;
$C(w)$ is called a \emph{cell} and is
isomorphic to the quotient
$P_{\Theta}^{\mathrm{opp}}  w
P_{\Theta}^{\mathrm{opp}} /  P_{\Theta}^{\mathrm{opp}}$. By
Lemma~\ref{lemma:PThetaequalsPforWTheta}, it is also the image
of the following map
\begin{align*}
  \mathfrak{u}_{\Theta}^{\mathrm{opp}} & \longrightarrow \mathsf{F}_\Theta \\
  X & \longmapsto \exp(X) w \cdot \mathfrak{p}_{\Theta}^{\mathrm{opp}}.
\end{align*}
Recall that $\mathfrak{u}^{\mathrm{opp}}_{\Theta}$ has the following decomposition
\( \mathfrak{u}^{\mathrm{opp}}_{\Theta} = \bigoplus_{\mathclap{\alpha \in \Sigma^{+}_{\Theta}}} \mathfrak{g}_{-\alpha}. \)

For $w\in W$ we denote\index{$\Sigma^{+}_{w\prec}$ the set  $\Sigma^+ \cap
  w\cdot  \Sigma^-$}%
\index{$\Sigma^{+}_{w\succ}$ the set $\Sigma^+ \cap w\cdot  \Sigma^+$}%
\index{$\Sigma^{+}_{\Theta, w\prec}$ the set $\Sigma^{+}_{w\prec} \cap
  \Sigma^{+}_{\Theta}$}%
\index{$\Sigma^{+}_{\Theta, w\succ}$ the set $\Sigma^{+}_{w\prec} \cap
  \Sigma^{+}_{\Theta}$}%
\index{$\mathfrak{u}^{\mathrm{opp}}_{\Theta,w\prec}$ the sum $ \bigoplus_{\mathclap{\alpha\in \Sigma^{+}_{\Theta, w\prec}}}
 \mathfrak{g}_{-\alpha}$ }%
\index{$\mathfrak{u}^{\mathrm{opp}}_{\Theta,w\succ}$ the sum $\bigoplus_{\mathclap{\alpha\in \Sigma^{+}_{\Theta, w\succ}}}  \mathfrak{g}_{-\alpha}$}
\begin{align*}
  \Sigma^{+}_{w\prec} & =  \Sigma^+ \cap w\cdot  \Sigma^-,
  \,  \Sigma^{+}_{w\succ} = \Sigma^+ \cap w\cdot  \Sigma^+,
  \text{ so that } \Sigma^+ = \Sigma^{+}_{w\prec} \sqcup \Sigma^{+}_{w\succ} ,\\
  \Sigma^{+}_{\Theta, w\prec} & = \Sigma^{+}_{w\prec} \cap \Sigma^{+}_{\Theta},
  \,  \Sigma^{+}_{\Theta, w\succ} = \Sigma^{+}_{w\prec} \cap
                                \Sigma^{+}_{\Theta},\, \text{so that } \Sigma^{+}_{\Theta} = \Sigma^{+}_{\Theta, w\prec} \sqcup
  \Sigma^{+}_{\Theta, w\succ},\\
  \intertext{and at the level of Lie algebras }
  \mathfrak{u}^{\mathrm{opp}}_{\Theta,w\prec} & = \bigoplus_{\mathclap{\alpha\in \Sigma^{+}_{\Theta, w\prec}}}
                             \mathfrak{g}_{-\alpha},
         \, \mathfrak{u}^{\mathrm{opp}}_{\Theta,w\succ} = \bigoplus_{\mathclap{\alpha\in
                             \Sigma^{+}_{\Theta, w\succ}}} 
                             \mathfrak{g}_{-\alpha},\,\text{so that } \mathfrak{u}^{\mathrm{opp}}_\Theta = \mathfrak{u}^{\mathrm{opp}}_{\Theta,w\prec} \oplus
\mathfrak{u}^{\mathrm{opp}}_{\Theta,w\succ},
\end{align*}
 and, for all~$X$ in~$\mathfrak{u}_{\Theta, w\succ}$,
$\exp(X)w$ belongs to $w P^{\mathrm{opp}}_{\Theta}$.

It is well known that the length of~$w$ is the cardinality of
$\Sigma_{w\prec}^{+}$ \cite[ch.~VI, \S1, p.~158, cor.~2]{BourbakiLie456} and
the following equalities hold \cite[Lemme~3.4]{BorelTits}
\begin{align}
\notag  \Sigma_{s_\alpha w\prec}^{+} &= \{\alpha\} \sqcup s_\alpha\cdot \Sigma_{w\prec}^{+}
                             \text{ if } \ell( s_\alpha w) = 1+ \ell(w),\\
\label{eq:Sigma_xw}  \Sigma_{x w\prec}^{+} &= \Sigma_{x \prec}^{+} \sqcup x^{-1}\cdot \Sigma_{w\prec}^{+}
                             \text{ if } \ell( x w) = \ell(x)+ \ell(w).
\end{align}

\begin{lemma}
  \label{lem:param-Bruhat-cell}
  Let~$w$ be in~$W(\Theta)$. The map
  \begin{align*}
    f_w\colon \mathfrak{u}^{\mathrm{opp}}_{\Theta, w\prec}& \longrightarrow C(w) \\
    X & \longmapsto \exp(X) w\cdot P_{\Theta}^{\mathrm{opp}}
  \end{align*}
  is a diffeomorphism.
\end{lemma}
\begin{proof}
  Clearly this map is $C^\infty$. It also results from classical facts about
  nilpotent Lie algebras that
  \begin{align*}
    \mathfrak{u}^{\mathrm{opp}}_{\Theta, w\prec} \times \mathfrak{u}^{\mathrm{opp}}_{\Theta, w\succ}
    & \longrightarrow U_\Theta\\
    (X,Y) & \longmapsto \exp(X)\exp(Y)
  \end{align*}
  is a diffeomorphism. Since, for all~$Y$ in~$\mathfrak{u}^{\mathrm{opp}}_{\Theta, w\succ}$,
  $\exp(Y) w P_{\Theta}^{\mathrm{opp}} =w P_{\Theta}^{\mathrm{opp}}$, we have that $f_w$ is onto. By the very choice of the
  space $\mathfrak{u}^{\mathrm{opp}}_{\Theta, w\prec}$, the map~$f_w$ is a local diffeomorphism
  at~$0$. By equivariance with respect to the element
  $\exp(\sum_{\alpha\in\Delta} H_\alpha)$ of the Cartan subgroup that acts
  as a contracting transformation on~$\mathfrak{u}^{\mathrm{opp}}_{\Theta, w\prec}$ we deduce that
  $f_w$~is a diffeomorphism.
\end{proof}

From this we deduce how the dimensions of the cells $C(w)$ jump:
\begin{proposition}
  \label{prop:jump-dimensions-Cw}
  Let $w$ be in 
  the Coxeter group $(W(\Theta), R(\Theta))$ and let
  $\alpha$ be in~$\Theta$ such that $\sigma_\alpha w$ is a reduced product
  in $W(\Theta)$. Then
  \begin{align*}
    \dim C(\sigma_\alpha w) - \dim C(w)
    & = \dim \mathfrak{u}_\alpha\\
    \intertext{which means}
    \dim C(\sigma_\alpha w) - \dim C(w)
    & = 1 \text{ if $\alpha \neq \alpha_\Theta$},\\
    \dim C(\sigma_{\alpha_\Theta} w) - \dim C(w)
    & = \dim \mathfrak{u}_{\alpha_\Theta}  \text{ if $\alpha = \alpha_\Theta$}.
  \end{align*}
\end{proposition}
\begin{proof}
  Denote by~$\delta$ the difference of dimensions. The hypothesis that
  $\sigma_\alpha w$ is reduced in $W(\Theta)$ implies that
  $\sigma_\alpha w$ is also a reduced product in~$W$ (Lemma~\ref{lem:reduced-exp-inWTheta-and-W}).
  The previous lemma, the equation~\eqref{eq:Sigma_xw}
  (applied with $x=\sigma_\alpha$), and the equalities $\dim
  \mathfrak{g}_{-\beta} = \dim \mathfrak{g}_{\beta}$ imply that
  \[ \delta=\!\!\!\!\!\! \sum_{\beta \in \Sigma^{+}_{\Theta}, \sigma_{\alpha}^{-1} \beta
      \in \Sigma^-}\!\!\!\!\!\! \dim \mathfrak{g}_\beta.\]
  But  Lemma~\ref{lem:Sigma-plus-for-a-in-J} below
  establishes 
  \[ \{ \beta \in \Sigma^{+}_{\Theta}\mid \sigma_{\alpha}^{-1} \beta
      \in \Sigma^- \} = \{ \beta\in \Sigma^+ \mid \beta-\alpha \in \Span(
      \Delta\setminus\Theta)\},\]
    hence
    \[ \delta = \dim\!\!\!\!\!\! \bigoplus_{\substack{\beta\in \Sigma^+,\\ \beta-\alpha \in \Span(
        \Delta\setminus\Theta)}}\!\!\!\!\!\! \mathfrak{g}_{\beta} = \dim \mathfrak{u}_\alpha. \qedhere\]
\end{proof}

\begin{lemma}
  \label{lem:Sigma-plus-for-a-in-J}
  For all~$\alpha$ in~$\Theta$, one has
  \[ \{ \beta \in \Sigma^{+}_{\Theta}\mid \sigma_{\alpha}^{-1} \beta
      \in \Sigma^- \} = \{ \beta\in \Sigma^+ \mid \beta-\alpha \in \Span(
      \Delta\setminus\Theta)\}. \]
\end{lemma}
\begin{proof}
  In the case $\alpha\neq \alpha_\Theta$, $\sigma_\alpha=s_\alpha$ and it is already known that
  \[\Sigma^{+}_{s_\alpha\prec} = \{ \beta \in \Sigma^{+}\mid \sigma_{\alpha}^{-1} \beta
    \in \Sigma^- \} =\{ \alpha\}.\]
  Since $\alpha$ belongs to
  $\Sigma^{+}_{\Theta}$, this proves the equality $\Sigma^{+}_{\Theta,
    s_\alpha\prec} = \Sigma^{+}_{s_\alpha\prec} \cap \Sigma^{+}_{\Theta} =\{\alpha\}$. Since $\alpha$ is not connected to
  $\Delta\setminus \Theta$, the only root in the affine subspace $\alpha +
  \Span( \Delta\setminus\Theta)$ is~$\alpha$ and this proves the wanted
  equality.

  Let us treat the case $\alpha = \alpha_\Theta$. The element
  $\sigma_{\alpha_\Theta} = w_{\{\alpha_\Theta\} \cup \Delta \setminus \Theta}
  w_{\Delta\setminus\Theta}$ belongs to the subgroup $W_{\{\alpha_\Theta\}
    \cup \Delta \setminus \Theta}       $ of~$W$.

     The set
  \[ \{ \beta\in \Sigma^+ \mid \beta-\alpha_\Theta \in \Span(
    \Delta\setminus\Theta)\}\]
  is contained in $\Span(\{\alpha_\Theta\}
    \cup \Delta \setminus \Theta )$ and thus  in the root system generated by $\{\alpha_\Theta\}
    \cup \Delta \setminus \Theta$.
    
  From the fact that, for every simple root~$\beta$ in
  $\Theta\setminus \{ \alpha_\Theta\}$ and for every~$\alpha$ in $\{\alpha_\Theta\}
    \cup \Delta \setminus \Theta$, $s_\alpha (\beta)-\beta$ belongs to $\Span(
    \{\alpha_\Theta\}
    \cup \Delta \setminus \Theta)$, we have that $w( \Sigma^+ \smallsetminus \Span(
    \{\alpha_\Theta\}
    \cup \Delta \setminus \Theta)) = \Sigma^+ \smallsetminus \Span(
    \{\alpha_\Theta\}
    \cup \Delta \setminus \Theta)$ for all~$w$ in $W_{\{\alpha_\Theta\}\cup
      \Delta \smallsetminus \Theta}$. In particular
    the set
    \[ \{ \beta \in \Sigma^{+}_{\Theta}\mid \sigma_{\alpha_\Theta}^{-1} \beta
      \in \Sigma^- \}\]
  must be contained in  the root system generated by $\{\alpha_\Theta\}
  \cup \Delta \setminus \Theta$. 

  Therefore we can and will assume that $\Delta = \{\alpha_\Theta\}
    \cup \Delta \setminus \Theta$, i.e.\ that we are in the Hermitian tube
    type case. In this case it is known that the longest length element $w_{\{\alpha_\Theta\}
      \cup \Delta \setminus \Theta}$ is $-\id$ so that
    \[ \{ \beta \in \Sigma^{+}_{\Theta}\mid \sigma_{\alpha_\Theta}^{-1} \beta
      \in \Sigma^- \} = \{ \beta \in \Sigma^{+}_{\Theta}\mid
      w_{\Delta\setminus \Theta} \beta
      \in \Sigma^+ \}. \]
   Furthermore, since for every~$\beta$ in $ \Sigma^{+}_{\Theta}$,
   $w_{\Delta\setminus\Theta} \beta -\beta$ belongs to $\Span(
   \Delta\setminus\Theta)$ (and again since every root is a sum of simple
   roots with coefficients all of the same sign), this set is
   equal to $\Sigma^{+}_{\Theta}$.
   However, here we are dealing with a root system of type $\lietype{C}_{d+1}$, and the equality $\Sigma^{+}_{\Theta} = \{
   \beta \in \Sigma^+ \mid \beta-\alpha_\Theta \in \Span(
   \Delta\setminus\Theta)\}$ is satisfied; indeed it follows from the equality
   $\mathfrak{u}_\Theta=\mathfrak{u}_{\alpha_\Theta}$ at the level of Lie
   algebras.
                                                \end{proof}

\subsection{The nontrivial cone and the Bruhat decomposition}
\label{sec:nontr-cone-bruh}

We give now a precise description of the Bruhat cell with respect to the $P_\Delta^{\mathrm{opp}} \times P_\Theta^{\mathrm{opp}}$ action, containing the
image of the cone~$\mathring{c}_{\alpha_\Theta}$ by the exponential map. 

We assume in this section that $\Theta\neq \Delta$ and adopt the notation of
the previous parts: $\alpha_\Theta$ is the special root, $\Delta\setminus
\Theta = \{ \chi_1, \dots, \chi_d\}$ (with $\chi_1$ connected
to~$\alpha_\Theta$ in the Dynkin diagram), and the elements $Z_0,\dots, Z_d$
of the Lie algebra (Section~\ref{sec:putting-an-hand}); the reflection in~$W$
associated with the simple root~$\alpha_\Theta$ is
denoted by~$s_0$ and the reflections associated with~$\chi_1, \dots, \chi_d$
are denoted by $s_1, \dots, s_d$ respectively.

The
longest length element in $W_{\{\alpha_\Theta\} \cup
  \Delta\setminus\Theta}$ is
\[ w_{\{\alpha_\Theta\}\cup \Delta\setminus\Theta} = s_0\, s_{0}^{s_1} \cdots
  s_{0}^{s_1 \cdots s_d},\]
 again with the notation $x^y = y^{-1} xy$, so that $s_{0}^{ s_1\cdots s_i} = s_i
\cdots s_1 s_0 s_1 \cdots s_i$, the above is a reduced expression
of~$w_{\{\alpha_\Theta\}\cup \Delta\setminus\Theta}$, and these elements
$s_{0}^{ s_1\cdots s_i}$ pairwise commute.

Let us introduce also the following elements of~$W$:\index{$w_j $ the product $ s_0\, s_{0}^{s_1} \cdots
  s_{0}^{s_1 \cdots s_{j-1}}$}%
\index{$w_I $ the product $ \prod_{i\in I} s_{0}^{ s_1\cdots s_i}$}%
\begin{align}
  \label{eq:w_j}
  w_0 &=e, \ w_j = s_0\, s_{0}^{s_1} \cdots
        s_{0}^{s_1 \cdots s_{j-1}}, \text{ for } j\in \llbracket 1, d+1\rrbracket\\
  \intertext{and, for every subset~$I$ of $\llbracket 0, d \rrbracket$,}
 \label{eq:w_I} w_I &= \prod_{i\in I} s_{0}^{ s_1\cdots s_i},
\end{align}
so that $w_\emptyset = e$ is the neutral element of~$W$, $w_{d+1}= w_{\llbracket 0,  d\rrbracket} = w_{\{ \alpha_\Theta\}\cup \Delta\setminus\Theta}$, and the
 expressions~(\ref{eq:w_j}) and~(\ref{eq:w_I}) are reduced. By Lemma~\ref{lem:theta-length-on-reduced-expression}, we have
\begin{equation}
 \ell_\Theta (w_j) =j,\ \ell_\Theta(w_I)= \sharp I, \ \text{for all } j \text{
   and for all }  I.\label{eq:Theta-length-wj-wI}
\end{equation}

The $w_j$ belong to different Bruhat cells:
\begin{lemma}
  \label{lem:w_j-different-cells}
  Let~$j$ and~$k$ be in $\llbracket 0,d\rrbracket$. If
  $P_{\Theta}^{\mathrm{opp}} w_j P_{\Theta}^{\mathrm{opp}} =
  P_{\Theta}^{\mathrm{opp}} w_k P_{\Theta}^{\mathrm{opp}}$, then $j=k$.
\end{lemma}
\begin{proof}
  One has $P_{\Delta}^{\mathrm{opp}} w_j P_{\Delta}^{\mathrm{opp}} \subset P_{\Theta}^{\mathrm{opp}}
      w_j P_{\Theta}^{\mathrm{opp}}$ and
  \begin{equation*}
     P_{\Theta}^{\mathrm{opp}}
      w_k P_{\Theta}^{\mathrm{opp}} =\bigcup_{\mathclap{x_1, x_2 \in
        W_{\Delta\setminus\Theta}}} P_{\Delta}^{\mathrm{opp}} x_1 w_k x_2
      P_{\Delta}^{\mathrm{opp}}.
    \end{equation*}
  (The last union is justified by property~(\ref{item:2:Bruhat_dec}) of
  Section~\ref{sec:bruhat-decomposition} and may not be disjoint.) Hence an equality $P_{\Theta}^{\mathrm{opp}} w_j
  P_{\Theta}^{\mathrm{opp}} = P_{\Theta}^{\mathrm{opp}} w_k
  P_{\Theta}^{\mathrm{opp}}$ implies, and is in fact equivalent to,  the
  existence of $x_1, x_2$ in $W_{\Delta\setminus\Theta}$
  such that
  \[ w_j = x_1 w_k x_2.\]
  By
  Lemma~\ref{lem:theta-length-invariant}, $\ell_\Theta( x_1 w_k x_2)=\ell_\Theta(
  w_k)$,
  hence $j =\ell_\Theta( w_j) =\ell_\Theta(
  w_k)=k$.
\end{proof}

We start by determining the Bruhat cells for linear combinations with nonnegative coefficients of the
elements~$Z_i$. 

\begin{lemma}
  \label{lem:bruh-spanZ_i}
  Let~$Y$ be a linear combination with nonnegative coefficients of the~$Z_i$:
  $Y=\sum_{i=0}^{d} \lambda_i Z_i$,
  i.e.\ $Y$~belongs to $c_{\alpha_\Theta} \cap \bigoplus \R Z_i$ (Lemma~\ref{lem:inter-spanZ_i-with-cone}). Let
  $I\subset \llbracket 0, d\rrbracket$ be the set of indices of the nonzero entries in
  $(\lambda_0, \dots, \lambda_d)$ (in formula $I=\{ i\in \llbracket 0, d\rrbracket \mid
  \lambda_i>0\}$). Then
  \[ \exp(Y) \text{ belongs to } P_{\Delta}^{\mathrm{opp}} w_I P_{\Delta}^{\mathrm{opp}}.\]
\end{lemma}
\begin{proof}
 Since the elements~$Z_i$ pairwise commute (Proposition~\ref{prop:u_alpha-are-abel}), we have that
  \[ \exp(Y) = \exp( \lambda_0 Z_0) \exp( \lambda_1 Z_1) \cdots \exp(
    \lambda_d Z_d).\]
  By property~(\ref{item:4:Bruhat_dec}) of
  Section~\ref{sec:bruhat-decomposition}, if $\lambda_0>0$, then $\exp(
  \lambda_0 Z_0)$ belongs to $P_{\Delta}^{\mathrm{opp}} s_0 P_{\Delta}^{\mathrm{opp}}$. For all
  $i$, one has $Z_i = \Ad( \dot{s}_i \cdots \dot{s}_1) Z_0$ so that $\exp(
  \lambda_i Z_i) = \dot{s}_{i} \cdots \dot{s}_{1} \exp( \lambda_i
  Z_0) \dot{s}_{1}^{-1} \cdots \dot{s}_{i}^{-1}$. Successive applications of the
  property~(\ref{item:3:Bruhat_dec}) of Section~\ref{sec:bruhat-decomposition}
  and the fact that $s_i \cdots s_1 s_0 s_1 \cdots s_i$ is a reduced expression
  imply that, when $\lambda_i>0$, $\exp( \lambda_i Z_i)$ belongs to
  \[ P_{\Delta}^{\mathrm{opp}} s_i\cdots s_1 s_0 s_1\cdots s_i P_{\Delta}^{\mathrm{opp}} =
    P_{\Delta}^{\mathrm{opp}} s_{0}^{s_1\cdots s_i} P_{\Delta}^{\mathrm{opp}}.\]
  Successive applications again of property~(\ref{item:3:Bruhat_dec}) and the fact
  that Equation~\eqref{eq:w_I} is a reduced expression imply that
  \[ \exp(Y) \text{ belongs to } P_{\Delta}^{\mathrm{opp}}
    \Bigl(\prod_{\mathclap{i\in I}}
    s_{0}^{s_1\cdots s_i} \Bigr)P_{\Delta}^{\mathrm{opp}},\]
  hence the result.
\end{proof}

We can now determine the Bruhat cell with respect to the left-right action of $P_{\Delta}^{\mathrm{opp}}$ pcorresponding to elements in the open
cone.
\begin{proposition}
  \label{prop:open-cone-bruh}
  Let~$X$ be in the open cone~$\mathring{c}_{\alpha_\Theta}$. Then $\exp(X)$
  belongs to $P_{\Delta}^{\mathrm{opp}} w_{\{ \alpha_\Theta\}\cup \Delta\setminus \Theta}P_{\Delta}^{\mathrm{opp}}$.
\end{proposition}
\begin{proof}
  Recall that the element $E_{\alpha_\Theta} = Z_0+Z_1+\cdots + Z_d$ belongs
  to~$\mathring{c}_{\alpha_\Theta}$ and that its stabilizer in~$L_{\Theta}^{\circ}$ (for
  the adjoint action) contains $K\cap L^{\circ}_{\Theta}$. Furthermore the action
  of~$L^{\circ}_{\Theta}$ on $\mathring{c}_{\alpha_\Theta}$ is transitive: there is an
  element~$g$ in~$L^{\circ}_{\Theta}$ such that $\Ad(g) E_{\alpha_\Theta} = X$.

  The Iwasawa decomposition for~$L^{\circ}_{\Theta}$ states the equality
  \[ L^{\circ}_{\Theta} = (P_{\Delta}^{\mathrm{opp}}\cap L^{\circ}_{\Theta}) ( K\cap L^{\circ}_{\Theta}).\]
  There are thus elements~$p$ in $P_{\Delta}^{\mathrm{opp}}\cap L^{\circ}_{\Theta} $ and~$k$ in $
  K\cap L^{\circ}_{\Theta}$ such that $g=pk$. Since $\Ad(k) E_{\alpha_\Theta} =
  E_{\alpha_\Theta}$, one has $\Ad(p) E_{\alpha_\Theta} =X$. By the
  compatibility of the exponential map with the adjoint action and the
  conjugation action, we get $\exp(X) = p \exp( E_{\alpha_\Theta}) p^{-1}$ so
  that $\exp(X)$ and $\exp(E_{\alpha_\Theta})$ belong to the same Bruhat
  cell. By Lemma~\ref{lem:bruh-spanZ_i}, $\exp(E_{\alpha_\Theta})$ belongs to
  $P_{\Delta}^{\mathrm{opp}} w_{\llbracket 0, d\rrbracket} P_{\Delta}^{\mathrm{opp}} = P_{\Delta}^{\mathrm{opp}}
  w_{\{\alpha_\Theta\} \cup \Delta\setminus\Theta} P_{\Delta}^{\mathrm{opp}}$. This
  concludes that $\exp(X)$ belongs to $P_{\Delta}^{\mathrm{opp}}
  w_{\{\alpha_\Theta\} \cup \Delta\setminus\Theta} P_{\Delta}^{\mathrm{opp}}$.
\end{proof}

For the elements in the closure of the cone, it will be enough for our purpose
to determine their class under the left-right action of
$P^{\mathrm{opp}}_{\Theta}$.

\begin{proposition}
  Let~$X$ be an element of $c_{\alpha_\Theta}$, There is then a unique~$j$ in
  $\llbracket 0, d\rrbracket$ such that
  \[ \exp(X) \in P_{\Theta}^{\mathrm{opp}} w_j P_{\Theta}^{\mathrm{opp}}.\]
\end{proposition}
\begin{proof}
  Uniqueness follows from Lemma~\ref{lem:w_j-different-cells}.

    Since
  $c_{\alpha_\Theta}$~is the closure of~$\mathring{c}_{\alpha_\Theta}$, there
  exists a sequence $(X_n)_{n\in\N}$ in~$\mathring{c}_{\alpha_\Theta}$ that
  converges to~$X$. For every~$n$ in~$\N$, let~$g_n$ be an element
  in~$L_{\Theta}^{\circ}$ such that $X_n = \Ad(g_n)E_{\alpha_\Theta}$.

  The Cartan decomposition in~$L_{\Theta}^{\circ}$ gives 
  \[ L_{\Theta}^{\circ} = K_{\Theta}^{\circ} \exp( \mathfrak{b}^+) K_{\Theta}^{\circ},\]
  where $K_{\Theta}^{\circ} = K\cap L_{\Theta}^{\circ}$ and $\mathfrak{b}^+ = \{ A\in \mathfrak{a}
  \mid \chi_i(A)\leq 0,\ \forall i\in \llbracket 1, d\rrbracket \}$ is a closed Weyl chamber for
  the reductive Lie group~$L_{\Theta}^{\circ}$. There exist thus, for all~$n$ in~$\N$,
  elements~$k_n$, $k^{\prime}_{n}$ in~$K_{\Theta}^{\circ}$, and~$A_n$
  in~$\mathfrak{b}^+$ such that $g_n = k_n \exp(A_n) k^{\prime}_{n}$. Up to
  extracting we will assume that the sequence $(k_n)_{n\in\N}$ is converging
  and its limit will be denoted~$k_\infty$. One has therefore, for all~$n$,
  \[ \Ad( \exp(A_n)) E_{\alpha_\Theta} = \Ad( k_{n}^{-1}) \Ad(g_n)
    E_{\alpha_\Theta} = \Ad( k_{n}^{-1}) X_n,\]
  so the sequence $( \Ad( \exp(A_n)) E_{\alpha_\Theta})_{n\in \N}$ is
  converging and its limit is $Y=\Ad(k_{\infty}^{-1}) X$. As  $K_{\Theta}^{\circ} \subset L_{\Theta}^{\circ} \subset
  P^{\mathrm{opp}}_{\Theta}$, it will be enough to determine the Bruhat cell
  of $\exp(Y)$.

  For all~$n$ in~$\N$ and for all~$i$ in $\llbracket 0, d\rrbracket$, let $$\lambda_{i,n} =
  \exp\bigl( (\alpha_\Theta +2\chi_1 +\cdots +2\chi_{i})(A_n)\bigr)$$ so
  that
  for all~$n$, as $A_n$~belongs to~$\mathfrak{b}^+$, $\lambda_{0,n}\geq \lambda_{1,n}\geq \cdots \geq
  \lambda_{d,n}>0$ and
  \begin{align*}
    \sum_{i=0}^{d} \lambda_{i,n} Z_i &= \sum_{i=0}^{d} \exp( \ad(A_n)) Z_i\\
    &= \exp( \ad(A_n)) \sum_{i=0}^{d} Z_i= \Ad(\exp( A_n)) E_{\alpha_\Theta}.
  \end{align*}
  This last equality implies that, for all~$i$ in $\llbracket 0, d\rrbracket$, the sequence $(\lambda_{i,n})_{n\in \N}$
  converges in~$\R_{\geq 0}$ and its limit will be denoted
  $\lambda_{i}$. We have hence
  \[ Y =\sum_{i=0}^{d} \lambda_{i} Z_i, \]
   $\lambda_{0} \geq \lambda_{1}\geq\cdots \geq
  \lambda_{d}\geq 0$, and the set of indices of the nonzero entries in
  $(\lambda_{0},\dots, \lambda_{d})$ has the form $
  \llbracket 0, j-1\rrbracket$ for some~$j$ in $\llbracket 0,
  d+1\rrbracket$ (this set is~$\emptyset$ when $j=0$). Lemma~\ref{lem:bruh-spanZ_i} says that $\exp(Y) \in P_{\Delta}^{\mathrm{opp}}
  w_j P_{\Delta}^{\mathrm{opp}}$ and this implies the wanted result.
\end{proof}

One can rephrase the previous result:

\begin{corollary}
  \label{coro:nontr-cone-bruh-J-length}
  Let~$Y$ be in~$c_{\alpha_\Theta}$ and let~$w$ be the element of~$W$ such
  that $\exp(Y)\in P_{\Delta}^{\mathrm{opp}} w P_{\Delta}^{\mathrm{opp}}$. Then
  \begin{enumerate}
  \item $\ell_\Theta(w)\leq d+1$;
  \item if $\ell_\Theta(w)= d+1$ then $Y$ belongs to~$\mathring{c}_{\alpha_\Theta}$
    and $w=w_{\{\alpha_\Theta\} \cup \Delta\setminus\Theta}$.
  \end{enumerate}
\end{corollary}

\subsection{The Bruhat cells of nonzero elements}
\label{sec:nontr-space-bruh}

The previous paragraph determines the Bruhat cell of an element of the form
$\exp(X)$ when $X$~belongs to the cone $c_{\alpha_\Theta}$. We now consider
the case when $X$~is only supposed to be a nonzero element
of~$\mathfrak{u}_{\alpha_\Theta}$.

\begin{lemma}
  \label{lem:nontr-space-bruh}
  Let~$X$ be a nonzero element of~$\mathfrak{u}_{\alpha_\Theta}$ and let~$x$
  be the element of~$W$ such that $\exp(X) = P_{\Delta}^{\mathrm{opp}} x P_{\Delta}^{\mathrm{opp}}$. Then
  $\ell_\Theta(x) >0$.
\end{lemma}

\begin{remark}
   The element $x$ belongs to $W_{\{ \alpha_\Theta\} \cup
    \Delta\setminus\Theta}$ as $\exp(X)$ belongs to~$S_{\Theta\smallsetminus \{\alpha_\Theta\}}$.
\end{remark}

\begin{proof}
  One has (Lemma~\ref{lemma:theta-length-0=W_Levi})
  \( W_{\Delta\setminus\Theta} = \{ x\in W \mid \ell_\Theta(x)=0\}.\)
  Furthermore, the following equality holds:
  \[ \bigcup_{x\in W_{\Delta\setminus\Theta}} P_{\Delta}^{\mathrm{opp}} x
    P_{\Delta}^{\mathrm{opp}} =P^{\mathrm{opp}}_{\Theta}.\]
  Thus the statement will be established if we can prove that $\exp(X)$ does
  not belong to $P_{\Theta}^{\mathrm{opp}}$. This last property is a
  consequence of the 
  fact that the map:
  \begin{align*}
    \mathfrak{u}_\Theta \times P_{\Theta}^{\mathrm{opp}} & \longrightarrow G\\
    (X, g) &\longmapsto \exp(X)g
  \end{align*}
  is an embedding.
\end{proof}

\section{Parametrizing the unipotent positive semigroup}
\label{sec:positive_semigroup} 

In this section, given~$G$ a simple Lie group admitting a
$\Theta$-positive structure, we give an explicit  parametrization of the unipotent
positive semigroup $U_\Theta^{>0}$.   First we give explicit parametrizations
of subsets of $U_\Theta$ associated with a reduced expression of the longest length element in the $\Theta$-Weyl group
$W(\Theta)$. This first step is analogous to the strategy in Lusztig's work on total
positivity \cite{lusztigposred}. 
In \cite{lusztigposred} the next step is to show that the image of this
parametrization does not depend on the chosen reduced expression, and once
this is shown, additional properties of the positive unipotent semigroup
are established. 

Here, in the next step we first prove some of these crucial topological properties. In particular, we show that the set we parametrize is a connected component of the intersection of the standard unipotent
subgroup~$U_\Theta$ with the Bruhat cell\index{$\Omega^{\mathrm{opp}}_{\Theta}
  =$ the open Bruhat cell $P^{\mathrm{opp}}_{\Theta} w_\Delta P^{\mathrm{opp}}_{\Theta} $} $\Omega^{\mathrm{opp}}_{\Theta} =
P^{\mathrm{opp}}_{\Theta} w_\Delta P^{\mathrm{opp}}_{\Theta} $.

To show
that the connected component does not depend on the reduced expression, we make use of the split subgroup~$H_\Theta\subset
G$ of type~$W(\Theta)$.

We fix, for each $\alpha \in \Theta$,  an $L_{\Theta}^{\circ}$-invariant convex cone $c_{\alpha}\subset
\mathfrak{u}_{\alpha}$.

\subsection{Maps associated with reduced expressions in \texorpdfstring{$W(\Theta)$}{W(Θ)}}

In this section we define maps from products of the cones to the unipotent group $U_\Theta$, which are associated to reduced expressions in the $\Theta$-Weyl group.  

We consider the element $w^{\Theta}_{\max} \in W$.  In
Section~\ref{sec:Weyl} we saw that $W(\Theta)$ equipped with its generating
set~$R(\Theta)$ (cf.\ Definition~\ref{defi:theta-weyl-group}) is a Coxeter
system of 
Lie type (cf.\ Corollary~\ref{coro:WTheta}), and that
$w^{\Theta}_{\max}$ is the longest length element in~$W(\Theta)$.

We denote by~$N$ the length of~$w^{\Theta}_{\max}$ and by
$\mathbf{W}\subset \Theta^{N}$ the set of reduced expressions of~$w^{\Theta}_{\max}$.\index{$\mathbf{W}$ the set of reduced expressions of~$w^{\Theta}_{\max}$}

For every~$\gammab$ in~$\mathbf{W}$, the product of cones $ c_{\gamma_1}\times
\cdots \times c_{\gamma_N}$ is denoted~$c_\gammab$.\index{$c_\gammab$ the product of cones $ c_{\gamma_1}\times
\cdots \times c_{\gamma_N}$ ($\gammab=(\gamma_1, \dots, \gamma_N)$)} We define
the map\index{$F_\gammab$ the ``parametrization'' of $U_{\Theta}^{\geq 0}$}
\begin{align*}
  F_\gammab \colon c_\gammab & \longrightarrow U_\Theta \\
(v_{1}, \dots, v_{N}) & \longmapsto \exp(v_{1})\cdots
\exp(v_{N}).
\end{align*}

The image of $F_\gammab$ clearly lies in the nonnegative unipotent semigroup $U_\Theta^{\geq 0}$. 
The interior $\mathring{c}_\gammab$ is the product $ \mathring{c}_{\gamma_1}\times
\cdots \times \mathring{c}_{\gamma_N}$ and the restriction of~$F_\gammab$ to~$\mathring{c}_\gammab$ is denoted by~$\mathring{F}_\gammab$. Similar notation,
$c_{\gammab}^{\mathrm{opp}}$, $F_{\gammab}^{\mathrm{opp}}$, and
$\mathring{F}_{\gammab}^{\mathrm{opp}}$ are adopted for the maps into the opposite unipotent group.

The goal of this section is to show that the image of~$\mathring{F}_\gammab$
is independent of~$\gammab$ and coincides with the positive semigroup   $U_{\Theta}^{>0}$: 

\begin{theorem}\label{thm:positive_semigroup}
For all~$\gammab$
  in~$\mathbf{W}$,   the map $\mathring{F}_\gammab$ is a diffeomorphism
  from~$\mathring{c}_\gammab$ onto the {unipotent positive semigroup}
   $U_{\Theta}^{>0}$, 
and $\mathring{F}^{\mathrm{opp}}_{\gammab}$ is a diffeomorphism from
$\mathring{c}_{\gammab}^{\mathrm{opp}}$ onto
the opposite unipotent positive semigroup~$U_{\Theta}^{\mathrm{opp},>0}$.
\end{theorem}

In order to prove Theorem~\ref{thm:positive_semigroup} we are going to establish several properties of the map $F_\gammab$ and $\mathring{F}_\gammab$ in the following sections. 

\subsection{Equivariance, transversality, injectivity, and openness}
\label{sec:transv-inject-maps}

We first note that the map $F_\gammab$ is $L_{\Theta}^{\circ}$-equivariant with respect to  the adjoint action on~$c_\gammab$ and the conjugation action on~$U_\Theta$. 

\begin{lemma}
  \label{lemma:equivarianceFgamma}
  For every~$\gammab$ in~$\mathbf{W}$ the map~$F_\gammab$ is
  $L_{\Theta}^{\circ}$-equivariant.
\end{lemma}
\begin{proof}
  Let $\ell$ be in~$L_{\Theta}^{\circ}$ and let $\vb=(v_1, \dots, v_N)$ be in~$c_\gammab$. The
  action of~$\ell$ on~$\vb$ is
  \[ \ell \cdot \vb = ( \Ad(\ell) v_1, \dots, \Ad(\ell)v_N)\]
  and is again in $c_\gammab$ since the cones~$c_\alpha$ are
  $L_{\Theta}^{\circ}$-invariant. The compatibility of the exponential map with
  the adjoint and conjugation actions implies
  \begin{align*}
    F_\gammab (\ell \cdot \vb)
    &= \exp(\Ad(\ell) v_1) \cdots \exp ( \Ad(\ell)v_N)\\
    &= \ell\exp( v_1)\ell^{-1} \cdots \ell\exp( v_N)\ell^{-1}\\
    &= \ell\exp( v_1) \cdots \exp( v_N)\ell^{-1} = \ell F_\gammab(
      \vb)\ell^{-1},
  \end{align*}
  which is the wanted equivariance property.
\end{proof}

We further establish transversality and injectivity properties of $\mathring{F}_{\gammab}$. 

\begin{proposition}
  \label{prop:transv-inject-mapsFgamma}
  For every~$\gammab$ in~$\mathbf{W}$, the image of the
  map~$\mathring{F}_{\gammab}$ is contained in  the open Bruhat cell $P_{\Delta}^{\mathrm{opp}} w_{\max}^{\Theta}
  P_{\Theta}^{\mathrm{opp}} = P_{\Theta}^{\mathrm{opp}} w_{\max}^{\Theta}
  P_{\Theta}^{\mathrm{opp}}=\Omega^{\mathrm{opp}}_{\Theta}$ and the map
  \[\mathring{c}_\gammab \longrightarrow \mathsf{F}_\Theta \mid \vb \longmapsto
    \mathring{F}_{\gammab}(\vb)\cdot \mathfrak{p}_{\Theta}^{\mathrm{opp}}\]
   is injective.
\end{proposition}

  Note that, since the map $U_\Theta \to \mathsf{F}_\Theta \mid  u \mapsto
  u\cdot \mathfrak{p}_{\Theta}^{\mathrm{opp}}$ is an embedding, the injectivity
  of the above map is equivalent to the injectivity of~$\mathring{F}_\gammab$.

Proposition~\ref{prop:transv-inject-mapsFgamma} will be proved thanks to an inductive process whose results are of independent interest. For this it will be
  a little more convenient to have a
decreasing numbering for the indices of~$\gammab$: $\gammab
  = (\gamma_N, \gamma_{N-1}, \dots, \gamma_1)$. 
  With this notation the
  cone $c_\gammab$ is the product $c_{\gamma_N}\times \cdots \times
  c_{\gamma_1}$.

Set, for all~$j$ in~$\llbracket 1, N\rrbracket $, $x_j=\sigma_{\gamma_j} \cdots
  \sigma_{\gamma_1}$, this is an element of~$W(\Theta)$.\index{$x_j$ the
    element $\sigma_{\gamma_j} \cdots
  \sigma_{\gamma_1}$ (given $\gammab=(\gamma_N, \dots, \gamma_1)$)}  
  \begin{proposition}
    \label{prop:transv-inject-maps-induction}
    For every~$j$ in $\llbracket 1, N\rrbracket $,
  \begin{enumerate}[leftmargin=*]
  \item\label{item:2proof-theo-trans} The image of the
    map\index{$\mathring{F}_j$ the map to $U_\Theta$ associated with $x_j$}
    \begin{align*}
      \mathring{F}_j\colon \mathring{c}_{\gamma_j} \times \cdots \times \mathring{c}_{\gamma_1} & \longrightarrow
                                                      U_\Theta\\
    (X_j,\dots, X_1) &\longmapsto \exp(X_j)\cdots \exp(X_1)
    \end{align*}
     is contained
    in $P_{\Delta}^{\mathrm{opp}} x_j P_{\Theta}^{\mathrm{opp}}
    =P_{\Theta}^{\mathrm{opp}} x_j P_{\Theta}^{\mathrm{opp}}$ 
     (the equality follows
    from Lemma~\ref{lemma:PThetaequalsPforWTheta} since $x_j$~belongs
    to~$W(\Theta)$);
  \item\label{item:1proof-theo-trans} The map
    \begin{align*}
       \mathfrak{u}_{\gamma_j} \times \mathring{c}_{\gamma_{j-1}} \times \cdots \times \mathring{c}_{\gamma_1} & \longrightarrow
                                                      \mathsf{F}_\Theta\\
    (X_j,\dots, X_1) &\longmapsto \exp(X_j)\cdots \exp(X_1)\cdot \mathfrak{p}_{\Theta}^{\mathrm{opp}}
    \end{align*}
    is injective.
      \end{enumerate}
\end{proposition}
\begin{proof}
  We prove first~(\ref{item:2proof-theo-trans}) by induction on~$j$ in
  $\llbracket 1, N\rrbracket$. For $j=1$ in the proof below $F_{j-1}(X_{j-1},\dots, X_1)$
  and $x_{j-1}$ should be replaced by the identity element.

  Let $(X_j,\dots, X_1)$ be in $\mathring{c}_{\gamma_j}\times \cdots \times
  \mathring{c}_{\gamma_1}$.
  If $\gamma_j$ is not equal to~$\alpha_\Theta$, then
  $\exp(X_j)$ belongs to
  $P_{\Delta}^{\mathrm{opp}}
  s_{\gamma_j} P_{\Delta}^{\mathrm{opp}} = P_{\Delta}^{\mathrm{opp}}
  \sigma_{\gamma_j} P_{\Delta}^{\mathrm{opp}} \subset P_{\Delta}^{\mathrm{opp}} \sigma_{\gamma_j} P_{\Theta}^{\mathrm{opp}}$  (see
  Section~\ref{sec:bruhat-decomposition}). If $\gamma_j=\alpha_\Theta$, then $\sigma_{\gamma_j}
  = \sigma_{\alpha_\Theta}$ and, by Proposition~\ref{prop:open-cone-bruh}, $
  \exp(X_j)$ belongs to 
  \[ P_{\Delta}^\mathrm{opp} \sigma_{\alpha_\Theta} w_{\Delta\setminus\Theta}
    P_{\Delta}^\mathrm{opp} 
    \subset  P_{\Delta}^\mathrm{opp} \sigma_{\alpha_\Theta}
    P_{\Theta}^{\mathrm{opp}}.\]
  By induction (or trivially if $j=1$) $\mathring{F}_{j-1}( X_{j-1},\dots,
  X_1)$ belongs to $P_{\Delta}^{\mathrm{opp}} x_{j-1}
  P^{\mathrm{opp}}_{\Theta}$.

  We deduce that $\mathring{F}_j( X_j,\dots, X_1)$ belongs to
  \begin{align*}
     P_{\Delta}^\mathrm{opp} \sigma_{j} P_{\Theta}^{\mathrm{opp}} x_{j-1}
    P_{\Theta}^{\mathrm{opp}}
    &= P_{\Delta}^\mathrm{opp} \sigma_{\alpha_\Theta} P_{\Delta}^{\mathrm{opp}} x_{j-1}
      P_{\Theta}^{\mathrm{opp}},\\
    \intertext{where the equality holds by Lemma~\ref{lemma:PThetaequalsPforWTheta}  applied to~$x_{j-1}$,}
    &= P_{\Delta}^\mathrm{opp} \sigma_{\alpha_\Theta} x_{j-1}
      P_{\Theta}^{\mathrm{opp}}
      \intertext{since $x_j = \sigma_{\alpha_\Theta} x_{j-1}$ is a reduced
  product by Lemma~\ref{lem:reduced-exp-inWTheta-and-W}. This concludes the
      induction step for~(\ref{item:2proof-theo-trans}).}
  \end{align*}

  We now prove point~\eqref{item:1proof-theo-trans} by induction on~$j$ again.
  For $j=1$, injectivity follows from the fact that $\mathfrak{u}_\Theta
  \times P^{\mathrm{opp}}_{\Theta}\mid (X,g)\mapsto \exp(X)g$ is injective.

  Suppose now that $j\geq 2$ and that the inductive hypothesis has been
  established up to $j-1$. Let $(X_j,\dots, X_1)$ and $(Y_j, \dots, Y_1)$ be
  in $\mathfrak{u}_{\gamma_j} \times \mathring{c}_{\gamma_{j-1}}\times \cdots \times \mathring{c}_{\gamma_1}$ such that
  \[ {F}_j(X_j,\dots, X_1) P_{\Theta}^{\mathrm{opp}} = {F}_j(Y_j, \dots, Y_1) P_{\Theta}^{\mathrm{opp}},\]
  this means that 
  \begin{align}
    \label{eq:7}
    \mathring{F}_{j-1}(X_{j-1},\dots, X_1) P_{\Theta}^{\mathrm{opp}}
    &= \exp(-X_j)\exp(Y_j)\mathring{F}_{j-1}(Y_{j-1},
    \dots, Y_1) P_{\Theta}^{\mathrm{opp}}\\ 
   \notag &= \exp(Y_j-X_j)\mathring{F}_{j-1}(Y_{j-1},
    \dots, Y_1) P_{\Theta}^{\mathrm{opp}}, 
  \end{align}
  where the last equality holds since $\mathfrak{u}_{\gamma_j}$ is Abelian (Proposition~\ref{prop:u_alpha-are-abel}). 
  By point~(\ref{item:2proof-theo-trans}) we have that $\mathring{F}_{j-1}(X_{j-1},\dots, X_1)$ and
  $\mathring{F}_{j-1}(Y_{j-1}, \dots, Y_1)$ belong to $P_{\Theta}^{\mathrm{opp}}
  x_{j-1} P_{\Theta}^{\mathrm{opp}}$. Suppose that $X_j\neq Y_j$ then (by
  Lemma~\ref{lem:nontr-space-bruh})     we have that $\exp(Y_j-X_j)$ belongs
  to $P_{\Delta}^\mathrm{opp} x P_{\Delta}^\mathrm{opp}$ with $\ell_\Theta(x)>0$. Using again the
  multiplicative properties of Bruhat cells we deduce that the element $y$ such
  that 
  \[\exp(Y_j-X_j)\mathring{F}_{j-1}(Y_{j-1},
    \dots, Y_1) \in P_{\Delta}^\mathrm{opp} y P^\mathrm{opp}_{\Theta}\]
  satisfies $\ell_\Theta(y)> \ell_\Theta( x_{j-1})$ which is a
  contradiction with Equation~\eqref{eq:7}. Hence $X_j=Y_j$ and~\eqref{eq:7} together with the induction hypothesis 
  imply that 
  $X_i=Y_i$ for all~$i$ less than~$j$.
\end{proof}

We can now deduce that the maps $\mathring{F}_{\gammab}$ are open. More
precisely, using again the notation of the proposition, and recalling that $C(x_j)\subset \mathsf{F}_\Theta$ is the 
$U_{\Theta}^{\mathrm{opp}}$-orbit of the element $x_j\cdot
\mathfrak{p}_{\Theta}^{\mathrm{opp}}$,  we have: 

\begin{corollary}
  \label{coro:F_j-open}
  For all~$j$ the map
  \begin{align*}
    \mathring{c}_{\gamma_j}\times \cdots \times \mathring{c}_{\gamma_1}
    & \longrightarrow C(x_j) \subset \mathsf{F}_\Theta\\
    (X_j, \dots, X_1) &\longmapsto \mathring{F}_j(X_j,\dots, X_1)\cdot \mathfrak{p}_{\Theta}^{\mathrm{opp}}
  \end{align*}
  is open.
\end{corollary}
\begin{proof}
  We know that this map is injective. By Invariance of Domain, it is enough to
  show that the dimensions of the source and the range coincide. This can be
  proved again by induction on~$j$, using the relation $x_j =
  \sigma_{\gamma_j} x_{j-1}$ and using Proposition~\ref{prop:jump-dimensions-Cw}
  which says:   \[ \dim C(x_j) - \dim C(x_{j-1}) = \dim \mathfrak{u}_{\gamma_j} = \dim
    c_{\gamma_j}.\]
  This is the precise relation needed to show that the induction step holds true.
\end{proof}

\begin{corollary}\label{item2bis:thm:pos_defined} 
For any~$\gammab$ in~$\mathbf{W}$, the   map $\mathring{c}_\gammab \to \mathsf{F}_\Theta \mid \vb \mapsto
    \mathring{F}_{\gammab}(\vb)\cdot \mathfrak{p}_{\Theta}^{\mathrm{opp}}$ is open.
\end{corollary}

\subsection{Nontransversality at the boundary}
\label{sec:nontr-at-bound}

We now establish, that elements in the boundary of~${c_\gammab}$ are not sent in
the open Bruhat cell~$\Omega_{\Theta}^{\mathrm{opp}}$. 

 \begin{proposition}\label{prop:notinBruhat}
  Let
   $\vb $ be in~${c_\gammab}$. 
   If
   $\vb$~does not belong to~$\mathring{c}_{\gammab}$, then
    $F_{\gammab}(\vb)  $ does not belong to $ P_{\Delta}^{\mathrm{opp}}
   w^{\Theta}_{\max} P_{\Theta}^{\mathrm{opp}}$.
 \end{proposition}
 \begin{proof}
We use
again natural numbering for the components of~$\gammab$: $\gammab = (\gamma_1,
\dots, \gamma_N)$.
   We denote $(v_1,\dots, v_N)$ the components of~$\vb$. Since $\vb\notin
   \mathring{c}_\gammab$, there exists~$j$ such that $v_{j}$ is in
    $c_{\alpha_{j}} \setminus \mathring{c}_{\alpha_{j} }$.
    For all~$i$ in $
    \llbracket 1,N\rrbracket$ denote by~$t_i$ an element of~$W$ such
    that $\exp(v_i)$ belongs to $P^{\mathrm{opp}}_{\Theta} t_i
    P^{\mathrm{opp}}_{\Theta} $. We have  $\ell_\Theta(t_i)\leq
    \ell_\Theta(\sigma_{\alpha_i}) $. Since $t_j=e$ when $\alpha_j\neq
    \alpha_\Theta$ or by Corollary~\ref{coro:nontr-cone-bruh-J-length} when
    $\alpha_j = \alpha_\Theta$, we have that  $\ell_\Theta(t_j)<
    \ell_\Theta(\sigma_{\alpha_j}) $.

    Repeated applications of Lemma~\ref{lem:J-length-prod-bruhat-cells} show
    that the element~$x$ of~$W$ such that $F_\gammab( \vb)$ belongs to $
    P^{\mathrm{opp}}_{\Theta} x P^{\mathrm{opp}}_{\Theta}$ satisfies
    \begin{equation*}
      \ell_\Theta(x)
       \leq \sum_{i=1}^{N} \ell_\Theta(t_i)
       < \sum_{i=1}^{N} \ell_\Theta(\sigma_{\alpha_i}) = \ell_\Theta(x_{\max}^{\Theta}).
    \end{equation*}
    This implies the result.
 \end{proof}

\begin{corollary}
\label{item3:thm:pos_defined} We have
    $F_\gammab(c_\gammab \setminus \mathring{c}_\gammab) \subset U_\Theta
   \smallsetminus \Omega^{\mathrm{opp}}_{\Theta}$.
\end{corollary}

\begin{corollary}
\label{item4:thm:pos_defined} We have
    $F_\gammab( \mathring{c}_\gammab) = F_\gammab(c_\gammab ) \cap
    \Omega^{\mathrm{opp}}_{\Theta} \subset U_\Theta^{\geq0}\cap  \Omega^{\mathrm{opp}}_{\Theta}$.
\end{corollary}

\subsection{Properness}
\label{sec:properness}

To prove properness of the maps~$F_\gammab$ we consider the
composition of $\log\circ F_\gammab\colon c_\gammab\to \mathfrak{u}_\Theta$
with the projection $p_\alpha\colon \mathfrak{u}_\Theta \to 
\mathfrak{u}_\alpha$ for $\alpha$~varying in~$\Theta$.\index{ $p_\alpha$ the
  projection from  $\mathfrak{u}_\Theta$ to 
$\mathfrak{u}_\alpha$}
As the projection~$p_\alpha$ is a
Lie algebra homomorphism and by applying the Baker--Campbell--Hausdorff
formula, we have for all $\vb=(v_i)_{i\in\llbracket 1, N\rrbracket}$ in
$c_\gammab$
\begin{equation}
  \label{eq:palpha_log}
  p_\alpha \bigl( \log\circ F_\gammab(\vb)\bigr) = \sum_{{\mathclap{{\substack{
        i\in \llbracket 1, N\rrbracket\\ \gamma_i=\alpha}}}}} v_i.
\end{equation}
\begin{proposition}\label{prop:propernessFgamma}
  The map~$F_\gammab\colon c_\gammab \to U_{\Theta}$ is proper. More
  precisely, there is a constant~$C\geq 0$ such that, for any $\vb \in
  c_\gammab$, $\max_i \|v_i\| \leq C \max_{\alpha\in\Theta}\| p_\alpha\circ \log (F_\gammab(\vb))\|$.
\end{proposition}
\begin{proof}
                  For any acute closed convex
  cone~$c$ 
  (in a finite dimensional normed real vector space) and for any $n\geq 1$, the map
  $c^n \to c\mid (x_1, \dots , x_n)\mapsto x_1+\cdots+x_n$ is
  Lipschitz: there is constant~$D$ (depending on~$c$ and~$n$) such
  that $\max \| x_i\| \leq D \| x_1+\cdots +x_n\|$. Together with the formula~\eqref{eq:palpha_log}, this implies the
  precise control on $\vb \mapsto \bigl(p_\alpha\circ
  \log(F_\gammab(\vb))\bigr)_{\alpha\in \Theta}$.
\end{proof}

\subsection{Connectedness}
\label{sec:connectedness}

We now prove that the image of $\mathring{F}_{\gammab}$ is a connected component of the intersection of $U_\Theta$ with the open Bruhat cell~$\Omega^{\mathrm{opp}}_{\Theta}$.

\begin{theorem}\label{cor:conncomp}
  The set $ F_\gammab ( \mathring{c}_\gammab) $ is a connected
  component of $\Omega^{\mathrm{opp}}_{\Theta} \cap U_\Theta$. 
\end{theorem}
\begin{proof}
  Corollary~\ref{coro:F_j-open} and the fact that $U_\Theta\to
  C(w_{\max}^{\Theta}) \mid g \mapsto g \cdot \mathfrak{p}_{\Theta}^{\mathrm{opp}}$ is a
  diffeomorphism imply that the set $ F_\gammab({\mathring{c}_\gammab})$ is
  open in~$U_\Theta$. Since the map
  $F_{\gammab}|_{{c}_\gammab}$ is proper, we
  have that the closure of
  $F_\gammab(\mathring{c}_\gammab)$ in~$U_\Theta$
  is equal to $F_{\gammab}( {c}_\gammab)$.  Therefore the closure  of
  $F_\gammab(\mathring{c}_\gammab)$ in $\Omega^{\mathrm{opp}}_{\Theta}  \cap U_\Theta$ is equal
  to $F_\gammab({c}_\gammab)  \cap \Omega^{\mathrm{opp}}_{\Theta}$ hence to
  $F_\gammab(\mathring{c}_\gammab)$ by Corollary~\ref{item4:thm:pos_defined}. 
  Thus $F_\gammab({\mathring{c}_\gammab})$ is
  open and closed in $\Omega^{\mathrm{opp}}_{\Theta}  \cap U_\Theta$. Since it is connected, it is one connected component.
\end{proof}

\subsection{Independence on~\texorpdfstring{$\gammab$}{γ}}
\label{sec:independence-gammab}
A priori the connected component $ F_\gammab ( \mathring{c}_\gammab)$ could depend on the choice of $\gammab \in\mathbf{W}$. We show now that this is not the case. For this it is sufficient that given any two elements $\gammab, \boldsymbol{\gamma'} \in \mathbf{W}$, the intersection of  $ F_\gammab ( \mathring{c}_\gammab)$ and $ F_{\gammab'} ( \mathring{c}_{\boldsymbol{\gamma'}})$ is nonempty. 
To prove this, we first clarify the relation between~$F_\gammab$ and Lusztig's positive unipotent semigroup in the split real subgroup generated by the $\Theta$-system $(E_\alpha, F_\alpha, D_\alpha)_{\alpha \in \Theta}$.

\subsubsection{Lusztig's positive semigroup}
Recall that the Lie subalgebra $\mathfrak{h}_\Theta$ generated by the
$\Theta$-system $(E_\alpha, F_\alpha, D_\alpha)_{\alpha \in \Theta}$ is a
split real Lie algebra of type $W(\Theta)$. We denote by $H_\Theta < G$ the
connected Lie subgroup with Lie algebra~$\mathfrak{h}_\Theta$. Note that the
intersection $\mathfrak{h}_\Theta \cap \mathfrak{p}_\Theta$ is a standard
minimal parabolic subalgebra in $\mathfrak{h}_\Theta$, and
$\mathfrak{n}_\Theta  = \mathfrak{g}_\Theta \cap \mathfrak{u}_\Theta$ is its
unipotent radical. We denote by $N_\Theta$ the corresponding subgroup, one has
$N_\Theta = U_\Theta \cap H_\Theta$.

The Lie algebra  
 $\mathfrak{n}_\Theta$ is generated by the elements $(E_\alpha)_{ \alpha \in
   \Theta}$. Recall that $E_\alpha \in \mathring{c}_\alpha$ for all $\alpha
 \in \Theta$, more precisely $\mathfrak{n}_\Theta \cap \mathring{c}_\alpha = \R_{>0}E_\alpha$. 
Given an element $\gammab \in\mathbf{W}$ we can restrict the map $F_\gammab $
to $\mathring{c}^{H}_{\gammab}=\prod_{i=1}^{N} \R_{>0} E_{\gamma_i}$.

We have 
\begin{proposition}\label{prop:Lusztig}
The set $N^{>0}_\Theta = F_\gammab \bigl(\mathring{c}^{H}_{\gammab})$ is the unipotent positive semigroup of $H_\Theta$. It is independent of the choice of $\gammab \in \mathbf{W}$.
\end{proposition}
\begin{proof}[Sketch of proof.]
The proposition follows from Proposition~2.7.(b) in Lusztig's work
\cite{lusztigposred}. Using the material developed above, one can also argue
as follows:

The image $F_\gammab \bigl(\mathring{c}^{H}_{\gammab})$ is a connected
component of $\Omega_{\Theta}^{\mathrm{opp}}\cap N_\Theta$. It is thus enough to show that the
images $F_\gammab \bigl(\mathring{c}^{H}_{\gammab})$ and $F_\alphab
\bigl(\mathring{c}^{H}_{\alphab})$ intersect when $\alphab$ and $\gammab$
differ by a braid relation. Focusing on the entries where $\alphab$ and
$\gammab$ differ, we are reduced to the case of rank~$2$ split real Lie groups. The
situation is obvious when this group is $\SL_2(\R)\times \SL_2(\R)$.

For the
group $\SL_3(\R)$, denoting $x_1\colon \R\to U_\Delta \mid t \mapsto \id+t E_{1,2}$ and
$x_2\colon \R\to U_\Delta \mid t \mapsto \id+t E_{2,3}$ the unipotent $1$-parameter groups
corresponding to the simple roots $\alpha_1$ and $\alpha_2$, there are here
$2$~reduced expressions of the longest length element: $s_1 s_2 s_1 = s_2 s_1
s_2$ and the $2$~corresponding maps are $(t_1,t_2,t_3)\mapsto x_1(t_1)
x_2(t_2) x_1(t_3)$ and  $(t_1,t_2,t_3)\mapsto x_2(t_1)
x_1(t_2) x_2(t_3)$ and one checks that $x_1(1)
x_2(2) x_1(1) = x_2(1)
x_1(2) x_2(1)$ showing the result in this case. 

For the group $\Sp_4(\R)$, the unipotent $1$-parameter groups are $x_1\colon
\R\to U_\Delta \mid t \mapsto \id + t(E_{1,2}+E_{3,4})$ and $x_2\colon
\R\to U_\Delta \mid t \mapsto \id + tE_{2,3}$ and one checks the relation
$x_1(1) x_2(2) x_1(3)x_2(1) = x_2(1) x_1(3) x_2(2)x_1(1)$.

For the case of $\lietype{G}_2$, we refer to \cite{Berenstein_Zelevinsky}.
\end{proof}

\begin{remark}
Lusztig has an argument to reduce to the simply laced cases and this enables
him to treat only the case of $\SL_3(\R)$ where he writes down an explicit
formula for the change of coordinates between the images of the maps
associated to these two reduced expressions. This is more than what we used here. 
  \end{remark}

As a corollary we obtain the following proposition. 
\begin{proposition}\label{prop:image_independent}
The image $F_\gammab (\mathring{c}_\gammab)$ is independent of the choice of $\gammab \in \mathbf{W}$. 
\end{proposition}

\begin{proof}
Since $E_\alpha \in \mathring{c}_\alpha$ for all $\alpha \in \Theta$, and
$\mathring{c}_\alpha$ are convex cones, for any~$\gammab$ in~$
\mathbf{W}$ we have that $F_\gammab (\mathring{c}^{H}_{\gammab}) \subset F_\gammab(\mathring{c}_\gammab)$.
Let $\gammab, \boldsymbol{\gamma'} \in   \mathbf{W}$ be two different reduced
expressions of the longest length element in $W(\Theta)$, then 
$F_\gammab (\mathring{c}_\gammab) \cap F_{\boldsymbol{\gamma'}}
(\mathring{c}_{\boldsymbol{\gamma'}})$ contains $N^{>0}_\Theta = F_\gammab (\mathring{c}^{H}_{\gammab}) = F_{\boldsymbol{\gamma'}} (\mathring{c}^{H}_{\gammab}) $. Therefore $F_\gammab (\mathring{c}_\gammab)$ and $F_{\boldsymbol{\gamma'}} (\mathring{c}_{\boldsymbol{\gamma'}})$ have nonempty intersection. Since both sets are connected components of $\Omega^{\mathrm{opp}}_{\Theta} \cap U_\Theta$, they have to be equal.
\end{proof}

\begin{remark}
We can get in fact precise formulas for the change of coordinates. Two elements $\gammab$ and $\boldsymbol{\gamma'}$ differ by a sequence of braid
relation. In the split case with $\Theta = \Delta$ explicit formulas for the change of coordinates for the braid relations are given in \cite{lusztigposred} for the simply laced case, and in \cite{Berenstein_Zelevinsky} for the general case. 
  In a forthcoming article  \cite{GW_pos_2} we describe explicit systems of polynomial equations for the general $\Theta$-positive case following a similar approach. This gives rise to explicit coordinate changes as well. 
\end{remark}

\subsection{Parametrizations of the positive unipotent
  semigroup} \label{sec:consequences}
In this section we draw some consequences on the positive and nonnegative unipotent semigroups. 
In particular we prove that the maps $\mathring{F_\gammab}$ provide parametrizations of $U_\Theta^{>0}$ for any~$\gammab$ in~$\mathbf{W}$. 

By Proposition~\ref{prop:image_independent}, the set $V = F_\gammab
(\mathring{c}_\gammab)$ is independent of the choice of $\gammab \in
\mathbf{W}$.

We still have to show that this set coincides with $U_{\Theta}^{>0}$. For this we prove first.

\begin{proposition}\label{prop:invariance}
The closure~$\overline{V}$ is equal to $F_\gammab(c_\gammab)$ for
every~$\gammab$ in~$\mathbf{W}$. For all $\alpha \in \Theta$ and for all $v_\alpha \in c_\alpha$, 
$V$~is invariant under left and right multiplication by $\exp (v_\alpha)$.
\end{proposition}
\begin{proof}
  The properness of~$F_\gammab$ implies the first statement.
  
  We prove only the invariance by left multiplication, the case of right
  multiplication follows by entirely analoguous arguments.
  
Let $v_\alpha \in c_\alpha$ for some $\alpha \in \Theta$, and let $u$ be an element in $V$. We want to prove that then  $\exp (v_\alpha) u \in V$. 
Let us choose $\gammab \in \mathbf{W}$ such that $\gamma_1 = \alpha$. The existence
of 
such a reduced expression~$\gammab$ of the longest length element
$w_{\max}^{\Theta}$ in $(W(\Theta), R(\Theta))$ is a classical fact and is
established starting from the equality $w_{\max}^{\Theta} = \sigma_\alpha
(\sigma_\alpha w_{\max}^{\Theta})$ and from a reduced expression of
$\sigma_\alpha w_{\max}^{\Theta}$, indeed since $w_{\max}^{\Theta}$ is the
longest length element of $W(\Theta)$, the element $\sigma_\alpha
w_{\max}^{\Theta}$ is of smaller length and $\sigma_\alpha
(\sigma_\alpha w_{\max}^{\Theta})$ is a reduced product.
Let also~$\vb$ be
in~$\mathring{c}_\gammab$ such that $F_\gammab (\vb) = u$
Then we have 
  \begin{align*}
\exp (v_\alpha)  u &= 
\exp (v_\alpha) \exp(v_1)\cdots \exp(v_N) \\
    & = \exp (v_\alpha + v_1)\cdots \exp(v_N),
\end{align*}
where we used that $\gammab$ starts with $\alpha$ and that
$\mathfrak{u}_\alpha$~is Abelian. Since $c_\alpha$ is a
convex cone, $v_1 \in \mathring{c}_\alpha$ and $v_\alpha \in c_\alpha$ we have
that $v_\alpha + v_1  \in \mathring{c}_\alpha$. Therefore $\exp (v_\alpha)
\cdot u $ belongs to~$F_\gammab(\mathring{c}_\gammab) =V$.
\end{proof}

As a consequence of this proposition we obtain 
\begin{corollary}\label{cor:semigroup}
\begin{enumerate}
  \item\label{item1:cor:semigroup} The inclusions $U_{\Theta}^{\geq 0} V \subset
   V$, and $ V U_{\Theta}^{\geq 0} \subset    V$ hold.
  \item\label{item1bar:cor:semigroup} The inclusions $U_{\Theta}^{\geq 0} \overline{V} \subset
   \overline{V}$, and $ \overline{V} U_{\Theta}^{\geq 0} \subset \overline{V}$ hold.
 \item\label{item2:cor:semigroup} $V$ is a semigroup.
  \item\label{item3:cor:semigroup} One has $\overline{V}=U_{\Theta}^{\geq 0}$.
  \item\label{item4:cor:semigroup} One has ${V}=U_{\Theta}^{> 0}$.
   \end{enumerate}
 \end{corollary}
 
 \begin{proof}
  The first point~(\ref{item1:cor:semigroup}) is a direct consequence of 
  Proposition~\ref{prop:invariance} and the fact that the semigroup
  $U_{\Theta}^{\geq 0}$ is generated by $\bigcup_{\alpha\in \Theta}\exp( c_\alpha)$.

  By taking closure in $uV\subset V$ (for $u$ in $U_{\Theta}^{\geq 0}$), we
  get point~(\ref{item1bar:cor:semigroup}).

  Since $V\subset U_{\Theta}^{\geq 0}$, point~(\ref{item1:cor:semigroup})
  implies that $VV\subset V$, i.e.\ that $V$~is a semigroup.

  One has $\overline{V}=F_{\gammab}(c_\gammab)\subset U_{\Theta}^{\geq 0}$ and,
  using point~(\ref{item1bar:cor:semigroup}) with the fact that $e$~belongs
  to~$\overline{V}$, we also have $U_{\Theta}^{\geq 0}\subset
  \overline{V}$. This proves point~(\ref{item3:cor:semigroup}).

  Point~(\ref{item4:cor:semigroup}) now follows from Lemma~\ref{lemm:semigroups}.(\ref{item:3:lemm:semigroups-inter-of-closure}).
\end{proof}

We can now record the result proved so far with:
\begin{theorem}\label{thm:parametrization}
For any  $\gammab \in\mathbf{W}$, the map $\mathring{F_\gammab} \colon \mathring{c}_\gammab \to U_\Theta$ gives a parametrization of the positive unipotent semigroup $U^{>0}_\Theta$. 
\end{theorem}

And furthermore we proved also:

 \begin{corollary}\label{cor:semigroupUtheta}
  \begin{enumerate}[leftmargin=*]
  \item\label{item1:cor:semigroupUtheta} The inclusions $U_{\Theta}^{\geq 0} U_{\Theta}^{>0} \subset
    U_{\Theta}^{>0}$, and $ U_{\Theta}^{>0} U_{\Theta}^{\geq 0} \subset
    U_{\Theta}^{>0}$ hold.
  \item\label{item2:cor:semigroupUtheta}  The nonnegative unipotent semigroup $
    U_{\Theta}^{\geq 0}$ is equal to the closure of ${ U_{\Theta}^{>0}}$, in particular it is closed. 
    
   \item\label{item3:cor:semigroupUtheta}    The nonnegative unipotent semigroup $ U_{\Theta}^{\geq 0}$ is the image $F_\gammab(
    c_\gammab)$ for every~$\gammab$ in~$\mathbf{W}$.
In particular, every
    element in $
    U_{\Theta}^{\geq 0}$ can be written as a finite product (of length at most~$N$)  of elements of the form $\exp (v_\alpha)$ with $v_\alpha \in c_\alpha$. 
  \item\label{item4:cor:semigroupUtheta} The semigroup~$U_{\Theta}^{>0}$ is equal to
    the intersection of $U_{\Theta}^{\geq 0}$
    with~$\Omega_{\Theta}^{\mathrm{opp}}$.
  \end{enumerate}
\end{corollary}

\begin{remark}
Note that $ F_\gammab$ is in general not injective on ${c_\gammab}$, and thus,
point~(\ref{item3:cor:semigroupUtheta}) of Corollary~\ref{cor:semigroupUtheta} does not give a parametrization of $U_{\Theta}^{\geq 0}$. However, we expect that the restriction of  $ F_\gammab$ to appropriate subsets of ${c_\gammab}$ will give a parametrization of $U_\Theta^{\geq 0}$. We describe this in detail for the orhogonal group  $G = \mathrm{SO}(3,q)$, $q \geq 4 $ in Section~\ref{sec:orthogonal}.
\end{remark}

\subsection{Tangent cone of the semigroup}
\label{sec:tip-semigroup}

We now determine the tangent cone at~$e$ of the semigroup $U_{\Theta}^{\geq 0}$. 
For this we introduce some notation. 

Let~$d$ be the degree of nilpotency of the Lie algebra~$\mathfrak{u}_\Theta$
(i.e.\ the iterated Lie brackets of $d+1$ elements of~$\mathfrak{u}_\Theta$ is
always zero).
Let $\mathfrak{e}_{N,d}$ be the free degree~$d$ nilpotent real Lie algebra generated by elements
$e_{1}, \dots, e_{N}$. Then $\mathfrak{e}_{N,d}$ is a finite dimensional graded
Lie algebra whose degree~$1$ component is equal to $\bigoplus_{i=1}^{N} \R
e_{i}$.

Let~$\gammab$ be in~$\mathbf{W}$.

For any~$\vb=(v_1, \dots, v_N)$ in $\mathfrak{u}_{\gamma_1} \times
\cdots\times \mathfrak{u}_{\gamma_N}$ the unique Lie
algebra morphism $\mathfrak{e}_{N,d} \to \mathfrak{u}_\Theta$ sending, for
each~$i$ in $\llbracket 1, N\rrbracket $, $e_{i}$ to~$v_i$ will be denoted~$\Psi_\vb$. The
Lie algebra~$\mathfrak{u}_\Theta$ is also graded (with degree~$1$ component
equal to $\bigoplus_{\alpha\in \Theta} \mathfrak{u}_\alpha$) and the
map~$\Psi_\vb$ is in fact a graded morphism.

The simply connected Lie group with Lie algebra~$\mathfrak{e}_{N,d}$ will be
denoted~$E_{N,d}$. It is well known that $\exp\colon \mathfrak{e}_{N,d} \to E_{N,d}$ is a
diffeomorphism whose inverse will be denoted by~$\log$ and the map
$\mathfrak{e}_{N,d} \times \mathfrak{e}_{N,d} \to \mathfrak{e}_{N,d} \mid (X,Y) \mapsto
\log( \exp(X) \exp(Y))$ is given by the Baker--Campbell--Hausdorff formula.

Let us introduce the following element of~$\mathfrak{e}_{N,d}$
\[ \Gamma \coloneqq \log\bigl( \exp(e_1)\cdots \exp(e_N)\bigr).\]
Then, for every~$\vb$ in 
$\mathfrak{u}_{\gamma_1} \times
\cdots\times \mathfrak{u}_{\gamma_N}$, the following holds
\[ \log\bigl( \exp(v_1)\cdots \exp(v_N)\bigr) = \Psi_\vb(\Gamma).\]
The Baker--Campbell--Hausdorff formula implies that the degree~$1$ component
of~$\Gamma$ is equal to
\begin{equation}
  \label{eq:degree-one-Gamma}
  \Gamma_1 = \sum_{i=1}^{N} e_{i};
\end{equation}
thus $\Gamma-\Gamma_1$ the sum of components of
higher degree.

From the fact that $\Psi_\vb$ preserves the degree, we deduce:
\begin{lemma}\label{lem:logFgamma}
  Let $\|\cdot\|$ be a norm on $\mathfrak{u}_\Theta$. There is a constant~$C\geq 0$ such that, for every~$\vb$ in
    $\mathfrak{u}_{\gamma_1} \times
\cdots\times \mathfrak{u}_{\gamma_N}$,
  \begin{enumerate}[leftmargin=*]
  \item\label{item1:lem:logFgamma} 
    The
    degree~$1$ component of $\log\bigl( \exp(v_1)\cdots \exp(v_N)\bigr)$ is
    equal to
    \( \psi_{\mathbf{v}}(\Gamma_1)=\sum_{i=1}^{N} v_i\);
  \item\label{item2:lem:logFgamma}                        We have $\log\bigl( \exp(v_1)\cdots \exp(v_N)\bigr) - \sum_{i=1}^{N} v_i=
   \psi_{\mathbf{v}}( \Gamma-\Gamma_1)$ and
    \begin{equation*}
      \bigl\| \psi_{\mathbf{v}}( \Gamma-\Gamma_1) \bigr\| \leq
      C \max_{i} \| v_i\|^2.
    \end{equation*}
  \end{enumerate}
\end{lemma}

The projection $\mathfrak{u}_\Theta \to \bigoplus_{\alpha\in \Theta}
\mathfrak{u}_\alpha$\index{$p_1$ the projection from $\mathfrak{u}_\Theta$ to
  $ \bigoplus_{\alpha\in \Theta} \mathfrak{u}_\alpha$}
 according to the decomposition $\mathfrak{u}_\Theta =
\bigl( \bigoplus_{\alpha\in \Theta} \mathfrak{u}_\alpha\bigr) \oplus \bigl(
\bigoplus_{\alpha\in P( \mathfrak{u}_\Theta,
  \mathfrak{t}_\Theta)\smallsetminus\Theta} \mathfrak{u}_\alpha\bigr)$ is
denoted by~$p_1$.

\begin{proposition}\label{prop:tip-semigroup} 
  Let $c\coloneqq \bigoplus_{\alpha\in\Theta} c_\alpha\subset
  \mathfrak{u}_\Theta$. Then $c$ is the tangent cone at~$e$ of the semigroup $U_{\Theta}^{\geq 0}$, precisely
  \begin{enumerate}[leftmargin=*]
  \item\label{item1:prop:tip-semigroup} There is a constant~$C$ such that, for
    every~$X$ in
    $\log(U_{\Theta}^{\geq 0})\subset \mathfrak{u}_\Theta$, 
    $p_1(X)$~belongs to~$c$ and
    $\|X-p_1(X)\| \leq C \|X\|^2$.
  \item\label{item2:prop:tip-semigroup} For any continuous map $\gamma\colon [0,1]\to U_{\Theta}^{\geq 0}$, if
    $\gamma$ is differentiable at~$0$ then $\gamma'(0)$ belongs to~$c$.
  \item\label{item3:prop:tip-semigroup} For all~$Y$ in~$c$, there is~$X$ in $\log(U_{\Theta}^{\geq 0})$ such
    that $p_1(X)=Y$. 
  \item\label{item4:prop:tip-semigroup} For all~$Y$ in~$c$, the map $\gamma\colon \R_{\geq 0}\to U_\Theta \mid
    t \mapsto \exp(tX)$ is of class $C^\infty$, contained in~$U_{\Theta}^{\geq
    0}$ and its derivative at~$0$ is equal to~$Y$.
  \end{enumerate}
\end{proposition}
\begin{proof}
  The first item~(\ref{item1:prop:tip-semigroup}) is a direct consequence of
  the estimates given in  
  point~(\ref{item2:lem:logFgamma}) of Lemma~\ref{lem:logFgamma} and in  
  Proposition~\ref{prop:propernessFgamma}.

  The second
  item~(\ref{item2:prop:tip-semigroup}) is a consequence of the first one.

  The third item~(\ref{item3:prop:tip-semigroup})
  follows easily from the formula for the map $\vb \mapsto
  p_1\circ \log(F_\gammab(\vb))$ (cf.\ Section~\ref{sec:properness}).

  For item~(\ref{item4:prop:tip-semigroup}), it is   classical that this map
  is $C^\infty$ and that its derivative at~$0$ is~$Y$. Let $\gammab$ be
  in~$\mathbf{W}$ and let also~$\vb$ be in~$c_\gammab$ be such that
  $p_1\circ \log\bigl( F_\gammab(\vb)\bigr)=Y$. 
  Let $t$ be
  in~$\R_{\geq 0}$. For all~$n$ in~$\N_{>0}$, $t \vb/n$ belongs to~$c_\gammab$ so that
  \[ F_\gammab(t \vb /n) \text{ and } F_\gammab(t \vb/n )^n\]
  belong to $U_\Theta^{\geq 0}$. Since the sequence $(F_\gammab( t \vb/n )^n)_{n>0}$
  converges to $\exp(tY)$ and
  since $U_{\Theta}^{\geq 0}$ is closed, we deduce that $\exp(tY)$ belongs
  to~$U_{\Theta}^{\geq 0}$.
\end{proof}

\begin{remark}
   The tangent cone at~$e$ of~$U_{\Theta}^{\geq 0}$ is open if and only if $\Theta$~consists of one
element, i.e.\ if and only if the Lie group~$G$ is Hermitian of tube
type. Otherwise the tangent cone is of positive codimension. We refer to
Figure~\ref{fig:semigroup} (Section~\ref{sec:semigr-bigg-than}) for a
visualization of this phenomenon in the case of $\SL_3(\R)$.
A related phenomenon has already
been observed in the case of total positivity in \cite{Lusztig} where the
$1$-parameter semigroups contained in the unipotent positive semigroup are determined.
\end{remark}

As a consequence we also obtain a proof of 
\begin{theorem}\cite[Conjecture~4.11]{GW_ECM}\label{thm:nilpotent}
Let $v = \sum_{\alpha\in \Theta} v_\alpha$ with $v_\alpha \in \mathring{c}_\alpha$ for all $\alpha \in \Theta$. Then 
$\exp(v) \in U_\Theta^{>0}$. 
Conversely, if $v \in \mathfrak{u}_\Theta$ with $\exp(tv) \in U_\Theta^{>0}$
for all~$t$ in $(0,\varepsilon)$ (for some $\varepsilon>0$),
then $v = \sum_{\alpha\in \Theta} v_\alpha$ with $v_\alpha \in
\mathring{c}_\alpha$ for all $\alpha \in \Theta$.
\end{theorem}

\begin{proof}
Let us prove the first statement. Let $E=\sum_{\alpha\in\Theta}
q^{1/2}_{\alpha} E_\alpha$ be the element of the $\Theta$-principal
$\mathfrak{sl}_2$-triple $(E,F, D)$ given in Section~\ref{sec:sl2}. 
Since $L_{\Theta}^{\circ}$ acts transitively on $\Pi_{\alpha \in \Theta}
\mathring{c}_\alpha$ (Proposition~\ref{prop:homogeneity-Ltheta}), there exists $\ell \in L_{\Theta}^{\circ}$ such that $v = \Ad(\ell) E$. 
Since $\exp\colon \mathfrak{u}_\Theta  \rightarrow U_\Theta$ is equivariant
with respect to $L_{\Theta}^{\circ}$, and $U_{\Theta}^{>0}$ is
$L_{\Theta}^{\circ}$-invariant, it is enough to show that $\exp(E) \in
U_{\Theta}^{>0}$. It is clear that $\exp(E)$~belongs to~$U_{\Theta}^{\geq
  0}$. Let~$a$ and~$b$ be the points $\mathfrak{p}_\Theta$ and
$\mathfrak{p}_{\Theta}^{\mathrm{opp}}$ of~$\mathsf{F}_\Theta$. We need to
prove that $\exp(E)\cdot b$ is transverse to~$b$
(Corollary~\ref{cor:semigroupUtheta}.(\ref{item4:cor:semigroupUtheta})). From
equality valid in $\SL_2(\R)$, one has $\exp(E)\cdot b = \exp(F) \dot{s} \cdot
b$ with $\dot{s} = \exp\bigl(
\frac{\pi}{2}(E-F)\bigr)$. Lemma~\ref{lem:principal-weyl-group} implies that $\dot{s} \cdot
b =a$ and thus the wanted transversality.

Conversely, let~$v$ be in~$\mathfrak{u}_\Theta$ such that $\exp(tv)$ belongs
to~$U_{\Theta}^{>0}$ for all~$t$ in $(0,\varepsilon)$. This implies first that $v$~belongs to
the tangent cone at~$e$ to $U_{\Theta}^{\geq 0}$ and therefore
$v=\sum_{\alpha\in\Theta} v_\alpha$ with $v_\alpha\in c_\alpha$ for
all~$\alpha$ in~$\Theta$. Suppose by contradiction that there exists~$\alpha$
in~$\Theta$ such that $v_\alpha$~does not belong to $\mathring{c}_\alpha$. By
Proposition~\ref{prop:homogeneity-Ltheta}, the stabilizer of~$v$ in~$L_\Theta$ is not
compact; this would imply that for all~$t$ the centralizer of $\exp(tv)$
in~$L_\Theta$ is not compact, in contradiction with the properness of the
action of~$L_\Theta$ on the positive unipotent semigroup. 
\end{proof}
\begin{remark}
  As $U_{\Theta}^{>0}$ is a semigroup, the hypothesis in the theorem is
  equivalent to having  $\exp(tv) \in U_\Theta^{>0}$
for all~$t$ in~$\R_{>0}$.
\end{remark}

\section{The orthogonal groups}\label{sec:orthogonal}
In this section we discuss in a bit more detail the case when $G = \mathrm{SO}(3,q)$, $q\geq 4$. The description of the $\Theta$-positive structure for general orthogonal groups $\SO(p,q)$, $q>p>2$  restricts basically to this case and the description of the positive structure for $\SL_3(\R)$, see also the discussion in \cite[Section~4.5]{GW_ECM}. 
For the case when $G = \mathrm{SO}(3,q)$ ($q\geq 4$), we also provide a parametrization of the nonnegative semigroup~$U_{\Theta}^{\geq0}$. 

\subsection{The positive semigroup} 
We realize $\mathrm{SO}(3,q) = \mathrm{SO}(b_Q)$, where $b_Q$\index{$b_Q$ the
  quadratic form of signature $(3,q)$ associated with the matrix~$Q$} is the nondegenerate symmetric bilinear form of signature $(3,q)$ on $\R^{3+q}$ given by $b_Q(v,w) = {}^t\!v Q w$, with $Q = \begin{pmatrix} 0& 0& K\\ 0 &J &0 \\ -K&0& 0 \end{pmatrix}$, 
  $K= \begin{pmatrix} 0& 1\\ -1&0 \end{pmatrix}$, and $J = \begin{pmatrix} 0&0& 1\\ 0& -\id_{q-3} &0\\ 1 & 0& 0 \end{pmatrix}$. 
We denote by $b_J$ the form $b_J(x,y) = \frac{1}{2} {}^t x J y$ and set $q_J(x) = b_J(x,x)$. Note that $b_J$ is a nondegenerate symmetric bilinear form of signature $(1,q-2)$ on the corresponding subspace of $\R^{q+3}$. 

We choose the Cartan subspace $\mathfrak{a}\subset \mathfrak{so}(3,q)$
to be the intersection of the set of diagonal matrices with
$\mathfrak{so}(3,q)$. Concretely $\mathfrak{a}$~is the space of diagonal
matrices $\mathrm{diag}(\lambda_1, \lambda_2, \lambda_3, 0, \dots, 0,
-\lambda_3, -\lambda_2, -\lambda_1)$, with $\lambda_1$, $\lambda_2$,
and~$\lambda_3$ varying in~$\R$. 
Denote $e_i\colon \mathfrak{a}\to \R$ the linear form that associates to a
diagonal matrix its $i$-th diagonal coefficient. The set of (restricted) roots is
$\{
\pm e_i \pm e_j\}_{1\leq i< j\leq 3} \cup \{ \pm e_i\}_{1\leq i \leq 3}$; a
standard choice for the set of positive roots is $\{
 e_i \pm e_j\}_{1\leq i< j\leq 3} \cup \{  e_i\}_{1\leq i \leq 3}$ and 
the set of simple roots $\Delta = \{\alpha_1, \alpha_2, \alpha_3\}$ is given by $\alpha_i = e_i-e_{i+1}$, for $i = 1,2$, and $\alpha_3 = e_3$. 

We have $\Theta  = \{\alpha_1, \alpha_2\}$. Furthermore 
$\mathfrak{u}_{\alpha_1} \cong \R$ and $\mathfrak{u}_{\alpha_2} \cong
\R^{1,q-2}$. Then  the closed cones identify with $c_{1}  \cong \R_{\geq0}$
and $c_{2} \cong {c}^{1,q-2} =\{ x \in \R^{1,q-2}\mid q_J(x) \geq 0, \,
x_1\geq0\}$\index{${c}^{1,q-2}$ the closed Lorentz cone in signature $(1,q-2)$} and the closed cone $\mathring{c}_{1} \cong \R_{>0}$,
and $\mathring{c}_{2} \cong \mathring{c}^{1,q-2} = \{ x \in
\R^{1,q-2} \mid q_J(x) >0, \, x_1>0\}$.

The Weyl group $W$ is isomorphic to $\{\pm 1\}^{3} \rtimes {S}_3$, the
group of signed permutation matrices. It is presented by the 
generators $s_1, s_2, s_3$ associated with the simple roots, with the relations $s_i^2 = (s_1s_2)^3 = (s_1s_3)^2 = (s_2s_3)^4 = e$. 
The group $W(\Theta)$ is generated by $\sigma_1 = s_1$ and $\sigma_2 =
s_2s_3s_2$. It is a Weyl group of type $\lietype{B}_2$ and the longest length element is $w_{\max}^\Theta = \sigma_1\sigma_2\sigma_1\sigma_2 = \sigma_2\sigma_1\sigma_2\sigma_1$. 

Using that the map from $\mathfrak{u}_{\alpha_1} \times
\mathfrak{u}_{\alpha_2}\times \mathfrak{u}_{\alpha_1+\alpha_2} \times
\mathfrak{u}_{2\alpha_1+\alpha_2}$ to $U_\Theta$ given by the product of
exponentials is a diffeomorphism, every element in $U_\Theta$ can be uniquely
written as a matrix
\begin{multline*}
  U(a,e_1,e_2,c)=\\
  \begin{pmatrix}
    1 & a & {}^t(a e_1 -e_2)J & c +a q_J(e_1)
    & -ac - 2 a b_J(e_1,e_2) - q_J(e_2) \\
      & 1 & {}^t e_1J & q_J(e_1) & -c-2b_J(e_1, e_2)\\
      & & \id_{q-1} &  e_1 &  e_2 \\
      & & & 1 & a \\
      & & & & 1
  \end{pmatrix}.
\end{multline*}

The exponential maps $\R \to U_\Theta$ and $\R^{q-1} \to U_\Theta$ are
denoted $ x_{1},   x_{2}$ and given
by 
\begin{align*}
  x_1(s) &= U(s,0,0,0)\\
  x_2(v) &= U(0,v,0,0).
\end{align*}

Thus, setting $\gammab =(\alpha_1, \alpha_2, \alpha_1, \alpha_2)$, the positive semigroup $U_{\Theta}^{>0}$ in this case is $$U_{\Theta}^{>0} = F_{\gammab} (\mathring{c}_{1}\times \mathring{c}_{2}\times \mathring{c}_{1}\times \mathring{c}_{2}),$$ i.e.\ 
all matrices, that can be written as $x_{1}(s_1)x_{2}(v_1)x_{1}(s_2)x_{2}(v_2)$, with $s_1, s_2 \in \R_{>0}$ and $v_1, v_2 \in  \mathring{c}^{1,q-2}$.

Concretely, one has:
\begin{multline*}
  x_{1}(s_1)x_{2}(v_1)x_{1}(s_2)x_{2}(v_2) = \\U(s_1+s_2,  v_1+  v_2, s_2
  v_1,  -s_2 (q_J(v_1)+2b_J(v_1+v_2,v_2))).
\end{multline*}

A matrix $U(a,e_1, e_2, c)$ hence belongs to $U_{\Theta}^{>0}$ if and only if 
\begin{align*}
  -c-2b_J(e_1, e_2)  >0,& 
  -ac - 2 a b_J(e_1,e_2) - q_J(e_2)  >0 \\
  e_2 \in \mathring{c}^{1,q-2},& \ 
  q_J(e_2)e_1 +(c+2b_J(e_1, e_2))e_2 \in \mathring{c}^{1,q-2}.
\end{align*}

When we use the other reduced expression of the longest length element
$w_{\max}^{\Theta}= \sigma_2\sigma_1\sigma_2\sigma_1$, and
consider elements $x_{2}(w_2)x_{1}(t_2)x_{2}(w_1)x_{1}(t_1)$ with $t_1, t_2 \in \R_{>0}$ and $w_1, w_2 \in  \mathring{c}^{1,q-2}$, we parametrize as well the positive semigroup $U_{\Theta}^{>0}$.

The equation
\begin{equation}
  \label{eq:relation:sopq}
  x_{1}(s_1)x_{2}(v_1)x_{1}(s_2)x_{2}(v_2) = x_{2}(w_2)x_{1}(t_2)x_{2}(w_1)x_{1}(t_1),
\end{equation}
determines then the change between the two parametrizations; 
comparing the entries of the corresponding matrices, 
we get the following relations 
\begin{align}
  \label{eq:1:sopq}
  s_1 + s_2 &= t_1 + t_2\\
  \label{eq:2:sopq}
  v_1 + v_2 &= w_1 + w_2\\
  \label{eq:3:sopq}
  s_1(v_1 + v_2) + s_2 v_2 &= t_2 w_1\\
  \label{eq:4:sopq}
  s_1 q_J(v_1+v_2) + s_2 q_J(v_2) &= t_2 q_J(w_1),\\
  \label{eq:5:sopq}
  s_2v_1 &= t_2 w_2 + t_1 (w_1+w_2),\\
  \label{eq:6:sopq}
  s_2 q_J(v_1) &= t_2q_J(w_2) + t_1 q_J(w_1+w_2)\\
  \label{eq:7:sopq}
  s_1s_2 q_J(v_1) &= t_2 t_1 q_J(w_1)
\end{align}

\begin{remarks}
    \begin{enumerate}[leftmargin=*]
  \item Since the maps $x_1$, $x_2$ have some compatibility with transposition
    (namely with $D$ the block diagonal
    matrix with blocks $\id_2, J, \id_2$, for all~$v$ in~$\R^{q-1}$,
    ${}^t x_2(v) = x_2( Jv) = D x_2 (v)D^{-1}$  and, for all~$s$ in~$\R$,
    ${}^t x_1(s)= x_1(s) = Dx_1(s) D^{-1}$), Equation~\eqref{eq:relation:sopq}
    above is equivalent to
    \[ x_{1}(t_1)x_{2}(w_1)x_{1}(t_2)x_{2}(w_2) =
      x_{2}(v_2)x_{1}(s_2)x_{2}(v_1)x_{1}(s_1). \] Hence the chosen numbering
    makes the relation between $(s_i, v_i)$ and $(t_i, w_i)$ more symmetric.
  \item  Equations~\eqref{eq:1:sopq}--\eqref{eq:4:sopq} are the same than the
    ones obtained by Berenstein and Zelevinsky
    \cite[p.~140]{Berenstein_Zelevinsky} for $\Sp_4(\R)$, i.e. for a Weyl group of type $\lietype{C}_2 = \lietype{B}_2$,  where the squares in their formulas 
    there are replaced by the quadratic form~$q_J$ here.
  \item Equations~\eqref{eq:1:sopq}--\eqref{eq:4:sopq} determine the others:
    indeed multiplying \eqref{eq:1:sopq}$\times$\eqref{eq:2:sopq} and
    subtracting~\eqref{eq:3:sopq} gives~\eqref{eq:5:sopq}; the scalar product
    of~\eqref{eq:2:sopq} with the difference of~\eqref{eq:3:sopq}
    and~\eqref{eq:5:sopq} is the difference of~\eqref{eq:4:sopq}
    and~\eqref{eq:6:sopq}; applying~$q_J$ to~\eqref{eq:3:sopq} and subtracting
    the product of~\eqref{eq:1:sopq} and~\eqref{eq:4:sopq}
    gives~\eqref{eq:7:sopq}.
  \end{enumerate}
\end{remarks}
These equations can be solved as follow: the ratio of~\eqref{eq:7:sopq}
and~\eqref{eq:4:sopq} gives~$t_1$; the ratio of $q_J$\eqref{eq:3:sopq}
and~\eqref{eq:4:sopq} gives~$t_2$; once $t_2$~is determined, \eqref{eq:3:sopq}
gives~$w_1$; once $w_1$~is determined \eqref{eq:2:sopq}
gives~$w_2$. Explicitely we have

\begin{align*} 
t_1 &= \frac{s_1s_2 q_J(v_1)}{s_1 q_J(v_1+v_2) + s_2 q_J(v_2)}, \
 t_2 = \frac{q_J\bigl(s_1(v_1+v_2) + s_2 v_2\bigr)}{s_1 q_J(v_1+v_2) + s_2 q_J(v_2)},\\
w_1 &= \frac{s_1 q_J(v_1+v_2) + s_2 q_J(v_2)}{q_J(s_1(v_1+v_2) + s_2 v_2)} \bigl(s_1(v_1+v_2) + s_2 v_2\bigr),\\
  w_2 &= \frac{s_2}{q_J(s_1(v_1+v_2) + s_2 v_2)} \Bigl(s_2 q_J(v_2)v_1 + \Bigr.\\
  & \qquad\qquad \Bigl. s_1\bigl( q_J(v_1+v_2)v_1 - q_J(v_1)(v_1+v_2)\bigr)\Bigr).
\end{align*} 
Note that even though the formula for $w_2$ contains a minus sign, the
following lemma implies that $w_2$~belongs to $\mathring{c}^{1,q-2}$. 

\begin{lemma}
  For all~$v_1$ and~$v_2$ in~$\mathring{c}^{1,q-2}$, the element
  \[ a(v_1,v_2)= q_J(v_1+v_2) v_1 - q_J(v_1)(v_1+v_2)\]
  belongs to~$\mathring{c}^{1,q-2}$.
\end{lemma}
\begin{proof}
  We calculate $q_J(a(v_1,v_2))$. This gives
  \[ q_J(a(v_1,v_2)) = q_J(v_1+v_2) g_J(v_1)q_J(v_2).\]
  Thus $a(v_1,v_2)$ belongs to $\mathring{c}^{1,q-2} \sqcup
  -\mathring{c}^{1,q-2}$. Since $a(v_1,v_1)=2q_J(v_1)v_1$ a connectedness
  argument shows that, for all~$v_1$ and~$v_2$ in~$\mathring{c}^{1,q-2}$, $a(v_1,
  v_2)$ belongs to~$\mathring{c}^{1,q-2}$.
\end{proof}

\subsection{The nonnegative semigroup}\label{sec:example_nonnegative}
By point~(\ref{item3:cor:semigroupUtheta}) of
Corollary~\ref{cor:semigroupUtheta}, the nonnegative unipotent semigroup is
characterized as
$$U_{\Theta}^{\geq 0} = F_{\gammab} ({c}_{1}\times
{c}_{2}\times {c}_{1}\times {c}_{2}),$$ i.e.\ every element in
$U_{\Theta}^{\geq 0} $ can be written as
\[x_{1}(s_1)x_{2}(v_1)x_{1}(s_2)x_{2}(v_2),\]
with $s_1,s_2 \in \R_{\geq 0}$ and  $v_1, v_2 \in  {c}^{1,q-2}$. 

However this does not give a parametrization of $U_{\Theta}^{\geq 0}$ since one can easily check that $F_{\gammab}$ is not injective. For example take $v_1, v_2 \in {c}^{1,q-2}$ and $s_1 = s_2 = 0$. 
Then $$x_{1}(s_1)x_{2}(v_1)x_{1}(s_2)x_{2}(v_2) \\= x_{2}(v_1)x_{2}(v_2) = x_{2}(v_1+v_2) ,$$ hence 
$F_{\gammab} (0, v_1, 0, v_2)  = 
F_{\gammab} (0, v_1+v_2, 0, 0).$ 

Giving a parametrization of $U_{\Theta}^{\geq 0}$ is thus more subtle.
Already in the
case of split real groups, the maps $F_\gammab$ (defined here on $\R_{\geq
  0}^{N}$) are not injective (for the same reason than in the example above). A
way to remedy to this non-injectivity in this split case has been found by
Lusztig in \cite[Corollary~2.8]{Lusztig}; Starting from the fact 
that the nonnegative unipotent semigroup $U^{\geq 0}$ is the
disjoint union of subsets $U^{>0}(w)= U^{\geq 0} \cap P_{\Delta}^{\mathrm{opp}} w P_{\Delta}^{\mathrm{opp}}$, where $w$~varies over the elements of the
Weyl group~$W$. 
Then  parametrizations of
$U^{>0}(w)$ are obtained from  reduced expressions of~$w$, similarly to the
parametrizations of~$U^{>0}$. Of course Lusztig showed first that the image of
these parametrizations does not depend on the choice of reduced expression.
Lusztig showed first that this image does not depend on the choice of reduced expression. 

In the parametrization of the positive unipotent semigroup $U_{\Theta}^{>0}$
the role of the Weyl group $W$ in the split case is replaced by the Weyl
group $W(\Theta)$, so one might hope to get a decomposition of
$U_{\Theta}^{\geq 0}$ into a disjoint union of parametrized sets, indexed by
elements in $W(\Theta)$. The discussion in Section~\ref{sec:nontr-cone-bruh}
shows that this can not work already in the case of Hermitian Lie groups. 
In order to describe a decomposition of $U_{\Theta}^{\geq 0}$ we have to
consider a different object than the $\Theta$-Weyl group and to replace it by the double quotient $W_{\Delta\setminus
  \Theta}\backslash W / W_{\Delta\setminus\Theta}$.

We describe a decomposition of the nonnegative unipotent semigroup
for~$\mathrm{SO}(3,q)$ and a parameterization of it, using $W_{\Delta\setminus
  \Theta}\backslash W / W_{\Delta\setminus\Theta}$. It will be of interest to
explore this further in the general case.

Let us first consider the Bruhat decomposition of~$G$ with respect to the action of
$P_{\Theta}^{\mathrm{opp}} \times P_{\Theta}^{\mathrm{opp}}$:
$G = \bigsqcup_{a \in W_{\Delta\smallsetminus\Theta}\backslash W /
  W_{\Delta\smallsetminus\Theta} }P_{\Theta}^{\mathrm{opp}} a
P_{\Theta}^{\mathrm{opp}}$.  The next lemma determines representatives of
$W_{\Delta\smallsetminus\Theta}\backslash W/W_{\Delta\smallsetminus\Theta}$.

\begin{lemma}\label{lem:so3q:representative-W-double-class}
The following provides a list of smallest length
representatives for the $16$~classes in $W_{\Delta\smallsetminus\Theta}\backslash W/W_{\Delta\smallsetminus\Theta}$: 
\begin{enumerate}
\item $e$
\item $s_1$
\item $ s_2$
\item $ (s_2s_3s_2)$
\item $ s_2 s_1$
\item $ (s_2s_3s_2)s_1$
\item $ s_1s_2$ 
\item $ s_1 (s_2s_3s_2)$ 
\item $ (s_2s_3)s_1s_2$ 
\item $ s_2s_1(s_2s_3s_2)$
\item $ (s_2s_3s_2)s_1s_2$ 
\item $ (s_2s_3s_2)s_1(s_2s_3s_2)$ 
\item $ s_1s_2s_1$ 
\item $ s_1(s_2s_3s_2)s_1$
\item $ s_1(s_2s_3s_2)s_1s_2$ 
\item $ s_1(s_2s_3s_2)s_1(s_2s_3s_2)$. 
\end{enumerate}
\end{lemma}
\begin{proof}
  Here the group~$W$ can be realized as the group of signed permutation
  $3\times 3$-matrices and the subgroup $W_{\Delta\setminus\Theta}$ is the
  subgroup (isomorphic to $\Z/2\Z$) of diagonal matrices whose only
  possibly nontrivial diagonal coefficient is in the third position. The
  element~$s_1$ corresponds to the transposition $(12)$, $s_2$~corresponds to
  the transposition $(23)$, and $s_3$~is the nontrivial element in
  $W_{\Delta\setminus\Theta}$.

  A direct calculation shows then the result.
\end{proof}

In the present situation the flag variety $\mathsf{F}_\Theta$ can be realized as the space of partial
flags in $(\R^{3+q},b_Q)$ consisting of an isotropic line and an isotropic
$2$-plane containing it. For
concreteness the canonical basis of $\R^{3+q}$ used so far will be denoted\index{$(e_1, e_2, e_3,g_1,
  g_2, \dots, g_{q-3}, f_3, f_2, f_1)$ the basis of $\R^{3+q}$}
\[(e_1, e_2, e_3,g_1,
  g_2, \dots, g_{q-3}, f_3, f_2, f_1).\]
Then $E^+ = (\R e_1, \R e_1 \oplus \R
e_2)$ and  $E^- = (\R f_1, \R f_1 \oplus \R f_2)$ are partial isotropic flags,
 $P_\Theta= \mathrm{Stab} (E^+)$, and $P_{\Theta}^{\mathrm{opp}} =
\mathrm{Stab} (E^-)$. More generally, given a pair of linearly independent vectors $(v_1, v_2)$ in $\R^{3,q}$, 
the flag determined\index{$E_{(v_1, v_2)}$ the flag determined by the
  vectors~$v_1$ and~$v_2$} by $(v_1, v_2)$  is denoted $E_{(v_1, v_2)} =
(\R v_1, \R v_1 \oplus \R v_2)$; this flag depends only on the lines $\R v_1$
and $\R v_2$, it belongs to~$\mathsf{F}_\Theta$ if and only if the plane $\R
v_1 \oplus \R v_2$ is isotropic with respect to~$b_Q$.

\begin{lemma}\label{lem:so3q:representative-Flag-double-class}
  The $P_{\Theta}^{\mathrm{opp}}$-orbits in~$\mathsf{F}_\Theta$ are in
  one-to-one correspondence with $P_{\Theta}^{\mathrm{opp}} \times
  P_{\Theta}^{\mathrm{opp}}$-orbits in~$G$ and
  hence with the $16$~classes in
  $W_{\Delta\smallsetminus\Theta}\backslash W/W_{\Delta\smallsetminus\Theta}$.
  For each class, we give a flag in $\mathsf{F}_\Theta$ belonging to the
  corresponding orbit:
\noindent\begin{enumerate} 
\item $e$: $E_{(f_1, f_2)}=E^-$
\item $s_1$: $E_{(f_2, f_1)}$
\item $s_2$: $E_{(f_1, f_3)}$
    \item $(s_2s_3s_2)$: $E_{(f_1, e_2)}$
\item $s_2 s_1$: $E_{(f_3, f_1)}$
\item $(s_2s_3s_2)s_1$: $E_{(e_2, f_1)}$
\item $s_1s_2$: $E_{(f_2, f_3)}$  
\item $s_1 (s_2s_3s_2)$: $E_{(f_2, e_1)}$ 
\item $(s_2s_3)s_1s_2$ : $E_{(f_3, e_2)}$
\item $s_2s_1(s_2s_3s_2)$: $E_{(f_3, e_1)}$
\item $(s_2s_3s_2)s_1s_2$: $E_{(e_2, f_3)}$ 
\item $(s_2s_3s_2)s_1(s_2s_3s_2)$: $E_{(e_2, e_1)}$ 
\item $s_1s_2s_1$: $E_{(f_3, f_2)}$ 
\item $s_1(s_2s_3s_2)s_1$: $E_{(e_1, f_2)}$
\item $s_1(s_2s_3s_2)s_1s_2$: $E_{(e_1, f_3)}$
\item $s_1(s_2s_3s_2)s_1(s_2s_3s_2)$: $E_{(e_1, e_2)}=E^+$. 
\end{enumerate}
\end{lemma}

\begin{proof}
  The flag $E^- = E_{(f_1, f_2)}$ corresponds to the class of the trivial
  element. A flag belonging to the orbit   $P_{\Theta}^{\mathrm{opp}}
  [w] P_{\Theta}^{\mathrm{opp}}/P_{\Theta}^{\mathrm{opp}} \subset
  G/P_{\Theta}^{\mathrm{opp}} = \mathsf{F}_\Theta$, associated with~$w$ in~$W$, is then $\dot{w}\cdot
  E^{-}$ where $\dot{w}$ is a lift of~$W$ to~$\SO(3,q)$.

  For the explicit calculation, we can choose
  \begin{itemize}
  \item For the lift of~$s_1$ the matrix that exchanges~$e_1$ and~$e_2$,
    exchanges~$f_1$ and~$f_2$, and fixes the other basis vectors.
  \item For the lift of~$s_2$ the matrix that exchanges~$e_2$ and~$e_3$,
    exchanges~$f_2$ and~$f_3$, and fixes the other basis vectors.
  \item For the lift of~$s_3$ the matrix that sends~$e_3$ to~$f_3$,
    sends~$f_3$ to~$-e_3$, and fixes the other basis vectors.\qedhere
  \end{itemize}
\end{proof}

For every~$x$ in $W_{\Delta\setminus\Theta}\backslash W/W_{\Delta\setminus\Theta}$, define
$U^{>0}_{\Theta}(x)$ to be the intersection of $P^{\mathrm{opp}}_{\Theta} x
P^{\mathrm{opp}}_{\Theta}$ with $U^{\geq 0}_{\Theta}$. Let 
$F_\gammab$ be as above the map from $c_1\times c_2 \times c_1 \times c_2$ to
$U_\Theta\subset \SO(3,q)$. 
For each reduced expression~$w$ listed in
Lemma~\ref{lem:so3q:representative-W-double-class} we define~$F_w$ to be the
restriction\index{$F_w$ a
restriction of $F_\gammab$ associated with the element $w$} of~$F_\gammab$ to
the set~$D_w$\index{$D_w$ the domain of definition of $F_w$} where\\
\begin{enumerate}[leftmargin=*]
\item $D_e=\{0\}\times \{0\}\times \{0\}\times \{0\} $,
\item $D_{s_1} = \R_{>0}\times \{0\}\times \{0\}\times \{0\}$,
\item  $D_{s_2}=\{0\}\times(  \partial c_{2} \setminus \{ 0\}) \times \{0\}\times \{0\} $,
\item  $D_{s_2s_3s_2}=\{0\}\times \mathring{c}_2\times \{0\}\times \{0\} $,
\item $ D_{s_2 s_1}=\{0\}\times (\partial c_{2} \setminus \{ 0\}) \times
  \R_{>0} \times \{0\}$,
\item $ D_{(s_2s_3s_2)s_1} =\{0\}\times  \mathring{c}_{2} \times \R_{>0}
  \times \{0\}$,
\item $D_{s_1s_2}= \R_{>0}  \times  (\partial c_{2} \setminus \{ 0\} ) \times \{0\} \times \{0\}$,
\item $D_{s_1 (s_2s_3s_2)} =\R_{>0}  \times  \mathring{c}_{2} \times \{0\} \times \{0\}$,
\item $D_{(s_2s_3)s_1s_2} =\{ (0, v_1, s_1, v_2) \mid s_1> 0, v_1, v_2 \in
  \partial c_2,\, b_J(v_1,
  v_2)\neq 0\}$,
\item $D_{s_2s_1(s_2s_3s_2)} = \{0\} \times (\partial c_{2} \setminus \{ 0\} )
  \times \R_{>0}  \times  \mathring{c}_{2}$,
\item $D_{(s_2s_3s_2)s_1s_2}= \{0\}\times \mathring{c}_{2} \times \R_{>0}  \times (\partial c_{2} \setminus \{ 0\} )$,
\item $D_{(s_2s_3s_2)s_1(s_2s_3s_2)} = \{0\}\times \mathring{c}_{2} \times \R_{>0}  \times \mathring{c}_{2}$,
\item $D_{s_1s_2s_1}= \R_{>0} \times (\partial c_{2} \setminus \{ 0\})  \times
  \R_{>0} \times \{0\}$,
\item $D_{s_1(s_2s_3s_2)s_1} = \R_{>0} \times  \mathring{c}_{2} \times \R_{>0}
  \times \{0\}$,
\item $D_{ s_1(s_2s_3s_2)s_1s_2}= \R_{>0} \times  \mathring{c}_{2}   \times
  \R_{>0} \times(\partial c_{2} \setminus \{0\})$, 
      \item $D_{s_1(s_2s_3s_2)s_1(s_2s_3s_2)} =\R_{>0} \times  \mathring{c}_{2}   \times \R_{>0} \times\mathring{c}_{2}.$ 
\end{enumerate}

\begin{proposition}\label{prop:param_nonnegative}
  For every~$w$ in the list of Lemma~\ref{lem:so3q:representative-W-double-class} the map
  \begin{align*}
    D_w & \longrightarrow \mathsf{F}_\Theta \\
    z & \longmapsto F_w(z)\cdot E^-
  \end{align*}
  is injective; the image of~$F_w$ is $U^{\geq 0}_{\Theta} \cap
  P_{\Theta}^{\mathrm{opp}} w P_{\Theta}^{\mathrm{opp}} = U^{>0}_{\Theta}(x)$
  with $x=[w]$. 
The nonnegative semigroup $U_{\Theta}^{\geq 0}$ is the disjoint union the $U^{>0}_{\Theta}(x)$.
\end{proposition}
\begin{proof}
  The map $(s_1, v_1, s_2, v_2) \mapsto F_\gammab(s_1, v_1, s_2, v_2) \cdot
  E^-$ is   completely determined
  by the last two columns of the matrix
  $F_\gammab(s_1, v_1, s_2, v_2)$ that is
  \[ \begin{pmatrix} 
s_1 q_J(v_1+v_2)+ s_2 q_J(v_2)&s_1 s_2 q_J(v_1)\\
q_J(v_1+v_2)& s_2 q_J(v_1)\\
v_1+v_2&s_2 v_1\\
1&s_1+s_2\\
0&1
     \end{pmatrix}.\]
   We can now check by case by case consideration that the statement
   holds.
\end{proof}

\section{Invariant unipotent semigroups}
\label{sec:invar-unip-semigr}
In Section~\ref{sec:posit-flag-vari} we will consider a more geometric
approach to positivity in terms of triples in flag varieties. In
order to connect this with the algebraic approach we took so far, we will
make use of the results of the present section where 
we  show that essentially any $L_{\Theta}^{\circ}$-invariant
semigroup~$U^+$ in~$U_\Theta$ arises from a $\Theta$-positive structure.

We say that a semigroup~$U$ in a group~$G$ is \emph{sharp} if its only
invertible element is the neutral element (i.e.\ if $g\in U^+$ and
    $g^{-1}\in U^+$ then $g=e$).

\begin{theorem}
  \label{theo:invar-unip-semigr-implies-pos}
  Let~$G$ be a connected semisimple Lie group, and let $U_\Theta$ be a
  standard unipotent subgroup of~$G$ and $L_\Theta$ be the corresponding
  standard Levi
  subgroup. Suppose that there is $U^{+}\subset U_\Theta$ such that
  \index{$U^+$ a closed $L_{\Theta}^{\circ}$-invariant, sharp semigroup of
    nonempty interior in
    $U_\Theta$ (Theorem~\ref{theo:invar-unip-semigr-implies-pos})}
  \begin{enumerate}
  \item\label{item1:theo:invar-unip-semigr-implies-pos} $U^+$ is closed and of nonempty interior;
  \item\label{item2:theo:invar-unip-semigr-implies-pos} $U^+$ is $L_{\Theta}^{\circ}$-invariant;
  \item\label{item3:theo:invar-unip-semigr-implies-pos} $U^+$ is a semigroup;
  \item\label{item4:theo:invar-unip-semigr-implies-pos} $U^+$ is sharp.
  \end{enumerate}
  Then $G$ admits a $\Theta$-positive structure; for all~$\alpha$ in~$\Theta$,
  there is a unique $L_{\Theta}^{\circ}$-invariant closed nontrivial convex cone
  $c_\alpha\subset \mathfrak{u}_\alpha$ such that $\exp( c_\alpha)\subset
  U^+$. Defining the nonnegative unipotent semigroup with these cones (Definition~\ref{def:nonnegative}), the semigroup~$U^+$ contains the
  semi\-group~$U_{\Theta}^{\geq 0}$.
\end{theorem}

We treat now the case of a semigroup that is equal to the intersection
of~$U_\Theta$ with the open Bruhat cell:

\begin{theorem}
  \label{theo:invar-unip-semigr-implies-pos-open-case}
  Let~$G$ be a connected semisimple Lie group, and let $U_\Theta$ be a
  standard unipotent subgroup of~$G$ and $\Omega_{\Theta}^{\mathrm{opp}}$ the
  open Bruhat cell with respect to the opposite parabolic group
  $P_{\Theta}^{\mathrm{opp}}$. Suppose that there is $V^{+}\subset U_\Theta$
  such that\index{$V^+$  a
    connected component of $U_\Theta \cap \Omega_{\Theta}^{\mathrm{opp}}$ that
  is a semigroup (Theorem~\ref{theo:invar-unip-semigr-implies-pos-open-case})}
  \begin{enumerate}
  \item\label{item1:theo:invar-unip-semigr-implies-pos-open-case} $V^+$ is a
    connected component of $U_\Theta \cap \Omega_{\Theta}^{\mathrm{opp}}$.
  \item\label{item2:theo:invar-unip-semigr-implies-pos-open-case}  $V^+$ is a semigroup.
  \end{enumerate}
  Then $G$ admits a $\Theta$-positive structure; for every~$\alpha$
  in~$\Theta$, there is a unique $L_{\Theta}^{\circ}$-invariant closed nontrivial convex cone
  $c_\alpha\subset \mathfrak{u}_\alpha$ such that $\exp( c_\alpha)\subset
  \overline{V}^+$. Defining the positive unipotent semigroup with these cones (Definition~\ref{def:positive}), the semigroup~$V^+$ is
  equal to~$U_{\Theta}^{> 0}$.
\end{theorem}

As a corollary of these theorems, we have
\begin{corollary}
  \label{coro:semigroup-in-Bruhat-Theta}
  Let  $U^{+}$ be a closed
  $L_{\Theta}^{\circ}$-invariant semigroup in~$U_\Theta$. Assume that the
  interior of~$U^+$ is nonempty and contained in the open
  Bruhat cell with respect to $P_\Theta^{\mathrm{opp}}$, then the group~$G$
  has a $\Theta$-positive structure and (with the right choice of cones)
  $U^+ = U_{\Theta}^{\geq 0}$.
\end{corollary}

Section~\ref{sec:cones-associated-u+} derives a number of structural results
common to both theorems, then Section~\ref{sec:positive-structure} addresses
the proof of Theorem~\ref{theo:invar-unip-semigr-implies-pos} and
Section~\ref{sec:cones-associated-v+} addresses the proof of
Theorem~\ref{theo:invar-unip-semigr-implies-pos-open-case}. Section~\ref{sec:semigr-bigg-than}
gives examples of semigroups that satisfy the hypothesis of
Theorem~\ref{theo:invar-unip-semigr-implies-pos} and are not equal
to~$U_{\Theta}^{\geq 0}$.

\subsection{The cones associated with~$U^+$}
\label{sec:cones-associated-u+}

We work here with $G$, $\Theta$, and $U^+$ as in
Theorem~\ref{theo:invar-unip-semigr-implies-pos} but satisfying only the
hypothesis (\ref{item1:theo:invar-unip-semigr-implies-pos}),
(\ref{item2:theo:invar-unip-semigr-implies-pos}), and
(\ref{item3:theo:invar-unip-semigr-implies-pos}) of this theorem. We first
note that these three hypothesis are produced by
Theorem~\ref{theo:invar-unip-semigr-implies-pos-open-case}:

\begin{lemma}
  \label{lemma:them-open-case-implies123}
  Let $(G, \Theta, V^+)$ satisfying the hypothesis of
  Theorem~\ref{theo:invar-unip-semigr-implies-pos-open-case}. Then
  $U^+=\overline{V}{}^+$ is a closed $L_{\Theta}^{\circ}$-invariant semigroup
  of nonempty interior, i.e.\ it satisfies the hypothesis (\ref{item1:theo:invar-unip-semigr-implies-pos}),
(\ref{item2:theo:invar-unip-semigr-implies-pos}), and
(\ref{item3:theo:invar-unip-semigr-implies-pos}) of Theorem~\ref{theo:invar-unip-semigr-implies-pos}.
\end{lemma}
\begin{proof}
  In view of Lemma~\ref{lemm:semigroups}, it is enough to prove that the
  semigroup~$V^+$ is $L_{\Theta}^{\circ}$-invariant.   Indeed $U_\Theta$ and $\Omega_{\Theta}^{\mathrm{opp}}$ are invariant
  by~$L_{\Theta}^{\circ}$, and since $L_{\Theta}^{\circ}$ is connected,  every connected component of $U_\Theta \cap
  \Omega_{\Theta}^{\mathrm{opp}}$ is invariant by~$L_{\Theta}^{\circ}$.
\end{proof}

For all~$\alpha$ in~$\Theta$ we
introduce\index{$k_\alpha$ the trace of~$U^+$ on the factor
  $\mathfrak{u}_\alpha$ (Section~\ref{sec:cones-associated-u+})}
\[ k_\alpha = p_\alpha \bigl( \log(U^+)\bigr)\]
where, as above $p_\alpha\colon \mathfrak{u}_\Theta
\to \mathfrak{u}_\alpha$ is the $L_\Theta$-equivariant projection.

We will first show that
\begin{proposition}
  \label{prop:k-alpha-characterization}
  Let $X$ be an element of~$\mathfrak{u}_\alpha$. Then $X$~belongs
  to~$k_\alpha$ if and only if $\exp(X)$ belongs to~$U^+$.
\end{proposition}

\begin{proof}
  Let us define the following sequence $(A_n)_{n\in\N}$ in~$\mathfrak{a}$: for all
  $n\in \N$, $A_n$ is the element of~$\mathfrak{a}$ defined by the equalities
  \[ \alpha(A_n)=0, \gamma(A_n)=0 \ \forall \gamma \in \Delta\setminus\Theta
    \text{ and } \gamma(A_n)=-n \ \forall \gamma\in \Theta\setminus\{ \alpha\}. \]
  One has then, for all $X \in \mathfrak{u}_\alpha$
  \begin{align*}
    \ad(A_n) X &= 0\\
    \intertext{and for all $\beta $ in $P(\mathfrak{u}_\Theta,
    \mathfrak{t}_\Theta)\smallsetminus \{\alpha\}$ and for all~$Y$ in~$\mathfrak{u}_\beta$ }
    \ad(A_n) Y &= - s_\beta n Y
  \end{align*}
  for some $s_\beta>0$. Recall that $\mathfrak{u}_\beta$ is the sum of
  the~$\mathfrak{g}_\delta$ for $\alpha$ in $\Sigma \cap (\beta+ \Span( \Delta
  \smallsetminus \Theta))$ so that $\delta(A_n)=\beta(A_n)$, see Section~\ref{sec:acti-l-theta-on-u-theta}.
  
  The sequence $(g_n =\exp(A_n))_{n\in\N}$ belongs to
  $\exp(\mathfrak{a}) \subset L^{\circ}_{\Theta}$ and, for every~$X$
  in~$\mathfrak{u}_\Theta$, using the
    decomposition $X=\sum X_\beta$ according to the direct sum $\mathfrak{u}_\Theta =
    \bigoplus_\beta \mathfrak{u}_\beta$, one has, for all~$n$ in~$\N$,
  \begin{align*}
    \Ad(g_n) X
    &= \exp( \ad(A_n)) X_\alpha + \sum_{\beta\neq \alpha}
      \exp(\ad(A_n)) X_\beta\\
    &=  X_\alpha + \sum_{\beta\neq \alpha}
      e^{-s_\beta n} X_\beta.
  \end{align*}
  Thus the sequence $(\Ad(g_n)\cdot X)_{n\in \N}$ converges to $X_\alpha = p_\alpha(X)$.

    Let now $X_\alpha$ be in~$k_\alpha$. There exists then $X\in
  \mathfrak{u}_\Theta$ such that $\exp(X)$ belongs to~$U^+$ and $X_\alpha
    =p_\alpha(X)$. By the above, the sequence
    \[     g_n \exp(X) g_{n}^{-1} = \exp( \Ad(g_n) X)\]
  converges to $\exp(X_\alpha)$. Since $U^+$ is $L_{\Theta}^{\circ}$-invariant
  and closed, one has $\exp(X_\alpha)$ belongs to~$U^+$.

  Conversely, let~$X$ be in~$\mathfrak{u}_\alpha$ such that $\exp(X)$ belongs to~$U^+$, then
  \(p_\alpha\bigl( \log(
    \exp(X)) \bigr) = p_\alpha(X) =X\)
  belongs to~$k_\alpha$.
\end{proof}

From this we deduce:

\begin{corollary}
  \label{coro:k-alpha-cone}
  The set~$k_\alpha$ is a closed $L_{\Theta}^{\circ}$-invariant convex
  cone.
\end{corollary}
\begin{proof}
  By the previous proposition $k_\alpha = \log^{-1} \bigl( U^+ \cap
  \exp(\mathfrak{u}_\alpha)\bigr)$. Since $U^+ \cap \exp(
  \mathfrak{u}_\alpha)$ is closed and $L_{\Theta}^{\circ}$-invariant, this implies that $k_\alpha$ is closed and
  $L_{\Theta}^{\circ}$-invariant.

  Let $X\in k_\alpha$ and let $t>0$. Let $A\in \mathfrak{a}$ be defined
  by $\alpha(A)= \log t$ and $\gamma(A)=0$ for all $\gamma\in
  \Delta\setminus\{ \alpha\}$. Then
  \[ t X = \exp( \ad(A)) X = \Ad(\exp( A)) X\]
  belongs to~$k_\alpha$ since $\exp(A)$ belongs to $L^{\circ}_{\Theta}$ and
  $k_\alpha$~is $L^{\circ}_{\Theta}$-invariant. This means that $k_\alpha$ is a cone.

  Let $X$ and $Y$ be in $k_\alpha$. Then $\exp(X)
  \exp(Y)$ belongs to $U^+$. From the Baker--Campbell--Hausdorff
  formula, 
  \[ p_\alpha \circ \log\bigl( \exp(X)
    \exp(Y)\bigr) =X+Y\]
  belongs to~$k_\alpha$. This is the property that the cone~$k_\alpha$ is convex.
\end{proof}

The cones~$k_\alpha$ are nontrivial:
\begin{lemma}
  \label{lemm:k-alpha-nontrivial}
  For every~$\alpha$ in~$\Theta$, the cone $k_\alpha$ is not reduced
  to~$\{0\}$.
\end{lemma}
\begin{proof}
     Suppose that $k_\alpha =\{0\}$. Then $U^+$ would be contained in $\{ g \in
   U_\Theta \mid p_\alpha( \log(g))=0\}$ and would be of empty interior,
   contrary to the assumptions.
\end{proof}
\subsection{The positive structure}
\label{sec:positive-structure}

We now turn to the proof of Theorem~\ref{theo:invar-unip-semigr-implies-pos},
i.e.\ we furthermore assume
property~(\ref{item4:theo:invar-unip-semigr-implies-pos}).

The last point to make the relation with
Definition~\ref{defi:theta-pos-structure} is that the cones~$k_\alpha$ are
acute:

\begin{lemma}
  \label{lem:k-alpha-acute}
  For every~$\alpha$ in~$\Theta$, the cone $k_\alpha$  contains no line.
\end{lemma}
\begin{proof}
   Let~$X$ in~$\mathfrak{u}_\alpha$ be such that $X$ and $-X$ belong
   to~$k_\alpha$. Then $g=\exp(X)$ belongs to~$U^+$ as well as
   $g^{-1}=\exp(-X)$. Thus $g=e$ (since by assumption the semigroup~$U^+$ is sharp) and $X=\log(g)=0$. This means that the convex
   cone~$k_\alpha$ contains no line.
\end{proof}

Thus $G$ has a $\Theta$-positive structure and one can introduce the semigroup
$U_{\Theta}^{\geq 0}$ with the following choice of
invariant cones: for all~$\alpha$ in~$\Theta$ set  $c_\alpha=k_\alpha$.
Proposition~\ref{prop:k-alpha-characterization} and the semigroup property
 obviously imply the inclusion $U_{\Theta}^{\geq 0}\subset
U^+$. This concludes the proof of
Theorem~\ref{theo:invar-unip-semigr-implies-pos}.

\subsection{The cones associated with~$V^+$ and the induced positive structure}
\label{sec:cones-associated-v+}

We now turn to the proof of
Theorem~\ref{theo:invar-unip-semigr-implies-pos-open-case}.
Applying the results of Section~\ref{sec:cones-associated-u+} to $U^+
=\overline{V}{}^+$ (cf.\ Lemma~\ref{lemma:them-open-case-implies123}), we get
that,  for every~$\alpha$
in~$\Theta$
\[ k_\alpha = p_\alpha \bigl( \log(\overline{V}{}^+)\bigr)\]
is a closed convex, nonzero, $L_{\Theta}^{\circ}$-invariant cone and that, for
every~$X$ in~$\mathfrak{u}_\alpha$, $\exp(X)$~belongs to~$\overline{V}{}^+$ if
and only if $X$~belongs to~$k_\alpha$.

Applying Lemma~\ref{lemm:semigroups}, we get
$ V^+ \overline{V}{}^+ = V^+, \text{ and } \overline{V}{}^+ V^+ = V^+$ and
that $V^+$~is the interior of its closure~$\overline{V}{}^+$.

We also have
\begin{lemma}
  \label{lem:k_alpha-not-u_alpha}
  For all~$\alpha$ in~$\Theta$, $k_\alpha$ is an acute convex cone
  in~$\mathfrak{u}_\alpha$.
\end{lemma}
\begin{proof}
  Since $k_\alpha$ is a closed convex cone and is
  $L_{\Theta}^{\circ}$-invariant, the subspace of maximal dimension contained
  in~$k_\alpha$ is $L_{\Theta}^{\circ}$-invariant. Since the action
  of~$L_{\Theta}^{\circ}$ on $\mathfrak{u}_\alpha$ is irreducible (Theorem~\ref{theo:kostant-results}), this subspace is either~$\{0\}$
  or~$\mathfrak{u}_\alpha$. Hence we have to exclude the case when
  $k_\alpha=\mathfrak{u}_\alpha$.

  Suppose, by contradiction, that $k_\alpha=\mathfrak{u}_\alpha$. We then
  have   $\exp{
    \mathfrak{u}_\alpha} \subset \overline{V}{}^+$.
  Recall
  that the Lie algebra~$\mathfrak{u}_\Theta$ admits the following
  decomposition\index{$\mathfrak{u}_{\hat{\alpha}}$ the complement of
    $\mathfrak{u}_\alpha$ in $\mathfrak{u}_\Theta$}
  \begin{align*}
    u_\Theta = \mathfrak{u}_\alpha
    & \oplus \mathfrak{u}_{\hat{\alpha}}, \text{ where }
      \mathfrak{u}_{\hat{\alpha}} =\bigoplus_{\mathclap{\beta\in P(\mathfrak{u}_\Theta, \mathfrak{t}_\Theta) \smallsetminus \{\alpha\}}} \mathfrak{u}_\beta\\
    \intertext{and that the map}
    \mathfrak{u}_\alpha \times \mathfrak{u}_{\hat{\alpha}}
             & \longrightarrow U_\Theta\\
    (X,Y) & \longmapsto \exp(X)\exp(Y)
  \end{align*}
  is a diffeomorphism.

  Let now $x$~be an element of~$V^+$. There exists thus $(X,Y)$ in $\mathfrak{u}_\alpha\times \mathfrak{u}_{\hat{\alpha}}$ such that
  $x=\exp(X) \exp(Y)$. The equality $k_\alpha=\mathfrak{u}_\alpha$ implies
  that $\exp(-X)$ belongs to $\overline{V}{}^+$ and   $\exp(Y)= \exp(-X) x$ belongs to $\overline{V}{}^+ V^+ =V^+$. But this element
  $\exp(Y)$ does not belong to~$\Omega_{\Theta}^{\mathrm{opp}}$, in
  contradiction with the fact that $V^+$ is contained in
  $\Omega_{\Theta}^{\mathrm{opp}}$.
\end{proof}

As a consequence the group~$G$ has a $\Theta$-positive structure and we can
choose $c_\alpha=k_\alpha$ in the construction of $U_{\Theta}^{>0}$. The last
point is to notice the equality $V^+= U_{\Theta}^{>0}$.
As in Section~\ref{sec:positive-structure} we have $U_{\Theta}^{\geq 0}\subset
\overline{V}{}^+$ so that $U_{\Theta}^{> 0}\subset
\overline{V}{}^+$ and  $U_{\Theta}^{> 0}\subset {V}{}^+$ by
Lemma~\ref{lemm:semigroups}. Since these $2$~open sets are connected
components of $U_\Theta \cap \Omega_{\Theta}^{\mathrm{opp}}$, they are indeed equal.

\subsection{Semigroups bigger than \texorpdfstring{$U_{\Theta}^{\geq 0}$}{UΘ≥0}}
\label{sec:semigr-bigg-than}

We end this section with examples of closed semigroups contained in the unipotent
standard subgroup of $\SL_3( \R )$ that are bigger than the semigroup $U^{\geq 0}$ of
totally nonnegative unipotent matrices.

For each $r\in \R$ denote by $U_r$ the set of matrices\index{$U_r$ ($r\in
  [1,+\infty[$) a unipotent semigroup in $\SL_3(\R)$ }
\[
  \begin{pmatrix}
    1 & a & c \\
    0 & 1 & b \\
    0 & 0 & 1
  \end{pmatrix},\text{ with }a,b \geq 0\text{ and }0\leq c\leq r ab.
\]

\begin{figure}
  \centering
  \begin{tikzpicture}[x= {(1.3cm,-0.153cm)}, z={(0cm,1cm)}, y={(1cm,0.3cm)}]
    \draw[->] (0,0,0) -- (4,0,0);
    \draw[->] (0,0,0) -- (0,4,0);
    \draw[->] (0,0,0) -- (0,0,4);
    \draw (4,0,0) node[right]{$a$};
    \draw (0,4,0) node[right]{$b$};
    \draw (0,0,4) node[above]{$c$};
    \foreach \t in {0,0.2,...,3} { \draw
      [domain=0:1,smooth,thin] plot ({\x*\t}, {(1-\x)*\t},
      {\x*(1-\x)*\t*\t}); };
    \foreach \t in {0.9,0.5} {
    \draw [domain=0:1,smooth] plot ({\x*\t*3}, {\x*(1-\t)*3},
    {\x*\x*\t*(1-\t)*9});
    };
  \end{tikzpicture}
  \caption{The semigroup $U^{\geq 0}=U_1$ in the matrix coordinates is
    contained in the corner $a\geq 0$, $b\geq 0$, $c\geq 0$ and bounded by
    the surface $c=ab$ on which a few curves have been drawn.}
  \label{fig:semigroup}
\end{figure}
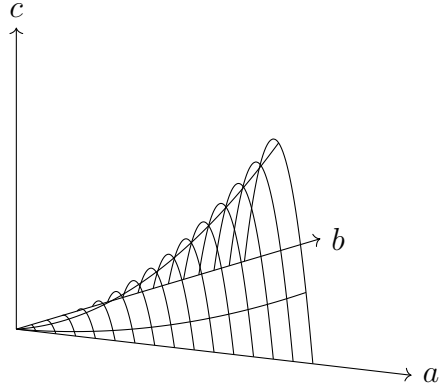

\begin{remark}
  The semigroup $U^{\geq 0}$ is~$U_1$. The matrix with $a=b=c=1$ is not in
  $\Omega_{\Delta}^\mathrm{opp}$ and belongs to the interior of~$U_r$ as soon
  as $r>1$; therefore, for $r>1$, $\mathring{U}_r$ is not contained
  in~$\Omega_{\Theta}^{\mathrm{opp}}$.
\end{remark}

\begin{lemma}
  For all $r\geq 1$, $U_r$ is a closed sharp semigroup     invariant by conjugation by diagonal matrices with positive
  coefficients. 
\end{lemma}
\begin{proof}
  Closedness and invariance are easily checked.
  To check the semigroup property, we calculate the product of two elements
  of~$U_r$:
  \begin{equation*}
    \begin{pmatrix}
      1 & a & c \\
      0 & 1 & b \\
      0 & 0 & 1
    \end{pmatrix}
    \begin{pmatrix}
      1 & a' & c' \\
      0 & 1 & b' \\
      0 & 0 & 1
    \end{pmatrix} =
    \begin{pmatrix}
      1 & a'' & c'' \\
      0 & 1 & b'' \\
      0 & 0 & 1
    \end{pmatrix}
  \end{equation*}
  then $a''=a+a'\geq 0$, $b''=b+b'\geq 0$ and $c'' = c+c'+ab'\geq 0$, thus
  \begin{equation*}
    r a'' b'' - c'' = (rab-c) + (ra'b'-c') + ra' b + (r-1)ab' \geq 0,    
  \end{equation*}
  and the product belongs to~$U_r$. The above formula shows that if an element
  \[
    \begin{pmatrix}
      1 & a & c \\
      0 & 1 & b \\
      0 & 0 & 1
    \end{pmatrix}
  \]
  is invertible in~$U_r$ then necessarily $a=b=0$ and for such an element the
  condition of being in~$U_r$ says $0\leq c\leq 0$, i.e.\ $c=0$. This proves
  that $U_r$~is sharp.
\end{proof}

\section{Positivity in the flag variety \texorpdfstring{$\mathsf{F}_\Theta$}{F_Θ}}
\label{sec:posit-flag-vari}
In this section we first consider $G$ having a $\Theta$-positive structure and
use the positive unipotent semigroups $U_\Theta^{>0}$ to define positive
$n$-tuples in the flag variety $\mathsf{F}_\Theta$.  The key notion is the one
of positive triples of flags, which can be defined using
$U_\Theta^{>0}$ or $U_\Theta^{\mathrm{opp}, >0}$. To express the properties of
positive triples, we use the language of ``diamonds''. In the simplest flag variety~$\RR\PP^1$ diamonds are intervals. The notion of diamond was coined in \cite{Labourie_Toulisse} in the case when $G$ is $\SO(2,n)$, in which case diamonds actually look like diamonds. 
In the general case, diamonds have a slightly more complicated geometry, but can still thought of as generalized intervals. 
The language of diamonds, which is also used in \cite{GLW},  
captures well all
the necessary geometric properties of positive triples and quadruples needed in
applications.

 We also prove the reverse direction: a family of diamonds gives
rise to a $\Theta$-positive structure and is produced by the
$\Theta$-positivity. This thus gives a geometric characterization of $\Theta$-positivity.

In the case of split real groups and Lusztig's total positivity, the positive
structure on the flag variety $\mathsf{F}_\Delta$ has been described in
\cite{Lusztig} and the notion and key properties of positive $n$-tuples have been established
by Fock and Goncharov in \cite{Fock_Goncharov}.

\subsection{A first diamond}
\label{sec:first-diamond}

We will denote by $\mathcal{O} \subset \mathsf{F}_\Theta$ the set of points
transverse to $\mathfrak{p}_\Theta$\index{$\mathcal{O}$ the set of points in $  \mathsf{F}_\Theta$
transverse to $\mathfrak{p}_\Theta$} and by $\mathcal{O}^{\mathrm{opp}}$ the set of points\index{$\mathcal{O}^{\mathrm{opp}}$ the set of points in $  \mathsf{F}_\Theta$
transverse to $\mathfrak{p}_{\Theta}^{\mathrm{opp}}$}
transverse to $\mathfrak{p}_{\Theta}^{\mathrm{opp}}$. These sets are open and
diffeomorphic to the unipotent Lie groups $U_\Theta$ and
$U_{\Theta}^{\mathrm{opp}}$; indeed the maps $U_\Theta \to \mathcal{O} \mid u
\mapsto u\cdot \mathfrak{p}_{\Theta}^{\mathrm{opp}}$ and  $U_{\Theta}^{\mathrm{opp}} \to \mathcal{O}^{\mathrm{opp}} \mid u
\mapsto u\cdot \mathfrak{p}_{\Theta}$ are diffeomorphisms.
Since $w_\Delta \cdot \mathfrak{p}_{\Theta} =
\mathfrak{p}_{\Theta}^{\mathrm{opp}}$ and $w_\Delta \cdot \mathfrak{p}_{\Theta}^{\mathrm{opp}} =
\mathfrak{p}_{\Theta}$, one has also $\mathcal{O} = \Omega_{\Theta} \cdot
\mathfrak{p}_{\Theta}$ and $\mathcal{O}^{\mathrm{opp}} =
\Omega_{\Theta}^{\mathrm{opp}} \cdot \mathfrak{p}_{\Theta}^{\mathrm{opp}}$,
where $\Omega_\Theta = P_\Theta w_\Delta P_\Theta$ and $\Omega_{\Theta}^{\mathrm{opp}} = P_{\Theta}^{\mathrm{opp}} w_\Delta P_{\Theta}^{\mathrm{opp}}$. 
A diamond will be a connected component of $\mathcal{O} \cap \mathcal{O}^{\mathrm{opp}}$, see Section~\ref{sec:axiomatic-diamonds}. 

We first show that $U_\Theta^{>0}$ or $U_\Theta^{\mathrm{opp}, >0}$  gives rise to the same connected component, hence the same diamond. 

\begin{proposition}
  \label{prop:first-diamond}
  One has the equality, in $\mathsf{F}_\Theta$,
  \[U_{\Theta}^{>0}\cdot
    \mathfrak{p}_{\Theta}^{\mathrm{opp}} = U_{\Theta}^{\mathrm{opp},>0}\cdot
    \mathfrak{p}_{\Theta}.\]
  More precisely these sets are equal to the same
  connected component of $\mathcal{O} \cap \mathcal{O}^{\mathrm{opp}}$.
\end{proposition}

\begin{proof}
By Corollary~\ref{cor:semigroupUtheta}.(\ref{item4:cor:semigroupUtheta}) we have  that $U_{\Theta}^{>0}\cdot
  \mathfrak{p}_{\Theta}^{\mathrm{opp}}$ is a connected component of
  $\mathcal{O} \cap \mathcal{O}^{\mathrm{opp}}$. Equally, $U_{\Theta}^{\mathrm{opp},>0}\cdot
  \mathfrak{p}_{\Theta}$ is a connected component of $\mathcal{O} \cap
  \mathcal{O}^{\mathrm{opp}}$. Hence we only need to see that these sets
intersect. This can be obtained along the same lines than
Theorem~\ref{thm:nilpotent} using the equality  $\exp(E)\cdot
  \mathfrak{p}_{\Theta}^{\mathrm{opp}} = \exp(F)\cdot \mathfrak{p}_\Theta $
  (with $(E,F,D)$ the $\Theta$-principal $\mathfrak{sl}_2$-triple, Section~\ref{sec:sl2}).
\end{proof}

\subsection{Axiomatic of diamonds}
\label{sec:axiomatic-diamonds}

In this section and the next one, we do not assume that $G$~has a
$\Theta$-positive structure. We will prove however that the definition below
forces the presence of a $\Theta$-positive structure.

We introduce now the expected properties for positive
triples 
in $\mathsf{F}_\Theta$. It will be a little easier to express these properties in terms of
``diamonds'' that are, fixing $a$ and~$b$ in~$\mathsf{F}_\Theta$, the
connected components of the set $\{ x \in \mathsf{F}_\Theta \mid (a,x,b)
\text{ is a positive triple}\}$ (cf.\ below Definition~\ref{defi:posit-tripl}).

Let $\Theta \subset \Delta$ be a subset invariant by the opposition involution
$\iota\colon \alpha \mapsto -w_\Delta\cdot \alpha$ (so that $\mathsf{F}_{\iota(\Theta)}
= \mathsf{F}_\Theta$ and we can speak of transverse pairs
in~$\mathsf{F}_\Theta$, cf.\ Section~\ref{sec:flag-variety}).
For a point~$a$ in $\mathsf{F}_\Theta$, let us denote by $\mathcal{O}_a$ the
(open) subset of $\mathsf{F}_\Theta$ whose points are those transverse
to~$a$ and by~$P_a$ its stabilizer in~$G$.\index{$\mathcal{O}_a$ the points in 
 $\mathsf{F}_\Theta$  transverse to~$a$}

\begin{definition}
  \label{defi:axiomatic-diamonds}
  A \emph{family of diamonds in}~$\mathsf{F}_\Theta$ is a
  family~$\mathcal{F}$\index{$\mathcal{F}$ a family of diamonds}
  of triples  $(D, a,b)$\index{$(D, a,b)$ a diamond with extremities $a$ and $b$} where $D$ is a subset of~$\mathsf{F}_\Theta$ and $a$, $b$ belong
  to $\mathsf{F}_\Theta$ (and will be called the \emph{extremities} of the diamond) such that
  \begin{enumerate}[leftmargin=*]
  \item\label{item:1:defi:axiomatic-diamonds} For every $(D,a,b)$ in~$\mathcal{F}$, $a$ is transverse to~$b$ and $D$
    is a connected component of $\mathcal{O}_a \cap \mathcal{O}_b$;
  \item\label{item:2:defi:axiomatic-diamonds} For every $(D,a,b)$ in~$\mathcal{F}$,  $(D,b,a)$ belongs
    to~$\mathcal{F}$;
  \item\label{item:3:defi:axiomatic-diamonds} For every $(D,a,b)$
    in~$\mathcal{F}$, and for every~$g$ in $\Aut_1(
    \mathfrak{g})$, $( g\cdot D, g\cdot a, g\cdot b)$ belongs
    to~$\mathcal{F}$;
  \item\label{item:4:defi:axiomatic-diamonds}  For every $(D,a,b)$
    in~$\mathcal{F}$, and for every~$x$ in~$D$, there exists a unique diamond
    $(D', a, x)$ in~$\mathcal{F}$ such that $D'$~is contained in~$D$.
  \end{enumerate}
 The symmetry condition~(\ref{item:2:defi:axiomatic-diamonds})
together with the last condition imply also
\begin{enumerate}[leftmargin=*,resume]
\item \label{item:5:defi:axiomatic-diamonds} For every $(D,a,b)$
    in~$\mathcal{F}$, and for every~$x$ in~$D$, there exists a unique diamond
    $(D', x, b)$ in~$\mathcal{F}$ such that $D'$~is contained in~$D$.
\end{enumerate}
\end{definition}

\begin{remark}
  It is natural to ask that the notion of diamond (or the notion of positive
  triple to come later) is invariant under all the automorphisms of the flag
  variety~$\mathsf{F}_\Theta$. This is why we require invariance under the
  group $\Aut_1( \mathfrak{g})$ in the definition and not only under~$G$.
\end{remark}

\pgfmathsetmacro{\ra}{0.5}
\begin{figure}
  \centering
  \begin{tikzpicture}[scale=0.5]
    \fill[color=gray!20] (0,0) -- (4,4) -- (7,1) -- (3,-3);
    \fill[color=gray!40] (0,0) -- (2.5,2.5) -- (4,1) -- (1.5,-1.5);
    \draw ({0-\ra}, {0-\ra}) -- ({4+\ra}, {4+\ra});
    \draw ({0-\ra}, {0+\ra}) -- ({3+\ra}, {-3-\ra});
    \draw ({3-\ra}, {-3-\ra}) -- ({7+\ra}, {1+\ra});
    \draw ({4-\ra}, {4+\ra}) -- ({7+\ra}, {1-\ra});
    \draw ({2.5-\ra}, {2.5+\ra}) -- ({4+\ra}, {1-\ra});
    \draw ({1.5-\ra}, {-1.5-\ra}) -- ({4+\ra}, {1+\ra});
    \draw (7,1) node[right]{$b$};
    \draw (4,1) node[right]{$c$};
    \draw (0,0) node[left]{$a$};
  \end{tikzpicture}
  \caption{The diamonds in the case of the group $\SO(2,2)$. The point~$c$ is
    contained in the diamond with extremities~$a$ and~$b$ (the light gray
    region); the gray region is the diamond with extremities~$a$ and~$c$
    contained in the first diamond. In the general case, the diamonds are curved as shown in Figure~\ref{fig:semigroup},  which pictures the tip of a diamond in the case of the group $\SL_3(\R)$.}
  \label{fig:diamond}
\end{figure}
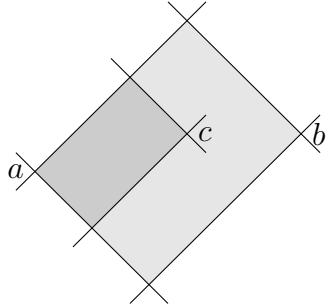

We can immediately note that:
\begin{itemize}[leftmargin=*]
\item For all~$a$ and~$b$ in~$\mathsf{F}_\Theta$, with $a$ transverse to~$b$,
  there is a least one diamond with extremities~$a$ and~$b$ (since the
  group $\Aut_1( \mathfrak{g})$ acts transitively on the space of
  pairs of  transverse points), and there are finitely many such diamonds
  (since $\mathcal{O}_a \cap \mathcal{O}_b$ has finitely many components);
\item For every diamond $(D,a,b)$ the set~$D$ is invariant under the 
  Lie group $L_{a,b}^{\circ}$ that is the neutral component of $L_{a,b} =
  P_\alpha \cap P_b$.
\end{itemize}

\subsection{From diamonds to \texorpdfstring{$\Theta$}{Θ}-positivity}
\label{sec:diamonds-semigroups}

We assume in this subsection that there is a family of diamonds~$\mathcal{F}$ in~$\mathsf{F}_\Theta$.
We immediately note that this family is produced from a semigroup
in~$U_\Theta$.

\begin{proposition}
  \label{prop:from-diamonds-to-semigroups}
  Let $a$ and $b$ be respectively the elements $\mathfrak{p}_\Theta$ and
  $\mathfrak{p}_{\Theta}^{\mathrm{opp}}$ of~$\mathsf{F}_\Theta$. Let $(D,a,b)$
  be a diamond in~$\mathcal{F}$ with extremities~$a$ and~$b$. Then the set
  \[ V^+ \coloneqq \{ v\in U_\Theta \mid v\cdot b \in D\}\]
  is a connected component of $U_\Theta \cap \Omega_{\Theta}^{\mathrm{opp}}$
  and is a semigroup.
\end{proposition}
\begin{proof}
  The fact that $V^+$ is a connected component of $U_\Theta \cap
  \Omega_{\Theta}^{\mathrm{opp}}$ follows directly from
  Condition~(\ref{item:1:defi:axiomatic-diamonds}) in Definition~\ref{defi:axiomatic-diamonds}.

  Proving that $V^+$ is a semigroup amounts to show that, for all~$v$
  in~$V^+$, $v V^+ \subset V^+$. This inclusion can be phrased in terms of
  diamonds: let $x=v\cdot b$ so that $x$~belongs to~$D$ and there is a unique
  diamond $(D',a,x)$ with extremities~$a$ and~$x$ with $D'\subset D$; with
  this notation, one wants to prove the equality $D' = v \cdot D$.

  Let $\ell_s = \exp( s X)$ ($s\in \R$) be the $1$-parameter subgroup associated with the element~$X$ of~$\mathfrak{a}$ such that
  $\alpha(X) = -1$ for all~$\alpha$ in~$\Delta$. Define $(v_t)_{t\in [0,1]}$ and $(x_t)_{t\in [0,1]}$
  by the equalities $v_0=e$,
  $v_t = \ell_{\log t} v \ell_{\log t}^{-1}$ for $t>0$, and $x_t = v_t \cdot b$. By the choice of~$X$,
  these paths are continuous and $x_0=b$, $x_1=x$.

  For all positive~$t$, $D^{\prime}_{t} = \ell_{\log t}\cdot D'$ is the
  diamond with extermities~$a$ and~$x_t$ contained in $\ell_{\log t}\cdot D=D$; i.e.\
  $D^{\prime}_{t}$ is the connected component of $\mathcal{O}_a \cap
  \mathcal{O}_{x_t}$ contained in~$D$. From the convergence of $x_t$ to~$b$,
  we deduce that $D'_t$ converges to~$D$ as $t$~tends to~$0$. Stated
  differently, setting $D'_0=D$, the family $(D'_t)_{t\in [0,1]}$ is
  continuous. This implies that the family $( v_{t}^{-1} \cdot D'_t)_{t\in
    [0,1]}$ of diamonds with extremities~$a$ and~$b$ is continuous. Since
  there are finitely many connected components in~$\mathcal{O}_a \cap
  \mathcal{O}_b$, we deduce that this family is constant equal to~$D$. In
  particular $D= v_{1}^{-1}\cdot D'_1 = v^{-1}\cdot D'$ which is the sought
  for equality.
\end{proof}

As a consequence of Proposition~\ref{prop:from-diamonds-to-semigroups}, we can apply Theorem~\ref{theo:invar-unip-semigr-implies-pos-open-case}
 and thus $G$~admits a $\Theta$-positive structure. Let us then fix,
for every~$\alpha$ in~$\Theta$, a nonzero acute $L_{\Theta}^{\circ}$-invariant
closed convex cone~$c_\alpha$ in~$\mathfrak{u}_\alpha$. We already noticed
that there are exactly two such cones, namely $c_\alpha$ and $-c_\alpha$. To
take into account all the possible choices of cones, let us introduce, for
every $\varepsilonb = ( \varepsilon_\alpha)_{\alpha \in \Theta}$
in $\{ \pm 1\}^\Theta$ the nonnegative unipotent
semigroup\index{$U_{\Theta}^{\varepsilonb\geq 0}$ the nonnegative unipotent semigroup
 generated by $\bigcup_{\alpha\in \Theta}\exp(
\varepsilon_\alpha c_\alpha)$}
$U_{\Theta}^{\varepsilonb\geq 0}$ generated by $\bigcup_{\alpha\in\Theta}\exp(
\varepsilon_\alpha c_\alpha)$ and the
positive unipotent semigroup\index{$U_{\Theta}^{\varepsilonb> 0}$ 
the interior of
$U_{\Theta}^{\varepsilonb\geq 0}$} $U_{\Theta}^{\varepsilonb> 0}$ that is
the interior (relative to~$U_\Theta$) of
$U_{\Theta}^{\varepsilonb\geq 0}$.

 Theorem~\ref{theo:invar-unip-semigr-implies-pos-open-case} 
gives the following corollary
\begin{corollary}
  \label{coro:all-diamonds-from-semigroup}
  The diamonds in~$\mathcal{F}$ with extremities $a=\mathfrak{p}_\Theta$
  and $b=\mathfrak{p}_{\Theta}^{\mathrm{opp}}$ are exactly
  the triples $( U_{\Theta}^{\varepsilonb> 0} \cdot b, a,b)$
  for~$\varepsilonb$ varying in~$\{ \pm 1\}^\Theta$.
\end{corollary}
\begin{proof}
  Indeed Proposition~\ref{prop:from-diamonds-to-semigroups} and
  Theorem~\ref{theo:invar-unip-semigr-implies-pos-open-case} prove that every
  diamond with extremities~$a$ and~$b$ has this form. The fact that all
  signs~$\varepsilonb$ are achieved is a consequence of the
  invariance under $\Aut_1( \mathfrak{g})$ and that there are elements in this group
  exchanging the cones (Proposition~\ref{prop:homog-under-aut}).  
\end{proof}

In fact the transitivity observed in Proposition~\ref{prop:homog-under-aut}
implies that there is exactly one orbit of diamonds under the action of $\Aut_1(
\mathfrak{g})$.

\begin{corollary}
  \label{coro:one-orbit-of-diamond}
  Let $a$ and $b$ be respectively
  $\mathfrak{p}_\Theta$ and
  $\mathfrak{p}_{\Theta}^{\mathrm{opp}}$. 
  Then   the
  diamonds are exactly the $(g\cdot U_{\Theta}^{> 0} \cdot b, g\cdot a, g\cdot b)$ for~$g$ varying
  in~$\Aut_1( \mathfrak{g})$.
\end{corollary}

\subsection{From \texorpdfstring{$\Theta$}{Θ}-postivity to diamonds}
\label{sec:from-theta-postivity}

We now consider the reverse direction of 
Corollary~\ref{coro:one-orbit-of-diamond}. Namely, we assume that $G$~has a
$\Theta$-positive structure and consider the family~$\mathcal{F}$ consisting 
of the $\Aut_1( \mathfrak{g})$-orbit of $( U_{\Theta}^{> 0} \cdot b,
a,b)$ where $a$ and $b$ are  respectively $\mathfrak{p}_\Theta$ and
$\mathfrak{p}_{\Theta}^{\mathrm{opp}}$.

Equivalently, the family~$\mathcal{F}$ can be defined by taking the union of
the  $G$-orbits
of  $( U_{\Theta}^{\varepsilonb>0}\cdot b, a,b)$
for~$\varepsilonb$ varying in~$\{\pm 1\}^\Theta$.
\begin{proposition}
  \label{prop:from-theta-postivity}
  The above family~$\mathcal{F}$ is a family of diamonds in~$\mathsf{F}_\Theta$.
\end{proposition}
\begin{proof}
  We need to check the conditions of
  Definition~\ref{defi:axiomatic-diamonds}. Condition~(\ref{item:1:defi:axiomatic-diamonds})
  follows from the corresponding property of~$U_{\Theta}^{> 0}$
  (see Proposition~\ref{prop:first-diamond}). Condition~(\ref{item:2:defi:axiomatic-diamonds})
  is a consequence of the equality $w_\Delta U_{\Theta}^{>0} w_{\Delta}^{-1} =
  U_{\Theta}^{\mathrm{opp}, >0}$ and of
  Proposition~\ref{prop:first-diamond}. Condition~(\ref{item:3:defi:axiomatic-diamonds})
  is there by construction. The existence part in
  Condition~(\ref{item:4:defi:axiomatic-diamonds}) follows from the
  fact that $U_{\Theta}^{> 0}$ is a semigroup: indeed let $(D,a,b)$ be
  in~$\mathcal{F}$ and
let~$x$ be in~$D$, by equivariance we can assume that $a=\mathfrak{p}_\Theta$, 
$b=\mathfrak{p}_{\Theta}^{\mathrm{opp}}$ and $D=U_{\Theta}^{> 0} \cdot b$; let~$u$ be the element of~$U_{\Theta}^{> 0}$ such that
$x=u\cdot b$ and set $D'=uU_{\Theta}^{> 0}$, since $(D', a, x) = u\cdot
(D,a,b) $, $D'$ is a diamond and the inclusion $D'\subset D$
follows from the fact that $U_{\Theta}^{> 0}$ is a semigroup.

  Let us now address the uniqueness in
  Condition~(\ref{item:4:defi:axiomatic-diamonds}). With the notation
  introduced in this condition, we can (by $\Aut_1( \mathfrak{g})$-invariance)
  assume that $a=\mathfrak{p}_\Theta$, $b=
  \mathfrak{p}_{\Theta}^{\mathrm{opp}}$, and $D= U_{\Theta}^{> 0}\cdot
  b$. Let~$u$ be the element of~$U_{\Theta}^{>0}$ such that $x=u\cdot b$ and
  let $(D',a,x)$ be a diamond in~$\mathcal{F}$ such that $D'\subset D$. Then
  $(u^{-1} D', a,b)$ is a diamond with extremities~$a$ and~$b$.

  There is thus an element~$\ell$ in $\Aut_1( \mathfrak{g})$ fixing~$a$ and~$b$
  and such that $u^{-1} \cdot D' = \ell \cdot D$. The element~$\ell$
  stabilizes all the spaces~$\mathfrak{u}_\alpha$ (for~$\alpha$ in~$\Theta$)
  and thus sends, for every~$\alpha$ in~$\Theta$, the cone~$c_\alpha$ to
  $\varepsilon_\alpha c_\alpha$ for some~$\varepsilon_\alpha$ in~$\{\pm
  1\}$. Setting $\varepsilonb =(\varepsilon_\alpha)_{\alpha\in
    \Theta}$ one has thus $u^{-1}\cdot D' =
  U_{\Theta}^{\varepsilonb>0} \cdot b$ and the inclusion $D'\subset
  D$ can be rewritten as $u U_{\Theta}^{\varepsilonb>0} \subset
  U_{\Theta}^{>0}$.

  The sought for uniqueness is now equivalently expressed in the
  equalities $\varepsilon_\alpha=1$ for all~$\alpha$ in~$\Theta$. Suppose by
  contradiction that there exists~$\alpha$ with $\varepsilon_\alpha=-1$. Then,
  using for example the parametrization of the semigroups, one can construct
  an element~$v$ in $U_{\Theta}^{\varepsilonb>0}$ such that
  $p_\alpha \bigl( \log(uv)\bigr)= p_\alpha(\log u) + p_\alpha(\log v) =0$ (where again $p_\alpha\colon
  \mathfrak{u}_\Theta \to \mathfrak{u}_\alpha$ is the projection on the
  factor~$\mathfrak{u}_\alpha$); this is incompatible with the fact that
  $uv$~belongs to~$U_{\Theta}^{>0}$ since we should have that $p_\alpha
  \bigl( \log(uv)\bigr)$ belongs to~$\mathring{c}_\alpha$.
\end{proof}

 Observe that, for
$\varepsilonb = (-1)_{\alpha\in \Theta}$ (i.e., for all~$\alpha$ in~$\Theta$,
$\varepsilon_\alpha=-1$), one has $U_{\Theta}^{\varepsilonb>0}
=(U_{\Theta}^{>0})^{-1}$.
In particular  $(
(U_{\Theta}^{>0})^{-1}\cdot b, a,b)$ is also in~$\mathcal{F}$.

\subsection{$U$-pinnings}
\label{sec:u-pinnings}

We explain here the relation between the diamonds introduced here and the
definition that is used in \cite[Definition~2.4]{GLW}.

Denote, for every~$a$ in~$\mathsf{F}_\Theta$ \index{$P_a$ the
  parabolic subgroup corresponding to~$a$ in~$\mathsf{F}_\Theta$}  by~$U_a$ the unipotent radical of~$P_a$\index{$U_a$ the unipotent radical of~$P_a$} (so that $U_a= U_\Theta$
when $a$ is $\mathfrak{p}_\Theta$). The group~$U_a$ is completely determined
by its Lie algebra and will be also identified with a subgroup of $\Aut_1(
\mathfrak{g})$ (namely the unipotent radical of the stabilizer of~$a$ in
$\Aut_1( \mathfrak{g})$). We will call a \emph{$U$-pinning} of~$U_a$
any homomorphism\index{$s_a$ a $U$-pinning} $s_a\colon U_\Theta \to U_a$ induced by the map $x\mapsto
gxg^{-1}$ where $g$~is an element of $\Aut_1( \mathfrak{g})$ such that $g\cdot
\mathfrak{p}_\Theta =a$.

\begin{lemma}
  \label{lemma:u-pinnings-and-diamonds}
  Let~$a$ and~$b$ be two transverse points of~$\mathsf{F}_\Theta$. Then the
  diamonds with extremities~$a$ and~$b$ are exactly the triples
  \[ ( s_a( U_{\Theta}^{>0})\cdot b, a, b)\]
  where $s_a$ runs through the $U$-pinnings of~$U_a$.
\end{lemma}

\begin{proof}
  Indeed every such triple   is a diamond and the
  family defined is invariant by the action of~$\Aut_1(\mathfrak{g})$. Hence the
  result holds by the transitivity observed in Corollary~\ref{coro:one-orbit-of-diamond}.
\end{proof}

\subsection{Opposite diamond}
\label{sec:opposite-diamond}

The notion of $U$-pinning allows us to introduce the notion of the opposite
of a diamond. Let $(D,a,b)$ be a diamond ant let $s_a\colon U_\Theta \to U_a$
be a $U$-pinning such that $D=s_a( U_{\Theta}^{>0})\cdot b$ then the triple
$(s_a( U_{\Theta}^{>0})^{-1} \cdot b,a,b)$ is a diamond that is called
\emph{opposite} to $(D,a,b)$.

\begin{lemma}
  \label{lem:opposite-diamond}
  There is a unique diamond $(D',a,b)$ opposite to $(D,a,b)$. 
\end{lemma}
\begin{proof}
Let $(D',a,b)$ be a diamond opposite to $(D,a,b)$ defined by a $U$-pinning~$s_a$. We 
 set $V= \{ u\in U_a \mid u\cdot b \in D\}$ and $V'= \{ u\in U_a \mid u\cdot
  b \in D'\}$. Then $V$ and $V'$ uniquely determine $D$ and $D'$. From the
definition we have, $V=s_a (U_{\Theta}^{>0})$ 
  \[ V' = s_a( (U_{\Theta}^{>0})^{-1}) = (s_a (U_{\Theta}^{>0}))^{-1} =
    V^{-1}.\]
  This proves that $V'$~does not depend on the choice of~$s_a$ but only on~$V$
  and in turn that $D'$~depends only on~$D$ and not on~$s_a$.
  \end{proof}

The diamond opposite to $(D,a,b)$ will be denoted $(D^\vee,a,b)$.\index{
  $(D^\vee,a,b)$ the diamond opposite to $(D,a,b)$ } One
obviously has $(D^\vee)^\vee=D$. Uniqueness implies equivariance:
\begin{corollary}
  \label{coro:opposite-diamond-equivariant}
  If $D$~is a diamond, and if $g$~belongs to $\Aut_1( \mathfrak{g})$, then the
  opposite of the diamond $g\cdot D$ is $g\cdot D^\vee$: $(
  g\cdot D )^\vee = g\cdot D^\vee$.
\end{corollary}

\begin{lemma}
  \label{lem:opposite-diamond-transverse}
  Let $(D,a,b)$ be a diamond. Then every~$x$ in~$D$ and every~$y$ in~$D^\vee$
  are transverse.
\end{lemma}
\begin{proof}
 By invariance we can assume that $a= \mathfrak{p}_\Theta$, $b=
  \mathfrak{p}_{\Theta}^{\mathrm{opp}}$ and $D= U_{\Theta}^{>0} \cdot
  b$. There are then~$u$ and $v$ in $U_{\Theta}^{>0}$ such that $x= u\cdot b$
  and $y=v^{-1}\cdot b$. Then the pair $(x,y)$ is in the same orbit as $( vu
  \cdot b,b)$. Since $vu$ belongs to $U_{\Theta}^{>0}$, $vu\cdot b$ belongs
  to~$D$ and is transverse to~$b$, thus $x$~is transverse to~$y$.
\end{proof}

The opposition of diamonds reverses inclusion:

\begin{lemma}
  \label{lem:opposite-diamond-reverse-inclusion}
  Let $(D,a,b)$ be a diamond and~$x$ belong to~$D$. Let $(D',a,x)$ be the
  diamond contained in~$D$. Then the opposite diamond $D^{\prime\vee}$
  contains~$D^\vee$.
\end{lemma}
\begin{proof}
  We can assume $a= \mathfrak{p}_\Theta$ and $b=
  \mathfrak{p}_{\Theta}^{\mathrm{opp}}$ and $D= U_{\Theta}^{>0}\cdot
  b$. Let~$u$ be in~$U_{\Theta}^{>0}$ such that $x=u\cdot b$ so that $D' =
  u\cdot D$. One thus has $D^\vee = (U_{\Theta}^{>0})^{-1} \cdot b$ and
  $D^{\prime \vee} = u \cdot D^\vee$, that is $u^{-1} \cdot D^{\prime \vee} =
  D^\vee$. The sought for inclusion is therefore equivalent to $u^{-1}
  (U_{\Theta}^{>0})^{-1} \subset (U_{\Theta}^{>0})^{-1}$ which is a
  consequence of the fact that $U_{\Theta}^{>0}$ is a semigroup.
\end{proof}

The last lemma implies
\begin{corollary}
  \label{coro:opposite-diamond-from-the-exterior}
  The set $D^\vee$ consists of those~$x$
  in~$\mathsf{F}_\Theta$ transverse to~$b$ and such that there is a diamond $(D',a,x)$
  containing~$D$.
\end{corollary}

\subsection{Positive triples in~\texorpdfstring{$\mathsf{F}_\Theta$}{FΘ}}
\label{sec:positive-triples-f_j}

We now use the family of diamonds to define positive triples of flags.

\begin{definition}
  \label{defi:posit-tripl}
  A triple $(f_1,f_2, f_3) \in (\mathsf{F}_\Theta)^3$ is \emph{positive} if
  $f_1$ is transverse to~$f_3$ and if there exists a diamond with
  extremities~$f_1$ and~$f_3$ that contains~$f_2$.
\end{definition}

The next proposition collects the main properties of positive triples.

\begin{proposition}
  \label{prop:posit-tripl}
  \begin{enumerate}[leftmargin=*]
  \item \label{item1:prop:posit-tripl} \emph{[Invariance]} For every
    $(f_1,f_2,f_3)$ and every~$g$ in~$\Aut_1( \mathfrak{g})$, the triple $(f_1,
    f_2, f_3)$ is positive if and only if the triple $(g\cdot f_1,
    g\cdot f_2, g\cdot  f_3)$ is positive.
  \item \label{item1bis:prop:posit-tripl} A triple $(f_i)_{i\in \Z/3\Z }$ is
    positive if and only if, for all $i\neq j$ in $\Z/3\Z$ there exists a
    diamond $D_{i,j}$ with extremities~$f_i$ and~$f_j$ with $D_{j,i} =
    D_{i,j}^{\vee}$ and $f_k$~belongs to~$D_{i,j}$ for all $(i,k,j)$
    cyclically ordered (i.e.\ $j-k=k-i=1$ in $\Z/3\Z$).
  \item   \label{item4:prop:posit-tripl} \emph{[Permutation]} For every
    permutation 
    $\sigma\in S_3$ and every $(f_1, f_2, f_3)$, the triple $(f_1, f_2,
    f_3)$ is positive if and only $(f_{\sigma(1)}, f_{\sigma(2)},
    f_{\sigma(3)})$ is positive.
  \item \label{item2:prop:posit-tripl} A triple is positive if and only if
    it is in the orbit (under $\Aut_1( \mathfrak{g})$) of $(
    \mathfrak{p}_\Theta, u\cdot \mathfrak{p}_{\Theta}^{\mathrm{opp}},
    \mathfrak{p}_{\Theta}^{\mathrm{opp}})$ for some~$u$ in~$U_{\Theta}^{>
      0}$.
  \item   \label{item3:prop:posit-tripl} A triple is positive if and only if
    it is in the orbit (under $\Aut_1( \mathfrak{g})$) of $(
    \mathfrak{p}_\Theta, u^{-1}\cdot \mathfrak{p}_{\Theta}^{\mathrm{opp}},
    \mathfrak{p}_{\Theta}^{\mathrm{opp}})$ for some~$u$ in~$U_{\Theta}^{> 0}$.
  \item   \label{item5:prop:posit-tripl} \emph{[Component]} The set of positive
    triples is a union of connected components of the space of pairwise
    transverse triples, in particular it is  open in $(\mathsf{F}_\Theta)^3$.
  \item   \label{item6:prop:posit-tripl} \emph{[Properness]} The group~$G$ acts
    properly on the space of positive triples. In particular stabilizers are
    compact.
  \end{enumerate}
\end{proposition}

\begin{proof}
  \begin{enumerate}[leftmargin=*]
  \item This property follows from the invariance of the family of diamonds.
  \item If $(f_1, f_2, f_3)$ satisfies these conditions, then $f_2$~belongs to
    $D_{1,3}$, a diamond with extremities~$f_1$ and~$f_3$ and the triple is
    positive.

    Conversely, if the triple $(f_1, f_2, f_3)$ is positive, there exists a
    diamond $D_{1,3}$ with extremities~$f_1$ and~$f_3$ and
    containing~$f_2$. Define $D_{1,2}$ to be the unique diamond with
    extremities~$f_1$ and~$f_2$ contained in~$D_{1,3}$ and $D_{2,3}$ to be the  unique diamond with
    extremities~$f_2$ and~$f_3$ contained in~$D_{1,3}$. For $i>j$ define
    $D_{i,j} = D_{j,i}^{\vee}$. Then the wanted membership properties of
    the~$f_k$ in the~$D_{i,j}$ follow from the choices made and from
    Corollary~\ref{coro:opposite-diamond-from-the-exterior}.
  \item The characterization of the previous point is clearly invariant by
    permutation, hence the result.
  \item This follows from the definition and from the fact that
    $(U_{\Theta}^{>0} \cdot b, a, b)$ is a diamond (again
    $a=\mathfrak{p}_\Theta$ and $b=\mathfrak{p}_{\Theta}^{\mathrm{opp}}$).
  \item This property follows from the previous one, the invariance by the
    transposition $(23)$, and the invariance by $\Aut_1(\mathfrak{g})$.
  \item This is a consequence of the connectedness properties of diamonds.
  \item The map $\mathcal{F} \to \mathsf{F}_{\Theta}^{2*}$ (where\index{$\mathsf{F}_{\Theta}^{2*}$ the space of transverse pairs}
    $\mathsf{F}_{\Theta}^{2*}$ is the space of transverse pairs) is continuous
    and equivariant. The sought for properness is then equivalent to the
    properness of the action of~$L_{\Theta}^{\circ}$ on the semigroup
    $U_{\Theta}^{>0}$; this properness is in turn a consequence of the fact
    that the parametrizations  are
    $L_{\Theta}^{\circ}$-equivariant (Lemma~\ref{lemma:equivarianceFgamma}) and from the already know properness of
    $L_{\Theta}^{\circ}$ on the product of cones
    (Proposition~\ref{prop:homogeneity-Ltheta}).\qedhere
  \end{enumerate}
\end{proof}

\begin{remark}
  In resonance with Remark~\ref{rem:homog-under-group-L}, we can establish
  (when $\Theta \neq \Delta$) that the space of positive triples has two
  connected components when $\sharp \Theta$ is odd and one connected component
  when $\sharp \Theta$ is even.
\end{remark}

Point~(\ref{item5:prop:posit-tripl}) can be used to prove other
characterizations of positive triples, for example
\begin{corollary}
  \label{coro:posit-tripl-charac}
  A triple is positive if and only if it is in the $\Aut_1( \mathfrak{g})$-orbit
  of a triple of the form $(vu\cdot b, v\cdot b, b)$ where $b=
  \mathfrak{p}_{\Theta}^{ \mathrm{opp}}$ and $u,v$ belong to~$U_{\Theta}^{>0}$.
\end{corollary}
\begin{proof}
  Let us prove first that any such triple is positive. By property of the
  semigroup $U_{\Theta}^{>0}$, the flags in $(vu\cdot b, v\cdot b, b)$ are
  pairwise transverse.
  Let  $X$ be the element of~$\mathfrak{a}$ such that
  $\alpha(X) = -1$ for all~$\alpha$ in~$\Delta$ and let $\ell_s = \exp( s X)$
  ($s\in \R$) be the associated $1$-parameter subgroup contained in the
  Cartan subspace. Then the family $\{ (v
  \ell_s u \ell_{s}^{-1} \cdot b, v\cdot b, b)\}_{s \geq 0}$ consists of
  pairwise transverse triples and converges, as $s\to \infty$, to the positive
  triple $(a,v\cdot b,b)$ (where $a=\mathfrak{p}_\Theta$). This implies
  (thanks to point~(\ref{item5:prop:posit-tripl}) of the above proposition)
  that $(vu \cdot b, v\cdot b, b)$ is positive.

  Let us now prove the reverse statement. Any positive triple is in the orbit of
  $(a, w\cdot b, b)$ for some~$w$ in $U_{\Theta}^{>0}$. The diamond with
  extremities~$a$ and~$b$ and containing $w\cdot b$ is our first diamond $D= U_{\Theta}^{>0} \cdot b =
  U_{\Theta}^{ \mathrm{opp}, >0}\cdot a$. For $x\in U_{\Theta}^{\mathrm{opp},
    >0}$, one has $x\cdot (a, w\cdot b, b) = (x\cdot a, x w\cdot b, b)$. The
  element $x\cdot a$ belongs to~$D$ and hence is of the form $r\cdot b$ for
  some~$r$ in~$U_{\Theta}^{>0}$.

  The element $x w \cdot b$ belongs also to~$D$: indeed there is $z\in
  U_{\Theta}^{\mathrm{opp}, >0}$ such that $w\cdot b = z\cdot a$; hence $x w
  \cdot b = (xz)\cdot a$ and $xz$~belongs to~$U_{\Theta}^{\mathrm{opp}, >0}$
  thus $xw \cdot b$ belongs to $U_{\Theta}^{\mathrm{opp}, >0}\cdot a
  =D$. There is therefore $v$~in $U_{\Theta}^{>0}$ such that $xw\cdot b =
  v\cdot b$.

  We finally prove that the element $u= v^{-1} r$ belongs
  to~$U_{\Theta}^{>0}$. 
  For this, note first that $u$ belongs to
  $\Omega_{\Theta}^{\mathrm{opp}}$ since $v^{-1} r \cdot b$ is transverse
  to~$b$; 
  choose $(w_t)_{t\in [0,1]}$ a continuous
  path satisfying $w_0=e$, $w_1 =w$ and $w_t\in U_{\Theta}^{>0}$ for
  all~$t$ in~$(0,1]$ (such a path can be constructed thanks to a
  parametrization $F_\gammab$). The path $(v_t)_{t\in [0,1]}$ in $U_\Theta$
  defined by the equality $v_t \cdot b= x w_t \cdot b$ is continuous and
  satisfies $v_0=e$, $v_1=v$ and $v_{t}^{-1} r$ belongs to
  $\Omega_{\Theta}^{\mathrm{opp}}$ for all~$t$. Hence $u=v^{-1} r= v_{1}^{-1} r$
  and $r= v_{0}^{-1} r$ belongs to the same connected component of $U_\Theta
  \cap \Omega_{\Theta}^{\mathrm{opp}}$. Since $U_{\Theta}^{>0}$ is a connected
  component of that intersection and since $r$~belongs to~$U_{\Theta}^{>0}$,
  we obtain that $u$ belongs to~$U_{\Theta}^{>0}$, as announced.
\end{proof}

\subsubsection{Hermitian Lie groups}
When $G$ is a Hermitian Lie group of tube type, and $\Theta = \{
\alpha_\Theta\}$, the flag variety $\mathsf{F}_\Theta$ coincides with the
Shilov boundary of the symmetric space of~$G$, and a triple is positive in the sense here if and only if it has maximal or minimal Maslov index. 
Note that for Hermitian Lie groups that are not of tube type, the Maslov index
of a triple of points in the Shilov boundary is still defined, and one can
look at triples of maximal (or minimal) Maslov index. In fact, it is shown in
\cite{Burger_Iozzi_Wienhard_ann}, that every maximal triple is contained in
the Shilov boundary of a unique maximal Hermitian symmetric subspace of tube
type. One important difference between the tube type and non tube type
situation is that in the first case, the space of pairwise transverse triples
in the Shilov boundary has several connected components, and the triples of
maximal or minimal Maslov index form connected components, and thus give rise to diamonds, whereas in the non-tube type situation the space of pairwise transverse triples in the Shilov boundary is connected see \cite[Section~3]{BurgerIozziWienhard_TG}, and thus the space of triples of maximal or minimal Maslov index does not form connected components.

\subsection{Positive quadruples}
\label{sec:positive-quadruples}

We investigate here properties of positive quadruples.

\begin{definition}
  \label{defi:positive-quadruples}
  A quadruple $(a,x,b,y)$ in~$\mathsf{F}_\Theta$ is said to be  \emph{positive} if there exists a diamond
  $(D,a,b)$ such that $x$~belongs to~$D$ and $y$~belongs to~$D^\vee$.
\end{definition}

\begin{proposition}
  \label{prop:positive-quadruples}
  \begin{enumerate}[leftmargin=*]
  \item \label{item1:prop:positive-quadruples} Let $(a,x,b,y)$ be a quadruple
    in~$\mathsf{F}_\Theta$ and let~$g$ be in $\Aut_1( \mathfrak{g})$. Then
    $(a,x,b,y)$ is positive if and only if $(g\cdot a,g\cdot x,g\cdot b,g\cdot
    y)$ is positive.
  \item \label{item2:prop:positive-quadruples} A quadruple is positive if and
    only if it is in the $\Aut_1(\mathfrak{g})$-orbit of $(a,u\cdot b, b,
    v^{-1}\cdot b)$ where $a=\mathfrak{p}_\Theta$,
    $b=\mathfrak{p}_{\Theta}^{\mathrm{opp}}$ and $u,v$ belong
    to~$U_{\Theta}^{>0}$.
  \item \label{item3:prop:positive-quadruples} A quadruple is positive if and
    only if it is in the $\Aut_1(\mathfrak{g})$-orbit of $(a, v u\cdot b, v\cdot b, b)$ where $a=\mathfrak{p}_\Theta$,
    $b=\mathfrak{p}_{\Theta}^{\mathrm{opp}}$ and $u,v$ belong
    to~$U_{\Theta}^{>0}$. 
  \item \label{item4:prop:positive-quadruples} Let $(a,x,b,y)$ be a quadruple
    in~$\mathsf{F}_\Theta$.  Then
    $(a,x,b,y)$ is positive if and only if there is a diamond $(D,a,y)$ such
    that $b$~belongs to~$D$ and $x$~belongs to~$D'$, the
    unique diamond with extremities~$a$ and~$b$ contained in~$D$.
  \item \label{item3bis:prop:positive-quadruples} A quadruple is positive if and
    only if it is in the $\Aut_1(\mathfrak{g})$-orbit of $(a, x\cdot a, xy\cdot a, b)$ where $a=\mathfrak{p}_\Theta$,
    $b=\mathfrak{p}_{\Theta}^{\mathrm{opp}}$ and $x,y$ belong
    to~$U_{\Theta}^{\mathrm{opp}, >0}$. 
  \item \label{item4bis:prop:positive-quadruples} Let $(a,x,b,y)$ be a quadruple
    in~$\mathsf{F}_\Theta$.  Then
    $(a,x,b,y)$ is positive if and only if there is a diamond $(D,a,y)$ such
    that $x$~belongs to~$D$ and $b$~belongs to~$D'$, the
    unique diamond with extremities~$x$ and~$y$ contained in~$D$.
  \item \label{item5:prop:positive-quadruples} A quadruple $(f_i)_{i\in \Z/4\Z }$ is
    positive if and only if, for all $i\neq j$ in $\Z/4\Z$ there exists a
    diamond $D_{i,j}$ with extremities~$f_i$ and~$f_j$ with $D_{j,i} =
    D_{i,j}^{\vee}$ and $f_k$~belongs to~$D_{i,j}$ for all $(i,k,j)$
    cyclically ordered (i.e.\ $(k-i, j-k)$ is $(1,1)$, $(1,2)$, or
    $(2,1)$).
  \item \label{item6:prop:positive-quadruples} The space of positive
    quadruples is invariant under the dihedral group $D_4\subset S_4$ (i.e.\
    the group generated by the $4$-cycle $(1,2,3,4)$ and the double
    transposition $(1,4)(2,3)$). Namely for $\sigma\in D_4$ and a quadruple
    $(f_1,f_2,f_3,f_4)$ then $(f_1,f_2,f_3, f_4)$ is positive if and only if
    $(f_{\sigma(1)},f_{\sigma(2)},f_{\sigma(3)},f_{\sigma(4)})$ is positive.
  \end{enumerate}
\end{proposition}

\begin{proof}
  \begin{enumerate}[leftmargin=*]
  \item This follows from the invariance of the family of diamonds and the
    equivariance of the map $D\mapsto D^\vee$
    (Corollary~\ref{coro:opposite-diamond-equivariant}).
  \item This follows from the definition and from the fact that, up to the
    action of $\Aut_1( \mathfrak{g})$, we can assume $a=\mathfrak{p}_\Theta$,
    $b=\mathfrak{p}_{\Theta}^{\mathrm{opp}}$, and $D= U_{\Theta}^{>0}\cdot
    b$.
  \item Applying the element $v\in U_{\Theta}^{>0}$ to the quadruple $(a,
    u\cdot b, b, v^{-1} b)$ gives the wanted result.
  \item This is the previous characterization stated in terms of diamonds.
  \item This follows by the same arguments as for ~(\ref{item3:prop:positive-quadruples}) above,  using this
    time the equality $U_{\Theta}^{\mathrm{opp}, >0}\cdot a = U_{\Theta}^{>0}
    \cdot b$.
  \item This is the previous characterization stated in terms of diamonds.
  \item If a quadruple satisfies the stated condition, then the condition in
     Definition~\ref{defi:positive-quadruples} is obviously satisfied and the quadruple is positive.

    Conversely, suppose that $(f_1, f_2, f_3, f_4)$ is positive. By
    definition, there is a diamond $D_{1,3}$ with extremities~$f_1$ and~$f_3$,
    containing~$f_2$ and such that $f_4$ belongs to $D_{3,1} \coloneqq
    D_{1,3}^{\vee}$. We define $D_{1,2}$ to be the diamond contained in
    $D_{1,3}$ and with extremities~$f_1$ and~$f_2$ and similarly $D_{2,3}$ is
    the diamond contained in $D_{1,3}$ with extremities~$f_2$ and~$f_3$.
    The characterization of point~(\ref{item4:prop:positive-quadruples}) gives also a diamond~$D_{1,4}$ with
    extremities~$f_1$ and~$f_4$ and containing~$f_2$ and~$f_3$. This diamond
    can be used to define the diamond $D_{2,4}$ and $D_{3,4}$. The other
    diamonds are defined thanks to the requirement $D_{j,i} =
    D_{i,j}^{\vee}$. All the wanted memberships are satisfied by construction
    and by Corollary~\ref{coro:opposite-diamond-from-the-exterior} except possibly that
    $f_3$~belongs to~$D_{2,4}$. However this membership follows from the
    characterization established in
    point~(\ref{item4bis:prop:positive-quadruples}).
  \item The  invariance under the dihedral group follows from the previous point.\qedhere
  \end{enumerate}
\end{proof}

Based on point~(\ref{item3:prop:positive-quadruples}) above one has
\begin{lemma}
  \label{lem:positive-quadruples-with-fixed-triple}
  Let $a$ and $b$ be the elements $\mathfrak{p}_\Theta$ and
  $\mathfrak{p}_{\Theta}^{\mathrm{opp}}$ of~$\mathsf{F}_\Theta$ and let~$v$ be
  in~$U_{\Theta}^{>0}$ and $x=v\cdot b$. For $y\in \mathsf{F}_\Theta$, the
  quadruple $(a,y,x,b)$ is positive if and only if there is~$u$
  in~$U_{\Theta}^{>0}$ such that $y=vu\cdot b$.
\end{lemma}

Diamonds associated with positive quadruples are properly contained one in another:
\begin{lemma}
  \label{lem:positive-quadruples-inclusion-diamond}
  Let $(a,x,b,y)$ be a positive quadruple and let $(D,a,y)$ be the diamond
  containing~$x$ and~$b$ and let $(D',x,b)$ be the diamond contained
  in~$D$. Then the closure~$\overline{D}{}'$ is contained in~$D$.
\end{lemma}

\begin{proof}
  It will be enough to prove that this closure is contained in
  the intersection $\mathcal{O}_a \cap \mathcal{O}_y$. By symmetry, we need only to prove the
  inclusion into~$\mathcal{O}_y$. We can assume that $a=\mathfrak{p}_\Theta$,
  $y=\mathfrak{p}_{\Theta}^{\mathrm{opp}}$ and $b=v\cdot y$, $x=vu \cdot y$
  with $u,v$ in~$U_{\Theta}^{>0}$. Any point in $\overline{D}{}'$ is of the
  form $vw \cdot y$ with $w\in U_{\Theta}^{\geq 0}$ (and $w^{-1} u$ must also
  belong to $U_{\Theta}^{\geq 0}$). In particular $vw$ belongs to
  $U_{\Theta}^{>0}$ (Corollary~\ref{cor:semigroupUtheta}.(\ref{item1:cor:semigroupUtheta})) and $vw\cdot y$ is transverse to~$y$ as wanted.
\end{proof}

Similarly to what has been established for triples we note
\begin{lemma}
  The space of positive quadruples is a union of connected components of the
  space of pairwise transverse quadruples.
\end{lemma}

\subsection{Positive tuples}
\label{sec:positive-tuples}

Let~$n$ be an integer~$\geq 3$. We introduce here the positive $n$-tuples in
generalizing point~(\ref{item1bis:prop:posit-tripl}) of
Proposition~\ref{prop:posit-tripl} or
point~(\ref{item5:prop:positive-quadruples}) of
Proposition~\ref{prop:positive-quadruples}. Positive tuples of flags have been
also extensively studied in\cite{GLW}.

\begin{definition}
  \label{defi:positive-tuples}
  A $n$-tuple $(f_i)_{i\in \Z/n\Z }$ of elements of~$\mathsf{F}_\Theta$ is
    \emph{positive} if, for all $i\neq j$ in $\Z/n\Z$, there exists a
    diamond $D_{i,j}$ with extremities~$f_i$ and~$f_j$ with $D_{j,i} =
    D_{i,j}^{\vee}$ and $f_k$~belongs to~$D_{i,j}$ for all $(i,k,j)$
    cyclically ordered (i.e.\ there are integers $a$, $b$, and~$c$
    representing $i$, $j$, and~$k$ respectively and such that
    $a<b<c<a+n$).
\end{definition}

Note that the diamonds are uniquely determined by the properties that
$f_{i+1}$ belongs to $D_{i,j}$ (when $j\notin \{ i,i+1\}$) and $D_{i,i+1} =
D_{i+1,i}^{\vee}$.

From the definition, it is obvious that
\begin{lemma}
  \label{lem:positive-tuples-first-prop}
  \begin{enumerate}[leftmargin=*]
  \item The set of positive $n$-tuples is invariant by the dihedral
    group~$D_n$ (i.e.\ by cyclic permutations and by the permutation $i
    \leftrightarrow n+1-i$).
  \item Any subconfiguration of a positive $n$-tuple consisting of $m$ elements is a positive $m$-tuple.
  \end{enumerate}
\end{lemma}

Here is one (among many) characterizations of positive tuples.

\begin{lemma}
  \label{lem:positive-tuples-charact}
  A $n$-tuple is positive if and only if it is in the
  $\Aut_1(\mathfrak{g})$-orbit of
  \[ (a, u_{1}\cdots u_{n-3} u_{n-2} \cdot b, u_{1}\cdots u_{n-3} \cdot b, \dots,
    u_1 \cdot b, b),\]
  where $a=\mathfrak{p}_\Theta$, $b=\mathfrak{p}_{\Theta}^{\mathrm{opp}}$,
  and, for all~$i$ in $\llbracket 1, n-2\rrbracket$, $u_i$~belongs to~$U_{\Theta}^{>0}$.
\end{lemma}

\begin{proof}
  Let $(f_1, \dots, f_n)$ be a positive tuple. Up to the $\Aut_1( \mathfrak{g})$
  action, we can assume that $f_1=a$, $f_n=b$ and $f_{n-1}= u_1\cdot
  b$. Repeated applications of Lemma~\ref{lem:positive-quadruples-with-fixed-triple}
  give the sequence $(u_2, \dots, u_{n-2})$.

  Conversely let $u_1, \dots, u_{n-2}$ be in~$U_{\Theta}^{>0}$ and let $(f_1,
  \dots, f_n)$ be
  \[ (a, \dots, u_1 \cdots u_i \cdot b, \dots, b).\]For all $i<
  j$ in $\llbracket 1, n\rrbracket$, let $D_{i,j}$ be the unique diamond with
  extremities~$f_i$ and~$f_j$ contained in $D=U_{\Theta}^{>0}\cdot b$ and
  $D_{j,i} = D_{i,j}^{\vee}$. The facts that, for
  $i<k<j$ in $\llbracket 1, n\rrbracket$, $f_k$ belongs to~$D_{i,j}$ come from
  Corollary~\ref{coro:posit-tripl-charac}, the other cases come from
  Corollary~\ref{coro:opposite-diamond-from-the-exterior}.
\end{proof}

Restated in terms of diamonds, the lemma says it is enough to check that some sub-quadruples are positive:

\begin{proposition}
  \label{prop:positive-tuples-charact}
  Let $(f_1, \dots, f_n)$ be in $(\mathsf{F}_\Theta)^n$. Then $(f_1, \dots,
  f_n)$ is positive if and only if, for every $i=2, \dots, n-2$, the quadruple
  $(f_1, f_i, f_{i+1}, f_n)$ is positive.
\end{proposition}

We also note that
\begin{lemma}
  The set of positive $n$-tuples is a union of connected components of the
  space of pairwise transverse $n$-tuples.
\end{lemma}

\subsection{Compatibility of positive structures}
\label{sec:comp-posit-struct}
As observed before (see Remarks~\ref{rem:split-and-2-str} and Examples~\ref{exam:nonnegativeSplitHermitian}), if $G$~is a split
real group whose Dynkin diagram has a double arrow, then $G$~carries two
positive structures, one with respect to~$\Delta$ and one with respect to a
proper subset $\Theta\subset \Delta$. Both positive structures give rise to a
notion of positive triples in the flag varieties~$\mathsf{F}_\Delta$,
respectively~$\mathsf{F}_\Theta$. Here we show that these two different
notions are compatible.  \index{$\mathsf{F}_\Delta$ the
  complete flag variety} This
concerns the following groups up to isogeny:
\begin{itemize}
\item the symplectic Lie group $\Sp(2n, \R)$, $n\geq 2$;
\item the orthogonal groups $\SO(n,n+1)$, $n\geq 3$ ;
\item the real split Lie group of type~$\lietype{F}_4$.
\end{itemize}
Note that in this case, as explained in Section~\ref{sec:putting-an-hand}, the choice of root vectors $X_\alpha$ for all $\alpha\in \Delta$ determine the cones ${c}_\alpha$ for $\alpha \in \Theta$.  Let $U_{\Delta}^{>0}$, $U_{\Delta}^{\mathrm{opp}, >0}$, and $U_{\Theta}^{>0}$, $U_{\Theta}^{\mathrm{opp}, >0}$ be the corresponding semigroups. 
The natural projection $\pi\colon \mathsf{F}_\Delta \to
\mathsf{F}_\Theta$ is $\Aut_1(\mathfrak{g})$-equivariant and $\pi( \mathfrak{p}_\Delta) = \mathfrak{p}_\Theta$ and $\pi(\mathfrak{p}_\Delta^{\mathrm{opp}}) = \mathfrak{p}_{\Theta}^{\mathrm{opp}}$. This projection behaves well with respect to the notion of positivity  introduced:
\begin{proposition}\label{prop:compatible}
  Let $(a,x,b)$ be a positive triple in~$\mathsf{F}_\Delta$. Then
  $( \pi(a), \pi(x), \pi(b))$ is a positive triple
  in~$\mathsf{F}_\Theta$.
\end{proposition}
\begin{proof}
  The positive semigroup~$U^{>0}_{\Delta}$ is determined here by elements~$X_\alpha$
  generating~$\mathfrak{g}_\alpha$ for~$\alpha$ in~$\Delta$. These elements
  can be used determine the semigroup~$U_{\Theta}^{>0}$, i.e.\ we choose the cones $c_\alpha\subset \mathfrak{u}_\alpha$ ($\alpha\in
  \Theta$) by the convention  that $X_\alpha\in c_\alpha$ (Remark~\ref{remark:normalize_cones}).
  
  First observe that, by the transitivity of the action of~$G$ on the space of
  transverse pairs and since $(\pi(\mathfrak{p}_\Delta)
  ,\pi(\mathfrak{p}_\Delta^{\mathrm{opp}})) = ( \mathfrak{p}_{\Theta},
  \mathfrak{p}_{\Theta}^{\mathrm{opp}})$, the image by~$\pi$ of every
  transverse pairs in~$\mathsf{F}_\Delta$ is a transverse pair
  in~$\mathsf{F}_\Theta$.

  Let now $(a,x,b)$ be a positive triple in~$\mathsf{F}_\Delta$.
  By the transitivity of the action of~$\Aut_1(\mathfrak{g})$, we can assume
  that $a= \mathfrak{p}_\Delta$, $b=\mathfrak{p}_\Delta^{\mathrm{opp}}$ and
  $x=u\cdot b$ for $u\in U^{>0}_{\Delta}$. Hence $\pi(x)= u\cdot \pi(b)$ and we need to
  show that $u\cdot \pi(b) \in U_{\Theta}^{>0}\cdot \pi(b)$. Since $\pi(x)$ is
  transverse to~$\pi(b)$ it is enough to prove that $u\cdot \pi(b) \in
   U_{\Theta}^{\geq 0}\cdot \pi(b)$.

   For this, we will  prove the inclusions $u\cdot   (U_{\Theta}^{\geq 0}\cdot \pi(b)) \subset U_{\Theta}^{\geq 0}\cdot
  \pi(b)$ for~$u$ in the generating set $\bigcup_{\alpha\in
    \Delta}\exp(\R_{\geq 0}X_\alpha)$ of~$U_{\Delta}^{\geq 0}$. Let thus
  $\alpha$ be in~$\Delta$, $t\geq 0$, and $u=\exp(tX_\alpha)$.
    When
  $\alpha$~belongs to~$\Theta$, then $u$ belongs to
  $U_{\Theta}^{\geq 0}$ and the inclusion comes from the fact
  that $U_{\Theta}^{\geq 0}$ is a semigroup. When $\alpha$ belongs
  to~$\Delta\setminus \Theta$, then $u$ belongs to~$L_{\Theta}^{\circ}$
  and the inclusion comes from the fact that $L_{\Theta}^{\circ}$
  normalizes~$U_{\Theta}^{\geq 0}$ and fixes~$\pi(b)$.
\end{proof}

\subsection{Positive maps}
\label{sec:positive-maps}

We extend the notion of positive tuples to the notion of positive maps. For
this let us denote by~$\Lambda$ a set equipped with a cyclic
ordering,\index{$\Lambda$ a set equipped with a cyclic ordering} i.e.\ there
is a notion of ``$y$ being between~$x$ and~$z$'' satisfying a number of
natural properties that are a bit similar to the properties requested for
diamonds, cf.\ \cite{Huntington} for example (typically $\Lambda$ is a subset
of the circle). This means that there is a subset $\Lambda^{3+}$
of~$\Lambda^3$ consisting of cyclically ordered triples. There is therefore a
notion of cyclically ordered $n$-tuples in~$\Lambda$ for every $n\geq 3$.

\begin{definition}\label{defi:positive-maps}
  A map $f\colon \Lambda \to \mathsf{F}_\Theta$ is said to be \emph{positive} if the
  image by~$f$ of every cyclically ordered $n$-tuple is a positive $n$-tuple.
\end{definition}

Proposition~\ref{prop:positive-tuples-charact} implies 
\begin{lemma}
  \label{lemma:positive-maps-char}
  Suppose $\Lambda$ is $\geq 4$. A map~$f$ is positive if and only if it sends every cyclically ordered
  quadruple to a positive quadruple.
\end{lemma}

The $\Theta$-principal embedding
$\pi_\Theta\colon \mathfrak{sl}_{2} \rightarrow \mathfrak{g}$ introduced in
Section~\ref{sec:sl2} induces an homomorphism
$\pi_\Theta\colon \SL_2(\R) \to G$ and hence an action of~$\SL_2(\R)$ on the
flag variety~$\mathsf{F}_\Theta$.
 \begin{lemma}
   The stabilizer of $\mathfrak{p}_\Theta\in \mathsf{F}_\Theta$ in $\SL_2(\R)$
   is the standard Borel subgroup~$B$ (whose Lie algebra is $\mathfrak{b} =\R E\oplus \R D$).
 \end{lemma}
 \begin{proof}
   Since $\pi_\Theta(E)$ and $\pi_\Theta(D)$ belongs to~$\mathfrak{p}_\Theta$,
   it is clear that the Lie algebra of this stabilizer contains $\mathfrak{b}
   $. Since the $\SL_2(\R)$-orbit of $\mathfrak{p}_\Theta$ is not
   trivial, the Lie algebra of the stabilizer must be equal
   to~$\mathfrak{b}$. Hence the stabilizer is either~$B$ or its neutral
   component~$B^\circ$. We thus need to prove that $-\id = \exp( \pi(E-F))$
   belongs to the stabilizer.
   
   The element $\dot{s}=\exp( \frac{\pi}{2}(E-F))$ of
   $\SL_2(\R)$ is sent by~$\pi_\Theta$ to (a representative of) the longest
   length element of~$H_\Theta$ (Lemma~\ref{lem:principal-weyl-group}). One
   has hence (cf.\ Lemma~\ref{lem:WHTheta-iso-WTheta} and
   Proposition~\ref{prop:longestelementWTheta}) that $\pi_\Theta(\dot{s}) \cdot
   \mathfrak{p}_{\Theta} = w_\Delta \cdot
   \mathfrak{p}_{\Theta} =\mathfrak{p}^{\mathrm{opp}}_{\Theta} $ and  $\pi_\Theta(\dot{s}) \cdot
   \mathfrak{p}^{\mathrm{opp}}_{\Theta} =\mathfrak{p}_{\Theta} $ which implies
   the sought for equality: $\pi_\Theta(\dot{s}^2) \cdot
   \mathfrak{p}_{\Theta} =\mathfrak{p}_{\Theta}$.
 \end{proof}

 From this, identifying $\mathbb{P}(\R^2)$ with
 $\SL_2(\R)/B$, we obtain an equivariant embedding $\mathbb{P}(\R^2)\to \mathsf{F}_\Theta$.
As a direct consequence of Theorem~\ref{thm:nilpotent} (and
since the cyclically ordered tuples of $\mathbb{P}(\R^2)$ are well understood) we obtain the following
 \begin{proposition}\label{prop:positive_circle} \cite[Conjecture~4.12]{GW_ECM}
 The embedding $\mathbb{P}(\R^2)$ into $\mathsf{F}_\Theta$ induced from the Lie algebra morphism $\pi_\Theta\colon \mathfrak{sl}_2(\R) \rightarrow \mathfrak{g}$ is a positive circle. 
 \end{proposition}
 
\begin{remark}
When $ \Theta \neq \Delta$, there are sometimes more than one embedding of
$\mathfrak{sl}_{2}$ that give rise to a positive circle in
$\mathsf{F}_\Theta$. For example when $G$~is a split real group, the principal
$3$-dimensional subalgebra of~$\mathfrak{g}$ by Proposition~\ref{prop:compatible} also determines a positive circle in~$\mathsf{F}_\Theta$ for any~$\Theta$ such that $G$~admits a $\Theta$-positive
structure. For Hermitian groups of tube type, any tight embedding of $\mathfrak{sl}_{2}$ into $\mathfrak{g}$ gives rise to a positive circle, see \cite{BurgerIozziWienhard_tight}.
\end{remark}

\appendix

\section{Longest length element in $\lietype{B}_{p+1}$}\label{app:bn}
In this first appendix we determine a reduced expression of the longest length
element in the Weyl group associated to a root system of type~$\lietype{B}_{p+1}$; it
is also the longest length
element for~$\lietype{C}_{p+1}$ since these  two Coxeter groups coincide. 

This is the type of the system of restricted roots of the groups $\SO(p+1,p+k)$
($p>0$, $k>1$). A choice of quadratic form which makes the calculations a little
easier is given by the matrix (cf.\ Section~\ref{sec:orthogonal}) 
 $Q = \begin{pmatrix} 0& 0& K\\ 0 &J &0 \\ {}^t\! K&0& 0 \end{pmatrix}$, where
 $K= \begin{pmatrix}  & & &    (-1)^p \\
         & &    \iddots & \\
   & 1 & &\\
   -1& & & \end{pmatrix}$, and
 $J = \begin{pmatrix} 0&0& 1\\ 0& -\id_{k-1} &0\\ 1 & 0& 0 \end{pmatrix}$.

 With this choice, a Cartan subspace~$\mathfrak{a}$ of $\so(p+1,p+k)$ is its
intersection with the space of diagonal matrices. A natural basis of~$\mathfrak{a}^*$
(with respect to this particular matrix realization of the group) are the
$e_i$ mapping a diagonal matrix to its $i$-th diagonal
entry ($i$ in $\llbracket 1, p+1\rrbracket$). We will use this to identify $\mathfrak{a}$ with $\R^{p+1}$ and describe the Weyl
group and its elements in $\GL_{p+1}(\R)$.

The roots are \[\{ \pm e_i\}_{i\in \llbracket 1, p+1\rrbracket}\cup \{ \pm e_j \pm e_k\}_{j,k \in \llbracket 1,p+1\rrbracket ,
  j<k}.\] For the lexicographic order the positive roots are
\[\{ e_i\}_{i\in\llbracket 1, p+1\rrbracket }\cup\{ e_j \pm e_k\}_{j,k \in \llbracket 1,p+1\rrbracket , j<k}\] and the simple roots
are $\alpha_1 = e_1 - e_2$, $\alpha_2 = e_2 - e_3$, \dots,
$\alpha_{p} = e_{p} - e_{p+1}$, $\alpha_{p+1} = e_{p+1}$.

The Weyl group~$W$ naturally identifies with $(\Z/2\Z)^{p+1} \rtimes S_{p+1}$ acting
on $\R^{p+1}$ by permuting the coordinates and changing their signs. Its
generators $s_1, \dots, s_{p+1}$ associated with $\alpha_1, \dots, \alpha_{p+1}$
are the transformations:
\begin{align*}
  s_1\colon  (x_1, x_2, \dots, x_{p+1})& \longmapsto (x_2, x_1,
         \dots, x_{p+1}),\\
  s_2\colon (x_1, x_2, x_3, \dots, x_{p+1})& \longmapsto (x_1, x_3, x_2, \dots,
         x_{p+1}),\\
       &\dots,\\
  s_{p}\colon (x_1,  \dots, x_{p}, x_{p+1})& \longmapsto (x_1, \dots, x_{p+1},
             x_{p})\\
  \text{and } s_{p+1}\colon (x_1, x_2, \dots, x_{p+1})& \longmapsto (x_1, x_2, \dots,
                     -x_{p+1}).
\end{align*}
Furthermore the longest length element of $W$ is $-\id_{p+1}$ as it exchanges
the Weyl chamber with its opposite.

Let $\Theta=\{ \alpha_1, \dots,\alpha_p\}$. With the notation of
Section~\ref{sec:special-root}, 
$\alpha_\Theta = \alpha_{p}$, the longest length element in
$W_{\Delta \smallsetminus \Theta} = \langle s_{p+1}\rangle$ is $s_{p+1}$ itself and
the subgroup of~$W$ generated by $s_{p}$ and $s_{p+1}$ is isomorphic to
$(\ZZ/2\ZZ)^2 \rtimes S_2$ (the Weyl group of $\SO(2,2+k)$). Its longest
length element is $s_p s_{p+1} s_p s_{p+1}  $ so that the
element $\sigma_{p} = \sigma_\Theta$ is $s_{p} s_{p+1} s_{p}$. Seen as an
element of $\GL_{p+1}(\R)$, $\sigma_{p}$ is the transformation
$(x_1, \dots, x_{p}, x_{p+1}) \mapsto (x_1, \dots, -x_{p}, x_{p+1})$

Recall that the group $W(\Theta)$ is the subgroup of $W$ generated by $s_1,
\dots, s_{p-1}$ and $\sigma_{p}$. As all these generators fix the last
coordinate, we can identify $W(\Theta)$ with a subgroup of
$\GL_p(\R)$. With this identification in mind and the above description,
it is apparent that this subgroup is the Weyl group of type $\lietype{B}_{p}$ and its
longest length element 
is, in this geometric realization, $-\id_p$.

Turning back to subgroups of $\GL_{p+1}(\R)$, the longest length element of $W(\Theta)$
is the transformation
\[(x_1, \dots, x_{p}, x_{p+1}) \mapsto (-x_1, \dots, -x_{p}, x_{p+1}).\] Since the
longest length element of $W$ is $-\id_{p+1}$ and the longest length element of
$W_{\Delta \smallsetminus \Theta}$ is
$s_{p+1}$ one gets
as well that $w_{\mathrm{max}}^\Theta$ is
\[(x_1, \dots, x_{p}, x_{p+1}) \mapsto (-x_1, \dots, -x_{p}, x_{p+1})\]
establishing, for the type~$\lietype{B}_{p+1}$, the fact that $w_{\max}^{\Theta}$
belongs to~$W(\Theta)$
(Proposition~\ref{prop:longestelementWTheta}).

Written as products of generators, these elements are: (we use the notation
$x^y = y^{-1}xy$ so that $(x^y)^z = x^{yz}$ and
$\sigma_{p}= s_{p+1}^{s_{s}}$)
\begin{align*}
  w_\Delta & = s_{p+1}^{s_{p} \cdots s_1} s_{p+1}^{s_{p} \cdots s_2} \cdots
             s_{p+1}^{s_{p}} \cdots s_{p+1}\\
  w_{\max}^{\Theta} & = s_{p+1}^{s_{p} \cdots s_1} s_{p+1}^{s_{p} \cdots s_2}
                      \cdots  s_{p+1}^{s_{p}}\\
           & = \sigma_{p}^{s_{p-1} \cdots s_1} \sigma_{p}^{s_{p-1}
             \cdots s_2} \cdots \sigma_{p}^{s_{p-1}}. 
             \sigma_{p},
\end{align*}
The above equalities are  established noting that
$s_{p+1}^{ s_{p} \cdots s_{k}}$ is
\[(x_1, \dots, x_k, \dots, x_{p+1}) \mapsto (x_1, \dots, -x_k,\dots, x_{p+1}).\]

One can also verify that the above decompositions are reduced: for example,
the length of $w_\Delta$ is the dimension of the complete flag variety for the
split group $\SO(p+1,p+2)$ and is thus equal to $(p+1)^2$; this number
coincides with the
length of the above product.

\section{Longest length element in $\lietype{F}_4$}\label{app:f4}
In this section we determine a reduced expression of the longest length
element in the Weyl group associated to a root system of type~$\lietype{F}_4$. 

The root system~$\lietype{F}_4$ is intimately related with the lattice~$\Lambda$
of~$\R^4$ generated by $\Z^4$ and the element $\frac{1}{2}(1,1,1,1)$. An
alternative description of~$\Lambda$ is the set of elements in $\R^4$ all of
whose coordinates have the same remainder, $0$ or $1/2$, modulo~$1$.

The elements of $\lietype{F}_4$ are the elements of $\Lambda$ whose Euclidean norms
are~$1$ or~$\sqrt{2}$. They can be explicitely listed: let
$(e_i)_{i \in \llbracket 1, 4\rrbracket }$ be the canonical basis of~$\R^4$, then
\[ \lietype{F}_4=\bigl\{ \pm e_i\bigr\}_{i \in \llbracket 1, 4\rrbracket} \cup \bigl\{ \pm e_k\pm
  e_\ell\bigr\}_{k,\ell \in \llbracket 1, 4\rrbracket, k<\ell} \cup \Bigl\{ \frac{1}{2}(\pm e_1 \pm e_2 \pm e_3
  \pm e_4)\Bigr\}.\] The positive roots
are (using colexicographic order)
\[ \bigl\{ e_i\bigr\}_{i \in \llbracket 1, 4\rrbracket}\cup \bigl\{\pm e_k+
  e_\ell\bigr\}_{k,\ell \in \llbracket 1, 4\rrbracket, k<\ell}\cup \Bigl\{ \frac{1}{2}(\pm e_1 \pm e_2 \pm e_3
  + e_4)\Bigr\},\] and the simple roots
are
\[ \alpha_1 = -e_2+e_3, \ \alpha_2 = -e_1+e_2, \alpha_3=e_1, \alpha_4 =
  \frac{1}{2}(-e_1-e_2-e_3+e_4).\]

The Weyl group is the subgroup of $\GL_4(\R)$ generated by the symmetries
$s_1, \dots, s_4$ associated with $\alpha_1, \dots, \alpha_4$. The
transformation $s_i$ is
$x \mapsto x- 2\frac{ \langle \alpha_i, x\rangle}{\langle \alpha_i, \alpha_i
  \rangle} \alpha_i$. In matrix coordinates
\[ s_4 = \frac{1}{2}
  \left(\begin{array}{rrrr}
          1 & -1 & -1 & 1\\
          -1 & 1 & -1 & 1\\
          -1 & -1 & 1 & 1\\
          1 & 1 & 1 & 1
        \end{array}\right)\ 
      .
\]
The element~$s_3$ is $(x_1, x_2, x_3, x_4) \mapsto (-x_1, x_2, x_3,
x_4)$. Finally $s_2$ and $s_1$ are respectively
$(x_1, x_2, x_3, x_4) \mapsto (x_2, x_1, x_3, x_4)$ and
$(x_1, x_2, x_3, x_4) \mapsto (x_1, x_3, x_2, x_4)$.
    
The relevant subset of the simple roots in this situtation is
$\Theta=\{ \alpha_1, \alpha_2\}$ and the root $\alpha_\Theta$ is $\alpha_2$.
The subgroup $W_{\Delta \smallsetminus \Theta}$ is generated by~$s_3$ and~$s_4$ and is isomorphic to $S_3$, its longest length element is~$s_3 s_4 s_3 $.

The subgroup generated by the symmetries indexed by $\Delta\smallsetminus \Theta$ and~$\alpha_\Theta$ is the group generated by $\{ s_2, s_3, s_4\}$ and is
isomorphic to the Weyl group~$\lietype{C}_3$, which is the same as the Weyl group $\lietype{B}_3$ (with the reindexation $1\mapsto 4,
2\mapsto 3, 3\mapsto 2$ with respect to the
previous appendix). Its longest length element is
$s_{2} s_{2}^{s_3} s_{2}^{s_3 s_4} = s_{2} s_3 s_{2} {s_3} s_4 s_3 s_{2} {s_3
  s_4}$. Hence the element $\sigma_2=\sigma_\Theta$ is:
\begin{align*}
  \sigma_2   &= s_{2} s_3 s_{2} {s_3} s_4 s_3 s_{2} (s_3 s_4 s_4 s_3) s_4\\
           &= s_{2} s_3 s_{2} {s_3} s_4 s_3 (s_{2} s_4)\\
           &= s_{2} s_3 s_{2} ({s_3} s_4 s_3 s_{4}) s_2\\
           &= s_{2} s_3 s_{2} s_4 s_3 s_2
           = s_{2} s_3 s_{4} s_2 s_3 s_2.
\end{align*}
And the last two expressions are reduced (this 
can be deduced
from the
fact below that $s_{2} s_3 s_{4} s_2 s_3 s_2$ is a subword of a reduced
expression of the longest length element).

One can calculate $\sigma_2$ in $\GL_4(\R)$:
\[ \sigma_2 = \frac{1}{2} \left(
    \begin{array}{rrrr}
      -1&-1&-1&1\\
      -1&-1&1&-1\\
      -1&1&1&1\\
      1&-1&1&1
    \end{array}
  \right),\] as well as (recall that $\sigma_1=s_1$)
\[ \sigma_1 \sigma_2 = \frac{1}{2} \left(
    \begin{array}{rrrr}
      -1&-1&-1&1\\
      -1&1&1&1\\
      -1&-1&1&-1\\
      1&-1&1&1
    \end{array}
  \right).\]
Its square is
\[ (\sigma_1 \sigma_2)^2 = \left(
    \begin{array}{rrrr}
      1&0&0&0\\
      0&0&1&0\\
      0&0&0&-1\\
      0&-1&0&0
    \end{array}
  \right),\]
showing that $\sigma_1 \sigma_2$ is of order~$6$ and that $W(\Theta)$ is of type~$\lietype{G}_2$.
                       
Finally, the following holds
\begin{multline*}
  s_3 s_4 s_3 = \frac{1}{2} \left(\begin{array}{rrrr}
                                    1 & 1 & 1 & -1\\
                                    1 & 1 & -1 & 1\\
                                    1 & -1 & 1 & 1\\
                                    -1 & 1 & 1 & 1
                                  \end{array}\right)\ \text{ and}\\
                                (\sigma_1 \sigma_2)^3 = \frac{1}{2} \left(
                                  \begin{array}{rrrr}
                                    -1&-1&-1&1\\
                                    -1&-1&1&-1\\
                                    -1&1&-1&-1\\
                                    1&-1&-1&-1
                                  \end{array}
                                \right) = -(s_3 s_4 s_3)\ .
                              \end{multline*}
We thus have
\begin{align*}
  (\sigma_1 \sigma_2)^3 s_3 s_4 s_3 & =(\sigma_2 \sigma_1)^3 s_3 s_4 s_3 =
                                      -\id_4\\
                                    &= s_2 s_3 s_2 s_4 s_3 s_2 s_1 \cdot  s_2 s_3 s_2 s_4 s_3 s_2 s_1 \cdot s_2
                                      s_3 s_2 s_4 s_3 s_2 s_1 \cdot s_3 s_4 s_3
\end{align*}
which is the longest length element of $\lietype{F}_4$ (since it sends the Weyl chamber to its
opposite) and a reduced decomposition of it (since the length of this
decomposition is equal to 24, the number of positive roots). This shows  Proposition~\ref{prop:longestelementWTheta} in this case too.


\bibliographystyle{amsalpha}
\bibliography{GeneralizingLusztigPositivity-revision}

\end{document}